\documentclass[11pt,letterpaper]{article}

\usepackage{euscript,amsfonts,amssymb,amsmath,amscd,amsthm,enumerate,hyperref}

\usepackage{tabularx}
\usepackage{caption}

\usepackage{graphicx}
\usepackage{subcaption}
\usepackage{tikz,tikz-cd}

\usetikzlibrary{decorations.pathreplacing}
\usetikzlibrary{arrows}

\usepackage{color}
\usepackage{stmaryrd}

\usepackage[margin=1in]{geometry}

\usepackage{bbm}
\usepackage{cleveref}
\usepackage{dsfont}
\usepackage{mathrsfs}

\newtheorem{theorem}{Theorem}[section]

\newtheorem*{theorem*}{Theorem}
\newtheorem{lemma}[theorem]{Lemma}

\newtheorem{definition}[theorem]{Definition}

\newtheorem{prop}[theorem]{Proposition}

\newtheorem{claim}[theorem]{Claim}
\newtheorem{corollary}[theorem]{Corollary}

\theoremstyle{remark}
\newtheorem{remark}[theorem]{Remark}
\newtheorem{example}[theorem]{Example}

\newcommand{\eps}{\varepsilon}
\newcommand{\E}{\mathbb E}

\newcommand{\B}{\mathcal B}

\newcommand{\sB}{\mathscr B}

\newcommand{\sO}{\mathscr O}

\newcommand{\sF}{\mathscr F}

\newcommand{\cE}{\mathcal E}

\newcommand{\cH}{\mathcal H}
\newcommand{\cS}{\mathcal S}
\newcommand{\sH}{\mathscr{H}}

\newcommand{\sL}{\mathscr L}

\newcommand{\sK}{\mathscr K}

\newcommand{\D}{\mathfrak D}
\newcommand{\hD}{\hat{\D}}

\newcommand{\bH}{\mathbf H}
\newcommand{\bEr}{\mathbf{Err}}

\newcommand{\C}{\mathbb C}

\newcommand{\Z}{\mathbb Z}
\newcommand{\PP}{\mathbb P}

\newcommand{\N}{\mathbb N}

\newcommand{\bX}{\mathbf X}
\newcommand{\bx}{\mathbf x}
\newcommand{\bY}{\mathbf Y}
\newcommand{\by}{\mathbf y}

\newcommand{\bA}{\mathbf A}

\newcommand{\bI}{\mathbf I}
\newcommand{\bF}{\mathbf F}

\newcommand{\Clo}{D}
\newcommand{\wes}{d}
\newcommand{\Ceps}{\hat{C}}

\newcommand{\bb}{\boldsymbol{\upsilon}}

\newcommand{\bu}{\mathbf u}

\newcommand{\bp}{\mathbf p}
\newcommand{\bL}{\mathbf L}

\newcommand{\bdel}{\boldsymbol{\delta}}
\newcommand{\bDel}{\mathbf \Delta}
\newcommand{\ttt}{\boldsymbol{\tau}}

\newcommand{\R}{\mathbb R}

\renewcommand{\P}{\mathcal P}
\newcommand{\hP}{\hat{\P}}
\newcommand{\baP}{\overline{\P}}

\newcommand{\fb}{\mathfrak b}

\newcommand{\cC}{\mathcal Q}
\newcommand{\cM}{\mathcal M}
\newcommand{\cL}{\mathcal L}

\newcommand{\bk}{\mathbf k}

\newcommand{\bQ}{\mathbf Q}

\newcommand{\fB}{\mathfrak B}

\newcommand{\bm}{\mathbf m}

\renewcommand{\d}{\mathrm d}

\newcommand{\AB}{\mathcal{A}^\beta}
\newcommand{\GT}{\mathcal G}

\newcommand{\don}{\mathds 1}
\DeclareMathOperator{\Err}{\Theta}

\newcommand{\rmA}{\mathrm A}
\newcommand{\rmB}{\mathrm B}
\newcommand{\rmC}{\mathrm C}
\newcommand{\rmI}{\mathrm I}
\newcommand{\rmII}{\mathrm {II}}
\newcommand{\rmIII}{\mathrm {III}}
\newcommand{\rmIV}{\mathrm {IV}}
\newcommand{\rmV}{\mathrm V}
\newcommand{\rmVIa}{\mathrm {VI.1}}
\newcommand{\rmVIb}{\mathrm {VI.2}}
\newcommand{\rmVIc}{\mathrm {VI.3}}

\title{Airy$_\beta$ line ensemble and its Laplace transform}

\author{Vadim Gorin
\thanks{Departments of Statistics and Mathematics, UC Berkeley, Berkeley, CA, USA. e-mail: vadicgor@gmail.com}
\and
Jiaming Xu
\thanks{Department of Mathematics, KTH Royal Institut of Technology, Stockholm, Sweden. e-mail: jxu0800@gmail.com}
\and
Lingfu Zhang
\thanks{Division of Physics, Mathematics and Astronomy, Caltech, Pasadena, CA, USA, and Department of Statistics, UC Berkeley, Berkeley, CA, USA. e-mail: lingfuz@caltech.edu}
}

\date{}

\begin{document}

\maketitle

\begin{abstract}
 The Airy$_\beta$ line ensemble is a random collection of continuous curves, which should serve as a universal edge scaling limit in problems related to eigenvalues of random matrices and models of 2d statistical mechanics. This line ensemble unifies many existing universal objects including Tracy-Widom distributions, eigenvalues of the Stochastic Airy Operator, Airy$_2$ process from the KPZ theory. Here $\beta>0$ is a real parameter governing the strength of the repulsion between the curves.

 We introduce and characterize the Airy$_\beta$ line ensemble in terms of the Laplace transform, by producing integral formulas for its joint multi-time moments. We prove two asymptotic theorems for each $\beta>0$: the trajectories of the largest eigenvalues in the Dyson Brownian Motion converge to the Airy$_\beta$ line ensemble; the extreme particles in the G$\beta$E corners process converge to the same limit.

 The proofs are based on the convergence of random walk expansions for the multi-time moments of prelimit objects towards their Brownian counterparts. The expansions are produced through Dunkl differential-difference operators acting on multivariate Bessel generating functions.
\end{abstract}

\setcounter{tocdepth}{2}
\tableofcontents

\section{Introduction}

\subsection{Overview}
Following the seminal work of Dyson, Mehta, Wigner, and others in 1950s-60s and the influential papers of Tracy and Widom in the 1990s, the central question of the random matrix theory is to describe the probability distributions for the scaling limits of eigenvalues of very large random self-adjoint matrices. The original focus (see \cite{Meh} for the review) has been on the eigenvalues in the bulk of the spectrum, motivated by the nuclear physics, where spacings between such eigenvalues served as a model for spacing between energy levels in heavy nuclei; later the same bulk eigenvalues were discovered to be connected to the statistics of the zeros of the Riemann zeta-function (see, e.g.\ \cite{montgomery1973pair,berry1999riemann}), thus, reinvigorating the interest to the topic. More recently, the largest eigenvalues attracted lots of attention, driven by their limiting distributions turning out to be relevant for the 2d-statistical mechanics (see, e.g.\ \cite{johansson2018edge,gorin2021lectures}), to interacting particle systems and growth models in the Kardar-Parisi-Zhang (KPZ) universality class (see e.g.\ \cite{corwin2012kardar,borodin2016lectures}), to asymptotic representation theory (see, e.g., \cite{okounkov2000random, borodin2000asymptotics, jeong2016limit, ahn2023airy}), and to the statistical testing procedures (see, e.g., \cite{johnstone2006high}). One remarkable discovery in the random matrix theory was the phenomenon of \emph{universality}, which is reminiscent of the universality of the Gaussian distribution in classical probability theory: it turns out that there is a very short list of distinguished limiting probability distributions (with the sine-process and Tracy-Widom distribution being two most celebrated examples) and the asymptotic behavior for each of the hundreds of systems, ranging from different classes of random matrices and to the stochastic PDEs and objects of the number theory, is governed by one of these few distributions, see, e.g., \cite{deift2006universality,tao2014random,erdHos2017dynamical}.

The main goal of this paper is to construct and investigate the most recent addition to this family of the universal limit distributions --- the Airy$_\beta$ line ensemble: for each $\beta>0$, this is an infinite collection of random continuous curves $\AB_1(\ttt)\ge \AB_2(\ttt)\ge \AB_3(\ttt)\dots$. For special choices of $\beta$ and by passing to the marginals, $\{\AB_i(\ttt)\}_{i=1}^{\infty}$ specializes to several important objects: $\AB_1(0)$ is the general $\beta$ Tracy-Widom distribution; for $\beta=2$, the Airy$_2$ line ensemble $\{\mathcal A^{2}_i(\ttt)\}_{i=1}^{\infty}$ plays  a central role in the KPZ universality class (see, e.g., the introduction in \cite{aggarwal2023strong} and references therein), it is also the same object as the extended Airy determinantal point process; $\{\AB_i(0)\}_{i=1}^{\infty}$ are eigenvalues of the Stochastic Airy Operator \cite{edelman2007random,RRV}; for $\beta=1,2$ the Airy$_\beta$ line ensemble is identified with the spectral edge for corners of time-dependent Wigner matrices \cite{sodin2015limit}. Hence, our $\{\AB_i(\ttt)\}_{i=1}^\infty$ combines all these objects and gives them a unified treatment. We prove two main types of results for the Airy$_\beta$ line ensemble:

\begin{itemize}
 \item We characterize $\{\AB_i(\ttt)\}_{i=1}^{\infty}$, by giving integral expressions for the moments of its (random) Laplace transforms, and deriving the continuity of the trajectories (see Theorem \ref{thm:main}). 
 \item We show that this object appears in two distinct edge limit transitions: for the largest particles in the general $\beta$ Dyson Brownian Motion (see also \cite{landon2020edge,HZ})  and for the Gaussian $\beta$ corners process.
\end{itemize}
Our approach to the proofs suggests numerous generalizations, and therefore predicts universality of $\{\AB_i(\ttt)\}_{i=1}^{\infty}$.
Towards this universality, in a companion paper \cite{KX}, the approach using actions of Dunkl operators similar to this text is applied to extend the one-time $\{\AB_{i}(0)\}_{i=1}^{\infty}$ edge universality to certain additions of Gaussian and Laguerre $\beta$-ensembles.
In another paper \cite{HZ} a general framework of proving convergence to $\{\AB_i(\ttt)\}_{i=1}^{\infty}$ is developed, and is applied to various continuous random processes.

\subsection{Beta ensembles and their $2d$--extensions}

The starting point for our text is the Gaussian $\beta$ Ensemble. It depends on an integer $N=1,2,\dots,$ and two parameters $\beta, \tau>0$ and is a probability distribution on $N$--tuples of reals $\lambda_1>\lambda_2>\dots>\lambda_N$ of density proportional to:
\begin{equation} \label{eq_GBE_t}
  \mathcal P(\lambda_1,\dots,\lambda_N)\sim \prod_{i<j} |\lambda_i-\lambda_j|^{\beta} \exp\left(-\frac{1}{2 \tau} \sum_{i=1}^N (\lambda_i)^2 \right).
\end{equation}
For $\beta=1,2,4$, the formula \eqref{eq_GBE_t} records the eigenvalue distribution of $N\times N$ self-adjoint Wigner matrices with i.i.d.\ mean $0$ Gaussian real/complex/quaternionic matrix elements, respectively, see, e.g., \cite{Meh,forrester2010log}. For general values of $\beta$, \eqref{eq_GBE_t} are eigenvalues of tridiagonal matrices of Dumitriu and Edelman \cite{DE}. The computation of the normalization constant which makes \eqref{eq_GBE_t} a probability distribution is the Mehta integral evaluation, which is a particular case of the famous Selberg integral, cf.\ \cite{forrester2008importance}. Other points of view on \eqref{eq_GBE_t} connect it to Calogero–Sutherland quantum many–body system and to Coulomb log-gases, cf.\ \cite{forrester_oxford,Serfaty24}.

A powerful strategy for studying $N\to\infty$ asymptotics of the random matrix-type distributions uses \eqref{eq_GBE_t} as the basic ingredient: we first investigate various limits of the Gaussian $\beta$ Ensemble by one set of tools and then use another set of tools to prove the universality, i.e.\ to compare other models with the Gaussian $\beta$ Ensemble and to show that the limits should be the same.  For classical values $\beta=1,2,4$, the formulas for the limits of \eqref{eq_GBE_t} both in the bulk and at the edge are quite old (see, e.g., \cite[Chapters 6,7,8]{Meh} or \cite[Chapter 7]{forrester2010log} and references therein); their development was based on the theory of determinantal and Pfaffian point processes. More recently they were shown to extend universally to the log-gases (of the form \eqref{eq_GBE_t} with other potentials replacing $ \exp\left(-\tfrac{1}{2 \tau} (\lambda)^2\right)$), to Wigner matrices with generic matrix elements, cf.\ \cite{erdHos2012universality,tao2014random}, and to many other random matrix models not mentioned here. The edge and bulk scaling limits of \eqref{eq_GBE_t} for general $\beta>0$ were found in \cite{edelman2007random,RRV} and \cite{valko2009continuum}, respectively, based on the tridiagonal matrix model of \cite{DE}. Subsequently, they were extended to log-gases with more general potential in \cite{bourgade2014edge,bourgade2014universality,bekerman2015transport,krishnapur2016universality}, to discrete log-gases in \cite{guionnet2019rigidity}, and to some $\beta$ matrix additions in \cite{KX}.

\medskip

The focus of this text is not on the Gaussian $\beta$ Ensemble per se, but rather on its two-dimensional extensions. There are two natural ways to add a second dimension to \eqref{eq_GBE_t}: we can either allow $\tau$ to vary, or we can allow $N$ to vary. Somewhat surprisingly, when it comes to edge limits, both approaches lead to exactly the same limiting objects, as our Theorems \ref{thm:cor-conv} and \ref{thm:dbm-conv} below indicate.

The dynamical point of view on \eqref{eq_GBE_t} goes back to \cite{dyson1962brownian} and interprets $\tau$ as a time parameter. For that we consider an $N$-dimensional diffusion known as the Dyson Brownian Motion (see
\cite[Chapter 9]{Meh}, \cite[Section 4.3]{AGZ} and references therein)

\begin{definition} \label{Definition_DBM} For a real $\beta>0$, the Dyson Brownian Motion (DBM) is the (unique,
strong) solution\footnote{For $0<\beta<1$ a special care is needed, because particles start to collide, see \cite{cepa1997diffusing}.} of the system of $N$ stochastic differential equations
\begin{equation}
 \label{eq_Dyson_BM}
\mathrm{d}Y_i(\tau )=\frac{\beta}{2}\,\sum_{j\neq i} \frac{1}{Y_i(\tau )-Y_j(\tau )}\,\mathrm{d}\tau +
\mathrm{d}W_i(\tau ), \quad i=1,2,\ldots,N,
\end{equation}
where $W_1,W_2,\ldots,W_N$ are independent standard Brownian motions.
\end{definition}
 If one starts
\eqref{eq_Dyson_BM} from zero initial condition, then the distribution of the solution at time $\tau$
is given by \eqref{eq_GBE_t}. For the special value $\beta=2$, the
diffusion \eqref{eq_Dyson_BM} admits two probabilistic interpretations.
First, it can be identified with a
system of $N$ independent standard Brownian motions conditioned never to collide, thereby making it a key tool to analyze KPZ universality class models (see e.g., \cite{DOV,DZ}). Besides, it also can be viewed as the evolution of the eigenvalues of a Hermitian
random matrix whose elements evolve as independent Brownian motions. Similarly, for $\beta=1$ and $4$, \eqref{eq_Dyson_BM} is the evolution of eigenvalues for self-adjoint symmetric and quaternionic matrices of Brownian motions, respectively. The DBM can be also linked to prominent conformally invariant objects: it is the driving function for multiple SLEs in a simultaneous parameterization; this observation can be traced back to \cite{cardy2003stochastic}\footnote{See also \cite{cardy2004calogero} for the correction of the correspondence between the parameters.}, and it was made more precise recently, see, e.g., \cite{healey2021n, wu2023multiple, feng2024multiple}. We recall that Schramm-Loewner Evolutions (SLEs) appear as conformally invariant scaling limits of interfaces in various critical models of 2d statistical mechanics, see, e.g., \cite{tsirelson2004wendelin}. SLEs depend on a parameter $\kappa$, which is related to the parameter of the DBM via $\beta=8/\kappa$. In particular, the relevant values of the parameter go beyond the classical $\beta=1,2,4$, for instance, the celebrated Ising model corresponds to $\kappa=3$ , i.e.\ $\beta=8/3$, cf.\ \cite{chelkak2014convergence}, while self-avoiding random walks should be related to $\kappa=8/3$, i.e. $\beta=3$, cf.\ \cite{lawler2002scaling}.

\smallskip

For an alternative 2d-extension of \eqref{eq_GBE_t}, we define a \emph{Gelfand--Tsetlin pattern of rank $N$ with top row $\lambda_1,\dots,\lambda_N$} as an array $\{y_{i}^n\}_{1\le i \le n \le N}$ of real numbers satisfying
$y^{n+1}_i\ge y^{n}_i \ge y^{n+1}_{i+1}$ and $(y_1^N,y_2^N,\dots,y_N^N)=(\lambda_1,\dots,\lambda_N)$. Denote by $\GT_N$ the space of all Gelfand--Tsetlin patterns of rank $N$ with arbitrary top rows.

\begin{definition}\label{def_betacorner} Fix $\beta>0$.\footnote{Compared to \cite[Definition 2.18]{GS1} or \cite[Definition 2.1]{BGC}, we set $\beta=2\theta$.} The \emph{$\beta$-corners process with top row $\lambda_1>\dots>\lambda_N$}
is a random array
$\{y^n_i\}_{1\leq i\leq n\leq N}\in \GT_N$ with top row $(\lambda_1,\dots,\lambda_N)$, in which
the $N(N-1)/2$ coordinates $y_i^n$, $1\le i \le n<N$, have the probability density
\begin{equation}
\label{eq_beta_corners_def}
\frac{1}{Z_{N; \beta}} \cdot
\prod_{n=1}^{N-1} \left[\prod_{1\le i<j\le n} |y_j^n-y_i^n|^{2-\beta}\right] \cdot \left[\prod_{a=1}^n \prod_{b=1}^{n+1}
|y^n_a-y^{n+1}_b|^{\beta/2-1}\right],
\end{equation}
where $Z_{N; \beta}$ is the normalization constant:
\begin{equation}
\label{eq_normalization}
Z_{N; \beta} =\left[\prod_{n=1}^N \frac{ \Gamma(\beta/2)^n}{\Gamma(k\beta/2)}\right] \cdot \prod_{1\le i < j \le N}
 |\lambda_j-\lambda_i|^{\beta-1}.
\end{equation}
\end{definition}
By taking  limits in the space of probability measures on $\GT_N$, we can allow equalities and extend the definition to arbitrary $\lambda_1\ge \lambda_2\ge\dots\ge \lambda_N$.

\begin{definition} \label{eq_GBE_corners}
 The G$\beta$E corners process of rank $N$ and variance $\tau$ is an $N(N+1)/2$ dimensional random array $\{y^n_i\}_{1\leq i\leq n\leq N}\in \GT_N$, such that the top row $(y_1^N,y_2^N,\dots,y_N^N)$ is distributed according to \eqref{eq_GBE_t} and the conditional distribution of $\{y^n_i\}_{1\leq i\leq n<N}$ given the top-row is that of the $\beta$--corners process of Definition \ref{def_betacorner}.
\end{definition}
For $\beta=1,2,4$, the distribution of G$\beta$E corners process coincides with that of eigenvalues of top-left $k\times k$ corners, $k=1,\dots,N$, of the Wigner matrix with Gaussian real, or complex, or quaternionic matrix elements, respectively: the ideas of this computation go back to \cite[Section 9.3]{GN}, with more recent treatments in \cite{baryshnikov2001gues,neretin2003rayleigh}. \Cref{eq_GBE_corners} is consistent\footnote{See, e.g., discussion after Definition 1.1 in \cite{GS1}.} over $N$: for $M<N$, the distribution of $\{y^k_i\}_{1\leq i\leq k\leq M}$ is again the G$\beta$E corners process (of rank $M$ and the same variance $t$). In this way, one can also set $N=\infty$ and think about infinite arrays $\{y^n_i\}_{1\leq i\leq n}$; these infinite random infinite arrays are examples of extreme ergodic Gibbs measures and play a role in the classification theorems for the latter, see \cite{OV,AN,BR}. For the special value $\beta=2$, the conditional density in \eqref{eq_beta_corners_def} is uniform, and the corresponding GUE-corners process (first studied in \cite{johansson2006eigenvalues,okounkov2006birth}) is universally appearing as a scaling model of many models of $2d$ statistical mechanics, see \cite[Lecture 20]{gorin2021lectures}, \cite{gorin2023six}, and references therein; some traces of this universality also exist for general values of $\beta$, cf.\ \cite{cuenca2021universal}.

\subsection{Edge limits via the Airy$_\beta$ line ensemble}  \label{ssec:convtoale}

We are interested in the largest particles (eigenvalues) in either of the DBM or the G$\beta$E corners process, and we now introduce the limiting object governing their asymptotic behavior.

\begin{theorem}\label{thm:main}
For any $\beta>0$, there is a unique family of time-dependent (stationary) continuous random processes --- denoted by $\{\AB_i(\ttt)\}_{i=1}^\infty$ --- such that they are ordered $\AB_1\ge \AB_2\ge \dots$; and for any $m\in \N$, $\vec \bk \in \R^m_+ $, and $\vec \ttt\in \R^m$ such that $\ttt_1\le \cdots \le \ttt_m$, the joint moment for the Laplace transforms
\begin{equation} \label{eq_expected_Laplace}
\E\left[ \prod_{\ell=1}^m \left(\sum_{i=1}^\infty
\exp(\bk_\ell \AB_i(\ttt_\ell)/2) \right) \right]
\end{equation}
equals $\bL_\beta(\vec\bk,\vec\ttt)$, which is defined below in \Cref{defn:core}\footnote{The factor of dividing by $2$ in \eqref{eq_expected_Laplace} is to simply the expressions in \Cref{defn:core}, avoiding various extra factors there. It also comes up naturally in the edge scaling limits.}.
\end{theorem}
We call such $\{\AB_i\}_{i=1}^\infty$ the Airy$_\beta$ line ensemble, and we also refer to the individual paths as particles.\footnote{Following our proofs, each $\AB_i$ is a continuous function from $\R$ to $\R\cup\{-\infty\}$, with topology which makes each $[-\infty, a)$, $a\in\mathbb R$ an open set.
However, it is implied by \cite{landon2020edge} ($\beta\ge 1$) and \cite{HZ} (any $\beta>0$) that, almost surely, $\AB_i(\ttt)\in \mathbb R$ for each $i\in\N$ and $\ttt\in\R$. Therefore, eventually the potential presence of $-\infty$ can be ignored.}
The expectations in \eqref{eq_expected_Laplace} can be rewritten as Laplace transforms of the correlation functions for the Airy$_\beta$ line ensemble, although extraction of compact expressions for correlation functions from \eqref{eq_expected_Laplace} is a separate task to be done, cf.\ \cite{okounkov2002generating}, \cite{borodin2016moments} for computations of such type at $\beta=2$. We plan to address this in the future work.

For $\beta=1,2$ representations of \eqref{eq_expected_Laplace} as infinite sums of integrals were previously studied in \cite{sodin2015limit,jeong2016limit} (following the earlier results on fixed time marginals, corresponding to $\ttt_1=\dots=\ttt_m=0$ in \eqref{eq_expected_Laplace} in \cite{soshnikov1999universality,okounkov2000random,okounkov2002generating}); the answers in these articles have distinct combinatorial structures for $\beta=1$ and $\beta=2$. On the other hand, our expressions for $\bL_\beta(\vec\bk,\vec\ttt)$ are analytic functions in $\beta$.
Another approach at $\beta=2$ is to characterize the Airy$_2$ line ensemble in terms of its correlation functions and determinantal formulas for the latter written in terms of the extended Airy kernel \cite{prahofer2002scale}; the interplay between determinantal formulas and combinatorial expressions for \eqref{eq_expected_Laplace} was explored in \cite{okounkov2002generating} to compute certain intersection numbers. Yet another approach to the Airy$_2$ line ensemble is to characterize it through the Brownian Gibbs property, see \cite{CH,aggarwal2023strong}. It is unknown how to extend either of the two last approaches to the general values of $\beta$. Very recently \cite[Problem 4]{arXiv:2411.01633} suggested an attractive and simple conjectural definition for the Airy$_\beta$ ensemble; though, at the end of Section 1 the authors report that ``numerical experiments do not yet provide strong evidence for these conjectures''. It is yet to be understood whether our Theorem \ref{thm:main} agrees with the conjectures of \cite{arXiv:2411.01633}.

The fixed time $\ttt_1=\dots=\ttt_m=0$ marginals of \eqref{eq_expected_Laplace} were previously computed in \cite{GS} leading to different expressions and using a method based on tridiagonal matrices, which does not generalize to the multi-time case. The only path to the Airy$_\beta$ line ensemble existing before this work was an approach through an infinite system of stochastic differential equations, making sense of \eqref{eq_Dyson_BM} at $N=\infty$. The initial attempt on this method goes back to \cite{osada2014infinite} for $\beta=1,2,4$; more recently results for all $\beta\ge 1$ were presented in \cite{landon2020edge} in the language of individual particles; and in the most recent (simultaneous with the current paper) \cite{HZ} in the language of an SDE satisfied by the Stieltjes transform. One inevitable difficulty of this method is that directly solving the resulting system of SDEs becomes a new (unsolved) hard task. On the other hand, Theorem \ref{thm:main} provides direct formulas for the Airy$_\beta$ line ensemble.

Our following two results demonstrate the appearance of the Airy$_\beta$ line ensemble as an edge scaling limit.

\begin{theorem}  \label{thm:cor-conv}
For fixed $\beta>0$, let $\{y_i^n\}_{1\le i\le n}$ be the (infinite rank) G$\beta$E-corners process of \Cref{eq_GBE_corners} with variance $1$.
As $N\to\infty$, the process
\[
(\ttt, i)\mapsto 2N^{2/3} \left( \frac{y_i^{\lfloor N-\ttt N^{2/3} \rfloor}}{\sqrt{2\beta(N-\ttt N^{2/3})}}-1\right)
\]
converges to the Airy$_\beta$ line ensemble in distribution in the topology of uniform convergence over compact subsets of $\mathbb R$ for $\ttt$--variable and jointly in finitely many $i$.
\end{theorem}

\begin{theorem}  \label{thm:dbm-conv}
For fixed $\beta>0$, let $\{Y_i(\tau)\}_{1\le i\le N, \tau>0}$ be the $N$ dimensional DBM of \Cref{Definition_DBM} started from zero initial condition.
As $N\to\infty$, the process
\[
(\ttt, i)\mapsto 2N^{2/3} \left( \frac{Y_i(2N/\beta + 2\ttt N^{2/3}/\beta)}{2\sqrt{(N + \ttt N^{2/3})N} }-1 \right)
\]
converges to the Airy$_\beta$ line ensemble in distribution in the topology of uniform convergence over compact subsets of $\mathbb R$ for $\ttt$--variable and jointly in finitely many $i$.
\end{theorem}

\begin{remark}\label{rem:multitimelevel}
 One can glue together the corners process and the DBM in a three-dimensional object --- time evolution of a $\GT_N$ array, see \cite{GS1}. Most of our results extend to \emph{monotone $2d$--sections} of these $3d$ object: we can deal with the joint distribution of the marginals at points $(N_1,\tau_1)$, $(N_2,\tau_2)$,\dots, $(N_k,\tau_k)$, with $N_1\ge N_2\ge \dots\ge N_k$, $\tau_1\le \tau_2\le\dots\le \tau_k$, but not with the full distribution.
\end{remark}

The proof of Theorems \ref{thm:cor-conv} and \ref{thm:dbm-conv} is split into two parts: in Theorems \ref{thm:multil} and \ref{thm:multit} we prove convergence of joint moments of Laplace transforms towards the expressions of Theorem \ref{thm:main}. Then in Section \ref{Section_process_convergence} we deduce the distributional convergence in the uniform in compacts topology in $\ttt$.

We refer to \cite{landon2020edge} and \cite{HZ} for alternative approaches to studying the multi-time edge limits of the DBM. On the other hand, no other approaches exist for investigating the edge limits of G$\beta$E-corners process at general $\beta>0$.

The coincidence of the edge scaling limits for the G$\beta$E corners process and for the DBM might look mysterious. At $\beta=2$ such a coincidence (first results of this type go back to \cite{ferrari2003step,johansson2005arctic}, where Airy$_2$ line ensemble was found in tiling and growth models --- and see \cite{nordenstam2006eigenvalues,aggarwal2022gaussian} and \cite[Lecture 20]{gorin2021lectures} for connections between tilings and GUE corners process) can be explained through the Brownian Gibbs property: one naturally expects both limiting objects to satisfy it and then \cite{aggarwal2023strong} shows that there is only one such object. For other values of $\beta$, one explanation is that both scaling limits should satisfy an infinite dimensional SDE, which can be viewed as either the $N\to\infty$ limit of \eqref{eq_Dyson_BM}, or a Markov property of the G$\beta$E corners process (which can be seen from \Cref{def_betacorner}).
From a general perspective, this coincidence serves as an indication of the \emph{universality} of the Airy$_\beta$ line ensemble. The coincidence itself was previously known for $\beta=1$ due to \cite{sodin2015limit} and for $\beta=\infty$ due to \cite{gorin2022universal}.

We expect that the Airy$_\beta$ line ensemble is also the edge limit of many other general $\beta$ objects, such as: Laguerre and Jacobi corners as in \cite{BP,DW,FN} and \cite{borodin2015general,sun2016matrix}; Macdonald processes as in \cite{borodin2014macdonald}; Jack-Gibbs measures as in \cite{GS1,dimitrov2024global} and multi-level versions of discrete $\beta$--ensembles as in \cite{guionnet2019rigidity}; various non-intersecting random walks as in \cite{huang2021beta,huang2021law}; DBM with general potentials as in \cite{li2013generalized,huang2019rigidity}; Laguerre and Jacobi dynamics as in \cite{bru1991wishart,konig2001eigenvalues} and \cite{demni2010beta}; additions of Gaussian and Laguerre $\beta$ ensembles as in \cite{KX}.
It would be extremely interesting to create proper tools for proving them, thus developing the universality of the Airy$_\beta$ line ensemble.
One approach is initiated in \cite{HZ}: it consists in first proving tightness and then checking that any subsequential limit satisfies an SDE, which also characterizes the Airy$_\beta$ line ensemble.
This strategy is being successfully applied in \cite{HZ} to establish the Airy$_\beta$ line ensemble limit for DBM with general potentials, and the Laguerre and Jacobi dynamics.

\subsection{Methodology} The classical approach for finding the local limits of eigenvalue distributions in random matrix theory goes back to the 60s and is based on the determinantal/Pfaffian formulas for the correlation functions; it is extensively covered in textbooks, see \cite{Meh,forrester2010log}. This method was very successful for $\beta=1,2,4$, but it failed to extend to generic $\beta>0$. An alternative proposed in \cite{DE,edelman2007random} was to use tridiagonal matrix models for all $\beta>0$ and this had led to computations of the scaling limits of $\beta$-ensembles in the bulk and in the edge through stochastic objects such as Brownian carousel \cite{valko2009continuum} and Stochastic Airy Operator \cite{RRV}. Adding time (as in Definition \ref{Definition_DBM}) or matrix size $N$ (as in Definition \ref{def_betacorner}) to these results turned out to be extremely difficult: the tridiagonalization procedure of \cite{DE} does not agree with these two-dimensional extensions (as discussed, e.g.\ after (9) in \cite{voit2023freezing}), while alternative matrix models (cf.\ \cite{allez2012invariant,allez2013diffusive,holcomb2017tridiagonal}) do not pave a way for the computations of local limits. The most recent attempt to add time to the tridiagonalization procedure is \cite{arXiv:2411.01633}, yet computing the edge scaling limit remains a hard open problem in that paper as well.

A very different approach for computing edge limits was pioneered in \cite{soshnikov1999universality} (see also \cite{furedi1981eigenvalues,bai1988necessary,boutet1995norm,sinai1998central,sinai1998refinement}): it suggests investigating the traces of very high powers of matrices (equivalently, eigenvalue power sums of large degrees). In subsequent years, these ideas were adapted to various classes of random matrices, e.g.\ in \cite{peche2009universality,sodin2014several,GS}, to Plancherel measure for symmetric groups in \cite{okounkov2000random,jeong2016limit}, to tensor products of representations of unitary groups in \cite{ahn2023airy}. In order to use this method for the Airy$_\beta$ line ensemble, the main challenge was to find a combinatorial expression for the expectations of (products of) power sums of eigenvalues, which simultaneously:
\begin{itemize}
\item Captures the information on multi-level/time distributions, as in Definitions \ref{Definition_DBM}, \ref{def_betacorner};
\item Is amenable to large $N$, large power limits, necessary to probe the edge.
\end{itemize}
It took us more than ten years and several attempts to figure out a path forward. We previously tried: an approach to expected power sums through Macdonald difference operators of \cite{borodin2015general} (obstacle: the combinatorial sums do not seem to be absolutely convergent in the large power limit), an approach to expected power sums through Negut's difference operators \cite{gorin2018interlacing} (obstacle: large powers required contour integrals of growing dimensions), and approach through powers of tridiagonal matrices \cite{GS} (obstacle: unclear how to add multilevel/time structures). More recently, we investigated in \cite{BGC,xu2023rectangular} the Laws of Large Numbers for empirical distributions of the eigenvalues of $\beta$--ensembles through the \emph{Dunkl differential-difference operators} of \cite{dunkl1989differential}. The key observation underlying the proofs in the present article is that the expressions for the expectations of the product of the power sums obtained through the Dunkl operators fit the bill and satisfy both desired properties. Compared to the fixed degree power sums calculation in \cite{BGC,xu2023rectangular}, the degrees of power sums are growing with $N$ (more precisely, of order $O(N^{\frac{2}{3}})$) in order to capture the edge behavior, which makes taking asymptotic of their expectation much more involved. After significant technical effort, we manage to use these operators to produce combinatorial expressions involving sums over decorated positive random walks for the prelimiting objects and then discover that the expressions essentially survive in the large $N$ limit: random walks are asymptotically replaced by Brownian motions, while ``decorations'' remain similar, and eventually we arrive at a series expansion in terms of ``number of decorations''. To our best knowledge, such combinatorial expression is novel, and gives explicit formulas even new for the one-time marginal $\{\AB_{i}(0)\}_{i=1}^{\infty}$. Simultaneously, similar developments and expressions were achieved in \cite{KX} in the study of limits of expectations of eigenvalue power sums for additions of Gaussian and Laguerre $\beta$-ensembles.

\medskip

\noindent\textbf{Acknowledgements.}
V.G.\ was partially supported by NSF grant DMS - 2246449.
L.Z.\ was partially supported by NSF grant DMS - 2246664, and the Miller Institute of Basic Research in Science.
This project was revived when L.Z. was visiting UW Madison in the fall of 2021, and he thanks them for their hospitality. The project was further developed at the Institute for Pure and Applied Mathematics, and Institut Mittag-Leffler, to which we are also very grateful.
We thank Jiaoyang Huang and B\'alint Vir\'ag for helpful discussions.
\medskip

\noindent\textbf{Notations.}
For any real numbers $a<b$, we denote $\llbracket a, b\rrbracket = [a, b]\cap\Z$.
We will also use Pochhammer symbols $(x)_n = x(x+1)\cdots(x+n-1)$, for any $x\in\R$ and $n\in \N$.
For a function $f$ whose domain of definition contains $x\in \R$ and a left (resp.~right) neighborhood of $x$,  we use $f(x-)$ (resp.~$f(x+)$) to denote the left limit of $f$ at $x$, if it exists.
For two real numbers $x$ and $y$, we sometimes denote $x\wedge y=\min(x, y)$, to make formulas shorter. If $y=\varnothing$, then $\min(x,y)=\max(x,y)=x$. For a set $A$, we let $|A|$ denote the number of elements in it.

Throughout the rest of this paper, we fix $\beta>0$ and treat it as a constant.

\section{Formulas for joint moments}  \label{sec:forjm}
In this section, we present formulas for $\bL_\beta(\vec \bk, \vec\ttt)$ of Theorem \ref{thm:main}.  We will later prove the convergence of moments of G$\beta$E corners process and DBM towards these quantities in \Cref{thm:multil,thm:multit} by demonstrating that the discrete versions of the constructions of this section converge in a diffusive scaling limit to the continuous versions introduced here.

\subsection{Blocks}
Our expression involves a combinatorial structure, which we call \emph{blocks} and illustrate in Figure \ref{fig:block}.
Blocks consist of several objects: a block process, a virtual block process, and a block height. Informally, block process and virtual block process are piecewise-constant constraints on a function, which is being conditioned to stay above them (later this function will become a Brownian bridge). The block height represents the values of this function in the most important points.

We next give formal definitions sequentially.
\begin{definition}  \label{defn:dbfdef}
For an interval $[a, b]$, we call a function $f:[a, b]\to \R_{\ge 0}$ a \emph{block function}, if it is right-continuous on $[a,b)$, and $f(a)=f(b)=f(b-)=0$, and the support of $f$ is connected, and the image $f([a,b])$ is a finite set.
We denote by $\fB([a, b])$ the space of all block functions $f$ on $[a,b]$.
\end{definition}
Clearly, every block function is a finite and positive linear combination of indicators of intervals.

We next define block processes, as families of block functions with certain constraints.
\begin{definition} \label{Definition_block_process}
Take $m\in \N$, $\vec \bk=(\bk_1,\ldots, \bk_m)\in \R_+^m$, and denote $\bQ_\ell=\sum_{\ell'=1}^\ell \bk_{\ell'}$ for each $\ell \in \llbracket 0, m\rrbracket$.
A \emph{block process with times $\vec \bk$}
is a family $\vec \bp=\{\bp_j\}_{j=1}^{m+\bu}$, where  $\bu\in \Z_{\ge 0}$, and each $\bp_j:[0, \bQ_m]\to \R_{\ge 0}$ is a function satisfying the following:
\begin{itemize}
    \item for each $\ell \in \llbracket 1, m\rrbracket$, $\bp_\ell\in \fB([0,\bQ_{m}])$, and $\bp_\ell= 0$ on $[\bQ_{\ell-1}, \bQ_m]$;
    \item for each $j\in \llbracket m+1, m+\bu\rrbracket$, $\bp_j \in \fB([0,\bQ_m])$ and $\bp_j$ is not identical zero.
    \item for each $j\in \llbracket 1, m+\bu\rrbracket$, $\ell \in \llbracket 1, m\rrbracket$, $\bp_j$ is left-continuous at $x=\bQ_{\ell-1}$ unless $j=\ell$.
    \item regularity condition spelled in Definition \ref{defn:blp} below.\footnote{Here is some intuition for the definition of a block process. It corresponds to the degrees of polynomials in Dunkl operators expansion, as shown later in \Cref{sec:gme}. The values of $\bp_\ell$ at points $[0,\bQ_{\ell-1}]$ for $\ell \in \llbracket 1, m\rrbracket$ relate to the degrees of variables $x_i$ with indices matching one of the Dunkl operators $\D_i^N$ (to be introduced in \Cref{sec:dunkl}) we are expanding, and $\bp_j$ for $j\in\llbracket m+1, m+\bu\rrbracket$ are the degrees of the other variables. In particular, the variable corresponding to the first operator would have zero degree after $\bQ_0=0$, so $\bp_1$ is identical zero.}
\end{itemize}
\end{definition}
In particular, the definition implies that $\bp_1$ is always identical zero. Given a block process and $\ell\in\llbracket 0, m\rrbracket$, we also denote
\begin{equation}\label{eq_p_ell}
\bp^\ell=\sum_{j=\ell+1}^{m+\bu} \bp_j.
\end{equation}

The block functions are encoded by the jumps at all the discontinuity points, for which we also set up notations. For a block process $\vec \bp$ as defined above, and  $j\in\llbracket 1, m+\bu\rrbracket$, $\ell\in\llbracket 1,m\rrbracket$, we let
\begin{equation}
\label{eq_discont_def_1}
 \bDel_{j,\ell}=\{x\in (\bQ_{\ell-1}, \bQ_\ell): \bp_j(x-)\neq \bp_j(x)\}, \qquad \bdel_{j,\ell}=|\bDel_{j,\ell}|\qquad  \text{ and}
\end{equation}
\begin{equation}
\label{eq_discont_def_2}
\bdel=\sum_{j=1}^{m+\bu}\sum_{\ell=1}^m \bdel_{j,\ell}, \quad  \bDel=\bigcup_{j=1}^{m+\bu}\bigcup_{\ell=1}^m \bDel_{j,\ell}.
\end{equation}

\begin{definition}  \label{defn:blp}
For a block process $\vec \bp$ as defined above, we say it is \emph{regular}, if all $\bDel_{j,\ell}$ are mutually disjoint\footnote{From definition of $\bDel$ it also follows that  the set $\bDel$ is disjoint with $\{0,\bQ_1,\bQ_2,\dots,\bQ_m\}$.}, and  $\min \bigcup_{\ell=1}^m\bDel_{j,\ell} < \min \bigcup_{\ell=1}^m\bDel_{j',\ell}$ for all $m+1\le j<j'\le m+\bu$.
\end{definition}

The condition on $\min \bigcup_{\ell=1}^m\bDel_{j,\ell}$ is introduced for convenience of enumerations: it breaks the symmetry between $\bp_j$ for $j\in \llbracket m+1, m+\bu\rrbracket$. We note that these requirements can be replaced by any other ordering. Throughout the text, by a block process, we always mean a regular one.

Another ingredient we need is virtual block process.
\begin{definition} \label{Definition_virtual_block}
For a block process $\vec \bp$, a compatible \emph{virtual block process} is a vector $\vec\bb=(\bb_1,\ldots, \bb_m)\in (\R_+\cup \{\varnothing\})^m$ satisfying the following conditions: for each $\ell \in \llbracket 1,m\rrbracket$,
\begin{itemize}
    \item if $\bb_\ell\neq \varnothing$, then $\bQ_{\ell-1}+\bb_\ell < \min\bigl( \bigcup_{j=1}^{m+\bu}\bDel_{j,\ell}\cup \{\bQ_\ell\}\bigr)$;
    \item if $\ell=1$, or $\ell>1$ and $\bp_\ell(\bQ_{\ell-1}-)=0$, then $\bb_\ell=\varnothing$.
\end{itemize}
\end{definition}

The virtual block process can be treated as a part of the block process from the previous definitions, which will be treated in a slightly different way in the following constructions.

The next ingredient is ``block height''. This object originates from considering a Brownian bridge, conditioned to stay above the function $\bp^0$ coming from the block process and conditioned on certain additional inequalities coming from the virtual block process. We are interested in the values of this conditional Brownian bridge at the various points of time appearing in the previous definitions, cf.\ Figure \ref{fig:block}, which are all encoded in the following definition.

\begin{definition} \label{Definition_block_height}
For a block process $\vec \bp$ and compatible virtual block process $\vec\bb$, a compatible \emph{block height} is a function $\bH:\bDel\cup\{\bQ_\ell\mid \ell\in  \llbracket 0, m\rrbracket\}\cup\{\bQ_{\ell-1}+\bb_\ell \mid \ell\in  \llbracket 1, m\rrbracket,\, \bb_\ell\ne \varnothing\} \to \R_{\ge 0}$, such that
\begin{itemize}
    \item $\bH(\bQ_\ell)=\bp^0(\bQ_\ell-)$ for each $\ell \in \llbracket 1, m\rrbracket$ and $\bH(0)=0$;
    \item $\bH(\bQ_{\ell-1}+\bb_\ell)=\bp^0(\bQ_{\ell-1}-)$ for each $\ell$ with $\bb_\ell\neq\varnothing$;
    \item for all $x\in \bDel$, we have $\bH(x)\ge  \max\bigl(\bp^0(x),\bp^0(x-)\bigr)$;
    \item if $j\in \llbracket 1, m+\bu\rrbracket$ and $x\in \bDel$ is a discontinuity point for $\bp_j$, then the jump $\bp_j(x)-\bp_j(x-)$ and the difference $\bH(x)-\bp^0(x-)-\bp_j(x)$ have the same sign;
    \item if $\bb_\ell=\varnothing$ for $\ell\in \llbracket 1, m\rrbracket$, then  $\bH\bigl(\min\bigl(\bigcup_{j=1}^{m+\bu}\bDel_{j,\ell}\cup \{\bQ_\ell\}\bigr)\bigr)\ge \bH(\bQ_{\ell-1})$.
\end{itemize}
\end{definition}
In particular, if $\bp_j(x-)>0$ and $\bp_j(x)=0$ for $x\in\bDel$, then we must have $\bH(x)=\bp^0(x-)$, because $\bH(x)\ge \max\bigl(\bp^0(x), \bp^0(x-)\bigr)$ on one hand, and $\bH(x)-\bp^0(x-)\le \bp_j(x)$ on the other hand.

\begin{figure}[!ht]
    \centering
\begin{subfigure}[b]{0.98\textwidth}
    \resizebox{0.95\textwidth}{!}{
    \begin{tikzpicture}
        \draw (0,0)--(28,0);
        \draw[thick] [dashed] (9,0)--(9,1);
        \draw[thick] [dashed] (17,0)--(17,1);
        \draw[brown] [line width=2.5pt] (0,0)--(28,0);
        \draw (0,0) node[anchor=north]{$0$};
        \draw (9,0) node[anchor=north]{$\bQ_1$};
        \draw (17,0) node[anchor=north]{$\bQ_2$};
        \draw (28,0) node[anchor=north]{$\bQ_3$};
        \draw (0,1) node[anchor=west, color=brown]{{\LARGE $\bp_1$}};
    \end{tikzpicture}}
\end{subfigure}
\par\bigskip
\begin{subfigure}[b]{0.98\textwidth}
    \resizebox{0.95\textwidth}{!}{
    \begin{tikzpicture}
        \draw (0,0)--(28,0);
        \draw[thick] [dashed] (9,0)--(9,3);
        \draw[thick] [dashed] (17,0)--(17,3);
        \draw[blue] [line width=2.5pt] (0,0)--(6,0);
        \draw[blue] [line width=2.5pt] (6,2)--(9,2);
        \draw[blue] [line width=2.5pt] (9,0)--(28,0);
        \draw[blue, fill=white] [ultra thick] (6,0) circle (3pt);
        \draw[blue, fill=blue] [ultra thick] (6,2) circle (3pt);
        \draw[blue, fill=blue] [ultra thick] (9,2) circle (3pt);
        \draw[blue, fill=white] [ultra thick] (9,0) circle (3pt);
        \draw (0,0) node[anchor=north]{$0$};
        \draw (9,0) node[anchor=north]{$\bQ_1$};
        \draw (17,0) node[anchor=north]{$\bQ_2$};
        \draw (28,0) node[anchor=north]{$\bQ_3$};
        \draw (0,3) node[anchor=west, color=blue]{{\LARGE $\bp_2$}};
    \end{tikzpicture}}
\end{subfigure}
\par\bigskip
\begin{subfigure}[b]{0.98\textwidth}
    \resizebox{0.95\textwidth}{!}{
    \begin{tikzpicture}
        \draw (0,0)--(28,0);
        \draw[thick] [dashed] (9,0)--(9,3);
        \draw[thick] [dashed] (17,0)--(17,3);
        \draw[red] [line width=2.5pt] (0,0)--(3,0);
        \draw[red] [line width=2.5pt] (3,1.5)--(12,1.5);
        \draw[red] [line width=2.5pt] (12,1)--(17,1);
        \draw[red] [line width=2.5pt] (17,0)--(28,0);
        \draw[red, fill=white] [ultra thick] (3,0) circle (3pt);
        \draw[red, fill=red] [ultra thick] (3,1.5) circle (3pt);
        \draw[red, fill=white] [ultra thick] (12,1.5) circle (3pt);
        \draw[red, fill=red] [ultra thick] (12,1) circle (3pt);
        \draw[red, fill=white] [ultra thick] (17,0) circle (3pt);
        \draw[red, fill=red] [ultra thick] (17,1) circle (3pt);
        \draw (0,0) node[anchor=north]{$0$};
        \draw (9,0) node[anchor=north]{$\bQ_1$};
        \draw (17,0) node[anchor=north]{$\bQ_2$};
        \draw (28,0) node[anchor=north]{$\bQ_3$};
        \draw (0,3) node[anchor=west, color=red]{{\LARGE $\bp_3$}};
    \end{tikzpicture}}
\end{subfigure}
\par\bigskip
\begin{subfigure}[b]{0.98\textwidth}
    \resizebox{0.95\textwidth}{!}{
    \begin{tikzpicture}
        \draw (0,0)--(28,0);
        \draw[thick] [dashed] (9,0)--(9,3);
        \draw[thick] [dashed] (17,0)--(17,3);
        \draw[orange] [line width=2.5pt] (0,0)--(15,0);
        \draw[orange] [line width=2.5pt] (15,1.7)--(24,1.7);
        \draw[orange] [line width=2.5pt] (24,0)--(28,0);
        \draw[orange, fill=white] [ultra thick] (15,0) circle (3pt);
        \draw[orange, fill=orange] [ultra thick] (15,1.7) circle (3pt);
        \draw[orange, fill=white] [ultra thick] (24,1.7) circle (3pt);
        \draw[orange, fill=orange] [ultra thick] (24,0) circle (3pt);
        \draw (0,0) node[anchor=north]{$0$};
        \draw (9,0) node[anchor=north]{$\bQ_1$};
        \draw (17,0) node[anchor=north]{$\bQ_2$};
        \draw (28,0) node[anchor=north]{$\bQ_3$};
        \draw (0,3) node[anchor=west, color=orange]{{\LARGE $\bp_4$}};
    \end{tikzpicture}}
\end{subfigure}
\par\bigskip
\begin{subfigure}[b]{0.98\textwidth}
    \resizebox{0.95\textwidth}{!}{
    \begin{tikzpicture}
        \draw[teal] [ultra thick] [dashed] (0,0)--(0.2,1)--(0.4,0.8)--(0.6,2)--(0.8,1.6)--(1,3)--(1.2,2.1)--(1.4,3.2)--(1.6,2.9)--(1.8,3.4)--(2,2)--(2.2,2.5)--(2.4,1.9)--(2.6,2.7)--(2.8,2.2)--
        (3.0,2.7)--(3.2,2.5)--(3.4,3.1)--(3.6,2.4)--(3.8,2.9)--(4.0,2.2)--(4.2,2.8)--(4.4,2.1)--(4.6,2.9)--(4.8,2.7)--(5.0,4.0)--(5.2,3.2)--(5.4,4.3)--(5.6,4.8)--(5.8,5.9)--(6.0,5.6)--(6.2,5.0)--(6.4,5.7)--(6.6,5.2)--(6.8,5.9)--(7.0,4.8)--(7.2,5.3)--(7.4,4.6)--(7.6,5.2)--(7.8,5.0)--(8.0,6.0)--(8.2,5.4)--(8.4,4.2)--(8.6,4.6)--(8.8,3.7)--(9.0,3.5)--(9.2,4.0)--(9.4,3.6)--(9.6,4.5)--(9.8,4.2)--(10.0,4.3)--(10.2,3.7)--(10.4,4.1)--(10.5,3.8)--(10.6,4.2)--(10.8,3.7)--(11.0,3.5)--(11.2,3.8)--(11.4,3.2)--(11.5,2.7)--(11.6,3.4)--(11.8,2.9)--(12.0,1.9)--(12.2,2.3)--(12.4,2.0)--(12.6,1.7)--(12.8,2.5)--(13.0,2.7)--(13.2,2.3)--(13.4,2.9)--(13.6,2.5)--(13.8,2.1)--(14.0,2.0)--(14.2,2.8)--(14.4,1.7)--(14.5,3.8)--(14.6,3.4)--(14.8,4.0)--(15.0,4.1)--(15.2,4.6)--(15.4,3.5)--(15.6,3.8)--(15.8,3.2)--(16.0,3.3)--(16.2,3.7)--(16.4,3.2)--(16.5,3.8)--(16.6,4.3)--(16.8,2.9)--(17.0,2.7)--(17.2,3.2)--(17.4,2.9)--(17.6,3.0)--(17.8,2.8)--(18.0,3.8)--(18.2,3.1)--(18.4,2.8)--(18.5,2.7)--(18.6,3.3)--(18.8,3.1)--(19.0,3.8)--(19.2,3.5)--(19.4,4.0)--(19.6,3.2)--(19.8,4.0)--(20.0,4.2)--(20.2,3.0)--(20.4,3.5)--(20.6,2.9)--(20.8,3.8)--(21.0,2.5)--(21.2,3.0)--(21.4,4.1)--(21.6,2.7)--(21.8,2.9)--(22.0,3.7)--(22.2,3.1)--(22.4,2.4)--(22.6,2.8)--(22.8,2.1)--(23.0,2.2)--(23.2,1.8)--(23.4,2.5)--(23.6,1.9)--(23.8,2.1)--(24.0,1.7)--(24.2,2.3)--(24.4,2.7)--(24.6,1.5)--(24.8,1.9)--(25.0,3.0)--(25.2,2.3)--(25.4,1.8)--(25.6,2.5)--(25.8,1.0)--(26.0,1.5)--(26.2,1.9)--(26.4,1.2)--(26.6,2.3)--(26.8,1.8)--(27.0,2.0)--(27.2,1.2)--(27.4,0.5)--(27.6,0.2)--(27.8,0.4)--(28.0,0.0);

        \draw (0,0)--(28,0);
        \draw[thick] [dashed] (9,0)--(9,6);
        \draw[thick] [dashed] (17,0)--(17,6);

        \draw[purple] [dashed] (9,3.5)--(11,3.5)--(11,0);

        \draw[purple] [dashed] (17,2.7)--(18.5,2.7)--(18.5,0);

        \draw[gray] [line width=2.5pt] (0,0)--(3,0);
        \draw[gray] [line width=2.5pt] (3,1.5)--(6,1.5);
        \draw[gray] [line width=2.5pt] (6,3.5)--(9,3.5);
        \draw[gray] [line width=2.5pt] (9,1.5)--(12,1.5);
        \draw[gray] [line width=2.5pt] (12,1)--(15,1);
        \draw[gray] [line width=2.5pt] (15,2.7)--(17,2.7);
        \draw[gray] [line width=2.5pt] (17,1.7)--(24,1.7);
        \draw[gray] [line width=2.5pt] (24,0)--(28,0);

        \draw[gray, fill=white] [ultra thick] (3,0) circle (3pt);
        \draw[gray, fill=gray] [ultra thick] (3,1.5) circle (3pt);

        \draw[gray, fill=white] [ultra thick] (6,1.5) circle (3pt);
        \draw[gray, fill=gray] [ultra thick] (6,3.5) circle (3pt);

        \draw[gray, fill=white] [ultra thick] (9,1.5) circle (3pt);
        \draw[gray, fill=gray] [ultra thick] (9,3.5) circle (3pt);

        \draw[gray, fill=white] [ultra thick] (12,1.5) circle (3pt);
        \draw[gray, fill=gray] [ultra thick] (12,1) circle (3pt);

        \draw[gray, fill=white] [ultra thick] (15,1) circle (3pt);
        \draw[gray, fill=gray] [ultra thick] (15,2.7) circle (3pt);

        \draw[gray, fill=gray] [ultra thick] (17,2.7) circle (3pt);
        \draw[gray, fill=white] [ultra thick] (17,1.7) circle (3pt);

        \draw[gray, fill=white] [ultra thick] (24,1.7) circle (3pt);
        \draw[gray, fill=gray] [ultra thick] (24,0) circle (3pt);

        \draw[cyan, fill=cyan] [ultra thick] (0,0) circle (2pt);
        \draw[cyan, fill=cyan] [ultra thick] (28,0) circle (2pt);
        \draw[cyan, fill=cyan] [ultra thick] (11,3.5) circle (2pt);
        \draw[cyan, fill=cyan] [ultra thick] (18.5,2.7) circle (2pt);
        \draw[cyan, fill=cyan] [ultra thick] (3,2.7) circle (2pt);
        \draw[cyan, fill=cyan] [ultra thick] (6,5.6) circle (2pt);
        \draw[cyan, fill=cyan] [ultra thick] (9,3.5) circle (2pt);
        \draw[cyan, fill=cyan] [ultra thick] (12,1.9) circle (2pt);
        \draw[cyan, fill=cyan] [ultra thick] (15,4.1) circle (2pt);
        \draw[cyan, fill=cyan] [ultra thick] (17,2.7) circle (2pt);
        \draw[cyan, fill=cyan] [ultra thick] (24,1.7) circle (2pt);

        \draw (11,0) node[anchor=north,color=purple]{$\bQ_1+\bb_2$};
        \draw (18.5,0) node[anchor=north,color=purple]{$\bQ_2+\bb_3$};
        \draw (0,0) node[anchor=north]{$0$};
        \draw (9,0) node[anchor=north]{$\bQ_1$};
        \draw (17,0) node[anchor=north]{$\bQ_2$};
        \draw (28,0) node[anchor=north]{$\bQ_3$};
        \draw (0,7) node[anchor=west, color=gray]{{\LARGE $\bp^0=\bp_1+\bp_2+\bp_3+\bp_4$}};
        \draw (0,6) node[anchor=west, color=cyan]{{\LARGE $\bH$}};
        \draw (0,5) node[anchor=west, color=purple]{{\LARGE $\vec\bb$}};
        \draw (19,5) node[anchor=west, color=teal]{{\LARGE Brownian bridge}};
    \end{tikzpicture}}
\end{subfigure}
    \caption{An illustration of blocks, with $m=3$, $\bu=1$, $\bb_1=\varnothing$.}
    \label{fig:block}
\end{figure}

All the above objects combine together into ``blocks''.
\begin{definition} ``Blocks'' is a triplet $(\vec\bp,  \vec\bb, \bH)$: block process, virtual block process, and compatible block height.
For each $\vec \bk\in \R_+^m$, we let $\sK=\sK[\vec\bk]$ be the space of all blocks with times $\vec\bk$.
\end{definition}

\subsection{Measure on blocks}

We next fix $m$, $\vec\bk$ and equip $\sK[\vec\bk]$ with a $\sigma$--finite measure. We recall the notations \eqref{eq_discont_def_1}, \eqref{eq_discont_def_2}.

\begin{definition}  \label{defn:measurecb}
 The measure $\d(\vec\bp,  \vec\bb,\bH)$ on $\sK[\vec\bk]$ views it as a disjoint union of subspaces with fixed values of $\bu=0,1,\dots$ and $\bdel=0,1,\dots$. For the fixed $(\bu,\bdel)$ the space is further partitioned according to the values of $\bdel_{j,\ell}$.
 We equip each part with the product of Lebesgue measures over all parameters of blocks, except that whenever only a single value for a parameter is involved, then we assign weight $1$ for this value and for $\R_+\cup \{\varnothing\}$ in the definition of virtual blocks, $\varnothing$ comes with the Dirac delta measure of weight $1$, while $\R_+$ comes with the Lebesgue measure.
\end{definition}

Let us decode the definition in more detail. We start with fixing the values of non-negative integers $\bu$ and $\bdel_{j,\ell}$, $j=1,\dots,m+\bu$, $\ell=1,\dots, m$. The measure $\d(\vec\bp,  \vec\bb,\bH)$ is separately defined for each choice of these values.\footnote{Therefore, the integral over this measure can be split into an infinite sum of integrals, where the summation goes over all possible values of $\bu$ and $\bdel_{j,\ell}$.} The space of all block processes with given $\bu$, $\{\bdel_{j,\ell}\}$ can be parameterized with finitely many real numbers. For each $j\in\llbracket 1, \bu\rrbracket$ there are two options: either $\sum_\ell \bdel_{j,\ell}=0$ implying that $\bp_j$ is identical zero (and then there are no parameters to choose for it), or we should choose $\sum_\ell \bdel_{j,\ell}+\sum_\ell \bdel_{j,\ell}$ reals to specify $\bp_j$ --- the first sum counts the parameters encoding the positions of all discontinuity points for $\bp_j$ and the second sum counts the parameters encoding the values of the jumps at these points. For each $j\in\llbracket \bu+1, \bu+m\rrbracket$, the first $\bp_j=0$ case is prohibited by our definitions, while in the second case we should choose  $\sum_\ell \bdel_{j,\ell}+\sum_\ell \bdel_{j,\ell}-1$ reals to specify $\bp_j$ ---  we subtracted $1$ because the last jump of $\bp_j$, $j>m$, is predetermined by the condition that $\bp_j$ is left-continuous at $\bQ_m$. We combine all parameters together in a vector of length $r$ and equip the coordinates of this vector with the Lebesgue measure; if $r=0$ (which is possible only when $\bu=\bdel=0$), then there is a unique choice of the blocks structure and we equip it with weight $1$.

Next, when we specify the virtual block process, we need to choose an $m$--dimensional vector $\vec\bb=(\bb_1,\ldots, \bb_m)\in (\R_+\cup \{\varnothing\})^m$. The coordinates of this vector are equipped with Lebesgue measure on $\R_+$ part and weight $1$ (i.e.\ the Dirac delta measure of such weight) on $\varnothing$ part.

Finally, we specify the block height, for which we need to choose $\bdel-\bu$ real parameters: in Definition \ref{Definition_block_height} there is freedom in choosing the values of $\bH$ at the discontinuity points from the set $\bDel$, and as remarked immediately after Definition \ref{Definition_block_height}, at the last discontinuity point the inequalities uniquely determine the value $\bH(x)=\bp^0(x-)$. (That remark does not apply at other points: $\bp_j$ does not take value $0$ except to the left from the first discontinuity for all $j$ and to the right of the last discontinuity for $j>m$.) We again equip the vector of these $\bdel-\bu$ real parameters with the Lebesgue measure.

\medskip

\begin{example} \label{Example_for_measure}
 Take $m=1$, $\bu=0$, $\bdel=0$. In this situation there is a unique blocks structure: $\bp_1=0$, $\bb_1=\varnothing$. This block structure comes with weight $1$.

 In Figure \ref{fig:block} we have $m=3$, $\bu=1$, $\bdel_{1,1}=\bdel_{1,2}=\bdel_{1,3}=0$, $\bdel_{2,1}=1$, $\bdel_{2,2}=\bdel_{2,3}=0$, $\bdel_{3,1}=\bdel_{3,2}=1$, $\bdel_{3,3}=0$, $\bdel_{4,1}=0$, $\bdel_{4,2}=\bdel_{4,3}=1$, leading to $\bdel=5$. For the block process, we choose two parameters for $\bp_2$, four parameters for $\bp_3$, and three parameters for $\bp_4$. For the virtual block process, we have $\bb_1=\varnothing$ and choose two parameters $\bb_2$ and $\bb_3$. For the block height, we choose four parameters (they are blue dots strictly above $\bp^0$ in the bottom panel of Figure \ref{fig:block}).
\end{example}

\subsection{Weight of blocks via Brownian bridges} \label{ssec:wobvbb}
The value of $\bL_\beta(\vec \bk, \vec\ttt)$ is an integral of a function $\bI_{\beta,\vec\bk}:\sK[\vec\bk]\to\R$.
Given the blocks $(\vec\bp, \vec\bb, \bH)$, this function can be roughly defined as follows.
Take a Brownian bridge on $[0, \bQ_m]$, and consider the following events based on it (see the bottom panel of \Cref{fig:block}): the trajectory is
\begin{itemize}
    \item[(a)] equal to $\bH$ on the set where $\bH$ is defined
    \item[(b)] $\ge \bp^0$ on $[0, \bQ_m]$,
    \item[(c)] $\ge \bp^0(\bQ_{\ell-1}-)$ on $[\bQ_{\ell-1}, \bQ_{\ell-1}+\bb_\ell]$ for each $\ell\in\llbracket 1,m\rrbracket$ with $\bb_\ell\neq\varnothing$,
    \item[(d)] $\ge \bp^0(\bQ_{\ell-1}-)$ on $[\bQ_{\ell-1},\min\bigl(\bigcup_{j=1}^{m+\bu}\bDel_{j,\ell}\cup \{\bQ_\ell\}\bigr)$   for each $\ell\in\llbracket 1,m\rrbracket$ with $\bb_\ell=\varnothing$.
\end{itemize}
Then $\bI_{\beta,\vec\bk}[(\vec\bp, \vec\bb, \bH]$ is defined as the product of:
\begin{itemize}
\item the `probability density' of the intersection of the above events; and
\item  the expectation of the exponentiated and multiplied by $\beta^{-1}$ area between the Brownian bridge and $\bp^0$, conditional on all of the above events; and
\item $\pm 1$, depending on whether some jumps of $\bp_0$ are positive or negative.
\end{itemize}

We next give a detailed definition.
For this, we remind the reader of the following classical stochastic processes. For $W:\R_{\ge 0}\to \R$ being a standard Wiener process,
\begin{itemize}
    \item for an interval $[a, b]$, and $g, h\in \R$, $x\mapsto B(x)= W(x-a)+g+\frac{x-a}{b-a}(-W(b-a) +h-g)$ is a Brownian bridge on $[a, b]$, with $B(a)=g$ and $B(b)=h$;
    \item for $W$ conditional on $W>-\epsilon$ on $[0, \epsilon^{-1}]$, its $\epsilon\to 0+$ limit is the Bessel$_3$ process;
    \item for $a>0$, the Brownian excursion on $[0, a]$ is the $\epsilon\to 0+$ limit of $W$ on $[0, a]$, conditional on $W>-\epsilon$ on $[0, a]$ and $W(a)<\epsilon$.
\end{itemize}
For more details on these stochastic processes, see e.g., the textbooks \cite{borodin2015handbook,morters2010brownian}.

We need the following functions built out of the above three processes.
\begin{definition}   \label{defn:bIs}
For $x>0$, and $h, g\ge 0$, we define
\[
\bF(x;h,g) = \frac{1}{\sqrt{2\pi x}} \left( \exp\left(-\frac{(g-h)^2}{2x}\right)-\exp\left(-\frac{(g+h)^2}{2x}\right)\right),
\]
and
\[
\bI(x;h,g) = \bF(x;h,g)\E\left[ \exp\left( \beta^{-1}\int_0^x B(y) \d y \right) \right],
\]
where $B:[0, x]\to \R$ is a Brownian bridge with $B(0)=g$ and $B(x)=h$, conditional on $B\ge 0$. We also define
\[
\bF_0(x;h) = \frac{2h}{\sqrt{2\pi x^3}} \exp\left(-\frac{h^2}{2x}\right), \qquad
\bI_0(x;h) =\bF_0(x;h)\E\left[ \exp\left( \beta^{-1}\int_0^x B_3(y) \d y \right) \right],
\]
where $B_3:[0, x]\to \R_{\ge 0}$ is a Bessel$_3$ process, conditional on $B_3(x)=h$.
Finally,
\[
\bF_{0,0}(x)=\frac{2}{\sqrt{2\pi x^3}}, \qquad \bI_{0,0}(x) = \frac{2}{\sqrt{2\pi x^3}} \E\left[ \exp\left( \beta^{-1}\int_0^x B_e(y) \d y \right) \right],
\]
where $B_e:[0, x]\to \R_{\ge 0}$ is a Brownian excursion.
\end{definition}
Note that $\bI_0(x;h)=\lim_{g\to 0}g^{-1}\bI(x;h,g)$, and $\bI_{0,0}(x)=\lim_{h\to 0}h^{-1}\bI_0(x;h)$.

\smallskip

Let us fix $\vec \bk\in \R_+^m$, recall that $\bQ_\ell=\sum_{\ell'=1}^\ell \bk_\ell$, and also fix the blocks structure $(\vec\bp,
\vec\bb, \bH)\in \sK[\vec\bk]$. Consider the set of points on which $\bH$ is defined --- $\bDel\cup\{\bQ_\ell\mid \ell\in  \llbracket 0, m\rrbracket\}\cup\{\bQ_{\ell-1}+v_\ell \mid \ell\in  \llbracket 2, m\rrbracket,\, v_\ell\ne \varnothing\}$ --- and let us denote them $(0=x_0<x_1<x_2<\dots<x_n)$. We further let $\Xi$ be the set of $n-1$ adjacent points: $\Xi=\{(x_0,x_1), (x_1,x_2),\dots,(x_{n-1},x_n)\}$. We subdivide $\Xi$ into four subsets $\Xi_1,\Xi_2,\Xi_3,\Xi_4$, depending on the types of the constraints Definition \ref{Definition_block_height} imposed on $\bH$ in these points. For that, we introduce $A=\bp^0(x-)$ if $x=Q_\ell$, and otherwise $A=\bp^0(x)=\bp^0(y-)$; $A$ has the meaning of a lower bound for the Brownian motion in $[x, y]$. We compare possible values of $\bH(x)$ and $\bH(y)$ (dictated by Definition \ref{Definition_block_height}) with $A$: $(x,y)\in \Xi_i$ if
\begin{enumerate}
    \item[$\Xi_1:$] $\bH(x)=\bH(y)=A$. This happens when $x=\bQ_{\ell-1}$ for some $\ell \in \llbracket 1,m\rrbracket$, and either $y=\bQ_\ell$,  or $y=\bQ_{\ell-1}+\bb_\ell$, or $y=\sup\{y': \bp_j(y')>0\}$ for some $j\in\llbracket m+1, m+\bu\rrbracket$;
    \item[$\Xi_2:$] $\bH(x)=A$, $\bH(y)>A$. This happens when $x=\bQ_{\ell-1}$ for some $\ell \in \llbracket 1,m\rrbracket$, and $y\in \bDel_{j,\ell}$ with either $j\in\llbracket 1, m\rrbracket$ or  $j\in\llbracket m+1, m+\bu\rrbracket$ and $y\ne \sup\{y': \bp_j(y')>0\}$ in the latter case;
    \item[$\Xi_3:$] $\bH(x)>A$, $\bH(y)=A$. This happens when $x
    \not\in \{\bQ_\ell\}_{\ell=0}^{m-1}$, and either $y=\bQ_\ell$ for some $\ell \in \llbracket 1,m\rrbracket$, or $y=\sup\{y': \bp_j(y')>0\}$ for some $j\in\llbracket m+1, m+\bu\rrbracket$;
    \item[$\Xi_4:$] $\bH(x)>A$, $\bH(y)>A$. This happens when $x
    \not\in \{\bQ_\ell\}_{\ell=0}^{m-1}$, and $y\in \bDel_{j,\ell}$ with either $j\in\llbracket 1, m\rrbracket$ or  $j\in\llbracket m+1, m+\bu\rrbracket$ and $y\ne \sup\{y': \bp_j(y')>0\}$ in the latter case.
\end{enumerate}
Schematically, these four cases correspond to four panels in \Cref{fig:fourcases}.

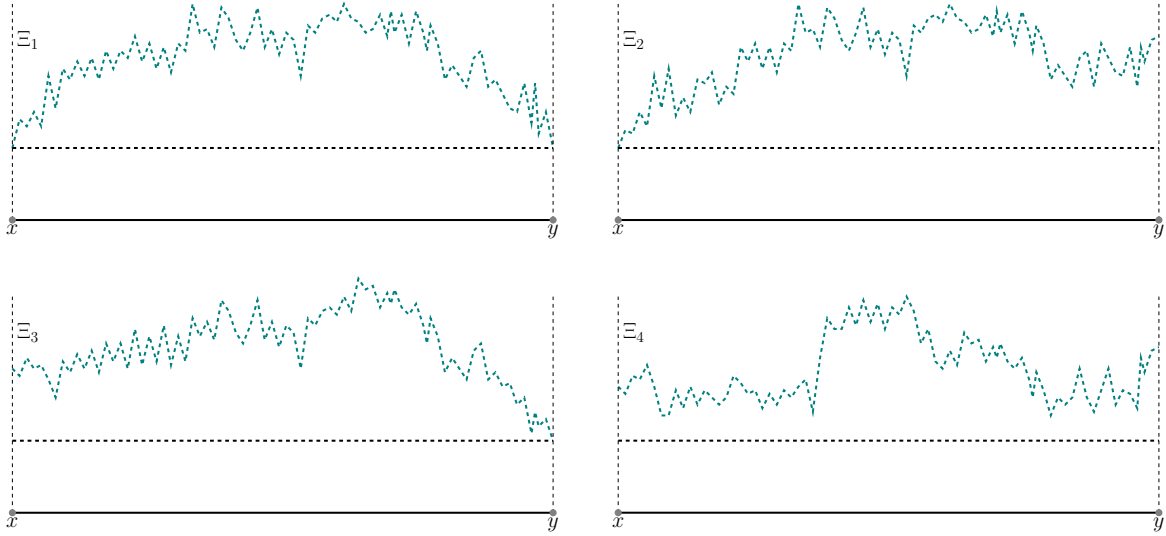
\begin{figure}[t]
    \centering
\begin{subfigure}{0.48\textwidth}
    \resizebox{0.95\textwidth}{!}{
    \begin{tikzpicture}
        \draw[teal] [dashed] [ultra thick] (0,0)--(0.2,0.8)--(0.4,0.6)--(0.6,1)--(0.8,0.6)--(1,2)--(1.2,1.1)--(1.4,2.2)--(1.6,1.9)--(1.8,2.4)--(2,2)--(2.2,2.5)--(2.4,1.9)--(2.6,2.7)--(2.8,2.2)--
        (3.0,2.7)--(3.2,2.5)--(3.4,3.1)--(3.6,2.4)--(3.8,2.9)--(4.0,2.2)--(4.2,2.8)--(4.4,2.1)--(4.6,2.9)--(4.8,2.7)--(5.0,4.0)--(5.2,3.2)--(5.4,3.3)--(5.6,2.8)--(5.8,3.9)--(6.0,3.6)--(6.2,3.0)--(6.4,2.7)--(6.6,3.2)--(6.8,3.9)--(7.0,2.8)--(7.2,3.3)--(7.4,2.6)--(7.6,3.2)--(7.8,3.0)--(8.0,2.0)--(8.2,3.4)--(8.4,3.2)--(8.6,3.6)--(8.8,3.7)--(9.0,3.5)--(9.2,4.0)--(9.4,3.6)--(9.6,3.5)--(9.8,3.2)--(10.0,3.3)--(10.2,3.7)--(10.4,3.1)--(10.5,3.8)--(10.6,3.2)--(10.8,3.7)--(11.0,2.9)--(11.2,3.8)--(11.4,3.2)--(11.5,2.7)--(11.6,3.4)--(11.8,2.9)--(12.0,1.9)--(12.2,2.3)--(12.4,2.0)--(12.6,1.7)--(12.8,2.5)--(13.0,2.7)--(13.2,1.7)--(13.4,1.9)--(13.6,1.5)--(13.8,1.1)--(14.0,1.0)--(14.2,1.8)--(14.4,0.7)--(14.5,1.8)--(14.6,0.4)--(14.8,1.0)--(15.0,0);

        \draw [dashed] (0,-2)--(0,4);
        \draw [dashed] (15,-2)--(15,4);
        \draw [ultra thick] [dashed] (0,0)--(15,0);

        \draw [ultra thick] (0,-2)--(15,-2);

        \draw (0,-2) node[anchor=north]{{\LARGE $x$}};
        \draw (15,-2) node[anchor=north]{{\LARGE $y$}};
        \draw (0,3) node[anchor=west]{{\LARGE $\Xi_1$}};
        \draw[gray, fill=gray] [ultra thick] (0,-2) circle (2pt);
        \draw[gray, fill=gray] [ultra thick] (15,-2) circle (2pt);
    \end{tikzpicture}}
\end{subfigure}
\begin{subfigure}{0.48\textwidth}
    \resizebox{0.95\textwidth}{!}{
    \begin{tikzpicture}
        \draw[teal] [dashed] [ultra thick] (0,0)--(0.2,0.5)--(0.4,0.4)--(0.6,1)--(0.8,0.6)--(1,2)--(1.2,1.1)--(1.4,2.2)--(1.6,0.9)--(1.8,1.4)--(2,1)--(2.2,1.9)--(2.4,1.8)--(2.6,2.1)--(2.8,1.2)--
        (3.0,1.7)--(3.2,1.5)--(3.4,2.8)--(3.6,2.4)--(3.8,2.9)--(4.0,2.2)--(4.2,2.8)--(4.4,2.1)--(4.6,2.9)--(4.8,2.7)--(5.0,4.0)--(5.2,3.2)--(5.4,3.3)--(5.6,2.8)--(5.8,3.9)--(6.0,3.6)--(6.2,3.0)--(6.4,2.7)--(6.6,3.2)--(6.8,3.9)--(7.0,2.8)--(7.2,3.3)--(7.4,2.6)--(7.6,3.2)--(7.8,3.0)--(8.0,2.0)--(8.2,3.4)--(8.4,3.2)--(8.6,3.6)--(8.8,3.7)--(9.0,3.5)--(9.2,4.0)--(9.4,3.6)--(9.6,3.5)--(9.8,3.2)--(10.0,3.3)--(10.2,3.7)--(10.4,3.1)--(10.5,3.8)--(10.6,3.2)--(10.8,3.7)--(11.0,2.9)--(11.2,3.8)--(11.4,3.2)--(11.5,2.7)--(11.6,3.4)--(11.8,2.9)--(12.0,1.9)--(12.2,2.3)--(12.4,2.0)--(12.6,1.7)--(12.8,2.5)--(13.0,2.7)--(13.2,1.7)--(13.4,2.9)--(13.6,2.5)--(13.8,2.1)--(14.0,2.0)--(14.2,2.8)--(14.4,1.7)--(14.5,2.8)--(14.6,2.4)--(14.8,3.0)--(15.0,3.1);

        \draw [dashed] (0,-2)--(0,4);
        \draw [dashed] (15,-2)--(15,4);
        \draw [ultra thick] [dashed] (0,0)--(15,0);

        \draw [ultra thick] (0,-2)--(15,-2);

        \draw (0,-2) node[anchor=north]{{\LARGE $x$}};
        \draw (15,-2) node[anchor=north]{{\LARGE $y$}};
        \draw (0,3) node[anchor=west]{{\LARGE $\Xi_2$}};
        \draw[gray, fill=gray] [ultra thick] (0,-2) circle (2pt);
        \draw[gray, fill=gray] [ultra thick] (15,-2) circle (2pt);
    \end{tikzpicture}}
\end{subfigure}
\par\bigskip
\begin{subfigure}[b]{0.48\textwidth}
    \resizebox{0.95\textwidth}{!}{
    \begin{tikzpicture}
        \draw[teal] [dashed] [ultra thick] (0,2)--(0.2,1.8)--(0.4,2.3)--(0.6,2)--(0.8,2.1)--(1,1.7)--(1.2,1.2)--(1.4,2.2)--(1.6,1.9)--(1.8,2.4)--(2,2)--(2.2,2.5)--(2.4,1.9)--(2.6,2.7)--(2.8,2.2)--
        (3.0,2.7)--(3.2,2)--(3.4,3.1)--(3.6,2.1)--(3.8,2.9)--(4.0,2.2)--(4.2,3.2)--(4.4,2.1)--(4.6,2.9)--(4.8,2.2)--(5.0,3.4)--(5.2,2.9)--(5.4,3.3)--(5.6,2.8)--(5.8,3.9)--(6.0,3.6)--(6.2,3.0)--(6.4,2.7)--(6.6,3.2)--(6.8,3.9)--(7.0,2.8)--(7.2,3.3)--(7.4,2.6)--(7.6,3.2)--(7.8,3.0)--(8.0,2.0)--(8.2,3.4)--(8.4,3.2)--(8.6,3.6)--(8.8,3.7)--(9.0,3.5)--(9.2,4.0)--(9.4,3.6)--(9.6,4.5)--(9.8,4.2)--(10.0,4.3)--(10.2,3.7)--(10.4,4.1)--(10.5,3.8)--(10.6,4.2)--(10.8,3.7)--(11.0,3.5)--(11.2,3.8)--(11.4,3.2)--(11.5,2.7)--(11.6,3.4)--(11.8,2.9)--(12.0,1.9)--(12.2,2.3)--(12.4,2.0)--(12.6,1.7)--(12.8,2.5)--(13.0,2.7)--(13.2,1.7)--(13.4,1.9)--(13.6,1.5)--(13.8,1.6)--(14.0,1.0)--(14.2,1.2)--(14.4,0.2)--(14.5,0.8)--(14.6,0.4)--(14.8,0.6)--(15.0,0);

        \draw [dashed] (0,-2)--(0,4);
        \draw [dashed] (15,-2)--(15,4);
        \draw [ultra thick] [dashed] (0,0)--(15,0);

        \draw [ultra thick] (0,-2)--(15,-2);

        \draw (0,-2) node[anchor=north]{{\LARGE $x$}};
        \draw (15,-2) node[anchor=north]{{\LARGE $y$}};
        \draw (0,3) node[anchor=west]{{\LARGE $\Xi_3$}};
        \draw[gray, fill=gray] [ultra thick] (0,-2) circle (2pt);
        \draw[gray, fill=gray] [ultra thick] (15,-2) circle (2pt);
    \end{tikzpicture}}
\end{subfigure}
\begin{subfigure}[b]{0.48\textwidth}
    \resizebox{0.95\textwidth}{!}{
    \begin{tikzpicture}
        \draw[teal] [dashed] [ultra thick] (0,3)--(0.2,2.8)--(0.4,3.3)--(0.6,3.2)--(0.8,3.6)--(1,3)--(1.2,2.2)--(1.4,2.2)--(1.6,2.9)--(1.8,2.4)--(2,3)--(2.2,2.5)--(2.4,2.9)--(2.6,2.7)--(2.8,2.5)--
        (3.0,2.7)--(3.2,3.3)--(3.4,3.1)--(3.6,2.8)--(3.8,2.9)--(4.0,2.4)--(4.2,2.8)--(4.4,2.5)--(4.6,2.9)--(4.8,2.7)--(5.0,3.0)--(5.2,3.2)--(5.4,2.3)--(5.6,3.8)--(5.8,4.9)--(6.0,4.6)--(6.2,4.6)--(6.4,5.2)--(6.6,4.7)--(6.8,5.4)--(7.0,4.8)--(7.2,5.3)--(7.4,4.6)--(7.6,5.2)--(7.8,5.0)--(8.0,5.5)--(8.2,5.1)--(8.4,4.2)--(8.6,4.6)--(8.8,3.7)--(9.0,3.5)--(9.2,4.0)--(9.4,3.6)--(9.6,4.5)--(9.8,4.2)--(10.0,4.3)--(10.2,3.7)--(10.4,4.1)--(10.5,3.8)--(10.6,4.2)--(10.8,3.7)--(11.0,3.5)--(11.2,3.8)--(11.4,3.2)--(11.5,2.7)--(11.6,3.4)--(11.8,2.9)--(12.0,2.2)--(12.2,2.7)--(12.4,2.5)--(12.6,2.7)--(12.8,3.5)--(13.0,2.7)--(13.2,2.3)--(13.4,2.9)--(13.6,3.5)--(13.8,2.5)--(14.0,2.9)--(14.2,2.8)--(14.4,2.4)--(14.5,3.8)--(14.6,3.4)--(14.8,4.0)--(15.0,4.1);

        \draw [dashed] (0,-0.5)--(0,5.5);
        \draw [dashed] (15,-0.5)--(15,5.5);
        \draw [ultra thick] [dashed] (0,1.5)--(15,1.5);

        \draw [ultra thick] (0,-0.5)--(15,-0.5);

        \draw (0,-0.5) node[anchor=north]{{\LARGE $x$}};
        \draw (15,-0.5) node[anchor=north]{{\LARGE $y$}};
        \draw (0,4.5) node[anchor=west]{{\LARGE $\Xi_4$}};
        \draw[gray, fill=gray] [ultra thick] (0,-0.5) circle (2pt);
        \draw[gray, fill=gray] [ultra thick] (15,-0.5) circle (2pt);
    \end{tikzpicture}}
\end{subfigure}
    \caption{An illustration of the four cases in $\Xi$.}
    \label{fig:fourcases}
\end{figure}

\smallskip

\begin{example} The bottom panel of Figure \ref{fig:block} has ten segments. Enumerating them from left to right: $\Xi_1$ contains segments $4$ and $8$; $\Xi_2$ contains segment $1$; $\Xi_3$ contains segments $3$, $7$, $9$, $10$; $\Xi_4$ contains segments $2$, $5$, $6$.
\end{example}

\smallskip

Using the above notations, we can now define the weight function for blocks.

\begin{definition}  \label{def:xibi} Using the above $\Xi$, $\Xi_1$, $\Xi_2$, $\Xi_3$, $\Xi_4$ notations, we define
\begin{equation} \label{eq:defnibk}
\bI_{\beta,\vec\bk}[\vec\bp,\vec\bb, \bH]=2^{-\bdel-|\{\ell\in\llbracket 1,m\rrbracket: \bb_\ell\neq\varnothing\}|}\prod_{(x,y)\in\Xi} \bI_\beta[x,y],
\end{equation}
where
\begin{equation*}
\bI_\beta[x,y]=
\begin{cases}
\exp\big(\beta^{-1}(y-x)(\bH(x)-\bp^0(x))\big)\bI_{0,0}(y-x), & (x,y) \in \Xi_1, \\
\exp\big(\beta^{-1}(y-x)(\bH(x)-\bp^0(x))\big) \bI_0(y-x; \bH(y)-\bH(x)), & (x,y) \in \Xi_2, \\
(-1)^{\don[\bp^0(x)<\bp^0(x-)]}\bI_0(y-x;\bH(x)-\bH(y)) , & (x,y) \in \Xi_3, \\
(-1)^{\don[\bp^0(x)<\bp^0(x-)]} \bI(y-x;\bH(x)-\bp^0(x), \bH(y)-\bp^0(y_-) ), & (x,y) \in \Xi_4.
\end{cases}
\end{equation*}
\end{definition}
Heuristically, the factor $\exp(\beta^{-1}(y-x)(\bH(x)-\bp^0(x)))$ in the first two cases accounts for the (exponentiated) area between $\bp^0$ and the red dashed horizontal line in Figure \ref{fig:block} corresponding to virtual blocks. This factor combines with $\bI_{0,0}$ or $\bI_{0}$ to produce the full area between the Brownian bridge and $\bp^0$. The power of $2$ in \eqref{eq:defnibk} does not have any particular meaning, it will come out of the rescaling of random walks to Brownian motion in the Section \ref{sec:gme}. 

\medskip

Here is the expression for the mixed moments of the Laplace transform, as in Theorem \ref{thm:main}.
\begin{definition}  \label{defn:core}
Take any $m\in \N$, $\vec \bk \in \R^m_+ $, and $\vec \ttt\in \R^m$ such that $\ttt_1\le \cdots \le \ttt_m$, and denote $\bQ_\ell = \sum_{\ell'=1}^\ell \bk_{\ell'}$ for each $\ell\in\llbracket 1,m\rrbracket$.
We let
\[
\bL_\beta(\vec\bk, \vec\ttt) = \mathrm{p.v.}\int_{\sK[\vec\bk]} \exp\left(\sum_{\ell=1}^{m-1}
(\ttt_\ell-\ttt_{\ell+1}) \bH(\bQ_\ell)/2 \right)\bI_{\beta,\vec\bk}[\vec \bp, \vec\bb, \bH] \d(\vec\bp,  \vec\bb,\bH),
\]
where the principal value integral is taken in the following sense.
For any small enough $\epsilon>0$, we let  $\sK_\epsilon=\sK_\epsilon[\vec\bk]$ denote the space of all blocks $(\vec\bp, \vec\bb, \bH)\in \sK[\vec \bk]$, such that for each $\ell\in\llbracket 1,m\rrbracket$, $\bp^0$ is constant in $(\bQ_{\ell-1}, \bQ_{\ell-1}+\epsilon)$, and either $\bb_\ell>\epsilon$, or $\bb_\ell=\varnothing$.
We integrate over $\sK_\epsilon$, then send $\epsilon\to 0+$.
\end{definition}
We note that it is essential to take a principal value integral here, since a direct integral over all blocks in $\sK[\vec \bk]$ would not be absolute convergent, because $\bI_{0,0}$ and $\bI_0(\cdot;h)$ (for any $h>0$) are not integrable as functions of $x$.
In fact, a priori, it is also unclear whether the principal value integral is well-defined (i.e., whether the $\epsilon\to 0+$ limit exists; moreover, why the integral over $\sK_\epsilon$ is well-defined). These will be justified late, see \Cref{prop:fixedepsconv} and \Cref{prop:epszerocov}.

\
\section{Joint moments through Dunkl operators}  \label{sec:dunkl}

In this section, we introduce our main technical tool for analyzing edge limits of both corners processes and DBM, which is the \emph{Dunkl operators.}

We deal with functions in variables $x_1,x_2,x_3,\dots$ and let $\sigma_{ij}$ denote the operator swapping $x_i$ with $x_j$:
 $$
 [\sigma_{ij}f](x_1,\dots,x_N)=f(x_1,\dots,x_N)_{x_i\leftrightarrow x_j}.
$$
We further set for each $N\in \N$ and $i\in\llbracket 1, N\rrbracket$
$$
 \D_i^N= \frac{\partial}{\partial x_i} + \frac{\beta}{2} \sum_{\begin{smallmatrix} j\in\llbracket 1,N\rrbracket,\\ j\ne i\end{smallmatrix}} \frac{1-\sigma_{ij}}{x_i-x_j}.
$$
\begin{lemma} \label{Lemma_Dunkls_commute} For each $N\in\N$, the operators $\D_1^N$, $\D_2^N$, \dots, $\D_N^N$ commute.
\end{lemma}
\begin{proof}
 See, \cite{dunkl1989differential}, or \cite{Ki_lect}, or \cite{Et_lect}.
\end{proof}
Next, define
$$
 \P_k^N=\sum_{i=1}^N (\D_i^N)^k.
$$

\begin{theorem} \label{Theorem_corners_moments}
 Let $\{y^n_i\}_{1\leq i\leq k\leq N}\in \GT_N$ be the G$\beta$E corners process of  variance $\tau$. Fix $m=1,2,\dots$ and take $2m$ positive integers $N_m\le N_{m-1}\le \dots \le N_1\le N$ and $k_1,k_2,\dots,k_m$. We have
 \begin{equation} \label{eq_corners_moments}
  \E \left[ \prod_{\ell=1}^m \left(\sum_{i=1}^{N_\ell} \bigl(y^{N_\ell}_i\bigr)^{k_\ell}\right) \right]= \P_{k_m}^{N_m} \cdots \P_{k_2}^{N_2} \P_{k_1}^{N_1} \left[\exp\left(\frac{\tau}{2} \sum_{i=1}^N (x_i)^2\right)\right]_{x_1=\dots=x_N=0}.
 \end{equation}
\end{theorem}
\begin{remark}
 While the operators $\P_{k_1}^N$ and $\P_{k_2}^N$ commute by \Cref{Lemma_Dunkls_commute}, the operators $\P_{k_1}^{N_1}$ and $P_{k_2}^{N_2}$ do not for $N_1\ne N_2$. Hence, the order of application of the operators in the right-hand side is important in \eqref{eq_corners_moments}.
\end{remark}

\begin{theorem}\label{Theorem_DBM_moments} Fix $N$ and let $(Y_1(\tau)\ge Y_2(\tau)\ge \dots\ge Y_N(\tau))$ denote the DBM started from zero initial condition. For each $m=1,2,\dots$ and each $0\le \tau_1\le \tau_2\le \dots\le \tau_m$, we have
 \begin{multline} \label{eq_DBM_moments}
  \E \left[ \prod_{\ell=1}^m \left(\sum_{i=1}^{N} \Bigl(Y_i(\tau_\ell)\Bigr)^{k_\ell}\right) \right]= \P_{k_m}^{N} \Biggl[\exp\left(\frac{\tau_m-\tau_{m-1}}{2} \sum_{i=1}^N (x_i)^2\right)  \P_{k_{m-1}}^{N} \Biggl[ \\  \cdots \P_{k_2}^{N} \Biggl[\exp\left(\frac{\tau_2-\tau_1}{2} \sum_{i=1}^N (x_i)^2\right) \P_{k_1}^{N} \Biggl[\exp\left(\frac{\tau_1}{2} \sum_{i=1}^N (x_i)^2\right)\Biggr]\dots\Biggr]_{x_1=\dots=x_N=0},
 \end{multline}
 where in the product we alternate applications of $\P_{k_\ell}^N$ and multiplications by $\exp(\cdot)$.
\end{theorem}

\subsection{Multivariate Bessel functions}
In the rest of this section, we prove \Cref{Theorem_corners_moments,Theorem_DBM_moments}.

\begin{definition} \label{Definition_Bessel_function}
For any reals $\lambda_1\ge \lambda_2\ge \dots \ge \lambda_N$, the \emph{multivariate Bessel function} $\B_{(\lambda_1, \ldots, \lambda_N)}(x_1, \ldots, x_N; \beta)$ is a function of $N$ complex variables $x_1,\dots,x_N$ defined as
\begin{equation}\label{eq_Bessel_combinatorial}
 \B_{(\lambda_1,\dots,\lambda_N)}(x_1,\dots,x_N;\,\beta)= \E_{\{y^k_i\}}\!\left[\exp\left(\sum_{k=1}^{N} x_k
 \cdot \left(\sum_{i=1}^{k} y_i^k-\sum_{j=1}^{k-1} y_j^{k-1}\right)  \right) \right],
\end{equation}
where the expectation is with respect to $\beta$-corners process with top row $(\lambda_1,\dots,\lambda_N)$ from Definition \ref{def_betacorner}.
\end{definition}
This definition (which can be found, e.g., in \cite{GK}) implies that
\begin{equation}
\label{eq_Bessel_normalization}
\B_{(\lambda_1, \dots, \lambda_N)}(0, \dots, 0; \beta) = 1.
\end{equation}
Alternatively, the multivariate Bessel function can be defined as the semiclassical limit of Jack symmetric polynomials (see, e.g., \cite[Section 4]{Ok_Olsh_shifted_Jack}):
\begin{equation}
\label{eq_J_to_B}
 \B_{\lambda_1,\dots,\lambda_N}(x_1,\dots,x_N; \beta)=\lim_{\eps\to 0} \frac{J_{\lfloor \eps^{-1} \lambda_1\rfloor ,\dots,\lfloor \eps^{-1} \lambda_N\rfloor }(1+\eps x_1,\dots,1+\eps x_N; \beta/2)}{J_{\lfloor\eps^{-1} \lambda_1\rfloor ,\dots,\lfloor \eps^{-1} \lambda_N\rfloor}(1,\dots,1; \beta/2)}.
\end{equation}
If we take \eqref{eq_J_to_B} as a definition, then \eqref{eq_Bessel_combinatorial} is obtained as a limit of the combinatorial formula for Jack polynomials.
\cite{O} shows that $\B_{(\lambda_1,\dots,\lambda_N)}(x_1,\dots,x_N;\beta)$ admits an analytic continuation on the $2N+1$ variables $\lambda_1, \dots, \lambda_N$, $x_1, \dots, x_N, \beta$, to an open subset of $\C^{2N+1}$ containing $\{(\lambda_1, \cdots, \lambda_N, x_1, \cdots, x_N, \beta)\in\C^{2N+1} \mid \mathrm{Re}(\beta)\ge 0\}$. In particular, for a fixed $\beta>0$,  $\B_{(\lambda_1,\dots,\lambda_N)}(x_1,\dots,x_N;\beta)$ is an entire function on the variables $\lambda_1, \cdots, \lambda_N, x_1, \cdots, x_N$.

Another important property is that the $\B_{(\lambda_1,\dots,\lambda_N)}(x_1,\dots,x_N;\,\beta)$ is \emph{symmetric} in its arguments $x_1,\dots,x_N$. In the particular case $\beta=2$, there is an explicit determinantal formula, which makes all these properties especially transparent:
\begin{equation}\label{eq_Bessel_1}
  B_{(a_1,\dots,a_N)}(x_1,\dots,x_N;\,2)= 1!\cdot 2! \cdots (N-1)! \cdot  \frac{\det\bigl[ e^{a_i x_j}\bigr]_{i,j=1}^N}{\prod_{i<j} (x_i-x_j)(a_i-a_j)}.
\end{equation}

We interested in Bessel functions because of their link to the Dunkl operators:
\begin{theorem}[\cite{O}]\label{thm:opdam}
For each $k,N=1,2,\dots,$ and each $\lambda_1\ge \lambda_2 \ge \dots \ge \lambda_N$, we have
\begin{equation}\label{eq_eigenrelation}
\P_k^N \B_{(\lambda_1,\dots,\lambda_N)}(x_1,\dots,x_N; \beta) = \left(\sum_{i=1}^N{\lambda_i^k}\right)\cdot \B_{(\lambda_1,\dots,\lambda_N)}(x_1,\dots,x_N; \beta).
\end{equation}
\end{theorem}

Let us point to several reviews covering various aspects of the Dunkl theory and its generalizations: \cite{A}, \cite{CM}, \cite{Et_lect},  \cite{Ki_lect}.

\subsection{Proof of Theorem \ref{Theorem_corners_moments}}
The proof is based on the following integral identity. In the terminology of \cite{BGC} this is a computation of the Bessel Generating Function (BGF) of the Gaussian $\beta$ Ensemble. In terminology of \cite{A} this is a computation of its (symmetric) Dunkl transform.

\begin{lemma}
 Let $\lambda_1>\lambda_2>\dots>\lambda_N$ be distributed as the G$\beta$E of \eqref{eq_GBE_t} with parameters $N$ and $\tau$. For any complex numbers $x_1,\dots,x_N$, we have
 \begin{equation}
  \label{eq_GBE_Fourier}
  \E_{\lambda_1,\dots,\lambda_N}\bigl[\B_{\lambda_1,\dots,\lambda_N}(x_1,\dots,x_N;\, \beta)\bigr]= \exp\left(\frac{\tau}{2} \sum_{i=1}^N (x_i)^2\right).
 \end{equation}
\end{lemma}
\begin{proof}
 At $N=1$, \eqref{eq_GBE_Fourier} is a well-known computation of the Laplace/Fourier transform of the Gaussian density. The general case and many references are discussed in \cite[Lemma 4.5]{BGC}.
\end{proof}

\begin{proof}[Proof of Theorem \ref{Theorem_corners_moments}]
 We start from the identity \eqref{eq_GBE_Fourier} in $N_1$ variables and apply $\P_{k_1}^{N_1}$ to both sides using the eigenrelation \eqref{eq_eigenrelation}. We get
  \begin{equation}
  \label{eq_GBE_Fourier_1}
  \E_{y_1^{N_1},\dots,y_{N_1}^{N_1}}\left[\sum_{i=1}^{N_1} \left(y_i^{N_1}\right)^{k_1}\cdot \B_{y_1^{N_1},\dots,y_{N_1}^{N_1}}(x_1,\dots,x_{N_1};\, \beta)\right]= \P_{k_1}^{N_1} \exp\left(\frac{\tau}{2} \sum_{i=1}^N (x_i)^2\right).
 \end{equation}
 Next, we plug $x_{N_2+1}=x_{N_2+2}=\dots=x_{N_1}=0$ into \eqref{eq_GBE_Fourier_1} and use the following identity: if $\{y_i^j\}_{1\le i \le j \le N}$ is the $\beta$--corners process of Definition \ref{def_betacorner} with arbitrary top row, then for each $M\le N$, we have
 \begin{equation}
 \label{eq_Bessel_branching}
  \B_{y_1^N,y_2^N\dots,y_N^N}(x_1,\dots,x_M, 0^{N-M};\, \beta)=\E_{y_1^M,y_2^M,\dots,y_M^M}\left[\B_{y_1^M,y_2^M,\dots,y_M^M}(x_1,\dots,x_M;\, \beta)\right].
 \end{equation}
 The identity \eqref{eq_Bessel_branching} is readily seen from \eqref{eq_Bessel_combinatorial}. Hence, the result of plugging zeros into \eqref{eq_GBE_Fourier_1} is
  \begin{multline}
  \label{eq_GBE_Fourier_2}
  \E_{y_1^{N_1},\dots,y_{N_1}^{N_1};\,\,\, y_1^{N_2},\dots,y_{N_2}^{N_2}}\left[\sum_{i=1}^{N_1} \left(y_i^{N_1}\right)^{k_1}\cdot \B_{y_1^{N_2},\dots,y_{N_2}^{N_2}}(x_1,\dots,x_{N_2};\, \beta)\right]\\= \left[\P_{k_1}^{N_1} \exp\left(\frac{\tau}{2} \sum_{i=1}^N (x_i)^2\right)\right]_{x_{N_2+1}=x_{N_2+2}=\dots=x_{N_1}=0}.
 \end{multline}
 Applying $\P_{k_2}^{N_2}$, we get
  \begin{multline}
  \label{eq_GBE_Fourier_3}
  \E_{y_1^{N_1},\dots,y_{N_1}^{N_1};\,\,\, y_1^{N_2},\dots,y_{N_2}^{N_2}}\left[\sum_{i=1}^{N_1} \left(y_i^{N_1}\right)^{k_1}\cdot \sum_{i=1}^{N_2} \left(y_i^{N_2}\right)^{k_2}\cdot \B_{y_1^{N_2},\dots,y_{N_2}^{N_2}}(x_1,\dots,x_{N_2};\, \beta)\right]\\= \P_{k_2}^{N_2} \left[\P_{k_1}^{N_1} \exp\left(\frac{\tau}{2} \sum_{i=1}^N (x_i)^2\right)\right]_{x_{N_2+1}=x_{N_2+2}=\dots=x_{N_1}=0}.
 \end{multline}
 We further plug  $x_{N_3+1}=x_{N_3+2}=\dots=x_{N_2}=0$, then apply $\P_{k_3}^{N_3}$, and continue recursively. Finally, after we have applied $\P_{k_m}^{N_m}$, we set $x_1=\dots=x_{N_m}=0$, use the normalization condition \eqref{eq_Bessel_normalization} and arrive at \eqref{eq_corners_moments}.
\end{proof}

\subsection{Proof of Theorem \ref{Theorem_DBM_moments}}
The proof is based on the representation of transition probabilities of the DBM in terms of the Bessel functions.
\begin{lemma}\label{Lemma_DBM_B}
 Fix $N$ and suppose that the DBM of \eqref{eq_Dyson_BM} is started at time $0$ from a deterministic configuration $(Y_1(0)\ge Y_2(0)\ge \dots Y_N(0))$. Then at time $\tau>0$, for any $x_1,\dots,x_N\in\mathbb C$, we have
 \begin{multline}
 \label{eq_DBM_by_B}
  \E_{Y_1(\tau),Y_2(\tau),\dots,Y_N(\tau)}\left[ \B_{Y_1(\tau),\dots, Y_N(\tau)} (x_1,\dots,x_N;\, \beta)\right]\\= \B_{Y_1(0),\dots, Y_N(0)} (x_1,\dots,x_N;\, \beta) \cdot \exp\left(\frac{\tau}{2} \sum_{i=1}^N (x_i)^2\right).
 \end{multline}
\end{lemma}
\begin{remark} If $(Y_1(0),\dots,Y_N(0))=(0,\dots,0)$, then the law of $Y_1(\tau),Y_2(\tau),\dots,Y_N(\tau)$ is the G$\beta$E of \eqref{eq_GBE_t}, and \eqref{eq_DBM_by_B} turns into \eqref{eq_GBE_Fourier}.
\end{remark}
\begin{proof}[Proof of Lemma \ref{Lemma_DBM_B}.]
We use a formula for the transition density of DBM. It reads for $z_1\ge \dots\ge z_N$ and $y_1\ge \dots\ge z_N$:
\begin{multline}
\label{eq_DBM_transition}
 P_\tau\bigl((z_1,\dots,z_N)\to(y_1,\dots,y_N);\, \beta\bigr)= \tau^{-\frac{N}{2}+\beta \frac{N(N-1)}{4}}\frac{N!}{(2\pi)^{N/2}} \prod_{i=1}^{N} \frac{\Gamma\left(1+\tfrac{\beta}{2}\right)}{\Gamma\left(1+j\tfrac{\beta}{2}\right)}\\ \times \exp\left(-\sum_{i=1}^N \frac{z_i^2+y_i^2}{2\tau}\right)\, \prod_{i<j} (y_i-y_j)^{\beta}\, \B_{z_1,\dots,z_N}\left(\frac{y_1}{\tau},\dots,\frac{y_N}{\tau}\right) \d y_1\d y_2\cdots \d y_N.
\end{multline}
The last formula can be taken as a definition of DBM: for the interplay between various points of view on DBM, see \cite{baker1997calogero}, \cite{rosler1998generalized}, \cite[Section 2]{andraus2012interacting}, \cite{voit2023freezing} and references therein.
We remark that for $N=1$, \eqref{eq_DBM_transition} is a familiar Gaussian density, which matches the transition probability of the standard Brownian motion:
\begin{equation}
\label{eq_DBM_transition_N1}
 P_\tau\bigl(z\to y;\, \beta\bigr)=\frac{1}{(2\pi \tau)^{1/2}} \exp\left(-\frac{z^2+y^2}{2\tau}\right) \exp\left(\frac{z y}{\tau}\right) \d y
\end{equation}
Equivalence of \eqref{eq_DBM_transition_N1} with $N=1$ version of \eqref{eq_DBM_by_B} is a standard computation of direct (or inverse) Fourier transform of the Gaussian density. Similarly, \eqref{eq_DBM_transition} is equivalent to \eqref{eq_DBM_by_B} by the inversion formula for the Dunkl transform, see, e.g.\ \cite[Section 3.3]{A} and references therein.
\end{proof}

\begin{proof}[Proof of \Cref{Theorem_DBM_moments}]
 We start from the identity \eqref{eq_GBE_Fourier} for $\tau=\tau_1$ variables and apply $\P_{k_1}^{N}$ to both sides using the eigenrelation \eqref{eq_eigenrelation}. We get
 \begin{equation}
  \label{eq_DBM_Fourier_1}
  \E_{Y_1(\tau_1),\dots,Y_{1}(\tau_1)}\left[\sum_{i=1}^{N} \left(Y_i(\tau_1)\right)^{k_1}\cdot \B_{Y_1(\tau_1),\dots,Y_1(\tau_1)}(x_1,\dots,x_{N};\, \beta)\right]= \P_{k_1}^{N} \exp\left(\frac{\tau_1}{2} \sum_{i=1}^N (x_i)^2\right).
 \end{equation}
Next, we apply the identity \eqref{eq_DBM_by_B} with $\tau=\tau_2-\tau_1$, transforming \eqref{eq_DBM_Fourier_1} into:
 \begin{multline}
  \label{eq_DBM_Fourier_2}
  \E_{Y_1(\tau_1),\dots,Y_{1}(\tau_1);\,\, Y_1(\tau_2),\dots,Y_{1}(\tau_2)}\left[\sum_{i=1}^{N} \left(Y_i(\tau_1)\right)^{k_1}\cdot \B_{Y_1(\tau_2),\dots,Y_1(\tau_2)}(x_1,\dots,x_{N};\, \beta)\right]\\= \exp\left(\frac{\tau_2-\tau_1}{2} \sum_{i=1}^N (x_i)^2\right) \P_{k_1}^{N} \exp\left(\frac{\tau_1}{2} \sum_{i=1}^N (x_i)^2\right).
 \end{multline}
 Applying $\P_{k_2}^{N}$ to both sides, we get:
\begin{multline}
  \label{eq_DBM_Fourier_3}
  \E_{Y_1(\tau_1),\dots,Y_{1}(\tau_1);\,\, Y_1(\tau_2),\dots,Y_{1}(\tau_2)}\left[\sum_{i=1}^{N} \left(Y_i(\tau_1)\right)^{k_1}\cdot \sum_{i=1}^{N} \left(Y_i(\tau_2)\right)^{k_2}\cdot \B_{Y_1(\tau_2),\dots,Y_1(\tau_2)}(x_1,\dots,x_{N};\, \beta)\right]\\= \P_{k_2}^{N} \left[\exp\left(\frac{\tau_2-\tau_1}{2} \sum_{i=1}^N (x_i)^2\right) \P_{k_1}^{N} \exp\left(\frac{\tau_1}{2} \sum_{i=1}^N (x_i)^2\right)\right].
 \end{multline}
 Further iterating this procedure and finally plugging $x_1=\dots=x_N=0$ after applying $\P_{k_m}^N$, we arrive at \eqref{eq_DBM_moments}.
\end{proof}

\section{Scaling limit of general moments}   \label{sec:gme}

\subsection{Asymptotic statements}

In this section we prove  the convergence of high moments, for both the G$\beta$E corners process and the DBM.
For the G$\beta$E corners process, we take the variance $\tau=2N/\beta$, so that  the top level particles would fill in the interval $[-2N, 2N]$ as $N\to\infty$.
For the DBM, we take slices around the time $2N/\beta$. In our notations, whenever possible, we use bold letters for the quantities related to the limiting objects and non-bold fonts for their prelimit counterparts.

\begin{theorem}  \label{thm:multil}
Choose $m\in \N$, $\vec \bk \in \R^m_+ $, and $\vec \ttt\in \R_{\ge 0}^m$ such that $\ttt_1\le \cdots \le \ttt_m$, any real $C_1>0$, and a large enough $N\in \N$. We take the rank $N$ G$\beta$E corners process $\{y^n_i\}_{1\leq i\leq n\leq N}$, with variance $\tau=2N/\beta$; and $\vec k \in \Z^m$, and integers $N\ge N_1\ge \cdots \ge N_m$,  such that
\begin{equation} \label{eq_x1}
|k_\ell-\bk_\ell N^{2/3}|<C_1, \qquad  |N - \ttt_\ell N^{2/3}-N_\ell|<C_1, \qquad \text{for each }\ell\in\llbracket 1,m\rrbracket.
\end{equation}
Then as $N\to\infty$, uniformly over the ($N$--dependent) choices of $\vec k$ and $\vec N$ satisfying \eqref{eq_x1}:
\[
 \E \left[ \prod_{\ell=1}^m \left(\frac{1}{2}\sum_{i=1}^{N_\ell} \left(\frac{y^{N_\ell}_i}{2\sqrt{N_\ell N}}\right)^{k_\ell}+\left(\frac{y^{N_\ell}_i}{2\sqrt{N_\ell N}}\right)^{k_\ell+1}\right) \right]\to \bL_\beta(\vec \bk, \vec\ttt).
\]
\end{theorem}
Note that here the variance is taken to be different from \Cref{thm:cor-conv}: 
one needs to replace $y_i^n$ by $\sqrt{2N/\beta} y_i^n$ to match the scaling there.

\begin{theorem}  \label{thm:multit}
Choose $m\in \N$, $\vec \bk \in \R^m_+ $, and $\vec \ttt\in \R_{\ge 0}^m$ such that $\ttt_1\le \cdots \le \ttt_m$, any real $C_1>0$, and a large enough $N\in \N$.
We take the $N$ dimensional DBM $\{Y_i(\tau)\}_{1\leq i\leq N, \tau>0}$ started from zero initial condition, and $\vec k \in \N^m$, $\tau_1\le \tau_2\le \dots \le \tau_m$ such that
\begin{equation}
\label{eq_x2}
|k_\ell-\bk_\ell N^{2/3}|<C_1, \qquad \left|\frac{2}{\beta}N + \frac{2}{\beta}\ttt_\ell N^{2/3}-\tau_\ell\right|<C_1,\qquad  \text{for each }\ell\in\llbracket 1,m\rrbracket.
\end{equation}
Then as $N\to\infty$, uniformly over the ($N$--dependent) choices of $\vec k$ and $\vec \tau$ satisfying \eqref{eq_x2}:
\[
 \E \left[ \prod_{\ell=1}^m \left(\frac{1}{2}\sum_{i=1}^N \left(\frac{Y_i(\tau_\ell)}{\sqrt{2\beta\tau_\ell N}}\right)^{k_\ell}+\left(\frac{Y_i(\tau_\ell)}{\sqrt{2\beta\tau_\ell N}}\right)^{k_\ell+1}\right) \right]\to \bL_\beta(\vec \bk, \vec\ttt).
\]
\end{theorem}
For the ease of our analysis, we slightly reformulate the Dunkl operators, as follows.
For $\tau>0$, $N\in\N$, and $i\in\llbracket 1,N\rrbracket$ we consider the following operators on polynomials of $x_1, x_2, \ldots$ (in the following formula $x_i$ means the operator of multiplication by $x_i$):
\begin{equation}  \label{eq:defhD}
 \hD_i^{N,\tau}= \frac{\partial}{\partial x_i} + \tau x_i + \frac{\beta}{2} \sum_{\begin{smallmatrix} j\in\llbracket 1,N\rrbracket,\\ j\ne i\end{smallmatrix}} \frac{1-\sigma_{ij}}{x_i-x_j}.
\end{equation}
\begin{equation}  \label{eq:defhP}
 \hP_k^{N,\tau}=\sum_{i=1}^N (\hD_i^{N,\tau})^k, \qquad k=1,2,\dots.
\end{equation}
The right-hand side of \eqref{eq_corners_moments} now equals the degree zero term of $\hP_{k_m}^{N_m,\tau} \cdots \hP_{k_1}^{N_1,\tau}$ applied to the function $1$;
similarly, the right-hand side of \eqref{eq_DBM_moments} equals the degree zero term of $\hP_{k_m}^{N,\tau_m} \cdots \hP_{k_1}^{N,\tau_1}$ applied to the function $1$.

By \Cref{Theorem_corners_moments,Theorem_DBM_moments},
we see that \Cref{thm:multil,thm:multit} are implied by the following two statements, respectively.
\begin{prop}  \label{prop:conv}
Take $m\in \N$, $\vec \bk \in \R^m_+ $, and $\vec \ttt\in \R_{\ge 0}^m$ such that $\ttt_1\le \cdots \le \ttt_m$, and $C_1>0$.
For any large enough $N\in \N$, we take $\vec k$ and $N\ge N_1\ge \cdots \ge N_m$ satisfying \eqref{eq_x1}.
Then as $N\to\infty$,
\begin{equation}   \label{eq:convpp}
2^{-m}\left(\frac{\hP_{k_m}^{N_m,2N/\beta}}{(2\sqrt{N_mN})^{k_m}}+\frac{\hP_{k_m+1}^{N_m,2N/\beta}}{(2\sqrt{N_mN})^{k_m+1}}\right) \cdots \left(\frac{\hP_{k_1}^{N_1,2N/\beta}}{(2\sqrt{N_1N})^{k_1}}+\frac{\hP_{k_1+1}^{N_1,2N/\beta}}{(2\sqrt{N_1N})^{k_1+1}}\right) [1]_{x_1=\dots=x_N=0}    
\end{equation}
converges to $\bL_\beta(\vec\bk,\vec\ttt)$.
\end{prop}
\begin{prop}  \label{prop:DBMconv}
Take $m\in \N$, $\vec \bk \in \R^m_+ $, and $\vec \ttt\in \R_{\ge 0}^m$ such that $\ttt_1\le \cdots \le \ttt_m$, and $C_1>0$.
For any large enough $N\in \N$, we take $\vec k\in\N^m$, $\tau_1\le \tau_2\le \dots\le \tau_m$ satisfying \eqref{eq_x2}.
Then as $N\to\infty$,
\[
2^{-m}\left(\frac{\hP_{k_m}^{N,\tau_m}}{(2\beta \tau_mN)^{k_m/2}}+\frac{\hP_{k_m+1}^{N,\tau_m}}{(2\beta\tau_mN)^{(k_m+1)/2}}\right) \cdots \left(\frac{\hP_{k_1}^{N,\tau_1}}{(2\beta\tau_1N)^{k_1/2}}+\frac{\hP_{k_1+1}^{N,\tau_1}}{(2\beta\tau_1N)^{(k_1+1)/2}}\right) [1]_{x_1=\dots=x_N=0}
\]
converges to $\bL_\beta(\vec\bk,\vec\ttt)$.
\end{prop}

\subsection{Strategy of the proof}

Sections \ref{ssec:setdisexp}--\ref{ssec:sob} are devoted to proving \Cref{prop:conv}, and our approach is to expand the product of the operators in terms of random walks and then study diffusive scaling limit for them, in which walks turn into Brownian motions. Section \ref{sec:DBMlimit} explains how the same arguments can be adapted to the Dyson Brownian Motion to get \Cref{prop:DBMconv}.

\smallskip

As a starting point, we write the product of operators in Proposition \ref{prop:conv} as a linear combination of  $(\hD_{i_m}^{N_m})^{k_m} \cdots (\hD_{i_1}^{N_1})^{k_1}$. In Section \ref{ssec:setdisexp} we look at how each $\hD_{i_\ell}^{N_\ell}$ acts on monomials, and convert this information into a combinatorial expansion for a general operator  $(\hD_{i_m}^{N_m})^{k_m} \cdots (\hD_{i_1}^{N_1})^{k_1}$. The expansion is written in terms of the objects that we call ``walks'', as defined in Section \ref{ssec:dwrep}. Each walk is a finite collection of discrete paths and we refer to Figures \ref{fig:nrw} -- \ref{fig:blp} for graphical illustrations. Generically, these paths have either $\pm 1$ steps or horizontal plateaus where they stay constant. However, at special points there are much larger jumps.

Our task is to send $N\to\infty$ in the sum over this collection of paths. We will eventually see that, generically, each path becomes a Brownian bridge (or maybe a Brownian excursion, if there are some positivity constraints), however, jumps play the central role in the analysis. The main idea is that all terms can be treated as finite perturbations of walks without jumps. In other words, the terms in the sum can be grouped by the number of jumps and the subsum of the terms with more than $\delta$ jumps becomes negligible as $\delta\to\infty$, uniformly in $N$. Therefore (in the $N\to\infty$ limit) it suffices to consider the paths with finitely many jumps. Careful encoding of the information contained in these jumps eventually gives rise to the combinatorial structure of ``blocks'' which we have already introduced in Section \ref{sec:forjm}.

\medskip

There is an important obstacle in directly developing such arguments. A naive sum over all walks fails to be \emph{absolutely} convergent in the $N\to\infty$ limit. In more detail, each walk $\vec{r}$ comes with a signed weight $w$, and while we are eventually interested in $\sum_{\vec{r}} w(\vec{r})$, the sum of absolute values $\sum_{\vec r} |w(\vec r)|$ grows faster as $N\to\infty$. We call this a blow-up issue and illustrate it in the examples in Section \ref{Section_examples_blow_ups}. Because of this feature, we need to keep track of certain cancelations as $N\to\infty$; this is also the reason why the limiting expression $\bL_\beta(\vec\bk,\vec\ttt)$ is represented in Definition \ref{defn:core} as a conditionally convergent improper integral.

\medskip

The first step of the asymptotic analysis of $\sum_{\vec{r}} w(\vec{r})$ is to better understand the combinatorics of walks and how it leads to different asymptotic behaviors of $w(\vec{r})$ and subsums corresponding to walks of different types. A combinatorial classification of walks is developed in Section \ref{ssec:classi}, and afterwards in Section \ref{ssec:generabd} we present a general upper-bound for a contribution of walks of a given type to sum $\sum_{\vec{r}} w(\vec{r})$. This bound is repeatedly used in subsequent subsections to produce various tail bounds and to understand which walks survive in $N\to\infty$ limit and whose contributions are asymptotically negligible. The proof of the upper-bound of Section \ref{ssec:generabd} relies on (somewhat standard, but quite tedious) counts for the random walks with various restrictions, which are presented in Section  \ref{ssec:expvar} of the appendix.

\medskip

In Sections \ref{ssec:buc} and \ref{ssec:pair} we explore various cancelations between walks of different combinatorial types in $\sum_{\vec{r}} w(\vec{r})$. The key result is Proposition \ref{prop:truncreg}: it introduces a subset $\sB_\epsilon^*$, $\epsilon>0$, of all walks, which satisfies two key features:
\begin{itemize}
 \item The subsum over  $\sB_\epsilon^*$ is absolutely convergent, i.e.\ the asymptotics scale of $\sum_{\vec{r}\in \sB_\epsilon^*} |w(\vec{r})|$ is the same as for the full signed sum $\sum_{\vec{r}} w(\vec{r})$.
 \item As $\epsilon\to 0+$ the walks outside  $\sB_\epsilon^*$ have negligible total contribution to $\sum_{\vec{r}} w(\vec{r})$.
\end{itemize}
Here $\epsilon$ has a role of a regularization parameter, and it is the same $\epsilon$ as the one in $\sK_\epsilon$ of Definition \ref{defn:core}. The constraint on the walks in $\sB_\epsilon^*$ is that on certain $\epsilon$--dependent intervals some paths are required to stay constant, while others are required to stay above certain lower bounds.

\medskip

In Section \ref{ssec:bdecom} we assign to each walk $\vec{r}$ in $\sB_\epsilon^*$ combinatorial data, which is a discrete version of the blocks structure of Section \ref{sec:forjm}. Proposition \ref{prop:IXHGest} shows that the subsum of $\sum_{\vec{r}\in \sB_\epsilon^*} w(\vec{r})$ corresponding to the walks whose data is close to a particular blocks structure is well-approximated by $\exp\left(\sum_{\ell=1}^{m-1}
(\ttt_\ell-\ttt_{\ell+1}) \bH(\bQ_\ell)/2 \right)\bI_{\beta,\vec\bk}[\vec \bp, \vec\bb,\bH]$, which is precisely the integrand in Definition \ref{defn:core}.
The proof is based on $N\to\infty$ approximation of Bernoulli random walks by Brownian motions; again somewhat standard, but quite technical properties of such approximations are presented in the appendix, see Sections \ref{Appendix_2} and \ref{ssec:couplgwbb}.

\medskip

Section \ref{ssec:sob} is the culmination of the proof of Proposition \ref{prop:conv}: we show that discrete sum $\sum_{\vec{r}\in \sB_\epsilon^*} w(\vec{r})$ approximates as $N\to\infty$ the integral over $\sK_\epsilon$ in Definition \ref{defn:core} and further send $\epsilon\to 0$ to complete the proof. On our way we also establish in Propositions \ref{prop:fixedepsconv} and \ref{prop:epszerocov} that $\bL_\beta(\vec\bk, \vec\ttt)$ from \Cref{defn:core} is well-defined as a conditionally convergent improper integral.

\medskip

For the Dyson Brownian Motion framework, i.e., for Proposition \ref{prop:DBMconv}, the proof follows exactly the same steps, and we only outline in Section \ref{sec:DBMlimit} the necessary modifications.

\bigskip

\noindent{\bf Conventions. }For the convenience of notations, from now on we write $\hD_i^{N_\ell}=\hD_i^{N_\ell,2N/\beta}$ and $\hP_k^{N_\ell}=\hP_k^{N_\ell,2N/\beta}$. Throughout the rest of this section, we fix $m\in \N$, and use $C$ and $c$ to denote large and small constants that may depend on $m$, while the specific values can change from line to line.
We will also use $C_1, C_2, \ldots$ to denote constants that are used locally within each statement (i.e., they are not changing inside a particular proof, but $C_1$ inside one lemma is allowed to be different from $C_1$ inside another lemma).
In particular, $C, c$ may also depend on $C_1, C_2, \ldots$, unless otherwise noted.

\subsection{Setup and discrete expansion}  \label{ssec:setdisexp}

\subsubsection{Terms in expansions}

We start analysis of the expression in Proposition \ref{prop:conv}. The product $\hP_{k_m}^{N_m} \cdots \hP_{k_1}^{N_1}$ can be expanded as the sum of $(\hD_{i_m}^{N_m})^{k_m} \cdots (\hD_{i_1}^{N_1})^{k_1}$,
over all $(i_1, \ldots, i_m) \in \prod_{\ell=1}^m\llbracket 1, N_\ell \rrbracket$. We now expand $(\hD_{i_m}^{N_m})^{k_m} \cdots (\hD_{i_1}^{N_1})^{k_1}$.
For a monomial $x_1^{d_1}\cdots x_N^{d_N}$ (where each $d_1, \dots, d_N\in\Z_{\ge 0}$), and any $j\neq i$, the operator $\frac{1-\sigma_{ij}}{x_i-x_j}$ acting on it leads to
\begin{equation}   \label{eq:explong1}
x_1^{d_1}\cdots x_N^{d_N}\sum_{\gamma=1}^{d_i-d_j} x_i^{-\gamma}x_j^{\gamma-1},
\end{equation}
if $d_i>d_j$; and
\begin{equation}   \label{eq:explong2}
-x_1^{d_1}\cdots x_N^{d_N}\sum_{\gamma=1}^{d_j-d_i} x_i^{\gamma-1}x_j^{-\gamma}
\end{equation}
if $d_i<d_j$; and zero if $d_i=d_j$.

We introduce the following notation: for $i\in\llbracket 1, N\rrbracket$ and monomial $f$ of $x_1,\ldots, x_N$, we use $\deg_i(f)$ to denote the degree of $x_i$ in $f$. We recall $\tau=\frac{2}{\beta}N$, set $Q_\ell=\sum_{\ell'=1}^\ell k_{\ell'}$ for each $\ell\in \llbracket 0, m\rrbracket$,
and expand $(\hD_{i_m}^{N_m})^{k_m} \cdots (\hD_{i_1}^{N_1})^{k_1}$:
\begin{definition}   \label{defn:tie}
A \emph{term in expansion} of $(\hD_{i_m}^{N_m})^{k_m} \cdots (\hD_{i_1}^{N_1})^{k_1}$ is a sequence $\{O(t), P(t)\}_{t=1}^{Q_m}$, where each $O(t)\in\{+,-,\top\}$ denotes the operation at each step, and each $P(t)$ is a monomial\footnote{In this paper, a monomial refers to a product of powers of variables with nonnegative integer exponents, and a non-zero constant. In particular, any non-zero constant is a monomial.} of $x_1,\ldots, x_N$. For each $\ell\in\llbracket 1, m\rrbracket$, and $t\in \llbracket Q_{\ell-1}+1, Q_\ell \rrbracket$, one of the following is required to hold for $P(t)$ (assuming that $P(0)=1$ for the ease of notations):
\begin{enumerate}
    \item[1:] $O(t)=+$, and $P(t)=\frac{2Nx_{i_\ell}P(t-1)}{\beta}$;
    \item[2:] $O(t)=-$, and $P(t)=\big(\deg_{i_\ell}(P(t-1))+\tfrac{\beta}{2} R(t-1)\big)\frac{P(t-1)}{x_{i_\ell}}$, where $R(t-1)$ is the number of $j\in\llbracket 1, N_\ell\rrbracket$ with $\deg_{i_\ell}(P(t-1))>\deg_j(P(t-1))$;\\
    (We note that in this case, necessarily $\deg_{i_\ell}(P(t-1))>0$, since otherwise $R(t-1)$ would also be zero, and $P(t)$ equals zero, which is not a monomial.)
    \item[3(a):] $O(t)=\top$, and $P(t)=x_{i_\ell}^{-\gamma} x_j^{\gamma-1}\frac{\beta P(t-1)}{2}$ for some $j\in\llbracket 1, N_\ell\rrbracket$ with $\deg_{i_\ell}(P(t-1))\ge \deg_j(P(t-1))+2$, and some $\gamma \in \llbracket 2, \deg_{i_\ell}(P(t-1))- \deg_j(P(t-1))\rrbracket$;
    \item[3(b):] $O(t)=\top$, and $P(t)=-x_{i_\ell}^{\gamma-1} x_j^{-\gamma}\frac{\beta P(t-1)}{2}$ for some $j\in\llbracket 1, N_\ell\rrbracket$ with $\deg_{i_\ell}(P(t-1))\le\deg_j(P(t-1))-1$, and some $\gamma \in \llbracket 1, \deg_j(P(t-1))-\deg_{i_\ell}(P(t-1))\rrbracket$.
\end{enumerate}
In addition, we require that $\deg_i(P(Q_m))=0$ for each $i\in\llbracket 1, N\rrbracket$, i.e., $P(Q_m)$ is a constant.
\end{definition}
We note that the four cases in the above definition corresponds to different parts of the operator $\hD_{i_\ell}^{N_\ell}=\hD_{i_\ell}^{N_\ell,2N/\beta}$:
\begin{enumerate}
    \item[1:] $\tau x_{i_\ell}=\frac{2}{\beta}Nx_{i_\ell}$;
    \item[2:] $\frac{\partial}{\partial x_{i_\ell}}$, as well as the $\frac{\beta}{2} x_{i_\ell}^{-1}$ (i.e., $\gamma=1$) terms in the expansion of $\frac{\beta}{2}\cdot\frac{1-\sigma_{i_\ell j}}{x_{i_\ell}-x_j}$ in \eqref{eq:explong1}, for each $j\in\llbracket 1, N_\ell\rrbracket$ with $\deg_{i_\ell}(P(t-1))>\deg_j(P(t-1))$;
    \item[3(a):] the $\frac{\beta}{2} x_{i_\ell}^{-\gamma}x_j^{\gamma-1}$ with $\gamma\ge 2$ terms in the expansion of $\frac{\beta}{2}\cdot\frac{1-\sigma_{i_\ell j}}{x_{i_\ell}-x_j}$ in \eqref{eq:explong1}, for each $j\in\llbracket 1, N_\ell\rrbracket$ with $\deg_{i_\ell}(P(t-1))\ge \deg_j(P(t-1))+2$;
    \item[3(b):] the $\frac{\beta}{2} x_{i_\ell}^{\gamma-1}x_j^{-\gamma}$ terms in the expansion of $\frac{\beta}{2}\cdot\frac{1-\sigma_{i_\ell j}}{x_{i_\ell}-x_j}$ in \eqref{eq:explong2}, for each $j\in\llbracket 1, N_\ell\rrbracket$ with $\deg_{i_\ell}(P(t-1))\le \deg_j(P(t-1))-1$.
\end{enumerate}
\begin{example}
The operator $\hD_2^5$ applied to $x_1^5x_2^2x_4$ leads to the sum of the following monomials:
\begin{enumerate}
    \item[1:] $\frac{2Nx_1^5x_2^3x_4}{\beta}$,
    \item[2:] $(2+3\beta/2)x_1^5x_2x_4$,
    \item[3(a):] $x_2^{-2}x_3\cdot\frac{\beta}{2} x_1^5x_2^2x_4$ and $x_2^{-2}x_5\cdot\frac{\beta}{2} x_1^5x_2^2x_4$,
    \item[3(b):] $-x_1^{-1}\cdot\frac{\beta}{2} x_1^5x_2^2x_4$, $-x_1^{-2}x_2\cdot\frac{\beta}{2} x_1^5x_2^2x_4$, and $-x_1^{-3}x_2^2\cdot\frac{\beta}{2} x_1^5x_2^2x_4$.
\end{enumerate}
\end{example}

It is straightforward to see that $(\hD_{i_m}^{N_m})^{k_m} \cdots (\hD_{i_1}^{N_1})^{k_1}$ applied to the function $1$ is precisely the sum of $P(Q_m)$ over all the terms in expansion.
Therefore the sum of $P(Q_m)$ over all the terms in expansion would give the zero degree term of $(\hD_{i_m}^{N_m})^{k_m} \cdots (\hD_{i_1}^{N_1})^{k_1}$ applied to $1$.

\subsubsection{Discrete walk representation}  \label{ssec:dwrep}

We next observe that  a term $\{O(t), P(t)\}_{t=1}^{Q_m}$ in expansion  of $(\hD_{i_m}^{N_m})^{k_m} \cdots (\hD_{i_1}^{N_1})^{k_1}$ can be encoded by the degrees of the monomials $P(t)$.
More precisely, we denote $r_i(t)=\deg_i(P(t))$, for each $t\in \llbracket 0, Q_m\rrbracket$ and $i\in \llbracket 1, N\rrbracket$.
We call such $\vec r=\{r_i\}_{i=1}^N$ a \emph{walk} of  $(\hD_{i_m}^{N_m})^{k_m} \cdots (\hD_{i_1}^{N_1})^{k_1}$.

\begin{definition}  \label{defn:wallblock}
A \emph{walk} of $(\hD_{i_m}^{N_m})^{k_m} \cdots (\hD_{i_1}^{N_1})^{k_1}$ consists of $\vec r=\{r_j\}_{j=1}^N$, with each $r_j:\llbracket 0, Q_m\rrbracket\to \Z_{\ge 0}$, satisfying the following conditions.
Denote $\sH(t)=\sum_{j=1}^N r_j(t)$.
\begin{enumerate}
    \item $|\sH(t)-\sH(t-1)|=1$ for each $t\in \llbracket 1, Q_m\rrbracket$, $\sH(t)\ge 0$, and  $\sH(0)=\sH(Q_m)=0$.
        \item For each $\ell\in \llbracket 1, m\rrbracket$ and $t\in \llbracket Q_{\ell-1}+1, Q_\ell\rrbracket$, we must have one of the two cases:
        \begin{itemize}
        \item Either, $r_{i_\ell}(t)=r_{i_\ell}(t-1)\pm 1$ and $r_j(t)=r_j(t-1)$ for all $j\ne i_\ell$;
        \item Or there exists one and only one index\footnote{This corresponds to the third option, i.e., $O(t)=\top$, in \Cref{defn:tie}.} $j\ne i_\ell$, such that $r_j(t)\neq r_j(t-1)$ .
    Besides, in this case, necessarily $\sH(t)=\sH(t-1)-1$, and $j\le N_\ell$, and the differences $r_j(t)-r_j(t-1)$ and $r_j(t)-r_{i_\ell}(t-1)+1/2$ are of opposite signs.

        \end{itemize}
\end{enumerate}
\end{definition}
 In other words, the last condition says that when $r_j(t)>r_j(t-1)$, i.e., $r_j$ increases, we have $r_j(t)<r_{i_\ell}(t-1)$, corresponding to 3(a) in \Cref{defn:tie}; and when $r_j(t)<r_j(t-1)$, i.e., $r_j$ decreases, we have $r_j(t)\ge r_{i_\ell}(t-1)$, corresponding to 3(b) in \Cref{defn:tie}. In particular, if $r_j(t)=0$, then we must have $r_{i_\ell}(t-1)=0$.

We use $\sB=\sB[i_1,\ldots,i_m]$ to denote the set of all walks of  $(\hD_{i_m}^{N_m})^{k_m} \cdots (\hD_{i_1}^{N_1})^{k_1}$.

We remark that in the continuous limit $\sH(t)$ turns into the trajectories of (positive) Brownian bridges, as in the bottom panel of Figure \ref{fig:block}, while the remaining data (block and virtual block processes) encodes how $\sH(t)$ is split into degrees of various variables.

There is a bijection between  terms in expansion and walks. In addition, each walk uniquely determines the value of $P(Q_m)$ --- the latter is a zero degree monomial and, hence, a constant.
\begin{definition} \label{defn:wr}
For $\vec r\in\sB$, we define its weight $w(\vec r)$ to be the constant $P(Q_m)$ in the corresponding term in expansion.
Equivalently, multiplying the factors from cases 1, 2, 3(a), 3(b) in Definition \ref{defn:tie}:
\begin{align}
\label{eq_x3} w(\vec r) = \prod_{\ell=1}^m \Biggl[ & (-1)^{\bigl|\{t\in\llbracket Q_{\ell-1}+1, Q_\ell \rrbracket:\, \sH(t)=\sH(t-1)-1,\, r_{i_\ell}(t)\ge r_{i_\ell}(t-1)\}\bigr|} \\ &\times(\sqrt{NN_\ell})^{k_\ell}\, N_\ell^{(\sH(Q_{\ell-1})-\sH(Q_\ell))/2-\bigl|\{t\in\llbracket Q_{\ell-1}+1, Q_\ell \rrbracket:\, \sH(t)=\sH(t-1)-1,\, r_{i_\ell}(t)\neq r_{i_\ell}(t-1)-1\}\bigr|}\notag \\
&\times \prod_{\substack{t\in \llbracket Q_{\ell-1}+1, Q_\ell \rrbracket: \\ \sH(t)=\sH(t-1)-1, \\r_{i_\ell}(t)= r_{i_\ell}(t-1)-1}}  \left(1+\frac{2r_{i_\ell}(t-1)}{\beta N_\ell}-\frac{|\{j\in \llbracket 1,N\rrbracket: r_j(t-1)\ge r_{i_\ell}(t-1)\}|}{N_\ell}\right)\Biggr].\notag
\end{align}
\end{definition}
Comparing with Definition  \ref{defn:tie}: the first case there contributes to the power of $N$ in the second line of  \eqref{eq_x3}; the second case contributes to the product in the last line of \eqref{eq_x3}; the case 3(b) contributes to the power of $-1$ in the first line of \eqref{eq_x3}. In addition the factors $\frac{\beta}{2}$ and $\frac{2}{\beta}$ from all four cases cancel out, because the walk starts and ends at $0$. Furthermore, the factors of $N_\ell$ in the second line of \eqref{eq_x3} precisely cancel with similar factors in the denominators in the third line. The exponent $N_\ell^{(\sH(Q_{\ell-1})-\sH(Q_\ell))/2}$ is responsible for the appearance in the $N\to\infty$ limit of $\exp\left(\sum_{\ell=1}^{m-1}
(\ttt_\ell-\ttt_{\ell+1}) \bH(\bQ_\ell)/2 \right)$ in Definition \ref{defn:core}.

We can now summarize the expansion of $\hP_{k_m}^{N_m} \cdots \hP_{k_1}^{N_1}$:
\begin{lemma}  \label{lem:exp}
We have
\begin{equation}  \label{eq:expsum}
\hP_{k_m}^{N_m} \cdots \hP_{k_1}^{N_1} [1]_{x_1=\dots=x_N=0}= \sum_{(i_1, \ldots, i_m) \in \prod_{\ell=1}^m\llbracket 1, N_\ell \rrbracket} \, \sum_{\vec r\in\sB[i_1,\ldots,i_m]} w(\vec r).
\end{equation}
\end{lemma}
We note that this sum can be viewed as a discrete analogue of the integral in \Cref{defn:core}, and the rest of the proof is to show that the discrete sum  converges to the principal value integral as $N\to\infty$.

\subsubsection{Scaling analysis and blow-up terms} \label{Section_examples_blow_ups}
For the sum \eqref{eq:expsum},
we expect it to be of order $\prod_{\ell=1}^m (2\sqrt{N_\ell N})^{k_\ell}$, since under our setup the extreme particles of the corners process at row $N_\ell$ ($y_i^{N_\ell}$ in Theorem \ref{thm:multil}) are approximately equal to $2\sqrt{N_\ell N}$.
We next heuristically explain the scale of the sum, through examples of certain types of walks. In doing so we illustrate that it would be crucial to exploit cancellation between the weights, which can be either positive or negative.

\bigskip
\noindent\textit{Walks without jumps.}
Let us consider all walks $\vec r$ for which $O(t)\in \{+, -\}$ for all $t$. Equivalently, for each $\ell \in \llbracket 1, m\rrbracket$ and  $j\neq i_\ell$, $r_j$ is constant on $\llbracket Q_{\ell-1}+1, Q_\ell \rrbracket$.
Imagine that we apply $\hD_{i_\ell}^{N_\ell}$ using \eqref{eq:defhD}. Then the walks without jumps correspond to taking $\frac{\partial}{\partial x_{i_\ell}}+\frac{2}{\beta}Nx_{i_\ell}+S \frac{\beta}{2} x_{i_\ell}^{-1}$, with the last term coming from $\frac{\beta}{2}\cdot\frac{1-\sigma_{i_\ell j}}{x_{i_\ell}-x_j}$ for each $j\in\llbracket 1, N_\ell\rrbracket$
where the degree of $x_j$ is smaller than the degree of $x_{i_\ell}$, and $S$ counts the number of such cases ($S$ changes as a function of $t$ and satisfies $N_\ell-m\le S\le N_\ell-1$). See \Cref{fig:nrw} for some illustrations of such walks.

\begin{figure}[t]
    \centering
\begin{subfigure}[b]{0.49\textwidth}
    \resizebox{0.95\textwidth}{!}{
    \begin{tikzpicture}
        \draw[line width=3pt] (0,0)--(1,1)--(1.5,0.5)--(3.5,2.5)--(4,2)--(5,3)--(5.5,2.5)--(7,4)--(9.5,1.5)--(10,2)--(11,1)--(11.5,1.5)--(12,2)--(14,0);

        \draw[thin] [dotted] [step=0.5] (0,0) grid (14,5);

        \draw[green] [ultra thick] (0,0)--(1,1)--(1.5,0.5)--(3.5,2.5)--(4,2)--(5,3)--(5.5,2.5)--(7,4)--(9.5,1.5)--(10,2)--(11,1)--(11.5,1.5)--(12,2)--(14,0);

        \draw (0,0) node[anchor=north]{{\LARGE $0$}};
        \draw (14,0) node[anchor=north]{{\LARGE $Q_1$}};
    \end{tikzpicture}}
\end{subfigure}
\par\medskip
\begin{subfigure}[b]{0.98\textwidth}
    \resizebox{0.95\textwidth}{!}{
    \begin{tikzpicture}
        \draw[line width=3pt] (0,0)--(2,2)--(2.5,1.5)--(5.5,4.5)--(6.5,3.5)--(7,4)--(8.5,2.5)--(9,3)--(9.5,2.5)--(10,3)--(11,2)--(11.5,2.5)--(12,2)--(14.5,4.5)--(15,4)--(15.5,4.5)--(17,3)--(18,4)--(20,2)--(21,3)--(22,2)--(22.5,2.5)--(23.5,1.5)--(24,2)--(25.5,0.5)--(26.5,1.5)--(28,0);

        \draw[thin] [dotted] [step=0.5] (0,0) grid (28,5);
        \draw[thick] [dashed] (11,0)--(11,5);
        \draw[thick] [dashed] (20,0)--(20,5);

        \draw[green] [ultra thick] (0,0)--(2,2)--(2.5,1.5)--(5.5,4.5)--(6.5,3.5)--(7,4)--(8.5,2.5)--(9,3)--(9.5,2.5)--(10,3)--(11,2)--(20,2)--(21,3)--(22,2)--(22.5,2.5)--(23.5,1.5)--(24,2)--(25.5,0.5)--(26.5,1.5)--(28,0);

        \draw[red] [ultra thick] (0,0)--(11,0)--(11.5,0.5)--(12,0)--(14.5,2.5)--(15,2)--(15.5,2.5)--(17,1)--(18,2)--(20,0)--(22,0)--(28,0);

        \draw (0,0) node[anchor=north]{{\LARGE $0$}};
        \draw (11,0) node[anchor=north]{{\LARGE $Q_1$}};
        \draw (20,0) node[anchor=north]{{\LARGE $Q_2$}};
        \draw (28,0) node[anchor=north]{{\LARGE $Q_3$}};
    \end{tikzpicture}}
\end{subfigure}
    \caption{
    Illustrations of walks without jumps (i.e., all $O(t)$ being $\pm$).\\
    Top panel: a walk of $(\hD_3^N)^{k_1}$ without jumps: the curve represents $\sH =r_3$, and all its steps have slopes $\pm 1$.
    (This walk corresponds to a term in \Cref{defn:core} with $m=1$, and all $\bp_j=0$.)
    \\
    Bottom panel: a walk $\vec r$ of $(\hD_5^N)^{k_3}(\hD_2^N)^{k_2}(\hD_5^N)^{k_1}$ without jumps: the green and red curves are $r_5$ and $r_2$ respectively, and the black curve is $\sH =r_2+r_5$.
    All the steps of the green curve in $\llbracket 0, Q_1\rrbracket \cup \llbracket Q_2, Q_3\rrbracket$ have slopes $\pm 1$, and all the steps of the red curve in $\llbracket Q_1, Q_2\rrbracket$ have slopes $\pm 1$. (The contribution of such walk vanishes as $N\to\infty$ due to cancellations in \Cref{ssec:pair} below.)
    }
    \label{fig:nrw}
\end{figure}
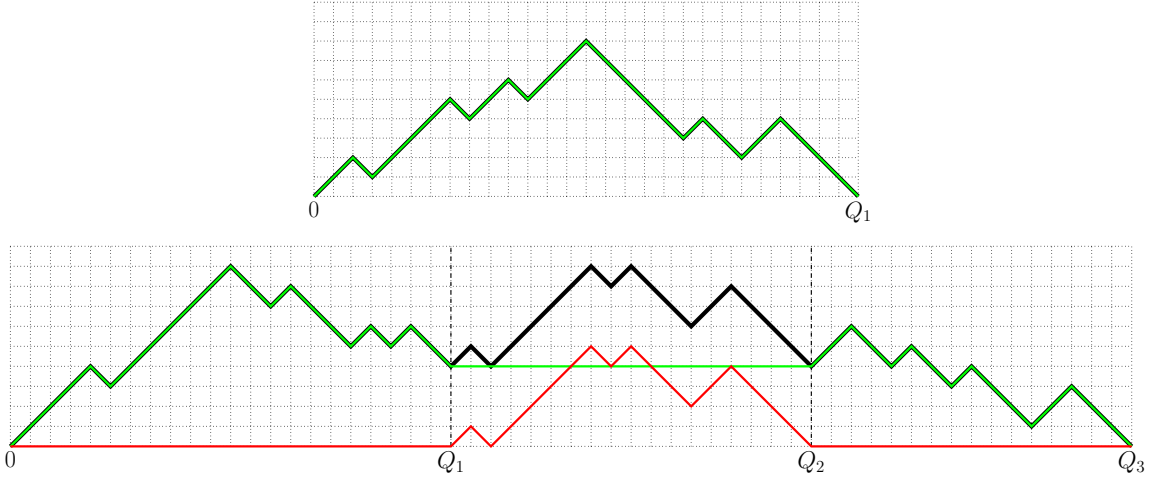

In this case, the non-constant part of each $r_j$ is a path of length $\sum_{\ell\in\llbracket 1,m\rrbracket: i_\ell = j}k_\ell$, starting from and ending at $0$, taking $\pm 1$ at each step, and staying non-negative.
Thus the number of such $\vec r$ is a product of Catalan numbers $\frac{1}{n+1}{{2n}\choose n}$. Plugging the asymptotic expansion of the binomial coefficients, the number of $\vec r$ is of order $2^{Q_m}N^{-\eta}$, because each $k_\ell$ is of order $N^{2/3}$, and where $\eta$ is the number of different numbers appearing in $(i_1,\ldots, i_m)$.

For a walk with each $O(t)\in \{+, -\}$, its weight roughly\footnote{Comparing with \Cref{defn:wr}, we omit $\frac{|\{j\in \llbracket 1,N_\ell\rrbracket: r_j(t-1)\ge r_{i_\ell}(t-1)\}|}{N_\ell}$, which is of order $N^{-1}$ and not necessarily zero, but does not change our rough bounds.} equals
\[
\prod_{\ell=1}^m (\sqrt{NN_\ell})^{k_\ell} \prod_{t\in \llbracket Q_{\ell-1}+1, Q_\ell\rrbracket: \sH(t)=\sH(t-1)-1}  \left(1+\frac{2r_{i_\ell}(t-1)}{\beta N_\ell}\right),
\]
which is of order $\prod_{\ell=1}^m (\sqrt{NN_\ell})^{k_\ell}$.
Finally, by multiplying this with $2^{Q_m}N^{-\eta}$, and summing over $\prod_{\ell=1}^m N_\ell$ choices for $(i_1,\ldots, i_m)$, we get the desired order $\prod_{\ell=1}^m (2\sqrt{N_\ell N})^{k_\ell}$.

We remark that if we take the $N\to\infty$ limit, but fix $\vec k$, the sum of these terms would give (after rescaling) the products of the moments of the semi-circle law, which are Catalan numbers.
On the other hand, if $\vec k$ also $\to\infty$ in the scaling limit, i.e., under the setting of \Cref{prop:conv}, these terms would give the integral over all $(\vec\bp,
\vec\bb,\bH)\in \sK[\vec\bk]$ with $\bDel=\emptyset$, in defining $\bL_\beta(\vec\bk, \vec\ttt)$ in \Cref{defn:core}. In particular, for $m=1$ this corresponds to the first case in Example \ref{Example_for_measure}.

\bigskip

We proceed to walks (equivalently, terms in expansion) with some $\top$ (i.e., jumps) in $\{O(t)\}_{t=1}^{Q_m}$.
For some of such walks, the sum of the weights, even when taking absolute values, is still of the desired order. For some other walks with $\top$, we may encounter a `blow-up' when summing over the absolute values of the weights; i.e., the sum of absolute values is of a much larger order.
The difference between these two types of walks (with jumps) is from the behavior to the right of each $Q_\ell$: in the `blow-up' case the total degree $\sH(t)$ is $\ge \sH(Q_\ell)$ for $t$ in a  small interval to the right from $Q_\ell$, for some $\ell\in\llbracket 1, m\rrbracket$ (once this interval is large enough, the blow-up disappears).  The precise formulation will be given shortly in \Cref{ssec:classi} below.

The walks with a blow-up correspond to blocks with non-trivial virtual block processes (i.e., $\vec\bb$ with at least one $\bb_\ell\neq\varnothing$), in Definition \ref{Definition_virtual_block};  other walks correspond to blocks with each virtual block coordinate $\bb_\ell=\varnothing$.

We next give some examples of walks with jumps, to illustrate the phenomena, and to explain how we resolve the blow-up issue using cancellations.

\bigskip
\noindent\textit{Walks with jumps, but no blow-up.}
\begin{example}
Consider the expansion of the operator $(\hD_1^N)^k$, and its walks $\vec r$ of the following form illustrated in \Cref{fig:nblpe}:
There exists a $j_*\in \llbracket 2, N\rrbracket$, such that $r_j=0$ for each $j\in \llbracket 2, N\rrbracket\setminus\{j_*\}$.
For $r_1$, it is integer-valued and non-negative on $\llbracket 0, k\rrbracket$. For some $s, s'\in \llbracket 2, k\rrbracket$, $s<s'$, we have
\begin{itemize}
    \item For each $t\in\llbracket 1,k\rrbracket\setminus\{s, s'\}$, we have $|r_1(t)-r_1(t-1)|=1$.
    \item $r_{j_*}=0$ on $\llbracket 0,s-1\rrbracket\cup\llbracket s', k\rrbracket$, and $r_{j_*}$ is constant and positive on $\llbracket s, s'-1\rrbracket$.
    Moreover, $r_1(s)+r_{j_*}(s)=r_1(s-1)+r_{j_*}(s-1)-1$ and $r_1(s')+r_{j_*}(s')=r_1(s'-1)+r_{j_*}(s'-1)-1$.
\end{itemize}
Any such $\vec r$ is uniquely determined by the following information: the sum $\sH =r_1+r_{j_*}$ (which is a path on $\llbracket 0, k\rrbracket$, starting and ending at $0$ and taking $\pm 1$ at each step, with $\sH(k)>0$), and the values of $s$ and $r_{j_*}(s)$.
The number $s'$ is uniquely determined as the smallest number in $\llbracket s+1, k\rrbracket$ with $\sH(s')<r_{j_*}(s)$.
Therefore, the number of such $\vec r$ is of the order of $2^{k}=2^{k} k^{-3/2}\cdot k\cdot k^{1/2}$, where $2^{k} k^{-3/2}$ is from counting $\sH $ (Catalan numbers), and $k$ and $k^{1/2}$ are from the number of choices of $s$ and $r_{j_*}(s)$, respectively.

On the other hand, for all such $\vec r$, we have that $w(\vec r)$ is negative, with $|w(\vec r)|$ in the order of $N^{k-2}$.
Then by summing over all such $\vec r$ one gets order $(2N)^{k}N^{-2}$. The variables $x_1$ and $x_{j_*}$ could have been any two other variables, and summing over all choices them, we finally arrive at the order of magnitude $(2N)^{k}$, matching what was announced at the start of Section \ref{Section_examples_blow_ups}.
\end{example}

\begin{example}
Consider the expansion of the operator $(\hD_2^N)^k(\hD_1^N)^k$, and its walks $\vec r$ of the following form illustrated in \Cref{fig:nblp}:
For each $j\in \llbracket 3, N\rrbracket$, we have $r_j=0$.
For $r_1$ and $r_2$, they are integer-valued and non-negative on $\llbracket 0, 2k\rrbracket$. For some $s\in \llbracket k+1, 2k\rrbracket$, we have
\begin{itemize}
    \item For each $t\in\llbracket 1,k\rrbracket$, we have $|r_1(t)-r_1(t-1)|=1$,  and $r_1$ is constant and positive on $\llbracket k, s-1\rrbracket$, and $r_1=0$ on $\llbracket s, 2k\rrbracket$.
    \item $r_2=0$ on $\llbracket 0,k\rrbracket$. For each $t\in \llbracket k+1, s-1\rrbracket$, or $t\in \llbracket s+1, 2k\rrbracket$, we have $|r_2(t)-r_2(t-1)|=1$. Also $r_2(s)=r_1(s-1)-1$.
\end{itemize}
Any such $\vec r$ is uniquely determined by the sum $\sH =r_1+r_2$, which is a path on $\llbracket 0, 2k\rrbracket$, starting and ending at $0$ and taking $\pm 1$ at each step, with $\sH(k)>0$.
The number $s$ is uniquely determined as the smallest number in $\llbracket k+1, 2k\rrbracket$ with $\sH(s)<\sH(k)$.
Therefore, the number of such $\vec r$ is of the order of $2^{2k} k^{-3/2}$, or
(equivalently) of the order of $2^{2k}N^{-1}$.

On the other hand, for all such $\vec r$, we have that $w(\vec r)$ is negative, with $|w(\vec r)|$ in the order of $N^{2k-1}$.
Then by summing over all such $\vec r$ one gets order $(2N)^{2k}N^{-2}$. The variables $x_1$ and $x_2$ could have been any two other variables, and summing over all choices them, we finally arrive at the order of magnitude $(2N)^{2k}$, which matches $\prod_{\ell=1}^2(2\sqrt{N_\ell N})^{k_\ell}$ announced at the start of Section \ref{Section_examples_blow_ups}.
\end{example}

\begin{figure}[t]
    \centering
    \resizebox{0.45\textwidth}{!}{
    \begin{tikzpicture}
        \draw[ultra thin] [dashed] (0,0) grid (14,6);
        \draw[line width=3pt] (0,0)--(4,4)--(5,3)--(7,5)--(11,1)--(12,2)--(14,0);
        \draw[ultra thick] [green] (0,0)--(4,4)--(5,1)--(7,3)--(10,0)--(11,1)--(12,2)--(14,0);
        \draw[ultra thick] [blue] (0,0)--(4,0)--(5,2)--(10,2)--(11,0)--(14,0);
        \draw (0,0) node[anchor=north]{{\LARGE $0$}};
        \draw (5,0) node[anchor=north]{{\LARGE $s$}};
        \draw (11,0) node[anchor=north]{{\LARGE $s'$}};
        \draw (14,0) node[anchor=north]{{\LARGE $k$}};
    \end{tikzpicture}
    }
    \caption{An illustration of a walk of $(\hD_1^N)^k$, with jumps but no blow-up: the green and blue curves are $r_1$ and $r_{j_*}$ respectively, for some $j_*\in\llbracket 2,N\rrbracket$; and the black curve is $\sH =r_1+r_{j_*}$.
    (This walk corresponds to a term in \Cref{defn:core} with $m=1$, $\bp_2$ having two discontinuous points ($\bdel_{2,1}=2$), and all other $\bp_j=0$.)}
    \label{fig:nblpe}
\end{figure}
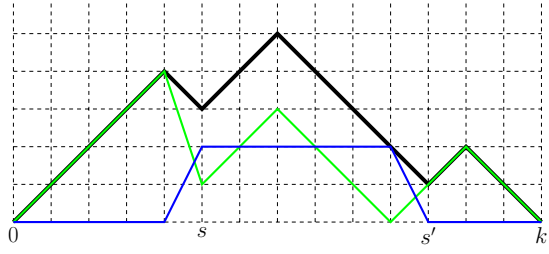

\begin{figure}[t]
    \centering
    \resizebox{0.9\textwidth}{!}{
    \begin{tikzpicture}
        \draw[ultra thin] [dashed] (0,0) grid (28,8);
        \draw[line width=3pt] (0,0)--(1,1)--(2,0)--(6,4)--(7,3)--(9,5)--(10,6)--(11,5)--(12,6)--(14,4)--(15,5)--(16,4)--(17,5)--(20,2)--(22,4)--(24,2)--(25,3)--(28,0);
        \draw[ultra thick] [dashed] (14,0)--(14,8);
        \draw[ultra thick] [green] (0,0)--(1,1)--(2,0)--(6,4)--(7,3)--(9,5)--(10,6)--(11,5)--(12,6)--(14,4)--(18,4)--(19,0)--(28,0);
        \draw[ultra thick] [red] (0,0)--(14,0)--(15,1)--(16,0)--(17,1)--(18,0)--(19,3)--(20,2)--(22,4)--(24,2)--(25,3)--(28,0);
        \draw (0,0) node[anchor=north]{{\LARGE $0$}};
        \draw (19,0) node[anchor=north]{{\LARGE $s$}};
        \draw (14,0) node[anchor=north]{{\LARGE $k$}};
        \draw (28,0) node[anchor=north]{{\LARGE $2k$}};
    \end{tikzpicture}
    }
    \caption{An illustration of a walk of $(\hD_2^N)^k(\hD_1^N)^k$, with jumps but no blow-up: the green and red curves are $r_1$ and $r_2$ respectively, and the black curve is $\sH =r_1+r_2$. (Note that this walk does not exist in the limit, due to cancellations in \Cref{ssec:pair} below.)}
    \label{fig:nblp}
\end{figure}
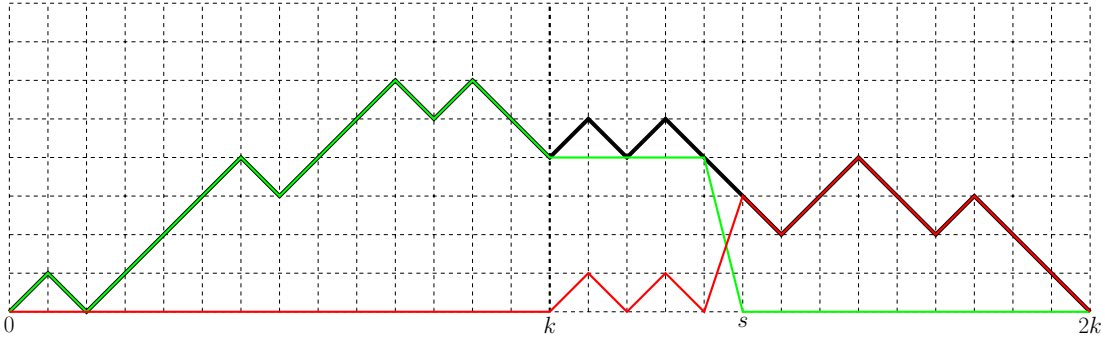

\bigskip
\noindent\textit{Blow up with jumps.}
We now turn to the `blow-up' issue, which can happen already when there is just one $\top$.
\begin{example}  \label{ex:blowup}
We again consider of the operator $(\hD_2^N)^k(\hD_1^N)^k$, and its walks $\vec r$ of the following form illustrated in the top panel of \Cref{fig:blp}.
For each $j\in \llbracket 3, N\rrbracket$, we have $r_j=0$.
For $r_1$ and $r_2$, they are integer-valued and non-negative on $\llbracket 0, 2k\rrbracket$.
For some $s\in \llbracket 2, k\rrbracket$ we have
\begin{itemize}
    \item $r_2=0$ on $\llbracket 0, s-1\rrbracket$, and $r_2$ is constant and positive on $\llbracket s, k\rrbracket$. For each $t\in \llbracket k+1, 2k\rrbracket$, we have $|r_2(t)-r_2(t-1)|=1$. Also $r_2(2k)=0$.
    \item For each $t\in\llbracket 1,s-1\rrbracket$ or $t\in \llbracket s+1, k\rrbracket$, we have $|r_1(t)-r_1(t-1)|=1$. We also have $r_1(s)+r_2(s)=r_1(s-1)+r_2(s-1)-1$, and $r_1=0$ on $\llbracket k, 2k\rrbracket$.
\end{itemize}
For any such $\vec r$ with given $s$, it is uniquely determined by the sum $\sH =r_1+r_2$, which is a path on $\llbracket 0, 2k\rrbracket$, starting and ending at $0$ and taking $\pm 1$ at each step, being non-negative, and staying $\ge \sH(k)$ on $\llbracket s, k\rrbracket$.
Therefore, the number of such $\vec r$ (with given $s$) is in the order of $2^{2k} k^{-3/2}(k-s+1)^{-1/2}$. The factor $2^{2k}k^{-3/2}$ comes from Catalan number, counting the number of all possible $\sH$; and $(k-s+1)^{-1/2}$ corresponds to the extra requirement of $\ge \sH(k)$ on $\llbracket s, k\rrbracket$ --- one could think about imposing this requirement on a uniformly random positive random walk bridge and then computing probability using the reflection principle. By summing over all possible choices for $s$, one gets order $2^{2k}k^{-1}$, which can be rewritten as  order $2^{2k}N^{-2/3}$.

Definition \ref{defn:wr} yields that for any such $\vec r$ the weight $w(\vec r)$ is negative and $|w(\vec r)|$ is of the order of $N^{2k-1}$.
Then by summing the weights over all such $\vec r$ one gets $(2N)^{2k}N^{-5/3}$. We also need to multiply by $N^2$, in order to account for all the possible choices of two variables which could replace $x_1$ and $x_2$ in our arguments. Hence, we eventually get a negative contribution of the order $(2N)^{2k}N^{1/3}$, which is much larger than the desired scale $(2N)^{2k}$, which we saw in the previous examples and announced at the start of Section \ref{Section_examples_blow_ups}.
\end{example}
Such blow-up issue in summing over $w(\vec r)$ is resolved by exploiting cancellations.
In particular, $\vec r$ of the above example are paired with the following walks whose weights have opposite signs.

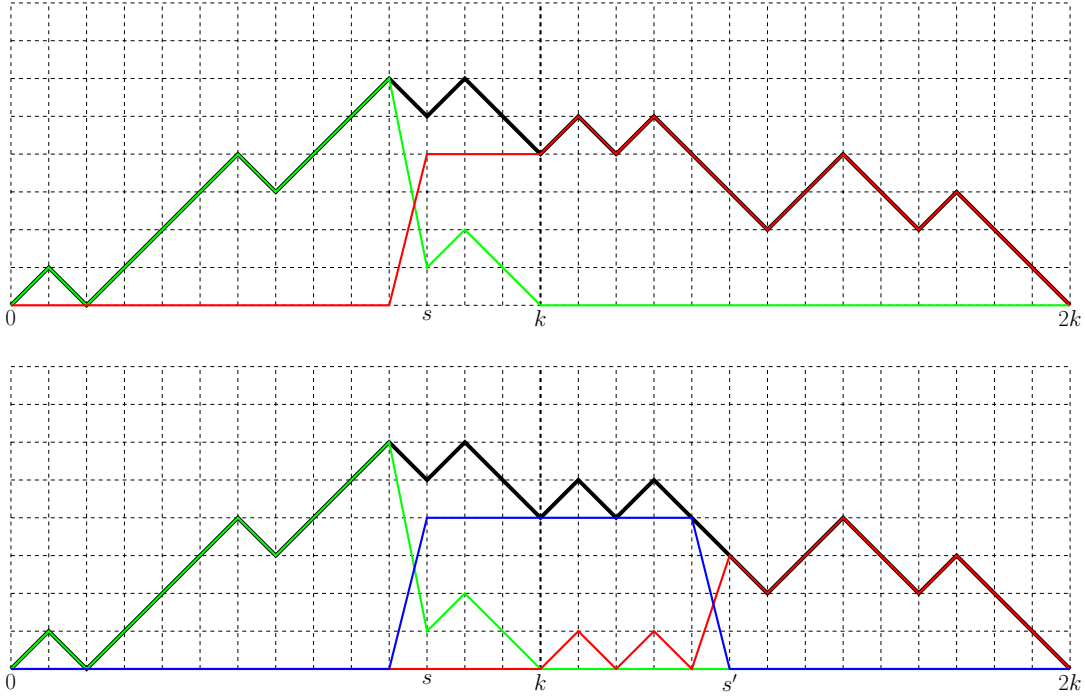
\begin{figure}[t]
\begin{subfigure}[b]{0.98\textwidth}
    \centering
    \resizebox{0.9\textwidth}{!}{
    \begin{tikzpicture}
        \draw[ultra thin] [dashed] (0,0) grid (28,8);
        \draw[line width=3pt] (0,0)--(1,1)--(2,0)--(6,4)--(7,3)--(9,5)--(10,6)--(11,5)--(12,6)--(14,4)--(15,5)--(16,4)--(17,5)--(20,2)--(22,4)--(24,2)--(25,3)--(28,0);
        \draw[ultra thick] [dashed] (14,0)--(14,8);
        \draw[ultra thick] [green] (0,0)--(1,1)--(2,0)--(6,4)--(7,3)--(10,6)--(11,1)--(12,2)--(14,0)--(28,0);
        \draw[ultra thick][red] (0,0)--(10,0)--(11,4)--(14,4)--(15,5)--(16,4)--(17,5)--(20,2)--(22,4)--(24,2)--(25,3)--(28,0);
        \draw (0,0) node[anchor=north]{{\LARGE $0$}};
        \draw (11,0) node[anchor=north]{{\LARGE $s$}};
        \draw (14,0) node[anchor=north]{{\LARGE $k$}};
        \draw (28,0) node[anchor=north]{{\LARGE $2k$}};
    \end{tikzpicture}
    }
\end{subfigure}
\par\bigskip
\begin{subfigure}[b]{0.98\textwidth}
    \centering
    \resizebox{0.9\textwidth}{!}{
    \begin{tikzpicture}
        \draw[ultra thin] [dashed] (0,0) grid (28,8);
        \draw[line width=3pt] (0,0)--(1,1)--(2,0)--(6,4)--(7,3)--(9,5)--(10,6)--(11,5)--(12,6)--(14,4)--(15,5)--(16,4)--(17,5)--(20,2)--(22,4)--(24,2)--(25,3)--(28,0);
        \draw[ultra thick] [dashed] (14,0)--(14,8);
        \draw[ultra thick] [green] (0,0)--(1,1)--(2,0)--(6,4)--(7,3)--(10,6)--(11,1)--(12,2)--(14,0)--(28,0);
        \draw[ultra thick] [red] (0,0)--(14,0)--(15,1)--(16,0)--(17,1)--(18,0)--(19,3)--(20,2)--(22,4)--(24,2)--(25,3)--(28,0);
        \draw[ultra thick] [blue] (0,0)--(10,0)--(11,4)--(14,4)--(18,4)--(19,0)--(28,0);
        \draw (0,0) node[anchor=north]{{\LARGE $0$}};
        \draw (11,0) node[anchor=north]{{\LARGE $s$}};
        \draw (14,0) node[anchor=north]{{\LARGE $k$}};
        \draw (19,0) node[anchor=north]{{\LARGE $s'$}};
        \draw (28,0) node[anchor=north]{{\LARGE $2k$}};
    \end{tikzpicture}
    }
\end{subfigure}
    \caption{An illustration of two walks of $(\hD_2^N)^k(\hD_1^N)^k$, where cancellation of weights happen and blow-up is resolved. \\
    Top panel: a walk $\vec r$, with the green and red curves indicating $r_1$ and $r_2$ respectively, and the black curve indicating $\sH =r_1+r_2$. \\
    Bottom panel: another walk $\vec r$, with the green, red, and blue curves indicating $r_1$, $r_2$, and $r_{j_*}$ (for some $j_*\in \llbracket 3,N\rrbracket$) respectively, and the black curve indicating $\sH =r_1+r_2+r_{j_*}$.\\
    (These two walks together correspond to a term in \Cref{defn:core} with $m=2$, $\bp_2$ having one discontinuous point in $(0,\bQ_1)$ ($\bdel_{2,1}=1$), and all other $\bp_j=0$, and the virtual block $\bb_2\neq \varnothing$.)}
    \label{fig:blp}
\end{figure}

\begin{example} \label{ex:blowup2}
For the same operator $(\hD_2^N)^k(\hD_1^N)^k$, consider walks $\vec r$ of the following form illustrated in the bottom panel of \Cref{fig:blp}.
There exists a $j_*\in \llbracket 3, N\rrbracket$, such that for each $j\in \llbracket 3, N\rrbracket\setminus\{j_*\}$, we have $r_j=0$.
For $r_1$, $r_2$, and $r_{j_*}$, they are integer-valued and non-negative on $\llbracket 0, 2k\rrbracket$.
For some $s\in \llbracket 2, k\rrbracket$ and $s'\in \llbracket k+1, 2k\rrbracket$ (necessarily $s'-k$ should be odd), we have
\begin{itemize}
    \item $r_{j_*}=0$ on $\llbracket 0, s-1\rrbracket\cup\llbracket s', 2k\rrbracket$, and $r_{j_*}$ is constant and positive on $\llbracket s, s'-1\rrbracket$.
    \item $r_2=0$ on $\llbracket 0, k\rrbracket$. For each $t\in \llbracket k+1, s'-1\rrbracket$ or $t\in \llbracket s', 2k\rrbracket$, we have $|r_2(t)-r_2(t-1)|=1$. Also $r_{j_*}(s')+r_2(s')=r_{j_*}(s'-1)+r_2(s'-1)-1$ and $r_2(2k)=0$.
    \item For each $t\in\llbracket 1,s-1\rrbracket$ or $t\in \llbracket s+1, k\rrbracket$, we have $|r_1(t)-r_1(t-1)|=1$. Also $r_1(s)+r_{j_*}(s)=r_1(s-1)+r_{j_*}(s-1)-1$, and $r_1=0$ on $\llbracket k, 2k\rrbracket$.
\end{itemize}
We note that for any such $\vec r$ with given $s$, it is uniquely determined by the sum $\sH =r_1+r_2+r_{j_*}$, which is a path on $\llbracket 0, 2k\rrbracket$, starting and ending at $0$ and taking $\pm 1$ at each step, and staying non-negative, and staying $\ge \sH(k)$ on $\llbracket s, k\rrbracket$.
The number $s'$ is the smallest integer in $\llbracket k+1, 2k\rrbracket$ with $\sH(s')<\sH(k)$.

Let $r^*_1=r_1$, $r^*_2=r_2+r_{j_*}$, and $r^*_j=0$ for each $j\in\llbracket 3, N\rrbracket$.
The map $\vec r\mapsto \vec r^*$ is onto, from the walks in this example to those in the previous example.
Each walk in the previous example has precisely $N-2$ pre-images under this map.
Also, $w(\vec r)>0$ and $w(\vec r^*)<0$, and one can check that $|(N-2)w(\vec r)+w(\vec r^*)|$ is of order at most $N^{2k-1}\cdot(s'-k)k^{1/2}N^{-1}$. The factor $N^{2k-1}$ is from the second line of \eqref{eq_x3}, which is roughly the same for $(N-2)w(\vec r)$ and $w(\vec r^*)$; and the factor $(s'-k)k^{1/2}N^{-1}$ is from the third line of \eqref{eq_x3}, which (for $(N-2)w(\vec r)+w(\vec r^*)$) is bounded by $N^{-1}\sum_{t=k+1}^{s'-1}\sH(t)$, whose asymptotics is estimated by assuming that random walks converge to Brownian objects under diffusive scaling.

Note that for given $s$ and $s'$, the number of $\vec r^*$ is of order at most $2^{2k} k^{-3/2}(k-s+1)^{-1/2}(s'-k)^{-3/2}$. Here (as in the previous example) $2^{2k}k^{-3/2}$ comes from the Catalan number, and the factors $(k-s+1)^{-1/2}$ and $(s'-k)^{-3/2}$ account for the requirements of $\ge \sH(k)$ on $\llbracket s, k\rrbracket$, and $\ge \sH(k)=\sH(s'-1)$ on $\llbracket k, s'-1\rrbracket$, respectively, which we estimated as probabilities for uniformly random positive random walk bridges.

Multiplying $N^{2k-1}\cdot(s'-k)k^{1/2}N^{-1}$ with  $2^{2k} k^{-3/2}(k-s+1)^{-1/2}(s'-k)^{-3/2}$ from the previous two paragraphs
and summing over $s$ and $s'$, one gets order $(2N)^{2k}N^{-2}$. Again, summing over all possible choices of a pair of variables which can replace $x_1$ and $x_2$, we get the desired order $(2N)^{2k}$, announced at the beginning of Section \ref{Section_examples_blow_ups}.
\end{example}

The situation in Example \ref{ex:blowup} corresponds as $N\to\infty$ to having $\bb_2>0$ for virtual block process in Definition \ref{Definition_virtual_block}. On the other hand, the situation in Example \ref{ex:blowup2} corresponds to $\bb_2=\varnothing$ case. The cancellation of such terms is also responsible for the necessity of considering the principle value integral in Definition \ref{defn:core}.

\subsection{Classification of beginnings and endings}  \label{ssec:classi}

For a walk $\vec r$ of $(\hD_{i_m}^{N_m})^{k_m} \cdots (\hD_{i_1}^{N_1})^{k_1}$, we set up the following notations:
\begin{itemize}
    \item For each $j\in\llbracket 1,N\rrbracket$ and $\ell\in\llbracket 1,m\rrbracket$, if $i_\ell\neq j$, then we let $\Delta_{j,\ell}$ to be the set of all ${t\in \llbracket Q_{\ell-1}+1, Q_\ell\rrbracket}$, such that $r_j(t-1)\neq r_j(t)$; otherwise we let $\Delta_{j,\ell}=\emptyset$.
    Let $\delta_{j,\ell}=|\Delta_{j,\ell}|$, $\delta_j=\sum_{\ell=1}^m \delta_{j,\ell}$, $\delta=\sum_{j=1}^N \delta_j$,
    and $\Delta_j=\bigcup_{\ell=1}^m\Delta_{j,\ell}$, $\Delta=\bigcup_{j=1}^N\Delta_j$.

    Note that $\Delta=\{t\in \llbracket 1, Q_m\rrbracket: O(t)=\top\}$ for $O$ from the corresponding term in expansion, and that this is a discrete version of $\bDel$ in \eqref{eq_discont_def_1}, \eqref{eq_discont_def_2}.
    \item Let $J$ be the collection of all $j$ which appear in $(i_1,\ldots, i_m)$. These are the indices of the main variables. For each $j\in J$, we let $a_j$ (resp.\ $b_j$) be the smallest (resp.\ largest) $\ell$ such that $j=i_\ell$.
    \item Let $U$ be the collection of all $j\in \llbracket 1, N\rrbracket \setminus J$, with $\sum_{\ell=1}^m \delta_{j,\ell}>0$. These are the indices of auxiliary variables; eventually they give rise to $\bp_\ell$ with $\ell>m$ in Definition \ref{Definition_block_process} of the block process.
\end{itemize}
We rewrite the expression in \Cref{defn:wr} in the new notations as
\begin{equation}  \label{eq:weight}
    \begin{split}
w(\vec r) = \prod_{\ell=1}^m & (-1)^{\bigl|\{t\in\llbracket Q_{\ell-1}+1, Q_\ell \rrbracket:\, \sH(t)=\sH(t-1)-1,\, r_{i_\ell}(t)\ge r_{i_\ell}(t-1)\}\bigr|} \\ &\times(\sqrt{NN_\ell})^{k_\ell} N_\ell^{(\sH(Q_{\ell-1})-\sH(Q_\ell))/2- \bigl|\llbracket Q_{\ell-1}+1, Q_\ell \rrbracket\cap\Delta\bigr| }\\
&\times \prod_{\substack{t\in \llbracket Q_{\ell-1}+1, Q_\ell \rrbracket\setminus\Delta \\ \sH(t)=\sH(t-1)-1 }} \left(1+\frac{2r_{i_\ell}(t-1)}{\beta N_\ell}-\frac{\bigl|\{j\in \llbracket 1,N\rrbracket: r_j(t-1)\ge r_{i_\ell}(t-1)\}\bigr|}{N_\ell}\right).
\end{split}
\end{equation}

We next introduce a notation, which is related to virtual blocks of Definition \ref{Definition_virtual_block} and plays an important role in the structure of cancellations.

\begin{definition}  \label{defn:discrete_virtual}
For each $j\in J$, we let $\vartheta_j=\min\{t\in \llbracket 1, \tfrac{1}{4}(Q_{a_j}-Q_{a_j-1})\rrbracket: \sH(Q_{a_j-1}+t)<\sH(Q_{a_j-1})\}$. If the last set is empty, then we let $\vartheta_j=\varnothing$.

For each $j\in J$, we let $\dot\vartheta_j=\min\{t\in \llbracket 1, (Q_{a_j}-Q_{a_j-1})/4\rrbracket: Q_{a_j-1}+t\in \Delta\}$. If the last set is empty, then $\dot \vartheta_j=\varnothing$.

\end{definition}
 Note that if $r_{j}(Q_{a_j-1})=0$, then $\dot\vartheta_j\le \vartheta_j$. The constant $\tfrac{1}{4}$ in the definition of $\vartheta_j$ is not important and can be replaced by any other positive number smaller than $\tfrac{1}{2}$.

Next, we take a walk $\vec r$ and consider the corresponding set $J$. We would like to classify points of $J$ in two ways. For the first way, summarized in Table \ref{tab:1}, we look into the behavior of $\vec r(t)$ for $t$ immediately to the right from $Q_{a_j-1}$, and partition $J$ into eight sets $J_\rmI$,  $J_\rmII$, $J_\rmIII$, $J_\rmIV$, $J_\rmV$, $J_\rmVIa$, $J_\rmVIb$, $J_\rmVIc$ (see \Cref{fig:beginning}). For the second way, we look into the behavior of each $r_j(t)$ at the points $t$ after which $r_j$ becomes zero and partition $J$ into three sets $J_{\rmA}$, $J_{\rmB}$, $J_{\rmC}$ (see \Cref{fig:ending}).

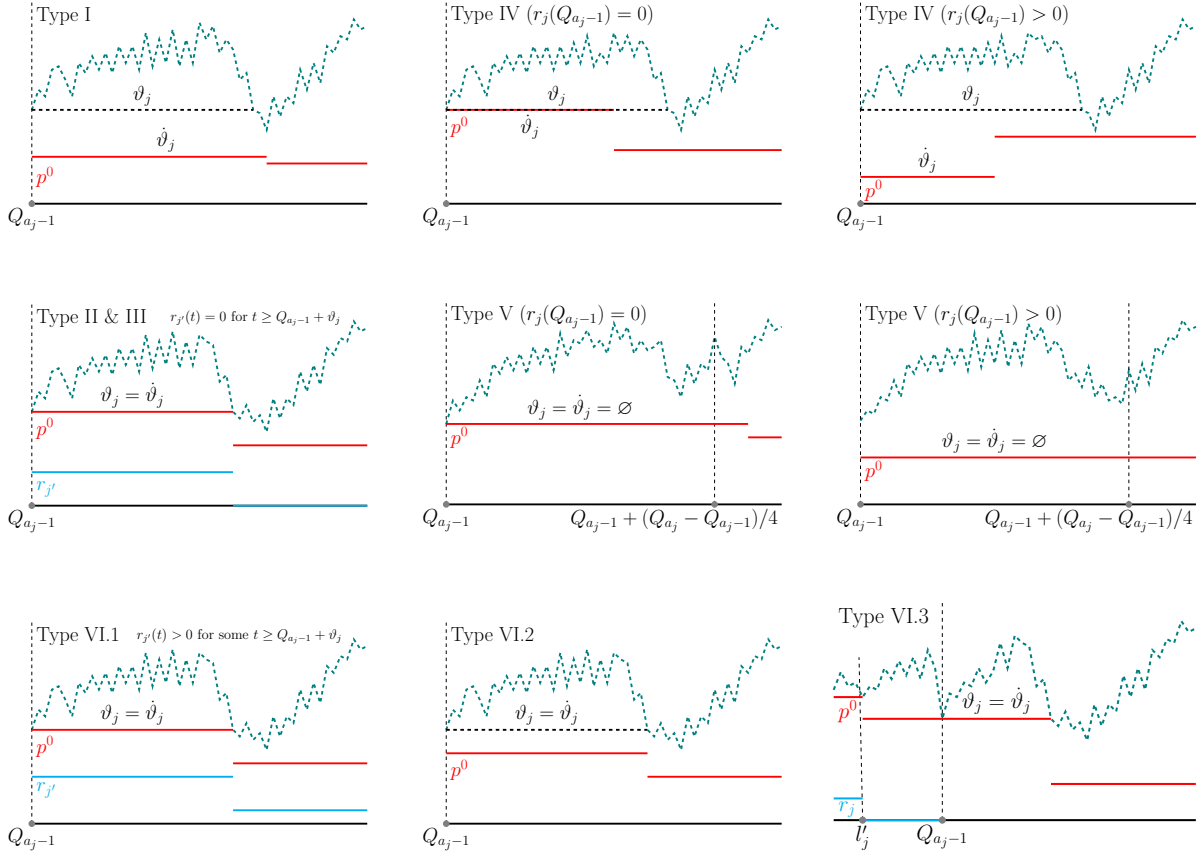
\begin{figure}[t]
    \centering
\begin{subfigure}{0.3\textwidth}
    \resizebox{0.99\textwidth}{!}{
    \begin{tikzpicture}
        \draw[teal] [dashed] [ultra thick] (0,0.8)--(0.2,1.4)--(0.4,1.2)--(0.6,2)--(0.8,2.1)--(1,1.7)--(1.2,1.2)--(1.4,2.2)--(1.6,1.9)--(1.8,2.4)--(2,2)--(2.2,2.5)--(2.4,1.9)--(2.6,2.7)--(2.8,2.2)--
        (3.0,2.7)--(3.2,2)--(3.4,3.1)--(3.6,2.1)--(3.8,2.9)--(4.0,2.2)--(4.2,3.2)--(4.4,2.1)--(4.6,2.9)--(4.8,2.2)--(5.0,3.4)--(5.2,2.9)--(5.4,3.3)--(5.6,2.5)--(5.8,2.9)--(6.0,2.6)--(6.2,1.7)--(6.4,1.9)--(6.6,0.8)--(6.8,0.7)--(7.0,0.2)--(7.2,1.2)--(7.4,0.6)--(7.6,1.2)--(7.8,1.0)--(8.0,2.0)--(8.2,1.4)--(8.4,2.2)--(8.6,1.6)--(8.8,2.7)--(9.0,2.5)--(9.2,3.0)--(9.4,2.9)--(9.6,3.5)--(9.8,3.2)--(10.0,3.3);
        \draw [dashed] (0,-2)--(0,4);
        \draw [ultra thick] [dashed] (0,0.8)--(6.6,0.8);

        \draw[red] [ultra thick] (0,-0.6)--(7,-0.6);

        \draw[red] [ultra thick] (7,-0.8)--(10,-0.8);

        \draw [ultra thick] (0,-2)--(10,-2);

        \draw (3.3,0.8) node[anchor=south]{{\LARGE $\vartheta_j$}};
        \draw (4,-0.6) node[anchor=south]{{\LARGE $\dot\vartheta_j$}};
        \draw (0,-2) node[anchor=north]{{\LARGE $Q_{a_j-1}$}};
        \draw (0,3.6) node[anchor=west]{{\LARGE Type I}};
        \draw [red] (0,-1.2) node[anchor=west]{{\LARGE $p^0$}};
        \draw[gray, fill=gray] [ultra thick] (0,-2) circle (2pt);
    \end{tikzpicture}}
\end{subfigure}
\hspace{0.3cm}
\begin{subfigure}{0.3\textwidth}
    \resizebox{0.99\textwidth}{!}{
    \begin{tikzpicture}
        \draw[teal] [dashed] [ultra thick] (0,0.8)--(0.2,1.4)--(0.4,1.2)--(0.6,2)--(0.8,2.1)--(1,1.7)--(1.2,1.2)--(1.4,2.2)--(1.6,1.9)--(1.8,2.4)--(2,2)--(2.2,2.5)--(2.4,1.9)--(2.6,2.7)--(2.8,2.2)--
        (3.0,2.7)--(3.2,2)--(3.4,2.7)--(3.6,2.1)--(3.8,2.8)--(4.0,2.2)--(4.2,3)--(4.4,2.1)--(4.6,2.5)--(4.8,2.2)--(5.0,3.1)--(5.2,2.5)--(5.4,3)--(5.6,2.4)--(5.8,2.9)--(6.0,2.6)--(6.2,1.7)--(6.4,1.9)--(6.6,0.8)--(6.8,0.7)--(7.0,0.2)--(7.2,1.2)--(7.4,0.6)--(7.6,1.2)--(7.8,1.0)--(8.0,2.0)--(8.2,1.4)--(8.4,2.2)--(8.6,1.6)--(8.8,2.7)--(9.0,2.5)--(9.2,3.0)--(9.4,2.9)--(9.6,3.5)--(9.8,3.2)--(10.0,3.3);
        \draw [dashed] (0,-2)--(0,4);

        \draw[red] [ultra thick] (0,0.8)--(5,0.8);

        \draw[red] [ultra thick] (5,-0.4)--(10,-0.4);

        \draw [ultra thick] [dashed] (0,0.8)--(6.6,0.8);

        \draw [ultra thick] (0,-2)--(10,-2);

        \draw (3.3,0.8) node[anchor=south]{{\LARGE $\vartheta_j$}};
        \draw (2.5,0.8) node[anchor=north]{{\LARGE $\dot\vartheta_j$}};
        \draw (0,-2) node[anchor=north]{{\LARGE $Q_{a_j-1}$}};
        \draw (0,3.6) node[anchor=west]{{\LARGE Type IV ($r_j(Q_{a_j-1})=0$)}};
        \draw [red] (0,0.3) node[anchor=west]{{\LARGE $p^0$}};
        \draw[gray, fill=gray] [ultra thick] (0,-2) circle (2pt);
    \end{tikzpicture}}
\end{subfigure}
\hspace{0.3cm}
\begin{subfigure}{0.3\textwidth}
    \resizebox{0.99\textwidth}{!}{
    \begin{tikzpicture}
        \draw[teal] [dashed] [ultra thick] (0,0.8)--(0.2,1.4)--(0.4,1.2)--(0.6,2)--(0.8,2.1)--(1,1.7)--(1.2,1.2)--(1.4,2.2)--(1.6,1.9)--(1.8,2.4)--(2,2)--(2.2,2.5)--(2.4,1.9)--(2.6,2.7)--(2.8,2.2)--
        (3.0,2.7)--(3.2,2)--(3.4,2.7)--(3.6,2.1)--(3.8,2.8)--(4.0,2.2)--(4.2,3)--(4.4,2.1)--(4.6,2.5)--(4.8,2.2)--(5.0,3.1)--(5.2,2.9)--(5.4,3)--(5.6,2.5)--(5.8,2.9)--(6.0,2.6)--(6.2,1.7)--(6.4,1.9)--(6.6,0.8)--(6.8,0.7)--(7.0,0.2)--(7.2,1.2)--(7.4,0.6)--(7.6,1.2)--(7.8,1.0)--(8.0,2.0)--(8.2,1.4)--(8.4,2.2)--(8.6,1.6)--(8.8,2.7)--(9.0,2.5)--(9.2,3.0)--(9.4,2.9)--(9.6,3.5)--(9.8,3.2)--(10.0,3.3);
        \draw [dashed] (0,-2)--(0,4);
        \draw [ultra thick] [dashed] (0,0.8)--(6.6,0.8);

        \draw[red] [ultra thick] (0,-1.2)--(4,-1.2);

        \draw[red] [ultra thick] (4,0)--(10,0);

        \draw [ultra thick] (0,-2)--(10,-2);

        \draw (3.3,0.8) node[anchor=south]{{\LARGE $\vartheta_j$}};
        \draw (2,-1.2) node[anchor=south]{{\LARGE $\dot\vartheta_j$}};
        \draw (0,-2) node[anchor=north]{{\LARGE $Q_{a_j-1}$}};
        \draw (0,3.6) node[anchor=west]{{\LARGE Type IV ($r_j(Q_{a_j-1})>0$)}};
        \draw [red] (0,-1.6) node[anchor=west]{{\LARGE $p^0$}};
        \draw[gray, fill=gray] [ultra thick] (0,-2) circle (2pt);
    \end{tikzpicture}}
\end{subfigure}
\par\bigskip\bigskip
\begin{subfigure}{0.3\textwidth}
    \resizebox{0.99\textwidth}{!}{
    \begin{tikzpicture}
        \draw[teal] [dashed] [ultra thick] (0,0.8)--(0.2,1.4)--(0.4,1.2)--(0.6,2)--(0.8,2.1)--(1,1.7)--(1.2,1.2)--(1.4,2.2)--(1.6,1.9)--(1.8,2.4)--(2,2)--(2.2,2.5)--(2.4,1.9)--(2.6,2.7)--(2.8,2.2)--
        (3.0,2.7)--(3.2,2)--(3.4,3.1)--(3.6,2.1)--(3.8,2.9)--(4.0,2.2)--(4.2,3.2)--(4.4,2.1)--(4.6,2.9)--(4.8,2.2)--(5.0,3.1)--(5.2,2.9)--(5.4,2.8)--(5.6,1.7)--(5.8,1.9)--(6,0.8)--(6.2,0.6)--(6.4,1)--(6.6,0.5)--(6.8,0.7)--(7.0,0.2)--(7.2,1.2)--(7.4,0.6)--(7.6,1.2)--(7.8,1.0)--(8.0,2.0)--(8.2,1.4)--(8.4,2.2)--(8.6,1.6)--(8.8,2.7)--(9.0,2.5)--(9.2,3.0)--(9.4,2.9)--(9.6,3.5)--(9.8,3.2)--(10.0,3.3);
        \draw [dashed] (0,-2)--(0,4);

        \draw[red] [ultra thick] (0,0.8)--(6,0.8);

        \draw[red] [ultra thick] (6,-0.2)--(10,-0.2);

        \draw [ultra thick] (0,-2)--(10,-2);

        \draw[cyan] [ultra thick] (0,-1)--(6,-1);

        \draw[cyan] [ultra thick] (6,-2)--(10,-2);

        \draw (3,0.8) node[anchor=south]{{\LARGE $\vartheta_j=\dot\vartheta_j$}};
        \draw (0,-2) node[anchor=north]{{\LARGE $Q_{a_j-1}$}};
        \draw (0,3.6) node[anchor=west]{{\LARGE Type II \& III}};
        \draw (4,3.6) node[anchor=west]{{\large $r_{j'}(t)=0$ for $t\ge Q_{a_j-1}+\vartheta_j$}};
        \draw [red] (0,0.3) node[anchor=west]{{\LARGE $p^0$}};
        \draw [cyan] (0,-1.5) node[anchor=west]{{\LARGE $r_{j'}$}};
        \draw[gray, fill=gray] [ultra thick] (0,-2) circle (2pt);
    \end{tikzpicture}}
\end{subfigure}
\hspace{0.3cm}
\begin{subfigure}{0.3\textwidth}
    \resizebox{0.99\textwidth}{!}{
    \begin{tikzpicture}
        \draw[teal] [dashed] [ultra thick] (0,0.4)--(0.2,1.1)--(0.4,0.8)--(0.6,1.2)--(0.8,1.1)--(1,1.7)--(1.2,1.2)--(1.4,2.2)--(1.6,1.9)--(1.8,2.4)--(2,2)--(2.2,2.5)--(2.4,1.9)--(2.6,2.7)--(2.8,2.2)--
        (3.0,2.7)--(3.2,2)--(3.4,3.1)--(3.6,2.4)--(3.8,2.9)--(4.0,2.8)--(4.2,3.2)--(4.4,2.7)--(4.6,2.9)--(4.8,2.5)--(5.0,3.4)--(5.2,2.9)--(5.4,3.3)--(5.6,2.5)--(5.8,2.9)--(6.0,2.6)--(6.2,2.7)--(6.4,2.9)--(6.6,1.8)--(6.8,1.7)--(7.0,1.2)--(7.2,2.2)--(7.4,1.6)--(7.6,2.2)--(7.8,2.0)--(8.0,3.0)--(8.2,2.4)--(8.4,2.2)--(8.6,1.6)--(8.8,2.7)--(9.0,2.5)--(9.2,3.0)--(9.4,2.9)--(9.6,3.5)--(9.8,3.2)--(10.0,3.3);
        \draw [dashed] (0,-2)--(0,4);
        \draw [dashed] (8,-2)--(8,4);

        \draw[red] [ultra thick] (0,0.4)--(9,0.4);

        \draw[red] [ultra thick] (9,0)--(10,0);

        \draw [ultra thick] (0,-2)--(10,-2);

        \draw (4,0.4) node[anchor=south]{{\LARGE $\vartheta_j=\dot\vartheta_j=\varnothing$}};
        \draw (0,-2) node[anchor=north]{{\LARGE $Q_{a_j-1}$}};
        \draw (6.8,-2) node[anchor=north]{{\LARGE $Q_{a_j-1}+(Q_{a_j}-Q_{a_j-1})/4$}};
        \draw (0,3.6) node[anchor=west]{{\LARGE Type V ($r_j(Q_{a_j-1})=0$)}};
        \draw [red] (0,0) node[anchor=west]{{\LARGE $p^0$}};
        \draw[gray, fill=gray] [ultra thick] (0,-2) circle (2pt);
        \draw[gray, fill=gray] [ultra thick] (8,-2) circle (2pt);
    \end{tikzpicture}}
\end{subfigure}
\hspace{0.3cm}
\begin{subfigure}{0.3\textwidth}
    \resizebox{0.99\textwidth}{!}{
    \begin{tikzpicture}
        \draw[teal] [dashed] [ultra thick] (0,0.5)--(0.2,0.8)--(0.4,0.7)--(0.6,1.2)--(0.8,1.1)--(1,1.7)--(1.2,1.2)--(1.4,2.2)--(1.6,1.9)--(1.8,2.4)--(2,2)--(2.2,2.5)--(2.4,1.9)--(2.6,2.7)--(2.8,2.2)--
        (3.0,2.7)--(3.2,2)--(3.4,3.1)--(3.6,2.1)--(3.8,2.9)--(4.0,2.2)--(4.2,3.2)--(4.4,2.1)--(4.6,2.9)--(4.8,2.2)--(5.0,3.4)--(5.2,2.9)--(5.4,3.3)--(5.6,2.5)--(5.8,2.9)--(6.0,2.6)--(6.2,1.7)--(6.4,1.9)--(6.6,1.5)--(6.8,2.1)--(7.0,1.2)--(7.2,1.6)--(7.4,1.1)--(7.6,1.2)--(7.8,1.0)--(8.0,2.0)--(8.2,1.4)--(8.4,2.2)--(8.6,1.6)--(8.8,2.7)--(9.0,2.5)--(9.2,3.0)--(9.4,2.9)--(9.6,3.5)--(9.8,3.2)--(10.0,3.3);
        \draw [dashed] (0,-2)--(0,4);
        \draw [dashed] (8,-2)--(8,4);

        \draw[red] [ultra thick] (0,-0.6)--(10,-0.6);

        \draw [ultra thick] (0,-2)--(10,-2);

        \draw (4,-0.6) node[anchor=south]{{\LARGE $\vartheta_j=\dot\vartheta_j=\varnothing$}};
        \draw (0,-2) node[anchor=north]{{\LARGE $Q_{a_j-1}$}};
        \draw (6.8,-2) node[anchor=north]{{\LARGE $Q_{a_j-1}+(Q_{a_j}-Q_{a_j-1})/4$}};
        \draw (0,3.6) node[anchor=west]{{\LARGE Type V ($r_j(Q_{a_j-1})>0$)}};
        \draw [red] (0,-1) node[anchor=west]{{\LARGE $p^0$}};
        \draw[gray, fill=gray] [ultra thick] (0,-2) circle (2pt);
        \draw[gray, fill=gray] [ultra thick] (8,-2) circle (2pt);
    \end{tikzpicture}}
\end{subfigure}
\par\bigskip\bigskip
\begin{subfigure}{0.3\textwidth}
    \resizebox{0.99\textwidth}{!}{
    \begin{tikzpicture}
        \draw[teal] [dashed] [ultra thick] (0,0.8)--(0.2,1.4)--(0.4,1.2)--(0.6,2)--(0.8,2.1)--(1,1.7)--(1.2,1.2)--(1.4,2.2)--(1.6,1.9)--(1.8,2.4)--(2,2)--(2.2,2.5)--(2.4,1.9)--(2.6,2.7)--(2.8,2.2)--
        (3.0,2.7)--(3.2,2)--(3.4,3.1)--(3.6,2.1)--(3.8,2.9)--(4.0,2.2)--(4.2,3.2)--(4.4,2.1)--(4.6,2.9)--(4.8,2.2)--(5.0,3.1)--(5.2,2.9)--(5.4,2.8)--(5.6,1.7)--(5.8,1.9)--(6,0.8)--(6.2,0.6)--(6.4,1)--(6.6,0.5)--(6.8,0.7)--(7.0,0.2)--(7.2,1.2)--(7.4,0.6)--(7.6,1.2)--(7.8,1.0)--(8.0,2.0)--(8.2,1.4)--(8.4,2.2)--(8.6,1.6)--(8.8,2.7)--(9.0,2.5)--(9.2,3.0)--(9.4,2.9)--(9.6,3.5)--(9.8,3.2)--(10.0,3.3);
        \draw [dashed] (0,-2)--(0,4);

        \draw[red] [ultra thick] (0,0.8)--(6,0.8);

        \draw[red] [ultra thick] (6,-0.2)--(10,-0.2);

        \draw [ultra thick] (0,-2)--(10,-2);

        \draw[cyan] [ultra thick] (0,-0.6)--(6,-0.6);

        \draw[cyan] [ultra thick] (6,-1.6)--(10,-1.6);

        \draw (3,0.8) node[anchor=south]{{\LARGE $\vartheta_j=\dot\vartheta_j$}};
        \draw (0,-2) node[anchor=north]{{\LARGE $Q_{a_j-1}$}};
        \draw (0,3.6) node[anchor=west]{{\LARGE Type VI.1}};
        \draw (3,3.6) node[anchor=west]{{\large $r_{j'}(t)>0$ for some $t\ge Q_{a_j-1}+\vartheta_j$}};
        \draw [red] (0,0.3) node[anchor=west]{{\LARGE $p^0$}};
        \draw [cyan] (0,-1) node[anchor=west]{{\LARGE $r_{j'}$}};
        \draw[gray, fill=gray] [ultra thick] (0,-2) circle (2pt);
    \end{tikzpicture}}
\end{subfigure}
\hspace{0.3cm}
\begin{subfigure}{0.3\textwidth}
    \resizebox{0.99\textwidth}{!}{
    \begin{tikzpicture}
        \draw[teal] [dashed] [ultra thick] (0,0.8)--(0.2,1.4)--(0.4,1.2)--(0.6,2)--(0.8,2.1)--(1,1.7)--(1.2,1.2)--(1.4,2.2)--(1.6,1.9)--(1.8,2.4)--(2,2)--(2.2,2.5)--(2.4,1.9)--(2.6,2.7)--(2.8,2.2)--
        (3.0,2.7)--(3.2,2)--(3.4,3.1)--(3.6,2.1)--(3.8,2.9)--(4.0,2.2)--(4.2,3.2)--(4.4,2.1)--(4.6,2.9)--(4.8,2.2)--(5.0,3.1)--(5.2,2.9)--(5.4,2.8)--(5.6,1.7)--(5.8,1.9)--(6,0.8)--(6.2,0.6)--(6.4,1)--(6.6,0.5)--(6.8,0.7)--(7.0,0.2)--(7.2,1.2)--(7.4,0.6)--(7.6,1.2)--(7.8,1.0)--(8.0,2.0)--(8.2,1.4)--(8.4,2.2)--(8.6,1.6)--(8.8,2.7)--(9.0,2.5)--(9.2,3.0)--(9.4,2.9)--(9.6,3.5)--(9.8,3.2)--(10.0,3.3);
        \draw [dashed] (0,-2)--(0,4);

        \draw [ultra thick] [dashed] (0,0.8)--(6,0.8);

        \draw[red] [ultra thick] (0,0.1)--(6,0.1);

        \draw[red] [ultra thick] (6,-0.6)--(10,-0.6);

        \draw [ultra thick] (0,-2)--(10,-2);

        \draw (3,0.8) node[anchor=south]{{\LARGE $\vartheta_j=\dot\vartheta_j$}};
        \draw (0,-2) node[anchor=north]{{\LARGE $Q_{a_j-1}$}};
        \draw (0,3.6) node[anchor=west]{{\LARGE Type VI.2}};
        \draw [red] (0,-0.4) node[anchor=west]{{\LARGE $p^0$}};
        \draw[gray, fill=gray] [ultra thick] (0,-2) circle (2pt);
    \end{tikzpicture}}
\end{subfigure}
\hspace{0.3cm}
\begin{subfigure}{0.3\textwidth}
    \resizebox{0.99\textwidth}{!}{
    \begin{tikzpicture}
        \draw[teal] [dashed] [ultra thick] (0,1.6)--(0.2,2.1)--(0.4,1.7)--(0.6,1.8)--(0.8,1.4)--(1,1.7)--(1.2,1.6)--(1.4,2.2)--(1.6,1.9)--(1.8,2.4)--(2,2)--(2.2,2.5)--(2.4,1.9)--(2.6,2.7)--(2.8,2.2)--
        (3.0,0.8)--(3.2,1.5)--(3.4,1.4)--(3.6,1.8)--(3.8,1.7)--(4.0,2.2)--(4.2,1.8)--(4.4,2.1)--(4.6,2.9)--(4.8,2.2)--(5.0,3.1)--(5.2,2.9)--(5.4,2.8)--(5.6,1.7)--(5.8,1.9)--(6,0.8)--(6.2,0.6)--(6.4,1)--(6.6,0.5)--(6.8,0.7)--(7.0,0.2)--(7.2,1.2)--(7.4,0.6)--(7.6,1.2)--(7.8,1.0)--(8.0,2.0)--(8.2,1.4)--(8.4,2.2)--(8.6,1.6)--(8.8,2.7)--(9.0,2.5)--(9.2,3.0)--(9.4,2.9)--(9.6,3.5)--(9.8,3.2)--(10.0,3.3);
        \draw [dashed] (3,-2)--(3,4);
        \draw [dashed] (0.8,-2)--(0.7,2.5);

        \draw [ultra thick] (0,-2)--(10,-2);

        \draw[red] [ultra thick] (0,1.4)--(0.8,1.4);

        \draw[red] [ultra thick] (0.8,0.8)--(6,0.8);

        \draw[cyan] [ultra thick] (0,0.6-2)--(0.8,0.6-2);

        \draw[cyan] [ultra thick] (0.8,0-2)--(3,0-2);

        \draw[red] [ultra thick] (6,-1)--(10,-1);

        \draw (4.5,0.8) node[anchor=south]{{\LARGE $\vartheta_j=\dot\vartheta_j$}};
        \draw (3,-2) node[anchor=north]{{\LARGE $Q_{a_j-1}$}};
        \draw (0.8,-2) node[anchor=north]{{\LARGE $l'_j$}};
        \draw (0,3.6) node[anchor=west]{{\LARGE Type VI.3}};
        \draw [red] (0,1) node[anchor=west]{{\LARGE $p^0$}};
        \draw [cyan] (0,-1.7) node[anchor=west]{{\LARGE $r_j$}};
        \draw[gray, fill=gray] [ultra thick] (3,-2) circle (2pt);
        \draw[gray, fill=gray] [ultra thick] (0.8,-2) circle (2pt);
    \end{tikzpicture}}
\end{subfigure}
    \caption{An illustration of the various beginning types.}
    \label{fig:beginning}
\end{figure}

\begin{enumerate}
\item[Type I:] Any $j\in J$ is of type I, if $r_{j}(Q_{a_j-1})>0$ and $\vartheta_j<\dot\vartheta_j$ or $\dot\vartheta_j=\varnothing$ (and $\vartheta_j\neq\varnothing$).
\item[Type II:] Any $j\in J$ is of type II, if $\sum_{\ell=1}^{a_j-1}\delta_{j,\ell}=0$ (thereby $r_{j}(Q_{a_j-1})=0$), and $\dot\vartheta_j=\vartheta_j \neq \varnothing$, and there exists $j'\in U$, such that $r_{j'}(t)$ is a positive constant for all $t\in \llbracket Q_{a_j-1}, Q_{a_j-1}+\vartheta_j-1\rrbracket$, but $r_{j'}(t)=0$ for each $t\ge Q_{a_j-1}+\vartheta_j$.

For type I or II, we also call $j$ a \emph{blow-up} index.

\item[Type III:] Any $j\in J$ is of type III,  if $\sum_{\ell=1}^{a_j-1}\delta_{j,\ell}=0$ (thereby $r_{j}(Q_{a_j-1})=0$), and $\dot\vartheta_j=\vartheta_j \neq \varnothing$, and there exists $j'\in J$, such that $r_{j'}(t)$ is a positive constant for all $t\in \llbracket Q_{a_j-1}, Q_{a_j-1}+\vartheta_j-1\rrbracket$, but $r_{j'}(t)=0$ for each $t\ge Q_{a_j-1}+\vartheta_j$.

We see that for each $j$ of type II or III, the corresponding $j'$ is unique.

\item[Type IV:] Any $j\in J$ is of type IV, if $\dot\vartheta_j<\vartheta_j$ or $\vartheta_j=\varnothing$ (and $\dot\vartheta_j\neq\varnothing$).

\item[Type V:] Any $j\in J$ is of type V, if $\dot\vartheta_j=\vartheta_j=\varnothing$.

\medskip

Any $j\in J$ that is not of type I to V is of type VI (miscellaneous).
More precisely, for each $j$ of type VI, $\dot\vartheta_j=\vartheta_j \ne\varnothing$, and one of the following happens:

    \item[Type VI.1:] $\sum_{\ell=1}^{a_j-1}\delta_{j,\ell}=0$ (thereby $r_{j}(Q_{a_j-1})=0$); and $r_{j'}(t)>0$ for some $t\ge Q_{a_j-1}+\vartheta_j$, where $j'\in \llbracket 1,N\rrbracket$ is the unique index with $r_{j'}(Q_{a_j-1})>r_{j'}(Q_{a_j-1}+\vartheta_j)$.
    \item[Type VI.2:] $r_{j}(Q_{a_j-1})>0$.
    \item[Type VI.3:] $\sum_{\ell=1}^{a_j-1}\delta_{j,\ell}>0$, but still $r_{j}(Q_{a_j-1})=0$.
\end{enumerate}

\begin{table}[t]
\begin{tabularx}{1.\textwidth} {
  | >{\centering\arraybackslash\hsize=.8\hsize}X
  | >{\centering\arraybackslash\hsize=.5\hsize}X
  | >{\centering\arraybackslash}X
  | >{\centering\arraybackslash\hsize=.8\hsize}X
  | >{\centering\arraybackslash\hsize=.8\hsize}X | }
 \hline
 &  $\vartheta_j<\dot\vartheta_j$ or $\dot\vartheta_j=\varnothing$ & $\dot\vartheta_j=\vartheta_j\neq\varnothing$ & $\dot\vartheta_j=\vartheta_j=\varnothing$ & $\dot\vartheta_j<\vartheta_j$ or $\vartheta_j=\varnothing$ \\
 \hline
$r_{j}(Q_{a_j-1})>0$  & I  & VI.2  & V & IV \\
\hline
$r_{j}(Q_{a_j-1})=0$ & None & II, III, VI.1, VI.3  & V & IV  \\
\hline
\end{tabularx}
\caption{A summary of the classification at the beginning.}\label{tab:1}
\end{table}

The guiding principle underlying this classification is the following: for Types I and II, there are cancellations between them that we shall exploit in order  to resolve the blow-up issue. For Types III, IV, V, each would make a significant contribution to the sum, which survives in the limit $N\to\infty$.  The sum of $w(\vec r)$ with $\vec r$ having any index of type VI would decay to zero as $N\to\infty$.

Coming back to Table \ref{tab:1}, let us comment why the bottom left entry is ``None'': indeed $\vartheta_j<\dot\vartheta_j$ together with $r_j(Q_{a_j-1})=0$ would imply $r_j(Q_{a_j-1}+\vartheta_j)=-1$, which is impossible.

\bigskip

\begin{figure}[t]
    \centering
\begin{subfigure}{0.40\textwidth}
    \resizebox{0.99\textwidth}{!}{
    \begin{tikzpicture}
        \draw[teal] [dashed] [ultra thick] (0,1.8)--(0.2,2.)--(0.4,1.2)--(0.6,2.2)--(0.8,2.1)--(1,1.7)--(1.2,1.2)--(1.4,2.2)--(1.6,1.9)--(1.8,2.4)--(2,2)--(2.2,2.5)--(2.4,1.9)--(2.6,2.7)--(2.8,2.2)--
        (3.0,2.7)--(3.2,2)--(3.4,3.1)--(3.6,2.1)--(3.8,2.9)--(4.0,2.2)--(4.2,3.2)--(4.4,2.1)--(4.6,2.9)--(4.8,2.2)--(5.0,3.1)--(5.2,2.9)--(5.4,2.8)--(5.6,1.7)--(5.8,1.9)--(6,1.8)--(6.2,1.6)--(6.4,2)--(6.6,1.5)--(6.8,1.7)--(7.0,1.2)--(7.2,2.2)--(7.4,1.4)--(7.6,1.7)--(7.8,1.0)--(8.0,2.0)--(8.2,2.1)--(8.4,1.2)--(8.6,1.6)--(8.8,0.7)--(9.0,1.5)--(9.2,1.0)--(9.4,1.1)--(9.6,0.2)--(9.8,0.5)--(10.0,0);
        \draw [dashed] (10,-2)--(10,4);

        \draw[red] [ultra thick] (0,0)--(10,0);

        \draw [ultra thick] (0,-2)--(10,-2);

        \draw (10,-2) node[anchor=north]{{\LARGE $Q_{b_j}$}};
        \draw (0,3.6) node[anchor=west]{{\LARGE Type A}};
        \draw [red] (0,-0.4) node[anchor=west]{{\LARGE $p^0$}};
        \draw[gray, fill=gray] [ultra thick] (10,-2) circle (2pt);
    \end{tikzpicture}}
\end{subfigure}
\hspace{1cm}
\begin{subfigure}{0.40\textwidth}
    \resizebox{0.99\textwidth}{!}{
    \begin{tikzpicture}
        \draw[teal] [dashed] [ultra thick] (0,0.8)--(0.2,1.4)--(0.4,1.2)--(0.6,2)--(0.8,2.1)--(1,1.7)--(1.2,1.2)--(1.4,2.2)--(1.6,1.9)--(1.8,2.4)--(2,2)--(2.2,2.5)--(2.4,1.9)--(2.6,2.7)--(2.8,2.2)--
        (3.0,2.7)--(3.2,2)--(3.4,3.1)--(3.6,2.1)--(3.8,2.9)--(4.0,2.2)--(4.2,3.2)--(4.4,2.1)--(4.6,2.9)--(4.8,2.2)--(5.0,3.1)--(5.2,2.9)--(5.4,2.8)--(5.6,1.7)--(5.8,1.9)--(6,0.8)--(6.2,0.6)--(6.4,1)--(6.6,0.5)--(6.8,0.7)--(7.0,0.2)--(7.2,1.2)--(7.4,0.6)--(7.6,1.2)--(7.8,1.0)--(8.0,2.0)--(8.2,1.4)--(8.4,2.2)--(8.6,1.6)--(8.8,2.7)--(9.0,2.5)--(9.2,3.0)--(9.4,2.9)--(9.6,3.5)--(9.8,3.2)--(10.0,3.3);
        \draw [dashed] (0,-2)--(0,4);

        \draw[red] [ultra thick] (0,0.8)--(6,0.8);

        \draw[red] [ultra thick] (6,-0.2)--(10,-0.2);

        \draw [ultra thick] (0,-2)--(10,-2);

        \draw[cyan] [ultra thick] (0,-1)--(6,-1);

        \draw[cyan] [ultra thick] (6,-2)--(10,-2);

        \draw (3,0.8) node[anchor=south]{{\LARGE $\vartheta_{j'}=\dot\vartheta_{j'}$}};
        \draw (0,-2) node[anchor=north]{{\LARGE $Q_{a_{j'}-1}$}};
        \draw (0,3.6) node[anchor=west]{{\LARGE Type B}};
        \draw (4,3.6) node[anchor=west]{{\large $r_j(t)=0$ for $t\ge Q_{a_{j'}-1}+\vartheta_{j'}$}};
        \draw [red] (0,0.4) node[anchor=west]{{\LARGE $p^0$}};
        \draw [cyan] (0,-1.5) node[anchor=west]{{\LARGE $r_j$}};
        \draw[gray, fill=gray] [ultra thick] (0,-2) circle (2pt);
    \end{tikzpicture}}
\end{subfigure}
    \caption{An illustration of the various ending types.}
    \label{fig:ending}
\end{figure}
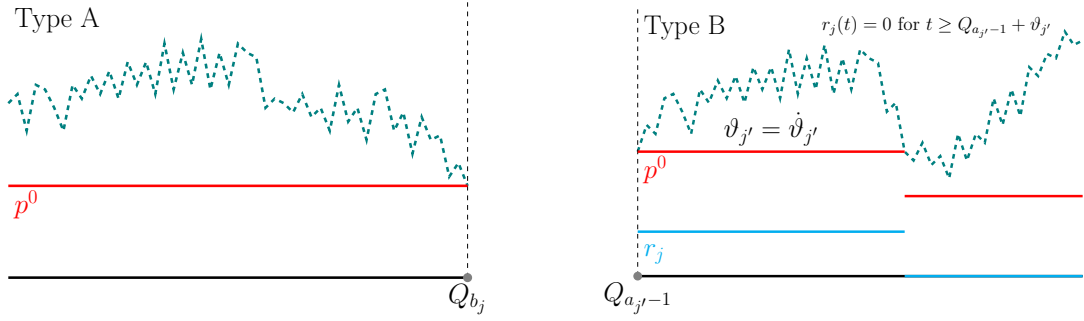

Here is the second classification: Any $j\in J$ is of
\begin{itemize}
    \item[Type A:] if $\sum_{\ell = b_j+1}^m \delta_{j,\ell}=0$.  Equivalently, the variable $x_j$ no longer appears in the monomials after the last times it has been a main variable;
    \item[Type B:] if  there is a $j'\in J_\rmIII$, such that $\dot\vartheta_{j'}=\vartheta_{j'} \neq \varnothing$, and $\sum_{\ell=1}^{a_{j'}-1}\delta_{j',\ell}=0$; and $r_{j}(t)$ being a positive constant for all $t\in \llbracket Q_{a_{j'}-1}, Q_{a_{j'}-1}+\vartheta_{j'}-1\rrbracket$, and $r_{j}(t)=0$ for each $t\ge Q_{a_j-1}+\vartheta_{j'}$;
    \item[Type C:] in all other cases.
\end{itemize}
Note that Type B definition implicitly includes $a_{j'}-1\ge b_j$. Further examining the definition, one sees a bijection between Type B and Type III: one type turns into another under $j\leftrightarrow j'$ correspondence, which implies $|J_{\rmIII}|=|J_{\rmB}|$. Types A and B would make a significant contribution to the sum, while the sum of $w(\vec r)$ with $\vec r$ having any index of type C would eventually decay to zero as $N\to\infty$.

\subsection{A general bound on walk weights}   \label{ssec:generabd}

We next give a general upper bound on the sum of $|w(\vec r)|$, with fixed index types.
This bound will be repeatedly used for various purposes.
In \Cref{prop:mbdc} we only consider those $\vec r$ with $J_\rmII=\emptyset$. 
Later in this subsection, \Cref{cor:mbdcII} extends the analysis to the case of non-empty $J_\rmII$.

As already stated, we always fix $m\in \N$.
In addition, we fix the following data: indices $i_1,\dots,i_m$ (recall that the set of walks $\sB=\sB[i_1,\ldots,i_m]$ was defined in Definition \ref{defn:wallblock}); disjoint sets $J,U\subset \llbracket 1, N\rrbracket$ from Section \ref{ssec:classi}; the partitioning of $J$ into $J_\rmI$,  $J_\rmII$, $J_\rmIII$, $J_\rmIV$, $J_\rmV$, $J_\rmVIa$, $J_\rmVIb$, $J_\rmVIc$ and into $J_{\rmA}$, $J_{\rmB}$, $J_{\rmC}$. We assume that this partition is such that $J_\rmII=\emptyset$. Furthermore, we fix the numbers $\vartheta_j$ for all $j\in J$ and also $\dot \vartheta_{j}$ for\footnote{Generally speaking, we want to fix as many parameters, as possible on this step. However, fixing $\dot \vartheta_j$, $j\in J_\rmI$, is harder than for other $j$, because fixing such $\dot \vartheta_j$ might lead to two cases which we would have to treat differently: it could correspond either to the last jump time, or not --- this notion will appear in Step 1 of the proof. Instead, we decided to avoid discussing additional cases and do not fix this parameter.} $j\in J\setminus J_\rmI$. Finally, we fix the total number of jumps, which is $\delta$.

 For any $h\ge 1$, we define $\sB^h\subset \sB$ to be the set of all walks agreeing with all of the above data and such that  $(h-1)N^{1/3}< \max_t \sH(t)<2hN^{1/3}$.

\begin{prop}  \label{prop:mbdc}
For any $C_1>0$, there exist $C_2=C_2(C_1)>0$ and $C_3=C_3(C_1)>0$, such that if $|N-N_\ell|<C_1N^{2/3}$ for all $\ell\in\llbracket 1, m \rrbracket$, 
then for each $h\ge 1$, we have
\begin{multline}  \label{eq:wbdfdt}
\sum_{\vec r\in\sB^h}|w(\vec r)|< \prod_{\ell=1}^m (2\sqrt{NN_\ell})^{k_\ell} N^{-|U|-|J|+\frac{1}{3}|J_\rmI| -\frac{1}{3}\bigl(|J_\rmIV|+|\{j\in J_\rmIV:\,\vartheta_j=\varnothing\}|+|J_\rmVIb|+\max\{|J_\rmVIa|,|J_\rmVIc|\}+|J_\rmC|\bigr)}
\\
\times C_2^{\delta+1} \frac{h^{\delta- 2|U|}}{(\delta-2|U|)!}\exp(-C_3h^2) \big((|U|+1)\log(|U|+2)\big)^{|U|} \big(h/(|U|+1)\big)^{C_3|U|+m\mathbf{1}_{h\ge |U|}} \prod_{j\in J:\vartheta_j \ne \varnothing} \vartheta_j^{-3/2}.
\end{multline}
\end{prop}
\begin{remark}
 There is some freedom in choosing a form for the upper bound \eqref{eq:wbdfdt}. The most important parts are the precise power of $N$ in the first line, as well as $\frac{1}{(\delta-2|U|)!}$ and $\exp(-C_3h^2)$ factors which lead to the fast decay as $\delta$ or $h$ grow.
\end{remark}

\begin{prop} \label{prop:mbdc2}
Subsums of \eqref{eq:wbdfdt} admit an improved upper bound:
\begin{itemize}
    \item[(i)] There is an extra factor\footnote{Depending on $t$, the factor might have been either $N^{-1/3}$ or $N^{-2/3}$; as in the previous footnote, this depends on whether $t$ is a last jump time. $N^{-1/3}$ bound works uniformly as the larger of the two.} of $N^{-1/3}$ in the right-hand side of \eqref{eq:wbdfdt},
if one sums over the subset of $\sB^h$ containing walks with a fixed jump position: we required $t\in\Delta$ for a chosen and fixed $t\in \llbracket 1, Q_m\rrbracket$, that is not in $\bigcup_{j\in J}\llbracket Q_{a_j-1}, Q_{a_j} \rrbracket \cup \llbracket (Q_{b_j-1}+Q_{b_j})/2, Q_{b_j}\rrbracket$.
\item[(ii)] There is an extra factor of $(y-x)^{-1/2}$  in the right-hand side of \eqref{eq:wbdfdt},
if one sums over the subset of $\sB^h$ containing walks such that $\sH(t)\ge \sH(x)$ for $t\in \llbracket x, y\rrbracket$ and $\llbracket x+1, y\rrbracket\cap\Delta=\emptyset$. Here for $x$ and $y$ we are allowed to choose $x=Q_{\ell-1}$ and $y\in \llbracket Q_{\ell-1}+1, Q_{\ell-1}+(Q_\ell-Q_{\ell-1})/4\rrbracket$ for some $\ell\in \llbracket 1,m\rrbracket \setminus \{a_j\}_{j\in J}$.
\end{itemize}
\end{prop}

\begin{proof}[Proof of Proposition \ref{prop:mbdc}]
In this proof $C,c>0$ may depend on $C_1$, and we take $C_2$ (resp.~$C_3$) to be large (resp.~small) enough depending on all $C,c$.

For each $j\in J_\rmA$, we denote $\hat\vartheta_j=\min\{t\in \llbracket 0, (Q_{b_j}-Q_{b_j-1})/4\rrbracket: Q_{b_j}-t\in\Delta, r_j(Q_{b_j}-t-1)>r_j(Q_{b_j}-t)\}\}$; if the set is empty, then $\hat\vartheta_j=\varnothing$. In addition to the data in definition of $\sB^h$, we fix all $\hat\vartheta_j$ for all $j\in J_\rmA$ and all $\delta_j$ for all $j\in J$. The summation over all possible values for these two parameters is performed at the very end of the proof. We refer to all the data in the definition of $\sB^h$ as well as the last two parameters as \emph{type data}.

We consider arbitrary $\vec r$ with fixed type data and
 observe from \eqref{eq:weight} and each $|N-N_\ell|<C_1N^{2/3}$ that
\begin{equation}
\label{eq_x5}
 |w(\vec r)|<CN^{-\delta}\prod_{\ell=1}^m (\sqrt{NN_\ell})^{k_\ell}\exp(Ch).
\end{equation}
In the rest of the proof we count the total number of the walks with fixed type data. Together with \eqref{eq_x5}, this count would imply \eqref{eq:wbdfdt}.

Here is the plan of the rest of the proof: We first refine the data, by additionally fixing ``jump time data'', which contains most of the set $\Delta$, as well $\Delta_j$ for $j\in J$, but not $\Delta_j$ for $j\in U$. We then count the number of walks with the refined data fixed. The way we count is by dividing the whole interval $\llbracket 0, Q_m\rrbracket$ into segments, and constructing the walk backwards in time from $Q_m$ down to $0$, segment by segment. Namely, for a segment $\llbracket v, w\rrbracket$, we count the number of possible values of $\vec r(t)$ for $t$ in this segment, with given $\vec r(t')$ for $t'\in \llbracket w, Q_m \rrbracket$: this part is subdivided into two steps --- choosing $\sH(t)$ and choosing the heights of the jumps. Finally, we multiply the counts over each segment, and then sum over (unfixed parts of) type data and jump time data.

\medskip
\noindent\textit{Step 1: jump time data.}
We make the following definitions:
\begin{itemize}
    \item For each $j\in U\cup J_\rmB\cup J_\rmC$, we let $l_j = \max \Delta_j$, i.e.\ the last jump time.
    \item For each $j\in J_\rmVIc$, we let $l_j'=\max \Delta_j\cap \llbracket 1, Q_{a_{j}-1} \rrbracket$. (See the bottom-right panel of \Cref{fig:beginning}.)
     \item We write $\vec l=\{l_j\}_{j\in U\cup J_\rmB\cup J_\rmC} \cup \{l_j'\}_{j\in J_\rmVIc}$.
    \item For each $j\in U$, or $j\in J$ with  $\sum_{\ell=1}^{a_j-1}\delta_{j,\ell}>0$, we let $f_j = \min \Delta_j$. For $j\in J$ with $\sum_{\ell=1}^{a_j-1}\delta_{j,\ell}=0$, we let $f_j=\varnothing$.
\end{itemize}

We define \emph{jump time data} for a walk $\vec r$, to consist of the following four parts:
\begin{enumerate}
    \item The set $\Delta\setminus \vec l$, and its intersection with $\Delta_j$ for each $j\in J$.
    \item $f_j$ for each $j\in U$.
    \item The relative order on the elements in $\vec l$: this set can be ordered in two ways --- by the values of $l_j$, $l'_j$ or by the values of indices $j$ --- the permutation between these two orderings is what we mean by the relative order.
\end{enumerate}

Given the information above, we consider a collection $\mathfrak{Q}$ of disjoint intervals contained in $\llbracket 1, Q_m\rrbracket$.
Namely, $\mathfrak{Q}$ contains all maximal by inclusion connected intervals $\llbracket x, y\rrbracket$, satisfying the following: it is contained in $\llbracket Q_{\ell-1}+1, Q_\ell \rrbracket$ for some $\ell \in \llbracket 1, m\rrbracket$;
it does not contain any number in $\Delta\setminus \vec l$;
for each $j\in J$, it does not contain $Q_{a_{j-1}}+\min(\vartheta_j, \lfloor (Q_{a_j}-Q_{a_j-1})/4\rfloor)$ and does not contain $Q_{a_{j-1}}+\min(\vartheta_j, \lfloor (Q_{a_j}-Q_{a_j-1})/4\rfloor)+1$;
for each $j\in J_\rmA$, it does not contain $Q_{b_j}-\min(\hat\vartheta_j,\lfloor (Q_{b_j}-Q_{b_j-1})/4\rfloor)$ and does not contain $Q_{b_j}-\min(\hat\vartheta_j,\lfloor (Q_{b_j}-Q_{b_j-1})/4\rfloor)+1$. (See the top panel of \Cref{fig:countb}.)

The last part of jump time data is the following:
\begin{enumerate}
    \item[4.] To which interval in $\mathfrak{Q}$ does each $l_j$ and $l'_j$ belong.
\end{enumerate}

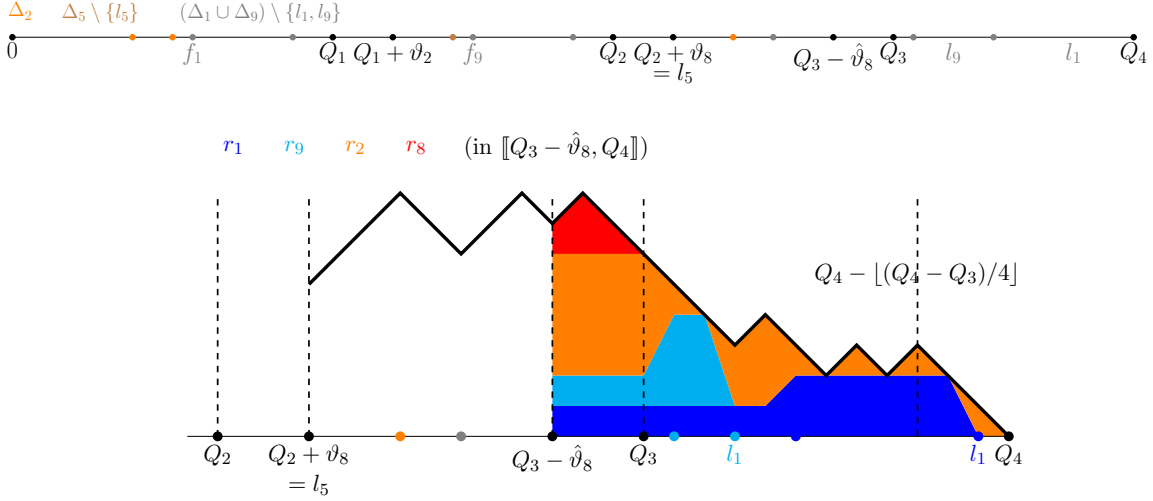
\begin{figure}[t]
    \centering
\begin{subfigure}[b]{0.98\textwidth}
    \resizebox{0.95\textwidth}{!}{
    \begin{tikzpicture}
        \draw (0,0)--(28,0);
        \draw (0,0) node[anchor=north]{\LARGE $0$};
        \draw (8,0) node[anchor=north]{\LARGE $Q_1$};
        \draw (9.5,0) node[anchor=north]{\LARGE $Q_1+\vartheta_2$};
        \draw (15,0) node[anchor=north]{\LARGE $Q_2$};
        \draw (16.5,0) node[anchor=north]{\LARGE $Q_2+\vartheta_8$};
        \draw (16.5,-0.5) node[anchor=north]{\LARGE $=l_5$};
        \draw (22,0) node[anchor=north]{\LARGE $Q_3$};
        \draw (28,0) node[anchor=north]{\LARGE $Q_4$};

        \draw (20.5,0) node[anchor=north]{\LARGE $Q_3-\hat\vartheta_8$};

        \draw [gray] (4.5,0) node[anchor=north]{\LARGE $f_1$};
        \draw [gray] (11.5,0) node[anchor=north]{\LARGE $f_9$};

        \draw [gray] (23.5,0) node[anchor=north]{\LARGE $l_9$};
        \draw [gray] (26.5,0) node[anchor=north]{\LARGE $l_1$};

        \filldraw[ultra thick] (0,0) circle (1.5pt);
        \filldraw[ultra thick] (8,0) circle (1.5pt);
        \filldraw[ultra thick] (9.5,0) circle (1.5pt);
        \filldraw[ultra thick] (15,0) circle (1.5pt);
        \filldraw[ultra thick] (16.5,0) circle (1.5pt);
        \filldraw[ultra thick] (22,0) circle (1.5pt);
        \filldraw[ultra thick] (28,0) circle (1.5pt);
        \filldraw[ultra thick] (20.5,0) circle (1.5pt);

        \filldraw[orange] [ultra thick] (3,0) circle (1.5pt);
        \filldraw[orange] [ultra thick] (4,0) circle (1.5pt);
        \filldraw[orange] [ultra thick] (18,0) circle (1.5pt);

        \filldraw[brown] [ultra thick] (11,0) circle (1.5pt);

        \filldraw[gray] [ultra thick] (4.5,0) circle (1.5pt);
        \filldraw[gray] [ultra thick] (7,0) circle (1.5pt);
        \filldraw[gray] [ultra thick] (11.5,0) circle (1.5pt);
        \filldraw[gray] [ultra thick] (14,0) circle (1.5pt);
        \filldraw[gray] [ultra thick] (19,0) circle (1.5pt);
        \filldraw[gray] [ultra thick] (22.5,0) circle (1.5pt);
        \filldraw[gray] [ultra thick] (24.5,0) circle (1.5pt);

        \draw (0.2,1) node[anchor=north, color=orange]{{\Large $\Delta_2$}};
        \draw (2.2,1) node[anchor=north, color=brown]{{\Large $\Delta_5\setminus\{l_5\}$}};
        \draw (6.2,1) node[anchor=north, color=gray]{{\Large $(\Delta_1\cup\Delta_9)\setminus\{l_1, l_9\}$}};
    \end{tikzpicture}}
\end{subfigure}
\par\bigskip
\begin{subfigure}[b]{0.98\textwidth}
    \centering
    \resizebox{0.7\textwidth}{!}{
    \begin{tikzpicture}

        \fill[blue,nearly transparent] (27.5,0)--(27,1)--(24.5,1)--(24,0.5)--(20.5,0.5) -- (20.5,0) -- cycle;

        \fill[cyan,nearly transparent] (23.5,0.5)--(23,2)--(22.5,2) -- (22,1) -- (20.5,1) -- (20.5,0.5)-- cycle;
        \fill[orange,nearly transparent] (22.5,2) -- (23,2) -- (22,3) -- (20.5,3) -- (20.5,1) -- (22,1) -- cycle;

        \fill[orange,nearly transparent] (23.5,1.5)--(24,2)--(25,1)--(25.5,1.5)-- (26,1)--(26.5,1.5)--(28,0) -- (27.5,0)--(27,1) -- (24.5,1) -- (24,0.5) -- (23.5,0.5) -- (23,2)-- cycle;

        \fill[red,nearly transparent] (21,4) -- (22,3) -- (20.5,3) -- (20.5,3.5) -- cycle;

        \draw (14.5,0)--(28,0);
        \draw[thick] [dashed] (15,0)--(15,4);
        \draw[thick] [dashed] (22,0)--(22,4);
        \draw[thick] [dashed] (16.5,0)--(16.5,4);
        \draw[thick] [dashed] (20.5,0)--(20.5,4);
        \draw[thick] [dashed] (26.5,0)--(26.5,4);
        \draw (15,0) node[anchor=north]{$Q_2$};
        \draw (16.5,0) node[anchor=north]{$Q_2+\vartheta_8$};
        \draw (16.5,-0.5) node[anchor=north]{$=l_5$};
        \draw (22,0) node[anchor=north]{$Q_3$};
        \draw (28,0) node[anchor=north]{$Q_4$};

        \draw (20.5,0) node[anchor=north]{$Q_3-\hat\vartheta_8$};

        \draw [blue] (27.5,0) node[anchor=north]{$l_1$};
        \draw [cyan] (23.5,0) node[anchor=north]{$l_1$};

        \filldraw[ultra thick] (15,0) circle (1.5pt);
        \filldraw[ultra thick] (16.5,0) circle (1.5pt);
        \filldraw[ultra thick] (22,0) circle (1.5pt);
        \filldraw[ultra thick] (28,0) circle (1.5pt);
        \filldraw[ultra thick] (20.5,0) circle (1.5pt);

        \filldraw[orange] [ultra thick] (18,0) circle (1.5pt);

        \filldraw[gray] [ultra thick] (19,0) circle (1.5pt);

        \filldraw[blue] [ultra thick] (24.5,0) circle (1.5pt);
        \filldraw[blue] [ultra thick] (27.5,0) circle (1.5pt);

        \filldraw[cyan] [ultra thick] (22.5,0) circle (1.5pt);
        \filldraw[cyan] [ultra thick] (23.5,0) circle (1.5pt);

        \draw[ultra thick] (16.5,2.5)-- (18,4)-- (19,3)-- (20,4)-- (20.5,3.5)-- (21,4)-- (22,3)--(23.5,1.5)--(24,2)--(25,1)--(25.5,1.5)-- (26,1)--(26.5,1.5)--(28,0);

        \draw (26.5,3) node[anchor=north]{ $Q_4-\lfloor (Q_4-Q_3)/4\rfloor$};

        \draw (15.2,5) node[anchor=north, color=blue]{{ $r_1$}};
        \draw (16.2,5) node[anchor=north, color=cyan]{{ $r_9$}};
        \draw (17.2,5) node[anchor=north, color=orange]{{ $r_2$}};
        \draw (18.2,5) node[anchor=north, color=red]{{ $r_8$}};
        \draw (20.5,5.15) node[anchor=north]{{ (in $\llbracket Q_3-\hat\vartheta_8, Q_4\rrbracket$)}};
    \end{tikzpicture}}
\end{subfigure}
    \caption{An illustration of the walk counting procedure, in the proof of \Cref{prop:mbdc}.
    Here we take $m=4$, $(i_1, i_2, i_3, i_4)=(5,2,8,2)$ (thereby $J=\{2,5,8\}$, and $a_5=b_5=1$, $a_2=2$, $b_2=4$, $a_8=b_8=3$), $U=\{1,9\}$, $J_\rmIV=\{5\}$, $J_\rmI=\{2\}$, $J_\rmIII=\{8\}$, $J_\rmB=\{5\}$, $J_\rmA=\{2,8\}$.\\
    \textit{Top panel:} jump time data. Note that the sets $\Delta_1\setminus\{l_1\}$ and $\Delta_9\setminus\{l_9\}$ are unknown, although $f_1$, $f_9$, and the union $(\Delta_1\cup\Delta_9)\setminus\{l_1, l_9\}$ is given. The numbers $l_1$ and $l_9$ are given up to an interval; on the other hand, $l_5$ is determined, as it coincides with $Q_2+\vartheta_8=Q_{a_8-1}+\vartheta_8$. \\
    The collection $\mathfrak{Q}$ of intervals can be obtained by breaking $\llbracket 1, Q_m\rrbracket$ at all the plotted points, as well as $\lfloor Q_1/4\rfloor$, $Q_1-\lfloor Q_1/4\rfloor$, and $Q_4-\lfloor (Q_4-Q_3)/4\rfloor$. (Each breaking point either belongs to the interval to its left, or neither.) \\
       \textit{Bottom panel:} the procedure of counting the number of $\vec r$, with given jump time data (from $Q_2$ to $Q_4$). The interval $\llbracket 0, Q_4\rrbracket$ is split into segments, which are processed in backwards order. In each segment, we first determine the sum $\sH $, and next the jump heights (thus determine $\vec r$).}
    \label{fig:countb}
\end{figure}

For any given type data, the number of jump time data is bounded by
\begin{multline} \label{eq_x6}
\frac{(CN^{2/3})^{\delta-|U|-|J_\rmB|-|J_\rmC|-|J_\rmVIc|}}{(\delta-|U|-|J_\rmB|-|J_\rmC|-|J_\rmVIc|)!}
\cdot \frac{(\delta-|U|-|J_\rmB|-|J_\rmC|-|J_\rmVIc|)!\delta^{2|J|}}{(\delta-2|U|-\sum_{j\in J}\delta_j)!\prod_{j\in J}\delta_j! }\cdot (\delta+|U|+5m)^{|U|+2m}
\\
\times \left(\frac{\delta}{N^{2/3}}\right)^{|\{j\in J_\rmA: \hat\vartheta_j\ne \varnothing\}| + |J_\rmIV| + |J_\rmVIb| + \max\{|J_\rmVIa|,|J_\rmVIc|\}-|J_\rmVIc|}.
\end{multline}
This bound is obtained as follows.
First, the number of possibilities of the set $\Delta\setminus \vec l$ is ${Q_m \choose \delta-|U|-|J_\rmB|-|J_\rmC|-|J_\rmVIc|}$, which is bounded by the first factor. Then given the set $\Delta\setminus \vec l$, it would be divided into several sets: $\Delta_j\setminus \vec l$ for each $j\in J$, and $\{f_j\}$ for each $j\in U$, and the remaining numbers (which are $\bigcup_{j\in U}\Delta_j\setminus \{f_j\}$).
The number of such partition is $\frac{(\delta-|U|-|J_\rmB|-|J_\rmC|-|J_\rmVIc|)!}{(1!)^{|U|}|\bigcup_{j\in U}\Delta_j\setminus \{f_j, l_j\}|!\prod_{j\in J} |\Delta_j\setminus \vec l|! }$, which is bounded by the second factor.
For the third and last parts of jump time data, one considers a collection of $|\mathfrak{Q}|$ intervals (with $|\mathfrak{Q}|\le \delta+3m$), and take a multiset of cardinality $|\vec l|\le |U|+2m$ from them, and determine the relative order on the elements in $\vec l$.
This number of these information is $(|\mathfrak{Q}|)_{|\vec l|}$. This is bounded by the third factor.

The last factor comes from that the set $\Delta\setminus \vec l $ is mandated to contain certain numbers:
\begin{itemize}
    \item $Q_{b_j}-\hat\vartheta_j$ for each $j\in J_\rmA$ with $\hat\vartheta_j\ne\varnothing$; $Q_{a_j-1}+\dot\vartheta_j$ for each $j\in J_\rmIV\cup J_\rmVIb$,
    \item $Q_{a_j-1}+\dot\vartheta_j$ for each $j \in J_{\rmVIa}$ with the corresponding $j'$ (i.e., the unique $j'$ with $r_{j'}(Q_{a_j-1})> r_{j'}(Q_{a_j-1}+\vartheta_j)$) not in $J_\rmVIc$ (otherwise it might be that $Q_{a_j-1}+\dot\vartheta_j=l_{j'}'$).
\end{itemize}
By rearranging  the factors and using $\delta^{2|J|}(\delta+|U|+5m)^{|U|+2m}<C^\delta (|U|+1)^{|U|}$, we upper bound \eqref{eq_x6} by
\begin{equation}  \label{eq:tbd1}
 N^{\frac{2}{3}(\delta - |U| - |\{j\in J_\rmA:\hat\vartheta_j\ne\varnothing\}| - |J_\rmB| - |J_\rmC| - |J_\rmIV| - |J_\rmVIb| - \max\{|J_\rmVIa|, |J_\rmVIc|\} )} \frac{C^\delta (|U|+1)^{|U|}}{(\delta-2|U|-\sum_{j\in J}\delta_j)!\prod_{j\in J}\delta_j! }.
\end{equation}

\medskip
\noindent\textit{Step 2: construction of segments.}
We next bound the number of walks, with given type data and jump time data.
Our strategy is to divide $\llbracket 0, Q_m\rrbracket$ into several segments, and count $\vec r$ on them sequentially.
More precisely, we let $D$ be the set consisting of the following numbers:
\begin{itemize}
    \item[] $\{Q_\ell\}_{\ell=0}^m$, $\{Q_{a_j-1} + \min(\vartheta_j, \lfloor (Q_{a_j}-Q_{a_j-1})/4\rfloor)
 \}_{j\in J}$, $\{Q_{b_j}-\min(\hat\vartheta_j,\lfloor (Q_{b_j}-Q_{b_j-1})/4\rfloor ) \}_{j\in J_\rmA}$.
\end{itemize}
The set $D$ splits $\llbracket 0, Q_m\rrbracket$ into $|D|-1\le 3m$ discrete intervals, each of the form of $\llbracket v, w\rrbracket$ with $\llbracket v, w\rrbracket \cap D = \{v, w\}$.
We inductively count the number of choices for the walk on each $\llbracket v, w\rrbracket$, backwards in time, i.e., starting from the rightmost interval.
For each segment $\llbracket v, w\rrbracket$ we split the count into two steps.
We first estimate the number of possible $\sH(t)$, $t\in \llbracket v, w\rrbracket$, given $\vec r(w)$. Then we argue conditionally on $\sH(t)$, $t\in \llbracket v, w\rrbracket$, thinking about the latter as being uniformly random and estimate the number of possible $\vec r(t)$, $t\in \llbracket v, w\rrbracket$, with given $\vec r(w)$ and $\sH $. Since jump time data is given and it includes jump times, the main task in the second step is to estimate the number of possible `jump heights'.

\medskip
\noindent\textit{Step 3: Choosing $\sH(t)$ in a fixed segment.} $\sH(t)$,  $t\in\llbracket v, w\rrbracket$ is a non-negative function with increments $\pm 1$, with given $\sH(w)$, and satisfies one of the following:
\begin{enumerate}
    \item $\sH(t)\ge \sH(w)$ for all $t\in\llbracket v, w\rrbracket$, if $v=Q_{b_j}-\hat\vartheta_j \wedge \lfloor (Q_{b_j}-Q_{b_j-1})/4\rfloor$ and $w=Q_{b_j}$ for some $j\in J_\rmA$.
    \item $\sH(t)\ge \sH(v)$ for all $t\in\llbracket v, w\rrbracket$, if $v=Q_{a_j-1}$ and $w=Q_{a_j-1} + \lfloor (Q_{a_j}-Q_{a_j-1})/4\rfloor$ for some $j\in J$ with $\vartheta_j=\varnothing$.
    \item $\sH(t)\ge \sH(v)$ for all $t\in \llbracket v, w\rrbracket$ and $\sH(v)=\sum_{j'\in J\cup U: j'\neq j, f_{j'} \not\in \llbracket Q_{a_j-1}+1, w\rrbracket} r_{j'}(w)$, if $v=Q_{a_j-1}$ and $w=Q_{a_j-1} + \lfloor (Q_{a_j}-Q_{a_j-1})/4\rfloor$ for some $j\in J$ with $\vartheta_j=\varnothing$ and additionally $f_j=\varnothing$ or $j\in J_\rmVIc$, and $\llbracket Q_{a_j-1}+1, w\rrbracket \cap \Delta \subset F$, where
    \begin{equation} \label{eq:defF}
F = \left(\bigcup_{j'\in J: f_{j'}\in \llbracket Q_{a_j-1}+1, w\rrbracket}\Delta_{j'}\right)\cup \left(\bigcup_{j'\in U: f_{j'}\in \llbracket Q_{a_j-1}+1, w\rrbracket} \{f_{j'},l_{j'}\}\right).
\end{equation}
This is because, for each $j'\in J\cup U$ with $j'\neq j$ and $f_{j'} \not\in \llbracket Q_{a_j-1}+1, w\rrbracket$, $r_{j'}$ must be constant on $\llbracket v, w\rrbracket$; while $r_{j'}(v)=0$ for all other $j'$.
    \item $\sH(t)\ge \sH(v)=\sH(w)+1$ for all  $t\in\llbracket v, w-1\rrbracket$, if $v=Q_{a_j-1}$ and $w=Q_{a_j-1} + \vartheta_j$ for some $j\in J$ with $\vartheta_j\ne\varnothing$.
    \item $\sH(v)=\sum_{j'\in J\cup U: j'\neq j, f_{j'} \not\in \llbracket Q_{a_j-1}+1, w\rrbracket} r_{j'}(w)-1$, if $v=Q_{a_j-1}+\vartheta_j$ for some $j\in J$ with $\vartheta_j\ne\varnothing$, and $f_j=\varnothing$ or $j\in J_\rmVIc$, and $\llbracket Q_{a_j-1}+1, w\rrbracket \cap \Delta \subset F$, where $F$ is the same as above.\\
    This is because in this case, for each $j'\in J\cup U$ with $j'\neq j$ and $f_{j'} \not\in \llbracket Q_{a_j-1}+1, w\rrbracket$, $r_{j'}$ must be constant on $\llbracket Q_{a_j-1}, w\rrbracket$; while $r_{j'}(Q_{a_j-1})=0$ for all other $j'$.
    Therefore $\sH(Q_{a_j-1})$ is determined by $\vec r(w)$. On the other hand there is $\sH(Q_{a_j-1})=\sH(v)+1$.
    \item No further constraints: the remaining cases.
\end{enumerate}
The number of possible values for $\sH(t)$, $t\in\llbracket v, w-1\rrbracket$,  is then (up to a constant factor) bounded by
\begin{equation} \label{eq:EHbd}
\begin{split}
& 2^{w-v}(\hat\vartheta_j^{-1/2}+ N^{-1/3}), \quad 2^{w-v} N^{-1/3}, \quad 2^{w-v} N^{-2/3},\\
& 2^{w-v} \vartheta_j^{-3/2}, \quad 2^{w-v} N^{-1/3},\quad  2^{w-v},
\end{split}
\end{equation}
in these cases, respectively.
These bounds are by \Cref{lem:countwk} (stated and proved in \Cref{ssec:expvar}), which counts the number of Bernoulli walks under various constraints. More precisely, we use (iii), (iii), (iv), (v), (ii), (i) there, in these 6 cases respectively, with $M=0$.
Note that by directly applying  \Cref{lem:countwk}, the bounds would be $2^{w-v}$ multiplying a power of $(w-v)^{-1}$.
It is convenient to keep $w-v$ in the power of $2$, but replace it with an order of magnitude in the polynomial factors; this is how we get \eqref{eq:EHbd}.

\medskip
\noindent\textit{Step 4: choosing jump heights in a fixed segment.}
We count the number of possible $\vec r(t)$, $t\in \llbracket v, w\rrbracket$, given $\sH(t)$, $t\in \llbracket v, w\rrbracket$.
Consider
\[
K = \llbracket v+1, w\rrbracket \cap \left(\Delta\setminus \vec l\right) .
\]
The main task would be to determine $\vec r(s-1)$ and $\vec r(s)$ for each $s\in K$, and $\vec r(v)$;
and from there, $\vec r$  (as well as the set $\llbracket v+1, w\rrbracket \cap  \vec l$) would be uniquely determined, as we will see shortly.

\begin{figure}[!ht]
    \centering
    \resizebox{0.55\textwidth}{!}{
    \begin{tikzpicture}

    \fill[cyan,nearly transparent] (18,-5)--(18,-6)--(1,-6) -- (1,-5)-- cycle;
    \fill[blue,nearly transparent] (18,-5)--(18,-3)--(1,-3) -- (1,-5)-- cycle;
    \fill[orange,nearly transparent] (16,0)--(17,-3)--(1,-3) -- (1,0)-- cycle;
    \fill[red,nearly transparent] (10,0)--(9,3)--(1,3) -- (1,0)-- cycle;
    \fill[violet,nearly transparent] (7,3)--(6,4)--(1,4) -- (1,3)-- cycle;

     \draw[thick] (1,-6)--(18,-6);
          \draw[thick] [dashed] (1,-6)--(1,5);
          \draw[thick] [dashed] (18,-6)--(18,0);
        \draw[ultra thick] (1,5) -- (2,6)--(3,5)--(4,4)--(5,5)--(7,3)--(8,4)--(11,1)--(12,2)--(14,0)--(15,1)--(17,-1)--(18,0);
        \draw (1,-6) node[anchor=north]{\Large{$s$}};
        \draw (18,-6) node[anchor=north]{\Large{$s'$}};

    \fill[orange] (1.1,0)--(1.1,-3)--(.9,-3) -- (.9,0)-- cycle;
    \fill[red] (1.1,0)--(1.1,3)--(.9,3) -- (.9,0)-- cycle;
    \fill[violet] (1.1,4)--(1.1,3)--(.9,3) -- (.9,4)-- cycle;
        \draw (1,-1.5) node[anchor=east, color=orange]{{\Large $r_9(s)$}};
        \draw (1,1.5) node[anchor=east, color=red]{{\Large $r_3(s)$}};
        \draw (1,3.5) node[anchor=east, color=violet]{{\Large $r_8(s)$}};
        \draw (1,-4) node[anchor=east, color=blue]{{\Large $r_1(s)$}};
        \draw (1,-5.5) node[anchor=east, color=cyan]{{\Large $r_4(s)$}};
        \draw (18,-4) node[anchor=west, color=blue]{{\Large $r_1(s')$}};
        \draw (18,-5.5) node[anchor=west, color=cyan]{{\Large $r_4(s')$}};
    \end{tikzpicture}
    }
    \caption{An illustration of determining $\vec r$ in an interval $\llbracket s, s'\rrbracket$ satisfying $\llbracket s+1, s'\rrbracket\cap\Delta\subset \vec l$, with given $\vec r(s')$ and $\sH $ on this interval.
    Here it is known that $\llbracket s+1, s'\rrbracket\cap\Delta=\{l_3,l_8,l_9\}$, with $l_8<l_3<l_9$. Necessarily $r_1(s)=r_1(s')$ and $r_4(s)=r_4(s')$, while
    $r_3(s)$, $r_8(s)$, and $r_9(s)$ can be any positive integers satisfying $r_9(s)\ge 3$, $r_3(s)+r_8(s)+r_9(s)\le 8$.\\ Then given these numbers $\vec r(s)$, $l_8$ is the smallest $t$ with $\sH(t)<r_1(s)+r_3(s)+r_4(s)+r_8(s)+r_9(s)$, $l_3$ is the smallest $t$ with $\sH(t)<r_1(s)+r_3(s)+r_4(s)+r_9(s)$, and $l_9$ is the smallest $t$ with $\sH(t)<r_1(s)+r_4(s)+r_9(s)$.
    }
    \label{fig:detl}
\end{figure}
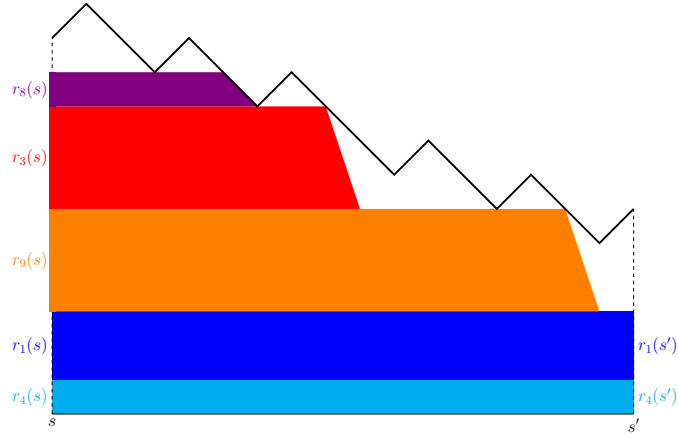

Recall that, from jump time data, it may already be known that $w\in \vec l$, when $w=Q_{a_{j}-1}+\vartheta_{j}$ for some $j\in J$.
We let $w'=w-1$ in this case  \footnote{Note that in this case, $\vec r(w')=\vec r(w-1)$ is determined by $\sH $ and $\vec r(w)$.},
and $w'=w$ otherwise.

Take any $s\in K\cup\{v\}$, let $s'$ denote the smallest number in $\llbracket s, w'\rrbracket$ with $s'+1\in K$, or $s'=w'$ if no such number exists (i.e., $s=\max K\cup\{v\}$).
We now determine $\vec r$ on $\llbracket s, s'\rrbracket$ given $\vec r(s')$.

Denote $u[s]=| \Delta \cap \llbracket s+1, s'\rrbracket |$.
Note that $\Delta\cap\llbracket s+1, s'\rrbracket\subset\vec l$. It remains to determine $r_j(s)$ for $j$ with $l_j$ or $l_j'$ in $\Delta\cap\llbracket s+1, s'\rrbracket$\footnote{Note that the collection of such $j$, and the ordering of these $l_j$ and $l_j'$ are already determined by jump time data}; from there $\vec r$ on $\llbracket s, s'\rrbracket$ would be determined (see \Cref{fig:detl}).
A constraint is that their sum is at most $\max_{\llbracket s, s'\rrbracket} \sH - \sum_{j\neq i_\ell}r_j(w')$ while (a particular) one of them is at least $\min_{\llbracket s, s'\rrbracket} \sH - \sum_{j\neq i_\ell}r_j(w')$, where $\ell$ is the number with $\llbracket v, w\rrbracket\subset\llbracket Q_{\ell-1}, Q_\ell\rrbracket$.
Besides, since $\sH $ on $\llbracket v, w\rrbracket$ is given, sometimes this already determines the sum of some entries of $\vec r(s)$.
More precisely, there are three cases:
\begin{itemize}
    \item[(a)] $v\in \{Q_{a_j-1}, Q_{a_j-1}+\vartheta_j\}$ for some $j\in J$ with $f_j=\varnothing$ or $j\in J_\rmVIc$, and
$\Delta\cap\llbracket Q_{a_j-1}+1, s\rrbracket\subset F$ while $\Delta \cap \llbracket s+1, s'\rrbracket\not\subset F$ (for $F$ defined in \eqref{eq:defF}).\\
In this case necessarily $u[s]\ge 1$, and $\sum_{j'\in J\cup U: f_{j'}\not\in \llbracket Q_{a_j-1}+1, s \rrbracket }r_{j'}(s)=\sH(Q_{a_j-1})$. Thereby, the number of possible choices of $\vec r(s)$ would be bounded by $\frac{(\max_{\llbracket s, s'\rrbracket} \sH - \min_{\llbracket s, s'\rrbracket} \sH +1)^{u[s]-1}}{(u[s]-1)!}$.
\item[(b)] In the remaining case, the number of possible choices of $\vec r(s)$ would be bounded by $\frac{(\max_{\llbracket s, s'\rrbracket} \sH - \min_{\llbracket s, s'\rrbracket} \sH +1)^{u[s]}}{(u[s])!}$.
\end{itemize}

We now take any $s\in K$, and determine $\vec r(s-1)$ given $\vec r(s)$. There are also three cases:
\begin{itemize}
    \item[(d)] $s=f_j$ for some $j\in J\cup U$. In this case $\vec r(s-1)$ is determined by $\vec r(s)$.
    \item[(e)] $v\in \{Q_{a_j-1}, Q_{a_j-1}+\vartheta_j\}$ for some $j\in J$ with $f_j=\varnothing$ or $j\in J_\rmVIc$, and
$\llbracket Q_{a_j-1}+1, s-1\rrbracket\cap \Delta\subset F$, while $s\not\in F$. In this case, necessarily $\sum_{j'\not\in J\cup U: f_{j'}\in \llbracket Q_{a_j-1}+1, s-1 \rrbracket }r_{j'}(s-1)=\sH(Q_{a_j-1})$, and to ensure this there are at most $|J|+|U|$ possible choices of $\vec r(s-1)$.
    \item[(f)] In the remaining case, the number of possible choices of $\vec r(s-1)$ would be at most $4hN^{1/3}$.
    This is because for every $j\in J\cup U$, assuming that $s\in \Delta_j$, there are at most $2r_j(s)$ possible choices of $r_j(s-1)$; and $\sum_{j\in J\cup U} 2r_j(s)<4hN^{1/3}$.
\end{itemize}
We note that the bound of $ChN^{1/3}$ in the last case (f) is the main reason for us to exclude the information of the set $\Delta_j$ (for $j\in U$) from jump time data, and to construct $\sH $ and $\vec r$ backwards.
More precisely, specifying each $\Delta_j$ ahead of time would lead to an extra factor of roughly $|U|^{\delta}$; while if we consider the number of possible $\vec r(s)$ with given $\vec r(s-1)$, it would only be bounded by $2|U|h N^{1/3}$, also leading to an extra factor of $|U|^{\delta}$ when taking the product over all $s$.

As another remark, in case (a) and case (e) we see that there are constraints on the sum of $r_{j'}(s)$ or $r_{j'}(s-1)$ for some $j'$.
These complement case 3 and case 5 in counting the number of $\sH $, in Step 3.

Putting the above cases together, we have that the number of $\vec r$ on $\llbracket v, w\rrbracket$, given $\sH $ on the same interval, is bounded by
\begin{equation}   \label{eq:uminus}
\prod_{s\in K\cup \{v\}} \frac{(\max_{\llbracket s, s'\rrbracket} \sH - \min_{\llbracket s, s'\rrbracket} \sH +1)^{u_-[s]}}{(u_-[s])!} \cdot (|J|+|U|) (4hN^{1/3})^{|K\setminus \{f_j\}_{j\in J\cup U}|-u_-},
\end{equation}
where each $u_-[s]=u[s]-1$ or $u[s]$, depending on whether $s$ is of case (a) or (b); and $u_-=0$ or $1$, depending on the existence of any $s$ of case (e).

We apply \Cref{lem:eva} (stated and proved in \Cref{ssec:expvar}) in order to upper bound \eqref{eq:uminus}. Note that we have counted the number of possible choices for $\sH$ on the previous step, and now we treat \eqref{eq:uminus} as a random variable with respect to the uniform measure on all these choices. This random variable and this measure are the content of  \Cref{lem:eva}. Note that all six cases from \eqref{eq:EHbd} of Step 3 are covered by the lemma applied to $\sH$ or its shifts by a constant.
More precisely, we always take $X=w-v$ or $w-v-1$.
Then in case 1 and 2, we take $G=0$, and sum over $H$;
in case 3 we take $G=0$ and a fixed $H$;
in case 4 we take $H=G=0$;
in case 5 we take fixed $H$ and $G$;
in case 6 we fix $G$ and sum over all $H$.
The number $L$ is chosen to fit the constraint $\max \sH< 2hN^{1/3}$.

We conclude that, given type data and jump time data, the number of $\vec r$ on $\llbracket v, w\rrbracket$ (with given $\vec r(w)$) is at most the number of choices for $\sH(t)$ (upper bounded on the previous step), multiplied by (with the convention that $0^0=1$)
\begin{equation}  \label{eq:buv}
C^{u(v)} \log(u(v)+1)^{u(v)} (u(v)^{-u(v)/2}  + u(v)^{-u(v)}h^{u(v)} )  N^{u(v)/3} \cdot (|J|+|U|) (4hN^{1/3})^{u'(v)} ,
\end{equation}
where
\[u(v)=\sum_{s\in K\cup \{v\}}u[s]= |\vec l \cap \llbracket v+1, w' \rrbracket |,\] and \[u'(v)= |K\setminus \{f_j\}_{j\in J\cup U}|= |(\Delta \setminus \{f_j\}_{j\in J\cup U}) \cap  \llbracket v+1, w' \rrbracket| - u(v) .\]
And there is an extra factor of $N^{-1/3}$, under the following condition:
\begin{itemize}
    \item[(ae)] $v\in \{Q_{a_j-1}, Q_{a_j-1}+\vartheta_j\}$ for some $j\in J$ with $f_j=\varnothing$ or $j\in J_\rmVIc$, and  $\llbracket Q_{a_j-1}+1, w\rrbracket\cap \Delta\setminus F$ is not empty, with its smallest number in $\llbracket v+1, w'\rrbracket$. This holds if one of (a) and (e) happens for one $s$.
\end{itemize}
\smallskip
\noindent\textit{Step 4': choosing jump heights when the increment is large.} In the next step we also need a version of the above count with an additional requirement that $\max_{\llbracket v, w\rrbracket}\sH  - \min_{\llbracket v, w\rrbracket}\sH  > (h-1)N^{1/3}/(3m)$. We claim that in this situation the same bounds \eqref{eq:EHbd} and \eqref{eq:buv} hold with an extra overall factor of $C\exp(-ch^2)$.
Indeed, to see that we consider two situations. If the boundary values of $\sH$ satisfy $|\sH(w) - \sH(v)|  > (h-1)N^{1/3}/(6m)$ in cases 1--3, 5, 6 of Step 3, then there would be an extra factor of $C\exp(-ch^2)$ on Step 3, by taking $M=(h-1)N^{1/3}/(6m)$ in \Cref{lem:countwk}. 
(In particular, in case 3 where \Cref{lem:countwk}(iv) is applied, we have $H=|\sH(w) - \sH(v)|>M$; and in case 5 where \Cref{lem:countwk}(ii) is applied, we have $|G-H|=|\sH(w) - \sH(v)|>M$.)
Otherwise, if  $|\sH(w) - \sH(v)|  \le (h-1)N^{1/3}/(6m)$, then we apply \eqref{eq:eva22} in \Cref{lem:eva} (with $K=(h-1)N^{1/3}/(12m)$) and get an extra factor of $C\exp(-ch^2)$ in the expectation in Step 4.

\medskip
\noindent\textit{Step 5: product of contributions of segments.}
We can now bound the number of possible choices for $\vec r(t)$, $t\in \llbracket 0,Q_m\rrbracket$ (with given type data and jump time data), by multiplying \eqref{eq:EHbd}, and \eqref{eq:buv}, over all $\llbracket v, w\rrbracket$.

We first take the product of \eqref{eq:buv}. Note that the sum of all $u(v)$ is at least $|U|$ and at most $|U|+m$. The former implies $u(v)\ge \frac{1}{3m} |U|$ for at least one segment $\llbracket v,w\rrbracket$.
 For this $\llbracket v,w\rrbracket$, then we have
\begin{equation}\label{eq_lastjumpchoice1}
u(v)^{-u(v)/2} + u(v)^{-u(v)}h^{u(v)} \le (3m/|U|)^{|U|/(6m)} + \Big[(3mh/|U|)^{|U|/(3m)}\ \text{or}\ (3mh/|U|)^{|U|+m\mathbf{1}_{h\ge |U|}}\Big];
\end{equation}
where for the first term we use the fact that $x\mapsto x^{-\frac{x}{2}}$ is decreasing and upper bound it at $x=\frac{1}{3m}|U|$, and the two cases for the second term correspond to the regime $h<\frac{1}{3m}|U|$ (so $\frac{h}{u(v)}\le \frac{3mh}{|U|}<1$ and we minimize the exponent) or $h\ge \frac{1}{3m}|U|$ (so $\max(\frac{h}{u(v)},1)\le \frac{3mh}{|U|}$ and we maximize the exponent at $|U|+m$).

For other $\llbracket v,w\rrbracket$'s, we use the fact that the function $x\mapsto x^{-x}h^x$, $x>0$, achieves its maximum at $x=\frac{h}{e}$: it is increasing when $x<h/e$, and decreasing when $x>h/e$. So instead we have
\begin{equation}\label{eq_lastjumpchoice2}
u(v)^{-u(v)/2} + u(v)^{-u(v)}h^{u(v)} \le 1 + e^{h/e},
\end{equation}
and the product of all such $\llbracket v, w\rrbracket$ is at most $(1+e^{h/e})^{3m}<C^h$.
Combining (\ref{eq_lastjumpchoice1}) and (\ref{eq_lastjumpchoice2}), we upper bound $u(v)$ by $|U|+C$ and slightly adjust the expression to avoid 0 in $\log(\cdot)$ and denominator, and get
\[
\prod_{\llbracket v,w\rrbracket} \log(u(v)+1)^{u(v)} (u(v)^{-u(v)/2} + u(v)^{-u(v)}h^{u(v)}) < C^{|U|+h}\log(|U|+2)^{|U|} (h/(|U|+1))^{c|U|+m\mathbf{1}_{h\ge |U|}},
\]
where $c$ is a positive constant between $\frac{1}{6m}$ and $1$.
Therefore, the product of \eqref{eq:buv} over $\llbracket v, w\rrbracket$ is bounded from above by
\begin{equation}  \label{eq:29}
    C^{\delta+h}\log(|U|+2)^{|U|} (h/(|U|+1))^{c|U|+m\mathbf{1}_{h\ge |U|}} h^{\delta-2|U|} N^{(\delta-|U|-|\{j\in J:f_j\ne \varnothing\}| -|\{j\in J: Q_{a_j-1}+\vartheta_j\in \vec l\}| )/3},
\end{equation}
where we use that $\sum_{\llbracket v,w\rrbracket} u'(v)\le \delta-2|U|$, and
\begin{multline*}
\sum_{\llbracket v,w\rrbracket} u(v)+u'(v) = \sum_{\llbracket v,w\rrbracket} |(\Delta \setminus \{f_j\}_{j\in J\cup U}) \cap  \llbracket v+1, w' \rrbracket| \\ = \delta-|U|-|\{j\in J:f_j\ne\varnothing\}| - |\{j\in J: Q_{a_j-1}+\vartheta_j\in \vec l\}|.
\end{multline*}
For $\llbracket v, w\rrbracket$ under (ae), there is an extra factor of $N^{-1/3}$.
The number of $\llbracket v, w\rrbracket$ under (ae) is the same as the number of $j\in J$, satisfying the following two conditions:
\begin{itemize}
    \item $f_j=\varnothing$ or $j\in J_\rmVIc$, and
    \item
    Let $w_j= Q_{a_j-1}+\lfloor (Q_{a_j}-Q_{a_j-1})/4\rfloor$ if $\vartheta_j=\varnothing$, and let $w_j$ be the smallest number in the set $D$ (from Step 2) larger than $Q_{a_j-1}+\vartheta_j$ if  $\vartheta_j\neq\varnothing$. Let $F_j$ be the set $F$ of \eqref{eq:defF} for $\llbracket Q_{a_j-1}+1, w_j\rrbracket$.\\
    We ask for $\llbracket Q_{a_j-1}+1, w_j\rrbracket\cap (\Delta\setminus F_j)$ to be non-empty;
    besides, when $Q_{a_j-1}+\vartheta_j\in \vec l$,
    the smallest number of the set $\llbracket Q_{a_j-1}+1, w_j\rrbracket\cap \Delta\setminus F_j$ should not be $Q_{a_j-1}+\vartheta_j$. \\ (This follows from (ae), and that $w'=Q_{a_j-1}+\vartheta_j-1$ for the interval $\llbracket Q_{a_j-1}+1, Q_{a_j-1}+\vartheta_j\rrbracket$.)
\end{itemize}
Then we get a factor at most
\begin{equation}  \label{eq:29a}
N^{-|\{j\in J: f_j=\varnothing \text{ or } j\in J_\rmVIc,\; \llbracket Q_{a_j-1}+1, w_j\rrbracket\cap \Delta\not\subset F_j, \;  Q_{a_j-1}+\vartheta_j\not\in \vec l \}|/3}.
\end{equation}
Multiplying \eqref{eq:29} and \eqref{eq:29a}, we get
\begin{multline}  \label{eq:29b}
        C^{\delta+h}\log(|U|+2)^{|U|} (h/(|U|+1))^{c|U|+m\mathbf{1}_{h\ge |U|}} h^{\delta-2|U|}\\ \times N^{(\delta-|U|-|\{j\in J:f_j\ne\varnothing\}|-|\{j\in J: f_j=\varnothing \text{ or } j\in J_\rmVIc,\; \llbracket Q_{a_j-1}+1, w_j\rrbracket\cap \Delta\not\subset F_j\}| )/3}.
\end{multline}

We next take the product of \eqref{eq:EHbd} over $\llbracket v,w\rrbracket$, and get
\begin{multline}  \label{eq:30}
C2^{Q_m}
N^{-|\{j\in J: f_j=\varnothing \text{ or } j\in J_\rmVIc, \llbracket Q_{a_j-1}+1, w_j\rrbracket\cap \Delta\subset F_j\}|/3} \\ \times
\exp(-ch^2) \prod_{j\in J:\vartheta_j \ne \varnothing} \vartheta_j^{-3/2} \prod_{j\in J: \vartheta_j=\varnothing} N^{-1/3}
\prod_{j\in J_\rmA, \hat\vartheta_j\ne\varnothing} \hat\vartheta_j^{-1/2}\prod_{j\in J_\rmA, \hat\vartheta_j=\varnothing}N^{-1/3}.
\end{multline}
Here the factor of $C\exp(-ch^2)$ is from Step 4': because $\max_t \sH(t)>(h-1)N^{1/3}$, we should necessary have  $\max_{\llbracket v,w\rrbracket} \sH  - \min_{\llbracket v,w\rrbracket} \sH  > (h-1)N^{1/3}/(3m)$ for at least one $\llbracket v,w\rrbracket$.

In summary, by multiplying \eqref{eq:29b} and \eqref{eq:30}, the number of possible $\vec r$ (with given type data and jump time data) is bounded by
\begin{multline}  \label{eq:tbd2}
2^{Q_m} C^{\delta+1} N^{-\frac{1}{3}|\{j\in J: f_j=\varnothing \text{ or } j\in J_\rmVIc\}|}N^{\frac{1}{3} (\delta - |U| - |\{j\in J: f_j\ne\varnothing\}| ) } h^{\delta - 2|U|} \log(|U|+2)^{|U|} (h/(|U|+1))^{c|U|+m\mathbf{1}_{h\ge |U|}}
 \\ \times \exp(-ch^2) \prod_{j\in J:\vartheta_j \ne\varnothing} \vartheta_j^{-3/2} \prod_{j\in J: \vartheta_j=\varnothing} N^{-1/3} \prod_{j\in J_\rmA, \hat\vartheta_j\ne \varnothing} \hat\vartheta_j^{-1/2}\prod_{j\in J_\rmA, \hat\vartheta_j=\varnothing}N^{-1/3}.
\end{multline}

\medskip

\noindent\textit{Step 6: Final summation.} By multiplying \eqref{eq:tbd1} and \eqref{eq:tbd2}, and summing over $\delta_j$ for $j\in J$ and $\hat\vartheta_j$ for $j\in J_\rmA$, we get
\begin{multline*}
2^{Q_m} N^{\delta - |U| - \frac{1}{3}(|J| + |\{j\in J:\vartheta_j=\varnothing\}| + |J_\rmA| + |J_\rmVIc|) - \frac{2}{3}(|J_\rmB| + |J_\rmC| + |J_\rmIV| + |J_\rmVIb| + \max\{|J_\rmVIa|, |J_\rmVIc|\})  } \\
\times C^{\delta+1} \frac{h^{\delta- 2|U|}}{(\delta-2|U|)!}\exp(-ch^2) ((|U|+1)\log(|U|+2))^{|U|} (h/(|U|+1))^{c|U|+m\mathbf{1}_{h\ge |U|}} \prod_{j\in J:\vartheta_j \ne\varnothing} \vartheta_j^{-3/2}.
\end{multline*}
Using that $|J_\rmI|+\cdots+|J_\rmVIc|=|J_\rmA|+|J_\rmB|+|J_\rmC|=|J|$, and that $|J_\rmIII|=|J_\rmB|$, and $|\{j\in J:\vartheta_j=\varnothing\}| = |\{j\in J_\rmIV:\vartheta_j=\varnothing\}| + |J_\rmV|$, we can rewrite the exponent of $N$, and the conclusion follows.
\end{proof}

\begin{proof}[Proof of Proposition \ref{prop:mbdc2}]
We explain how to adapt the previous proof to get the improvements, under the extra requirements.

\noindent\textbf{Improvement (i):} Two scenarios can happen if a fixed $t$ is contained in $\Delta$: it is contained in $\Delta\setminus \vec l$, or $\vec l$.
In the first case, there would be an extra factor of $\delta N^{-2/3}$ in \eqref{eq_x6}.
This would lead to an extra factor of $N^{-2/3}$ in the final bound.

In the second case,
in Step 4, in addition to (a) and (b), there is another case (b') where the fixed number $t\in \llbracket s+1, s'\rrbracket$. In this case (b'), necessarily $u[s]\ge 1$, and the number of possible choices of $\vec r(s)$ would be at most $C\frac{(\max_{\llbracket s, s'\rrbracket} \sH - \min_{\llbracket s, s'\rrbracket} \sH +1)^{u[s]-1}}{(u[s]-1)!}$.
Then in \eqref{eq:uminus}, we would have $u_-[s]=u[s]-1$ under (b').
We note that for any $\llbracket v, w\rrbracket$, at most one of (a), (b'), (e) can happen, and for at most one $s$.
Then \eqref{eq:buv}, there would be an extra factor of $N^{-1/3}$ when $t\in \llbracket v+1, w'\rrbracket$.
This would lead to an extra factor of $N^{-1/3}$ in the final bound.

\noindent\textbf{Improvement (ii):}
In Step 1, jump time data would contain the following extra information: for each $l_j$ or $l'_j$, whether it is contained in $\llbracket 1, y\rrbracket$, or $\rrbracket y+1, Q_m\rrbracket$.
In Step 2, the set $D$ would contain $y$.
In Step 3, when $\llbracket v, w\rrbracket = \llbracket x, y\rrbracket$, the function $\sH(t)$ for $t\in\llbracket v,w\rrbracket$ would be required to have $\sH(t)\ge \sH(v)$ for all $t\in \llbracket v, w\rrbracket$.
Then the number of such $\sH$ would be bounded by $C2^{w-v}(y-x)^{-1/2}$, which has an extra factor of order $(y-x)^{-1/2}$ compared to the last case in Step 3, corresponding to the probability for a Bernoulli random walk with $y-x$ steps to be minimized at $x$.
This would lead to an extra factor of $(y-x)^{-1/2}$ in \eqref{eq:30}, and thereby in \eqref{eq:tbd2} and the final bound.
\end{proof}

In \Cref{prop:mbdc} we only consider walks where $J_\rmII=\emptyset$. 
We can also deduce a bound without this assumption, which will also be useful.
\begin{corollary}  \label{cor:mbdcII}
Under the same setup as \Cref{prop:mbdc} except for that $J_\rmII$ is allowed to be non-empty, the estimate becomes
\begin{multline*}
\sum_{\vec r\in\sB^h}|w(\vec r)|< \prod_{\ell=1}^m (2\sqrt{NN_\ell})^{k_\ell} N^{-|U|-|J|+\frac{1}{3}|J_\rmI|+\frac{1}{3}|J_\rmII| -\frac{1}{3}\bigl(|J_\rmIV|+|\{j\in J_\rmIV:\,\vartheta_j=\varnothing\}|+|J_\rmVIb|+\max\{|J_\rmVIa|,|J_\rmVIc|\}+|J_\rmC|\bigr)}
\\
\times C_2^{\delta+1} \frac{h^{\delta- 2|U|}}{(\delta-2|U|)!}\exp(-C_3h^2) \big((|U|+1)\log(|U|+2)\big)^{|U|} \big(h/(|U|+1)\big)^{C_3|U|+m\mathbf{1}_{h\ge |U|}} \prod_{j\in J:\vartheta_j \ne \varnothing} \vartheta_j^{-3/2}.
\end{multline*}
\end{corollary}
To prove this, we note that each walk can be associated with another walk without type II indices, as follows.
\begin{definition}  \label{defn:fP}
For $\vec s\in\sB$, we construct another $\vec r\in\sB$ without type II indices, as follows.
For each type II index $j\in J_\rmII$, and the corresponding $j'$, we let $r_j=s_j+s_{j'}$ and $r_{j'}= 0$.
For the remaining $j\in J\cup U$ we let $r_j=s_j$.
We use $\cC$ to denote the map $\vec s\mapsto \vec r$.
If $J_\rmII=\emptyset$, then $\cC(\vec r)=\vec r$.
\end{definition}
Then we can reduce \Cref{cor:mbdcII} to \Cref{prop:mbdc}.
\begin{proof}[Proof of \Cref{cor:mbdcII}]
From \eqref{eq:weight} we see that
\[
|w(\vec r)| < CN^{-|J_\rmII|} \exp(Ch)|w(\cC(\vec r))|,
\]
for any $\vec r\in\sB^h$. Here the factor $N^{-|J_\rmII|}$ comes from the second line of \eqref{eq:weight}, since $\cC(\vec r)$ has $|J_\rmII|$ fewer jumps than $\vec r$; while the factor $\exp(Ch)$ is from the third line of \eqref{eq:weight}.

We shall apply  \Cref{prop:mbdc} to all $\cC(\vec r)$.
Note that  $J_\rmII$ and $U$ are fixed in $\sB^h$, so at most $|U|^{|J_\rmII|}\le |U|^m$ different walks $\vec r\in\sB^h$ would have the same image under $\cC$. 
On the other hand, in \Cref{prop:mbdc}, the sets $J, U$ and the partition of $J$ are fixed.
However, for different $\vec r\in \sB^h$, the sets $J, U$ and the partition of $J$ for $\cC(\vec r)$ are not necessarily the same; in fact, there are at most $|U|^{|J_\rmII|}\le |U|^m$ possibilities, and we need to apply \Cref{prop:mbdc} to each of them.
Then the upper bound follows, as we replace $|J_\rmI|$ by $|J_\rmI|+|J_\rmII|$ and $|U|$ by $|U|-|J_\rmII|$ in the right-hand side of \eqref{eq:wbdfdt}, and multiply the bound by $N^{-|J_\rmII|} \exp(Ch)$ and $|U|^m\cdot |U|^m=|U|^{2m}$.
The factors $\exp(ch)$ and $|U|^{2m}$ can be respectively absorbed into $\exp(-C_3h^2)$ and $C_2^{\delta+1}$, by taking larger $C_2$ and $C_3$.
\end{proof}

\subsection{Cancellations and the first regularization}  \label{ssec:buc}

In this section we apply the general bound of \Cref{prop:mbdc} to produce two uniform upper bounds for the subsums of the sum of weights \eqref{eq:expsum}.

\begin{prop}  \label{prop:csBJ}
For a fixed $i_1,\dots,i_m$, recalling $\sB=\sB[i_1,\dots,i_m]$ from \Cref{ssec:dwrep}, we have
\[
\left| \sum_{\vec r\in \sB} w(\vec r) \right| < C\prod_{\ell=1}^m (2\sqrt{NN_\ell})^{k_\ell} N^{-|J|}.
\]
\end{prop}

The proposition implies that the expression in \Cref{prop:conv} is of the desired order $\prod_{\ell=1}^m (2\sqrt{NN_\ell})^{k_\ell}$, since the number of possible $(i_1,\ldots,i_m)$ with given $|J|$ is of order $N^{|J|}$.

We recall that the announced limit $\bL_\beta(\vec\bk,\vec\ttt)$ is an improper integral, involving $\eps\to 0$ transition. Before producing this integral, we also introduce a bound which would imply its convergence, as well as uniform convergence of the sums representing its discrete counterpart.

\begin{definition} \label{defn:hDeps}
    Take any $\epsilon$ with $\Ceps^{-1}N^{-2/3}<\epsilon<\Ceps$ for a universal constant $\Ceps>0$; any $C, c$ can depend on $\Ceps$. We let $\sB_\epsilon=\sB_\epsilon[i_1,\ldots,i_m]$ be the set of all walks $\vec r$ of $(\hD_{i_m}^{N_m})^{k_m} \cdots (\hD_{i_1}^{N_1})^{k_1}$, with the following requirements:
\begin{itemize}
    \item $J_\rmVIa=J_\rmVIb=J_\rmVIc=J_\rmC=\emptyset$.
    \item For each $j\in J_\rmI\cup J_\rmII\cup J_\rmIV$, $\vartheta_j\wedge \dot\vartheta_j > \epsilon N^{2/3}$.
\end{itemize}
\end{definition}

\begin{prop}  \label{prop:sBJe}
We have
\[
\left| \sum_{\vec r\in \sB\setminus\sB_\epsilon} w(\vec r) \right| < C\sqrt{\epsilon}\prod_{\ell=1}^m (2\sqrt{NN_\ell})^{k_\ell} N^{-|J|}.
\]
\end{prop}
The rest of this section is devoted to the proofs of \Cref{prop:csBJ}, and \Cref{prop:sBJe}.

As already alluded to, there is a blow-up issue, which we now resolve by exploiting cancellations between type I and type II.

For this, we use the map $\cC$ from \Cref{defn:fP}.
Note that a walk $\vec r\in \sB_\epsilon$ if and only if $\cC(\vec r)\in \sB_\epsilon$.
For $\vec r$ satisfying $\cC(\vec r)=\vec r$, i.e., without type II indices,
the pre-image $\{\vec s: \cC(\vec s)=\vec r\}$ contains order $N^{|J_\rmI|}$ many walks, since for each $j$ that is of type I in $\vec r$, it can be made type II, with order $N$ many choices of corresponding $j'$.
(Note that this differs from the setting in \Cref{cor:mbdcII}, where $U$ and $J_\rmII$ are fixed.)

Now from \eqref{eq:weight}, the sum of $w(\vec s)$ over all the order $N^{|J_\rmI|}$ many walks can be bounded as follows.
\begin{lemma}  \label{lem:eq:caneq}
We have
\begin{equation}   \label{eq:caneq1}
\left|\sum_{\vec s: \cC(\vec s)=\vec r} w(\vec s)\right| <C\prod_{j\in J_I}\left( |U|N^{-1}+ N^{-1}\vartheta_{j}\max_t \sH(t) \right)|w(\vec r)|.
\end{equation}
\end{lemma}

\begin{proof}
We claim that
\begin{equation}   \label{eq:caneq11}
\left|\sum_{\vec s: \cC(\vec s)=\vec r} w(\vec s)\right| <\prod_{j\in J_\rmI} \left( C(|U|+1)N^{-1}+  1- (1+CN^{-1}\max_t \sH(t))^{-\vartheta_{j}} \right)|w(\vec r)|.
\end{equation}
Indeed, here the factors in the bound are from the following reason.
For each type I index $j$ of $\vec r$, if $j$ is type II for some $\vec s$, then in the $\ell=a_j$ factor in \eqref{eq:weight}, $w(\vec r)$ and $w(\vec s)$ differ in the following ways:
\begin{itemize}
    \item in the first line of \eqref{eq:weight}, $w(\vec s)$ has an extra factor of $-1$;
    \item in the second line of \eqref{eq:weight}, $w(\vec s)$ has an extra factor of $N_{a_j}^{-1}$;
    \item in the third line of \eqref{eq:weight}, $s_j(t-1)=r_j(t-1)-r_j(Q_{a_j-1})$, for $t\in \llbracket Q_{a_j-1}+1, Q_{a_j-1}+\vartheta_{j}\rrbracket$. (And $s_j(t-1)=r_j(t-1)$ for any other $t$.)
\end{itemize}
The first two above give a factor of $-N_{a_j}^{-1}$, while the last one gives a factor smaller than $1$, but at least $(1+CN^{-1}\max_t \sH(t))^{-\vartheta_{j}}$.
On the other hand, to make $j$ of type II in $\vec s$, the corresponding $j'$ has at least $N_{a_j}-|U|-m$ and at most $N_{a_j}$ many choices.
The product of this number of choices and the factors above, plus $1$, gives
\[
-(1-(|U|+m)N_{a_j}^{-1})(1+CN^{-1}\max_t \sH(t))^{-\vartheta_{j}} + 1,
\]
whose absolute value is bounded by $ C(|U|+1)N^{-1}+  \left( 1- (1+CN^{-1}\max_t \sH(t))^{-\vartheta_{j}} \right)$.

Hence, we have proven \eqref{eq:caneq11}. In order to deduce \eqref{eq:caneq1} from it, we notice that since since $CN^{-1}\max_t \sH(t)>0$, and $\vartheta_j > 0$ for each $j\in J_\rmI$, we have
\[
(1+CN^{-1}\max_t \sH(t))^{-\vartheta_{j}} \ge  1 - CN^{-1}\vartheta_j\max_t \sH(t). \qedhere
\]
\end{proof}

On the other hand, \Cref{prop:mbdc} and \Cref{prop:mbdc2} give the following.

Let $\sB^{{h}}_+$ be the collection of all $\vec r\in\sB$, satisfying $\max_t \sH(t) >{h}N^{1/3}$ and $J_\rmII=\emptyset$, and the following fixed: $J_\rmI$, $J_\rmIII$, $J_\rmIV$, $J_\rmV$, $J_\rmVIa$, $J_\rmVIb$, $J_\rmVIc$, and $\vartheta_j$ for each $j\in J$, $\dot\vartheta_j$ for each $j\in J\setminus J_\rmI$.
\begin{corollary}  \label{cor:sumoverUh}
For any $h'\ge 0$, we have
\begin{multline}  \label{eq:sumoverUh}
\sum_{\vec r\in\sB^{{h'}}_+}\prod_{j\in J_\rmI} \frac{|U|+\vartheta_j\max_t \sH(t)}{N} |w(\vec r)|<C\prod_{\ell=1}^m(2\sqrt{NN_\ell})^{k_\ell} N^{-|J| -\frac{1}{3}(|J_\rmVIb|+\max\{|J_\rmVIa|, |J_\rmVIc|\}+|J_\rmC|)} \\ \times    \exp(-c{h'}^2)
\prod_{j\in J_\rmI} \frac{\vartheta_j^{-1/2} }{N^{1/3}}
\prod_{j\in J_\rmIV,\vartheta_j\ne \varnothing} \frac{\vartheta_j^{-3/2} }{N^{1/3}}
\prod_{j\in J_\rmIV,\vartheta_j=\varnothing} \frac{1}{N^{2/3}}
\prod_{j\in J_\rmIII\cup J_\rmVIa\cup J_\rmVIb\cup J_\rmVIc} \vartheta_j^{-3/2}.
\end{multline}
There is an extra factor of $N^{-1/3}$ or $(y-x)^{-1/2}$ in the right-hand side, under the condition (i) or (ii), respectively, specified in \Cref{prop:mbdc2}.
\end{corollary}
\begin{proof}
In the upper bound \eqref{eq:wbdfdt}, by summing over $\delta\ge 2|U|$, we have
\[
\sum_{\delta = 2|U|}^\infty C^\delta \frac{h^{\delta-2|U|}}{(\delta-2|U|)!} = C^{2|U|}\exp(Ch).
\]
Next, for walks from $\sB^h$ of Proposition \ref{prop:mbdc}, we upper bound $|U|+\vartheta_j\max_t \sH(t)$ in the left-hand side of \eqref{eq:sumoverUh} with  $(|U|+1)\vartheta_j(2h+1) N^{1/3}$. Then by summing over $U$ (noting the number of possible $U$ with given $|U|$ is at most ${N \choose |U|}\le \frac{N^{|U|}}{|U|!}$), we have
\begin{align*}
&\sum_U (|U|+1)^{|J_\rmI|} N^{-|U|} C^{2|U|} \big((|U|+1)\log(|U|+2)\big)^{|U|} \big(h/(|U|+1)\big)^{c|U|+m\mathbf{1}_{h\ge |U|}} \\
&\le \sum_{|U|=0}^\infty \frac{C^{2|U|}}{|U|!} (|U|+1)^m \big((|U|+1)\log(|U|+2)\big)^{|U|} \big(h/(|U|+1)\big)^{c|U|+m\mathbf{1}_{h\ge |U|}}\\
& < C\exp(Ch\log(h+1)).
\end{align*}
We arrive at: 
\begin{multline*}
C\exp(-c{h}^2 + Ch\log(h+1))\prod_{\ell=1}^m(2\sqrt{NN_\ell})^{k_\ell} N^{-|J| -\frac{1}{3}(|J_\rmVIb|+\max\{|J_\rmVIa|, |J_\rmVIc|\}+|J_\rmC|)} \\ \times
\prod_{j\in J_\rmI} \frac{\vartheta_j^{-1/2}(2h+1)}{N^{1/3}}
\prod_{j\in J_\rmIV,\vartheta_j\ne \varnothing} \frac{\vartheta_j^{-3/2} }{N^{1/3}}
\prod_{j\in J_\rmIV,\vartheta_j=\varnothing} \frac{1}{N^{2/3}}
\prod_{j\in J_\rmIII\cup J_\rmVIa\cup J_\rmVIb\cup J_\rmVIc} \vartheta_j^{-3/2}.
\end{multline*}
Summing over dyadic growing $h=h', 2h', 4h', 8h',\dots$, the conclusion of the corollary follows.
Under the conditions (i) or (ii) specified in \Cref{prop:mbdc2}, the extra factors remain unchanged throughout the proof.
\end{proof}

By combining \Cref{lem:eq:caneq} and \Cref{cor:sumoverUh}, we get \Cref{prop:csBJ} and \Cref{prop:sBJe}.
\begin{proof}[Proof of \Cref{prop:csBJ}]
We first sum over $\vec r \in\sB^0_+$ in \Cref{lem:eq:caneq}, and apply \Cref{cor:sumoverUh} (using that $|U|<CN^{2/3}$), and get
\begin{multline}  \label{eq:csBJpf1}
C\prod_{\ell=1}^m(2\sqrt{NN_\ell})^{k_\ell} N^{-|J| -\frac{1}{3}(|J_\rmVIb|+\max\{|J_\rmVIa|, |J_\rmVIc|\}+|J_\rmC|)} \\ \times
\prod_{j\in J_\rmI} \frac{\vartheta_j^{-1/2} }{N^{1/3}}
\prod_{j\in J_\rmIV,\vartheta_j\ne\varnothing} \frac{\vartheta_j^{-3/2} }{N^{1/3}}
\prod_{j\in J_\rmIV,\vartheta_j=\varnothing} \frac{1}{N^{2/3}}
\prod_{j\in J_\rmIII\cup J_\rmVIa\cup J_\rmVIb\cup J_\rmVIc} \vartheta_j^{-3/2}.
\end{multline}
We claim that by summing over $\vartheta_j$ for each $j\in J$, $\dot\vartheta_j$ for each $j\in J\setminus J_\rmI$, each factor in the second line sums to order $1$.

Indeed, for each $j\in J_\rmI$, we would sum $\vartheta_j^{-1/2}$ for $\vartheta_j$ from $1$ and up to a number of order $N^{2/3}$, leading to a factor of order $N^{1/3}$.
For each $j\in J_\rmIII\cup J_\rmVIa\cup J_\rmVIb\cup J_\rmVIc$, we would sum $\vartheta_j^{-3/2}$ for $\vartheta_j$ from $1$ and up to a number of order $N^{2/3}$, giving a sum of order $1$ (note that here $\dot\vartheta_j=\vartheta_j$).
For each $j\in J_\rmV$ there is always $\dot\vartheta_j=\vartheta_j=\varnothing$, so nothing to sum.
For each $j\in J_\rmIV$, if $\vartheta_j=\varnothing$, the number of possible $\dot\vartheta_j$ is of order $N^{2/3}$, thus we get a factor of order $N^{2/3}$.
If $\vartheta_j\neq\varnothing$, we first sum $\vartheta_j^{-3/2}$ for $\vartheta_j$ from $\dot\vartheta_j$ up to a number of order $N^{2/3}$, and get a factor of order $\dot\vartheta_j^{-1/2}$; then we sum this for $\dot\vartheta_j$ from $1$ and up to a number of order $N^{2/3}$, getting a factor of order $N^{1/3}$.

Finally we sum over all possibly ways of partitioning $J$ into the various type sets $J_\rmI, \ldots,$ and $J_\rmA, \ldots$: we simply bound $N^{-|J| -\frac{1}{3}(|J_\rmVIb|+\max\{|J_\rmVIa|, |J_\rmVIc|\}+|J_\rmC|)}\le N^{-|J|}$, and note that the number of partitions is of constant order.
Thus the bound follows.
\end{proof}

\begin{proof}[Proof of \Cref{prop:sBJe}]
Similarly to \Cref{prop:csBJ}, we start from \eqref{eq:csBJpf1}. The difference is that this time we should sum it only over the parameters with (1) $\min\{\vartheta_j: j \in  J_\rmI\} \le \epsilon N^{2/3}$ or $\min\{ \dot\vartheta_j: j \in J_\rmIV\} \le \epsilon N^{2/3}$, or (2) $J_\rmVIa\cup J_\rmVIb\cup J_\rmVIc\cup J_\rmC$ is not empty.

For each $j\in J_\rmI$, the sum of $\frac{\vartheta_j^{-1/2} }{N^{1/3}}$ for $\vartheta_j$ from $1$ up to $\epsilon N^{2/3}$ is of order $\sqrt{\epsilon}$.
For each $j\in J_\rmIV$, the sum of $\frac{\vartheta_j^{-3/2} }{N^{1/3}}$ for $\vartheta_j>\dot\vartheta_j$, plus $\frac{1}{N^{2/3}}$ (for $\vartheta_j=\varnothing$) is bounded by $\frac{C\dot\vartheta_j^{-1/2} }{N^{1/3}} + \frac{1}{N^{2/3}}$; and a further sum of $\dot\vartheta_j$ from $1$ up to $\epsilon N^{2/3}$ is of order $\sqrt{\epsilon}$.
Thus the second line of \eqref{eq:csBJpf1} sum up to order $\sqrt{\epsilon}$.

For the sum over those walks where $J_\rmVIa\cup J_\rmVIb\cup J_\rmVIc\cup J_\rmC$ is non-empty, the factor $N^{ -\frac{1}{3}(|J_\rmVIb|+\max\{|J_\rmVIa|, |J_\rmVIc|\}+|J_\rmC|)}$ is at most $N^{-1/3}<\Ceps^{-1/2}\sqrt{\epsilon}$, since we take $\epsilon>\Ceps^{-1}N^{-2/3}$.

In summary, in either case we get an extra factor of order at most $\sqrt{\epsilon}$, so the conclusion follows.
\end{proof}

\subsection{Elimination of pair types and the second regularization}  \label{ssec:pair}

Type III and type B indices are in pairs.
Our next regularization is to show that these two types can be `eliminated', via a procedure similar to the cancellations between type I and type II.
The goal of this subsection is to write the action of $\hP_{k_m}^{N_m} \cdots \hP_{k_1}^{N_1}$ as a sum of $w(\vec r)$, over $\vec r\in\sB_\epsilon$ without type III indices or type B indices.

Again we take $\epsilon$ to be any number with $\Ceps^{-1}N^{-2/3}<\epsilon<\Ceps$, where $\Ceps>0$ is a universal constant. Any $C, c$ below can depend on $\Ceps$.

We let $\sB_\epsilon^*$ denote the set of all walks $\vec r$ of  a particular operator $(\hD_{m}^{N_m})^{k_m} \cdots (\hD_{1}^{N_1})^{k_1}$ 
(i.e., in the notations of \Cref{ssec:dwrep} for $\sB=\sB[i_1,\ldots,i_m]$, we set $i_\ell=\ell$ for each $\ell\in\llbracket 1, m\rrbracket$)
, such that $J_\rmB=J_\rmC=\emptyset$, $J_\rmIII=J_\rmVIa=J_\rmVIb=J_\rmVIc=\emptyset$, and $\vartheta_\ell\wedge \dot\vartheta_\ell>\epsilon N^{2/3}$ for each $\ell\in J_\rmI\cup J_\rmII\cup J_\rmIV$. 
Note that for such a walk, there is $J=\llbracket 1, m\rrbracket$, and $a_\ell = b_\ell = \ell$ for each $\ell \in \llbracket 1, m\rrbracket$.
\begin{prop}  \label{prop:truncreg}
We have
\[
\left| N^m \sum_{\vec r \in \sB_\epsilon^*} w(\vec r) - \hP_{k_m}^{N_m} \cdots \hP_{k_1}^{N_1} [1]_{x_1=\dots=x_N=0} \right| < C \prod_{\ell=1}^m (2\sqrt{NN_\ell})^{k_\ell} (\sqrt{\epsilon} + N^{-1/12}).
\]
\end{prop}
We note that here in the sum we have the restriction from $\epsilon$, i.e., $\vartheta_\ell\wedge \dot\vartheta_\ell>\epsilon N^{2/3}$. When sending $N\to\infty$ later in this section, this constraint is necessary to ensure absolute convergence, i.e., $\limsup_{N\to\infty}N^m \prod_{\ell=1}^m (2\sqrt{NN_\ell})^{-k_\ell} \sum_{\vec r \in \sB_\epsilon^*}|w(\vec r)|<\infty$ (see \Cref{prop:fixedepsconv}).

In the rest of this subsection we prove this proposition through a series of intermediate statements.

We shall consider walks of different operators.
Thus to be clear, we fix one operator $\hD=(\hD_{i_m}^{N_m})^{k_m} \cdots (\hD_{i_1}^{N_1})^{k_1}$. All the above set up notations, such as $\sB$ and $\sB_\epsilon$, refer to those associated with it (unless otherwise noted).

We define a partial order $\prec$ on $J$, such that for each $j\in J_\rmIII$ and the corresponding $j'\in J_\rmB$, we have $j'\prec j$.
Under this partial order, $J$ is a disjoint union of linearly-ordered chains.
By a 'minimal' element of a partially ordered set, we refer to any element such that no other element is smaller than it. In particular, any singleton element (which is not comparable to any other element) is minimal.
\begin{definition}
For $\vec s\in \sB$, we construct a walk $\vec r$ of another operator $(\hD_{i_m'}^{N_m})^{k_m} \cdots (\hD_{i_1'}^{N_1})^{k_1}$ without types III or B indices, as follows.
For each $j\in J$ that is minimal under $\prec$, we let $r_j=\sum_{j':j\preceq j'}s_{j'}$, and $i_\ell'=j$ for each $\ell$ with $j\preceq i_\ell$; and let $r_j= 0$ for each non-minimal $j$.
We use $\cM$ to denote the map $\vec s\mapsto \vec r$, and let $\bm(\vec s)$ be the number of non-minimal $j\in J$.
Note that $\bm(\vec s)$ equals the number of type III (or type B) indices of $\vec s$.
\end{definition}
Note that $\cM$ and $\cC$ (from \Cref{defn:fP}) commute.

\begin{lemma}  \label{eq:redutosum}
We have
\[
\left| \sum_{\vec r\in \sB_\epsilon} \left[w(\vec r) - (-N)^{-\bm(\vec r)} w(\cM(\vec r))\right] \right| < C\prod_{\ell=1}^m (2\sqrt{NN_\ell})^{k_\ell} N^{-|J|-1/3}.
\]
\end{lemma}
\begin{proof}
Similarly to \Cref{lem:eq:caneq}, for $\vec r\in \sB$ without type II indices, we have
\begin{multline}  \label{eq:caneq2}
\left|\sum_{\vec s: \cC(\vec s)=\vec r} \left[w(\vec s) - (-N)^{-\bm(\vec r)} w(\cM(\vec s))\right] \right|
< C\left(N^{-1/3} + \exp(CN^{-1/3}\max_t \sH(t)) N^{-2/3}\sum_{j\in J_\rmIII} \vartheta_j\right) \\ \times \prod_{j\in J_\rmI}\left( |U|N^{-1}+ N^{-1}\vartheta_{j}\max_t \sH(t) \right)|w(\vec r)|.
\end{multline}
Here the factor with $\prod_{j\in J_\rmI}$ is by the same reasoning as in the proof \Cref{lem:eq:caneq}.
The factor in the first line is from $(-N)^{-\bm(\vec r)} w(\cM(\vec s)) / w(\vec s)-1$, which equals $(-N)^{-\bm(\vec r)} w(\cM(\vec r)) / w(\vec r)-1$.
Namely, from \eqref{eq:weight}, $w(\vec r)$ and $w(\cM(\vec r))$ differ at $\ell=a_j$ for each $j$ of type III, in the following ways:
\begin{itemize}
    \item in the first line of \eqref{eq:weight}, $w(\cM(\vec r))$ has an extra factor of $-1$;
    \item in the second line of \eqref{eq:weight}, $w(\cM(\vec r))$ has an extra factor of $N_{a_j}$;
    \item in the third line of \eqref{eq:weight}, $r_j(t-1)=\cM(\vec r)_j(t-1)-\cM(\vec r)_j(Q_{a_j-1})$, for $t\in \llbracket Q_{a_j-1}+1, Q_{a_j-1}+\vartheta_{j}\rrbracket$. (And $r_j(t-1)=\cM(\vec r)_j(t-1)$ for any other $t$.)
\end{itemize}
These together imply that
\[
|(-N)^{-\bm(\vec r)} w(\cM(\vec r)) / w(\vec r)-1 | < \left|\prod_{j\in J_\rmIII}(1+CN^{-1}\max_t \sH(t) )^{\vartheta_j} - 1 \right| + 1 - (N_m/N)^{\bm(\vec r)}.
\]
This is further bounded by (using that $N^{-2/3}\sum_{j\in J_\rmIII} \vartheta_j<C$)
\[
CN^{-1/3}+\exp\left(CN^{-1}\sum_{j\in J_\rmIII} \vartheta_j \max_t \sH(t)\right) - 1 < CN^{-1/3} + C\exp(CN^{-1/3}\max_t \sH(t)) N^{-2/3}\sum_{j\in J_\rmIII} \vartheta_j.
\]
Thus we get \eqref{eq:caneq2}. From there, the conclusion follows similarly to the proof of \Cref{prop:sBJe}.
\end{proof}

We now take a operator $\hD'=(\hD_{i_m'}^{N_m})^{k_m} \cdots (\hD_{i_1'}^{N_1})^{k_1}$, and let $\sB(\hD')=\sB[i_1',\ldots,i_m']$ denote the set of walks of it, and $\sB_\epsilon(\hD')=\sB[i_1',\ldots,i_m']$ denote the set of those satisfying the conditions in \Cref{defn:hDeps}.
We assume that $\sB(\hD')$ contains $\cM(\vec r)$ for some $\vec r \in \sB$, and we use $\sO=\sO(\hD)$ to denote the collection of all such operators $\hD'$. 

More precisely, $\sO$ contains all $(\hD_{i_m'}^{N_m})^{k_m} \cdots (\hD_{i_1'}^{N_1})^{k_1}$, such that for each $j\in J$, either (1) $i_\ell'=j$ for each $\ell$ with $i_\ell=j$, or (2) there is a $j'\in J$, such that $i_\ell'=j'$ for each $\ell$ with $i_\ell=j$, and $\max\{\ell:i_\ell=j'\}<\min\{\ell:i_\ell=j\}$.

Moreover, we let $J(\hD')$ be the collection of all $j$ that appears in $(i_1',\ldots, i_m')$.
Then for $\vec r\in\sB$ with $\cM(\vec r)\in\sB(\hD')$, we have $\bm(\vec r)=|J|-|J(\hD')|$.

\begin{lemma}  \label{lem:diff2}
For $\hD'=(\hD_{i_m'}^{N_m})^{k_m} \cdots (\hD_{i_1'}^{N_1})^{k_1} \in \sO$, we have
\[
\left| \sum_{\vec r\in \sB_\epsilon, \cM(\vec r)\in \sB(\hD')} w(\cM(\vec r)) - \sum_{\vec r\in \sB_\epsilon(\hD'), \cM(\vec r)=\vec r} w(\vec r) \right| < C \prod_{\ell=1}^m (2\sqrt{NN_\ell})^{k_\ell} N^{-|J(\hD')|-1/12}.
\]
\end{lemma}
\begin{proof}
We note that $\cM$ is injective, and that $\cM(\vec r)\in \sB_\epsilon(\hD')$ whenever $\vec r\in \sB_\epsilon$.
However, there exist walks $\vec r\in \sB_\epsilon(\hD')$ with $\cM(\vec r)=\vec r$, but which are not in the image of $\cM$ from $\sB_\epsilon$; and the left-hand side is the sum of these walks.

We now describe these walks, by giving some necessary conditions any one of them must satisfy.
Any such walk $\vec r\in\sB_\epsilon(\hD')$ must not contain any type II index, and should satisfy one of the following (where $a_j$ is in terms of $(i_1,\ldots,i_m)$):
\begin{itemize}
    \item $t\in\Delta$ for some $t\in\llbracket Q_{a_j-1}, Q_{a_j-1} + N^{1/4}\rrbracket$ and $j\in J\setminus J(\hD')$.
    \item $\sH \ge \sH(Q_\ell)$ on $\llbracket Q_{a_j-1}, Q_{a_j-1} + N^{1/4}\rrbracket$, for some $j\in J\setminus J(\hD')$.
\end{itemize}
This is because, if for each $j\in J\setminus J(\hD')$, none of these two hold, then one can construct a $\vec r^* \in \sB_\epsilon$ with $\cM(\vec r_*)=\vec r$.
(We note that here we take $N^{1/4}$ to make the sum over $\vec r$ satisfying either of the two conditions small, as will be evident at the end of this proof. )

More precisely, we let $r^*_j=r_j$ for each $j\in U$; and for each $j\in J$, we take $r^*_j$ as follows.
If $j\in J(\hD')$, let $r^*_j=r_j$ on $\llbracket 0, Q_{b_j}\rrbracket$ (where $b_j$ is in terms of $(i_1,\ldots,i_m)$).
If $j\not\in J(\hD')$, let $r^*_j=0$ on $\llbracket 0, Q_{a_j-1}\rrbracket$, $r^*_j=r_{j'}-r_{j'}(Q_{a_j-1})$ on $\llbracket Q_{a_j-1}, t_j-1\rrbracket$, and $r^*_j=r_{j'}$ on $\llbracket t_j, Q_{b_j}\rrbracket$. Here $j'\in J(\hD')$ is the index with $j'=i'_{a_j}$, and $t_j$ is the smallest number in $\llbracket Q_{a_j-1}, Q_{a_j}\rrbracket$ with $\sH(t_j)<\sH(Q_{a_j-1})$ (which necessarily exists and is in $\llbracket Q_{a_j-1}+1, Q_{a_j-1} + N^{1/4}\rrbracket$, by the above conditions).
Now that we have defined $r^*_j$ on $\llbracket 0, Q_{b_j}\rrbracket$ for each $j\in J$, we let $r^*_j=0$ on $\llbracket Q_{\ell-1}, Q_\ell\rrbracket$ for each $\ell \ge b_j+2$, and choose $r^*_j$ on $\llbracket Q_{b_j}, Q_{b_j+1}\rrbracket$ (if $b_j<Q_m$) to ensure that the sum $\sum_{j\in J\cup U}r^*_j$ remains the same as $\sH $.

Therefore, the left-hand side in the statement is bounded by
\[
\sum_{\vec r} \left| \sum_{\vec s: \cC(\vec s)=\vec r} w(\vec s) \right|,
\]
where the sum is over all $\vec r$ satisfying at least one of the above two bulleted conditions.
Then the conclusion follows from bounding the sum using \Cref{lem:eq:caneq} and \Cref{cor:sumoverUh} with the additional conditions (i) and (ii), which are precisely 
due to the two bulleted conditions above.
\end{proof}

We next take a (weighted) sum of $\sum_{\vec r\in \sB_\epsilon(\hD')} w(\vec r)$ over all $\hD' \in \sO$.
From cancellations given in \Cref{eq:redutosum} and \Cref{lem:diff2}, the main remaining terms would be $\sum_{\vec r \in \sB_\epsilon, \cM(\vec r)=\vec r} w(\vec r)$.
\begin{lemma}  \label{lem:hDpbd}
We have
    \[
    \left| \sum_{\vec r \in \sB_\epsilon, \cM(\vec r)=\vec r} w(\vec r) - \sum_{\hD' \in \sO} N^{|J(\hD')|-|J|} \sum_{\vec r\in \sB_\epsilon(\hD')} w(\vec r)  \right|<C \prod_{\ell=1}^m (2\sqrt{NN_\ell})^{k_\ell} N^{-|J|-1/12}.
    \]
\end{lemma}
\begin{proof}
For each $\hD'\in\sO$, we let $\sO(\hD')\subset \sO$ be the collection of all operators $\hD''$, such that $\sB(\hD'')$ contains $\cM(\vec r)$ for some $\vec r\in\sB(\hD')$.
Then by \Cref{eq:redutosum} and \Cref{lem:diff2} we have
\begin{multline}  \label{eq:hDpbdpf}
\left|\sum_{\vec r\in \sB_\epsilon(\hD')} w(\vec r) - \sum_{\hD''\in\sO(\hD')} (-N)^{|J(\hD'')|-|J(\hD')|} \sum_{\vec r\in \sB_\epsilon(\hD''), \cM(\vec r)=\vec r}w(\vec r) \right| \\ < C \prod_{\ell=1}^m (2\sqrt{NN_\ell})^{k_\ell} N^{-|J(\hD')|-1/12}.
\end{multline}
Note that for each $\hD''\in\sO$, $\hD''\neq (\hD_{i_m}^{N_m})^{k_m} \cdots (\hD_{i_1}^{N_1})^{k_1}$,
we have
\[
\sum_{\hD'\in\sO: \hD''\in \sO(\hD')} (-1)^{|J(\hD')|}=0.
\]
This follows from M\"obius inversion formula for partially ordered sets (see e.g., \cite{rota,rotabook}). If we make $\sO$ a partially ordered set, with a partial order $\prec_\sO$ given by $\hD''\prec_\sO \hD'$ for any $\hD', \hD''\in \sO$ with $\hD''\in \sO(\hD')$, then it is straight-forward to check that $(\hD', \hD'')\mapsto (-1)^{|J(\hD')|-|J(\hD'')|}$ is the M\"obius function.
This identity can also be understood as an inclusion-exclusion principle on the set of indices.

Then by multiplying $N^{|J(\hD')|-|J|}$ in both sides of \eqref{eq:hDpbdpf} and summing over $\hD'$, the conclusion follows.
\end{proof}

\begin{proof}[Proof of \Cref{prop:truncreg}]
As stated in \Cref{lem:exp}, we expand $\hP_{k_m}^{N_m} \cdots \hP_{k_1}^{N_1}$ as a sum of operators $\hD=(\hD_{\bar i_m}^{N_m})^{k_m} \cdots (\hD_{\bar i_1}^{N_1})^{k_1}$, over all $(\bar i_1, \ldots, \bar i_m) \in \prod_{\ell=1}^m\llbracket 1, N_\ell \rrbracket$.
Then by \Cref{prop:csBJ} and $N-CN^{2/3}<N_m\le N_1\le N$, we have that the sum of $(\hD_{\bar i_m}^{N_m})^{k_m} \cdots (\hD_{\bar i_1}^{N_1})^{k_1}[1]_{x_1=\dots=x_N=0}$ over those $(\bar i_1, \ldots, \bar i_m)$ with $\max\{\bar i_\ell: \ell \in \llbracket 1,m\rrbracket\} > N_m$ is bounded by $C \prod_{\ell=1}^m (2\sqrt{NN_\ell})^{k_\ell} N^{-1/3}$. Therefore it suffices to sum over $(\bar i_1, \ldots, \bar i_m) \in \llbracket 1, N_m \rrbracket^m$.
By symmetry among the indices, $N-CN^{2/3}<N_m\le N_1\le N$, and further using \Cref{prop:csBJ}, we have
\[
\left| \hP_{k_m}^{N_m} \cdots \hP_{k_1}^{N_1} [1]_{x_1=\dots=x_N=0} - \sum_{\hD} N^{|J(\hD)|} \sum_{\vec r\in \sB(\hD)} w(\vec r) \right|<C \prod_{\ell=1}^m (2\sqrt{NN_\ell})^{k_\ell}N^{-1/3}.
\]
From this, \Cref{prop:sBJe}, and \Cref{lem:hDpbd}, the conclusion follows.
\end{proof}

\subsection{Discrete blocks decomposition}  \label{ssec:bdecom}
The outcome of the previous six subsections is that the expressions of \Cref{prop:conv} stay bounded as $N\to\infty$. The next step is to actually compute the limit and to identify it with $\bL_\beta(\vec\bk,\vec\ttt)$. This will be done in this and the next subsections: we will analyze $N^{m}\sum_{\vec r\in \sB_\epsilon^*}w(\vec r)$ as $N\to\infty$ and show that it converges towards the integral over $\sK_\epsilon[\vec\bk]$ in \Cref{defn:core}.

For this, in the present subsection we decompose a walk into `discrete blocks' --- which is a discrete analog of blocks defined in \Cref{sec:forjm}, and can be thought as the `skeleton' of a walk --- and a Bernoulli path. \Cref{fig:dblock} gives an illustration and more precise definitions follow, culminating in \Cref{prop:IXHGest}, which shows that the discrete structures converge to their continuous counterparts as $N\to\infty$.

Again, we emphasize that we do not work with arbitrary walks in this subsection, but only those in $\sB_\epsilon^*$, which remain after all the regularization steps (i.e., cancellations and pair type eliminations) as being discussed in \Cref{prop:truncreg}.
We recall the notation $Q_\ell=\sum_{\ell'=1}^\ell k_{\ell'}$ for each $\ell\in\llbracket 0,m\rrbracket$.

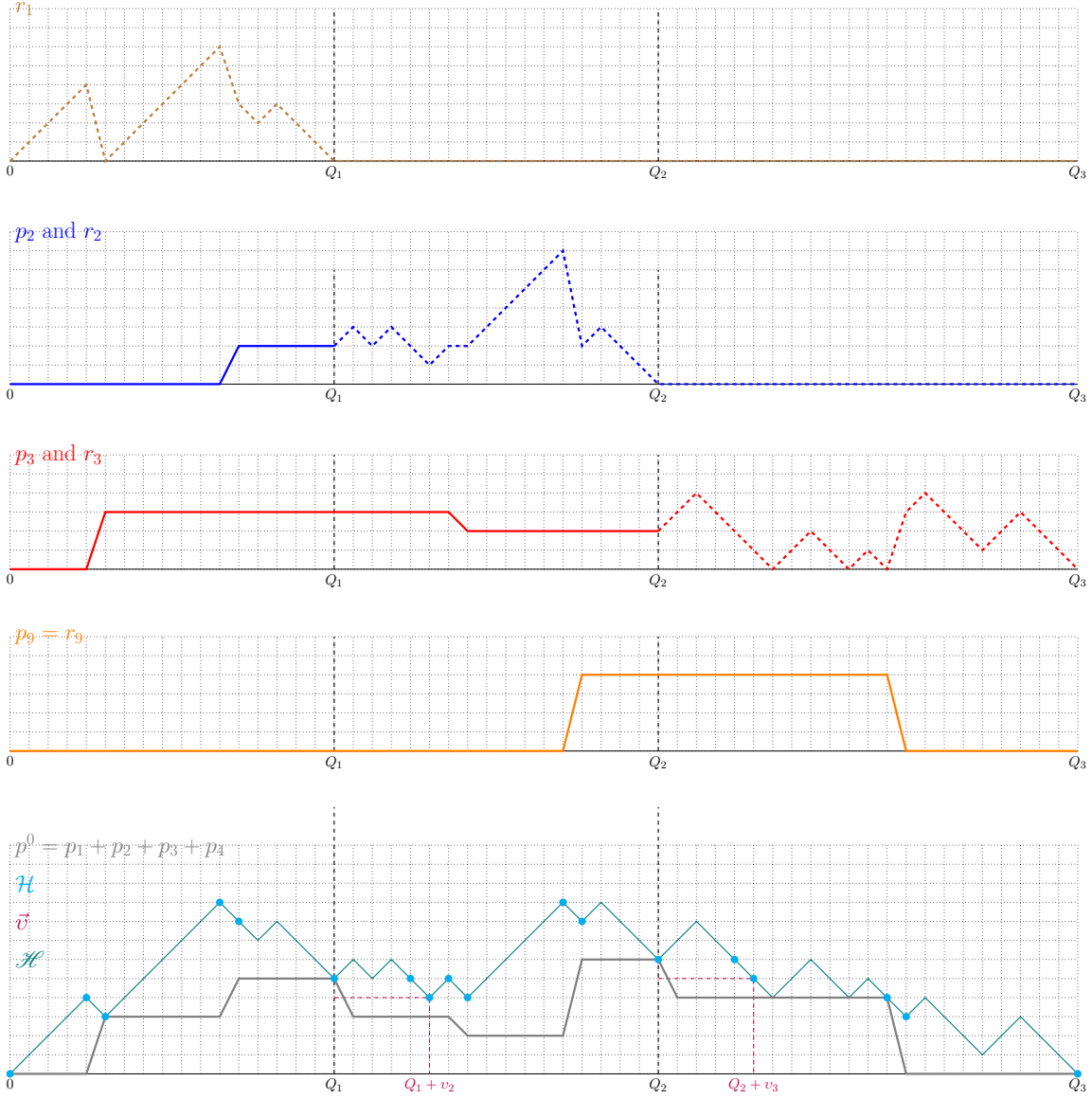
\begin{figure}[!ht]
    \centering
\begin{subfigure}[b]{0.98\textwidth}
    \resizebox{0.95\textwidth}{!}{
    \begin{tikzpicture}
        \draw[thin] [dotted] [step=0.5] (0,0) grid (28,4);
        \draw (0,0)--(28,0);
        \draw[thick] [dashed] (8.5,0)--(8.5,4);
        \draw[thick] [dashed] (17,0)--(17,4);
        \draw[brown] [ultra thick] [dashed] (0,0)--(2,2)--(2.5,0)--(5.5,3)--(6,1.5)--(6.5,1)--(7,1.5)--(8.5,0)--(28,0);
        \draw (0,0) node[anchor=north]{$0$};
        \draw (8.5,0) node[anchor=north]{$Q_1$};
        \draw (17,0) node[anchor=north]{$Q_2$};
        \draw (28,0) node[anchor=north]{$Q_3$};
        \draw (0,4) node[anchor=west, color=brown]{{\LARGE $r_1$}};
    \end{tikzpicture}}
\end{subfigure}
\par\bigskip
\begin{subfigure}[b]{0.98\textwidth}
    \resizebox{0.95\textwidth}{!}{
    \begin{tikzpicture}
        \draw[thin] [dotted] [step=0.5] (0,0) grid (28,4);
        \draw (0,0)--(28,0);
        \draw[thick] [dashed] (8.5,0)--(8.5,3);
        \draw[thick] [dashed] (17,0)--(17,3);
        \draw[blue] [ultra thick] (0,0)--(5.5,0) -- (6,1)--(8.5,1);
        \draw[blue] [ultra thick] [dashed] (8.5,1)--(9,1.5)--(9.5,1)--(10,1.5)--(11,0.5)--(11.5,1)--(12,1)--(14.5,3.5)--(15,1)--(15.5,1.5)--(17,0)--(28,0);
        \draw (0,0) node[anchor=north]{$0$};
        \draw (8.5,0) node[anchor=north]{$Q_1$};
        \draw (17,0) node[anchor=north]{$Q_2$};
        \draw (28,0) node[anchor=north]{$Q_3$};
        \draw (0,4) node[anchor=west, color=blue]{{\LARGE $p_2$ and $r_2$}};
    \end{tikzpicture}}
\end{subfigure}
\par\bigskip
\begin{subfigure}[b]{0.98\textwidth}
    \resizebox{0.95\textwidth}{!}{
    \begin{tikzpicture}
        \draw[thin] [dotted] [step=0.5] (0,0) grid (28,3);
        \draw (0,0)--(28,0);
        \draw[thick] [dashed] (8.5,0)--(8.5,3);
        \draw[thick] [dashed] (17,0)--(17,3);
        \draw[red] [ultra thick] (0,0)--(2,0)--(2.5,1.5)--(11.5,1.5)--(12,1)--(17,1);
        \draw[red] [ultra thick] [dashed] (17,1)--(18,2)--(20,0)--(21,1)--(22,0)--(22.5,0.5)--(23,0)--(23.5,1.5)--(24,2)--(25.5,0.5)--(26.5,1.5)--(28,0);
        \draw (0,0) node[anchor=north]{$0$};
        \draw (8.5,0) node[anchor=north]{$Q_1$};
        \draw (17,0) node[anchor=north]{$Q_2$};
        \draw (28,0) node[anchor=north]{$Q_3$};
        \draw (0,3) node[anchor=west, color=red]{{\LARGE $p_3$ and $r_3$}};
    \end{tikzpicture}}
\end{subfigure}
\par\bigskip
\begin{subfigure}[b]{0.98\textwidth}
    \resizebox{0.95\textwidth}{!}{
    \begin{tikzpicture}
        \draw[thin] [dotted] [step=0.5] (0,0) grid (28,3);
        \draw (0,0)--(28,0);
        \draw[thick] [dashed] (8.5,0)--(8.5,3);
        \draw[thick] [dashed] (17,0)--(17,3);
        \draw[orange] [ultra thick] (0,0)--(14.5,0)--(15,2)--(23,2)--(23.5,0)--(28,0);
        \draw (0,0) node[anchor=north]{$0$};
        \draw (8.5,0) node[anchor=north]{$Q_1$};
        \draw (17,0) node[anchor=north]{$Q_2$};
        \draw (28,0) node[anchor=north]{$Q_3$};
        \draw (0,3) node[anchor=west, color=orange]{{\LARGE $p_9=r_9$}};
    \end{tikzpicture}}
\end{subfigure}
\par\bigskip
\begin{subfigure}[b]{0.98\textwidth}
    \resizebox{0.95\textwidth}{!}{
    \begin{tikzpicture}
        \draw[teal] [thick] (0,0)--(2,2)--(2.5,1.5)--(5.5,4.5)--(6.5,3.5)--(7,4)--(8.5,2.5)--(9,3)--(9.5,2.5)--(10,3)--(11,2)--(11.5,2.5)--(12,2)--(14.5,4.5)--(15,4)--(15.5,4.5)--(17,3)--(18,4)--(20,2)--(21,3)--(22,2)--(22.5,2.5)--(23.5,1.5)--(24,2)--(25.5,0.5)--(26.5,1.5)--(28,0);

        \draw[thin] [dotted] [step=0.5] (0,0) grid (28,6);
        \draw (0,0)--(28,0);
        \draw[thick] [dashed] (8.5,0)--(8.5,7);
        \draw[thick] [dashed] (17,0)--(17,7);

        \draw[purple] [dashed] (8.5,2)--(11,2)--(11,0);

        \draw[purple] [dashed] (17,2.5)--(19.5,2.5)--(19.5,0);

        \draw[gray] [ultra thick] (0,0)--(2,0)--(2.5,1.5)--(5.5,1.5)--(6,2.5)--(8.5,2.5)--(9,1.5)--(11.5,1.5)--(12,1)--(14.5,1)--(15,3)--(17,3)--(17.5,2)--(23,2)--(23.5,0)--(28,0);

        \draw[cyan, fill=cyan] [ultra thick] (0,0) circle (2pt);
        \draw[cyan, fill=cyan] [ultra thick] (28,0) circle (2pt);
        \draw[cyan, fill=cyan] [ultra thick] (12,2) circle (2pt);
        \draw[cyan, fill=cyan] [ultra thick] (19,3) circle (2pt);
        \draw[cyan, fill=cyan] [ultra thick] (19.5,2.5) circle (2pt);
        \draw[cyan, fill=cyan] [ultra thick] (2,2) circle (2pt);
        \draw[cyan, fill=cyan] [ultra thick] (2.5,1.5) circle (2pt);
        \draw[cyan, fill=cyan] [ultra thick] (5.5,4.5) circle (2pt);
        \draw[cyan, fill=cyan] [ultra thick] (6,4) circle (2pt);
        \draw[cyan, fill=cyan] [ultra thick] (8.5,2.5) circle (2pt);
        \draw[cyan, fill=cyan] [ultra thick] (11,2) circle (2pt);
        \draw[cyan, fill=cyan] [ultra thick] (10.5,2.5) circle (2pt);
        \draw[cyan, fill=cyan] [ultra thick] (11.5,2.5) circle (2pt);
        \draw[cyan, fill=cyan] [ultra thick] (15,4) circle (2pt);
        \draw[cyan, fill=cyan] [ultra thick] (14.5,4.5) circle (2pt);
        \draw[cyan, fill=cyan] [ultra thick] (17,3) circle (2pt);
        \draw[cyan, fill=cyan] [ultra thick] (23,2) circle (2pt);
        \draw[cyan, fill=cyan] [ultra thick] (23.5,1.5) circle (2pt);

        \draw (11,0) node[anchor=north,color=purple]{$Q_1+\upsilon_2$};
        \draw (19.5,0) node[anchor=north,color=purple]{$Q_2+\upsilon_3$};
        \draw (0,0) node[anchor=north]{$0$};
        \draw (8.5,0) node[anchor=north]{$Q_1$};
        \draw (17,0) node[anchor=north]{$Q_2$};
        \draw (28,0) node[anchor=north]{$Q_3$};
        \draw (0,6) node[anchor=west, color=gray]{{\LARGE $p^0=p_1+p_2+p_3+p_4$}};
        \draw (0,5) node[anchor=west, color=cyan]{{\LARGE $\cH$}};
        \draw (0,4) node[anchor=west, color=purple]{{\LARGE $\vec\upsilon$}};
        \draw (0,3) node[anchor=west, color=teal]{{\LARGE $\sH $}};
    \end{tikzpicture}}
\end{subfigure}
    \caption{An illustration of discrete blocks $(\vec p, \vec\upsilon, \cH)$, with $m=3$, $U=\{9\}$, and the path $\sH $. These together encode the same information as $\vec r$. Discrete blocks are discrete analogs of blocks (from \Cref{sec:forjm}, and see \Cref{fig:block}).}
    \label{fig:dblock}
\end{figure}

\begin{definition}  \label{defn:discretebp}
A \emph{discrete block process} of $(\hD_{m}^{N_m})^{k_m} \cdots (\hD_{1}^{N_1})^{k_1}$ is a family of functions $\vec p=\{p_j\}_{j=1}^{m+u}$, where $u\in\Z_{\ge 0}$, and each $p_j:\llbracket 0, Q_m\rrbracket\to \Z_{\ge 0}$ is a function satisfying the following:
\begin{itemize}
    \item for each $\ell \in \llbracket 1, m\rrbracket$, $p_j(Q_0)=0$, and $p_\ell=0$ on $\llbracket Q_{\ell-1}+1, Q_m\rrbracket$;
    \item for each $\ell \in \llbracket m+1, m+u\rrbracket$, $p_j(Q_0)=p_j(Q_m)=0$, and $p_j$ is not identical zero;
    \item for each $t\in \llbracket Q_{\ell-1}+1, Q_\ell\rrbracket$, there is at most one $j$ such that $p_j(t)\neq p_j(t-1)$; and we must have $j\in\llbracket \ell+1, N_\ell\rrbracket$;
    \item $\min\{t\in\llbracket 0, Q_m\rrbracket: p_j(t)>0\} < \min\{t\in\llbracket 0, Q_m\rrbracket: p_{j'}(t)>0\}$ for all $m+1\le j < j'\le m+u$.
\end{itemize}
We also denote $p^\ell=\sum_{j=\ell+1}^{m+u} p_j$ for each $\ell\in\llbracket 0,m\rrbracket$.
\end{definition}
Note that (compared to the continuous counterpart defined in \Cref{Definition_block_process}) we do not require the support of each $p_j$ to be continuous. This is for the convenience of definitions in the discrete case, and does not affect the limiting objects: as $N\to\infty$ the total mass of the blocks with discontinuous support tends to $0$.

\begin{definition}
For a discrete block process $\vec p$, a compatible \emph{discrete virtual block process} is a vector $\vec\upsilon=\{\upsilon_\ell\}_{\ell=1}^m\in \prod_{\ell=1}^m (\llbracket 1, Q_\ell-Q_{\ell-1}\rrbracket\cup\{\varnothing\})$ satisfying the following conditions: for each $\ell\in\llbracket 1,m\rrbracket$,
\begin{itemize}
    \item if $\upsilon_\ell\neq\varnothing$, then $p^\ell$ is constant on $\llbracket Q_{\ell-1}, Q_{\ell-1}+\upsilon_\ell\rrbracket$;
    \item if $p_\ell(Q_{\ell-1})=0$, then $\upsilon_\ell=\varnothing$.
\end{itemize}
\end{definition}

\begin{definition}  \label{defn:discretehh}
For a discrete block process $\vec p$ and compatible discrete virtual block process $\vec\upsilon$, a compatible \emph{discrete block height} is a function
\[
\cH:\{Q_\ell\}_{\ell=0}^m\cup\bigcup_{\ell=1}^m \left(\bigcup_{t\in \llbracket Q_{\ell-1}+1, Q_\ell\rrbracket, p^\ell(t)\neq p^\ell(t-1)} \{t-1, t\}\right)\cup\bigcup_{\upsilon_\ell\ne \varnothing}\{Q_{\ell-1}+\upsilon_{\ell}-1, Q_{\ell-1}+\upsilon_{\ell}\} \to \Z_{\ge 0},
\]
such that
\begin{itemize}
    \item $\cH(Q_\ell)=p^0(Q_\ell)$ for each $\ell \in \llbracket 1,m\rrbracket$ and $\cH(0)=0$;
    \item $\cH(Q_{\ell-1}+\upsilon_\ell-1)=p^0(Q_{\ell-1})$ and $\cH(Q_{\ell-1}+\upsilon_\ell)=p^0(Q_{\ell-1})-1$ for each $\ell$ with $\upsilon_\ell\neq\varnothing$\footnote{It might happen that $p^\ell(Q_{\ell-1}+\upsilon_\ell)\neq p^\ell(Q_{\ell-1}+\upsilon_\ell+1)$ or $Q_{\ell-1}+\upsilon_\ell=Q_\ell$, so that there are more than one restrictions at $\cH(Q_{\ell-1}+\upsilon_\ell)$, given by the other conditions in this definition.};
     \item for each $\ell \in\llbracket 1,m\rrbracket$ and $t\in \llbracket Q_{\ell-1}+1, Q_\ell\rrbracket$ with $p_j(t)\neq p_j(t-1)$ for some $j>\ell$, we have $\cH(t-1)=\cH(t)+1$, $\cH(t)\ge p^\ell(t)$, $\cH(t-1)\ge p^\ell(t-1)$;
     \item if $j\in\llbracket 1, N\rrbracket$ and $t\in \llbracket Q_{\ell-1}+1, Q_\ell\rrbracket$ for some $\ell \in\llbracket 1,m\rrbracket$ with $j>\ell$, such that $p_j(t)\neq p_j(t-1)$, then the jump $p_j(t)-p_j(t-1)$ and the difference $\cH(t)+1/2-p^\ell(t-1)-p_j(t)$ have the same sign;
    \item if $\upsilon_\ell=\varnothing$,
and $t\in\llbracket Q_{\ell-1}, Q_\ell\rrbracket$ is the maximum number with $p^\ell$ being constant on $\llbracket Q_{\ell-1}, t\rrbracket$,  then $\cH(t)\ge \cH(Q_{\ell-1})$.
\end{itemize}
\end{definition}
In particular, if $p_j(t)=0$, we must have $\cH(t)=p^\ell(t-1)-1$.

We also note that for each $\ell\in\llbracket 1, m\rrbracket$, $p^\ell(t)=p^0(t)$ for each $t\in\llbracket Q_{\ell-1}+1, Q_\ell\rrbracket$,
while $p^\ell(Q_{\ell-1})=p^0(Q_{\ell-1})-p_\ell(Q_{\ell-1})$.
Therefore $p^\ell$ in the third and fourth point in \Cref{defn:discretehh} is actually in correspondence with $\bp^0$ in (the third and fourth point in) \Cref{Definition_block_height}.

\begin{definition}
``Discrete blocks'' is a triplet $(\vec p, \vec\upsilon, \cH)$: discrete block process, discrete virtual block process, and compatible discrete block height. It is \emph{valid in parity}, if for each $\ell\in\llbracket 1,m\rrbracket$, $Q_\ell+p^\ell(Q_\ell)$ is an even number, and $\upsilon_\ell$ is an odd number unless $\upsilon_\ell=\varnothing$; and for each $t\in \llbracket Q_{\ell-1}+1, Q_\ell\rrbracket$ with $p^\ell(t)\neq p^\ell(t-1)$, $t+\cH(t)$ is an even number.
\end{definition}

\medskip

\noindent\textbf{Connection to walks.}
Below we take $\epsilon$ to be any number with $N^{-1/10}<\epsilon<\Ceps$, for some universal $\Ceps>0$ (and any $C, c$ can depend on $\Ceps$)\footnote{We note that starting from here the lower bound on $\epsilon$ is taken to be larger than in the previous subsection (see e.g.~\Cref{defn:hDeps}), where $\epsilon$ is allowed to be of order $N^{-2/3}$. This is to simplify some computations later in this section.}.
Recall that $\sB_\epsilon^*$ denotes the set of all walks $\vec r$ of the particular operator $(\hD_{m}^{N_m})^{k_m} \cdots (\hD_{1}^{N_1})^{k_1}$, satisfying certain constraints, as specified above \Cref{prop:truncreg}.

For each walk $\vec r\in \sB_\epsilon^*$, we can construct discrete blocks $(\vec p, \vec\upsilon, \cH)$ which are valid in parity, as follows.
For each $\ell \in \llbracket 1, m\rrbracket$, we let $p_\ell=r_\ell$ on $\llbracket 0, Q_{\ell-1}\rrbracket$, and $p_\ell= 0$ on $\llbracket Q_{\ell-1}+1, Q_m\rrbracket$.
We let $\upsilon_\ell = \min\{t\in \llbracket 1, Q_\ell-Q_{\ell-1}\rrbracket: \sH(Q_{\ell-1}+t) < \sH(Q_{\ell-1}), \llbracket Q_{\ell-1}+1, Q_{\ell-1}+t\rrbracket \cap\Delta =\emptyset\}$, and let $\upsilon_\ell=\varnothing$ if the set is empty.

We let $u=|U|$ (recall from \Cref{ssec:classi} that $U$ is the set of $j\in \llbracket m+1,N\rrbracket$ where $r_j$ is not identical zero).
We define $p_j=r_{\sigma(j)}$ for each $j \in \llbracket m+1, m+u\rrbracket$, where $\sigma:\llbracket m+1, m+u\rrbracket \to U$ is a bijection, and is the unique one so that the last bullet point in \Cref{defn:discretebp} is satisfied.

We then let $\cH=\sH$ wherever it should be defined.

\begin{definition}
We let $\sL_\epsilon$ denote the collection of all discrete blocks $(\vec p, \vec\upsilon, \cH)$, that are valid in parity and such that for each $\ell\in \llbracket 1, m\rrbracket$, $p^\ell$ is constant on $\llbracket Q_{\ell-1}, Q_{\ell-1} + \epsilon N^{2/3}\rrbracket$ and either $\upsilon_\ell>\epsilon N^{2/3}$, or $\upsilon_\ell=\varnothing$.
\end{definition}

\begin{lemma}  \label{lem:defcL}
The construction above defines a correspondence $\cL: \vec r\mapsto (\vec p, \vec\upsilon, \cH)$,  which maps $\sB_\epsilon^*$ to $\sL_\epsilon$. 
\end{lemma}
\begin{proof}
Recall the definition of $\sB_\epsilon^*$ from \Cref{ssec:pair}: 
\begin{itemize}
    \item the requirement $J_\rmB=J_\rmC=\emptyset$ (thereby $J_\rmIII=\emptyset$) corresponds to $p_\ell=r_\ell=0$ on the segment $\llbracket Q_{\ell-1}+1, Q_m\rrbracket$ for each $\ell\in \llbracket 1, m\rrbracket$;
    \item as for the requirement that $\vartheta_\ell\wedge \dot\vartheta_\ell>\epsilon N^{2/3}$ for $\ell\in J_\rmI\cup J_\rmII\cup J_\rmIV$,
    it is equivalent to that each $\dot\vartheta_\ell>\epsilon N^{2/3}$ or  $\dot\vartheta_\ell=\varnothing$, and $\vartheta_\ell>\epsilon N^{2/3}$ or  $\vartheta_\ell=\varnothing$;
    the $\dot\vartheta_\ell$ part corresponds to that $p^\ell$ is constant on $\llbracket Q_{\ell-1}, Q_{\ell-1} + \epsilon N^{2/3}\rrbracket$; 
    the $\vartheta_\ell$ part corresponds to $\upsilon_\ell>\epsilon N^{2/3}$ or $\upsilon_\ell=\varnothing$.
\end{itemize}
All the other properties of discrete blocks process $\vec p$ and discrete block height $\cH$ are from properties of walks, in \Cref{defn:wallblock}, and the construction of $p_j$ for $j\in\llbracket m+1, m+u\rrbracket$, which ensures that $p_j$ is not identical zero, and the symmetry breaking in \Cref{defn:discretebp}.
All the other properties of discrete virtual block process $\vec\upsilon$ are ensured by its construction above. 
\end{proof}
The definition of $\cL$ also implies the following description of its inverse $\cL^{-1}$.
\begin{lemma}  \label{lem:defcLrev}
For $(\vec p, \vec\upsilon, \cH)\in \sL_\epsilon$, the pre-image $\cL^{-1}(\vec p, \vec\upsilon, \cH)$ either is empty (due to the requirement of $J_\rmVIa=J_\rmVIb=J_\rmVIc=\emptyset$ for $\sB_\epsilon^*$), or consists of all the $\vec r\in \sB_\epsilon^*$ obtained by specifying $\sH $ to be any function on $\llbracket 0, Q_m\rrbracket$, satisfying
\begin{itemize}
    \item $\sH(Q_\ell)=p^0(Q_\ell)$ for each $\ell\in\llbracket 0,m \rrbracket$, and in particular $\sH(0)=\sH(Q_m)=0$;
    \item $|\sH(t)-\sH(t-1)|=1$ and $\sH(t)\ge p^0(t)$ for each $t\in \llbracket 1, Q_m\rrbracket$;
    \item $\sH(t)=\sH(t-1)-1=\cH(t)$ for each $t\in \llbracket Q_{\ell-1}+1, Q_\ell\rrbracket$, $p^\ell(t)\neq p^\ell(t-1)$, or $t=Q_{\ell-1}+\upsilon_\ell$ when $\upsilon_\ell\neq\varnothing$;
\end{itemize}
and specifying the set $U\subset \llbracket m+1, N\rrbracket$ with $|U|=u$, as well as the bijection $\sigma:\llbracket m+1, m+u\rrbracket \to U$.
\end{lemma}
This follows by checking the definition of walks and $\sB_\epsilon^*$, and of blocks, as in the proof of \Cref{lem:defcL}. We omit the details.

We note that for any $\vec r\in \cL^{-1}(\vec p, \vec\upsilon, \cH)$, the sets $\Delta$, $\Delta_{j,\ell}$, and the numbers $\delta$ and $\delta_{j,\ell}$ (for each $j\in\llbracket 1,N\rrbracket$ and $\ell\in\llbracket 1,m\rrbracket$) are the same; therefore below we regard them as functions of $(\vec p, \vec\upsilon, \cH)$.

\bigskip

We next estimate the sum $\sum_{\vec r\in \cL^{-1}(\vec p, \vec\upsilon, \cH)} w(\vec r)$ for any $(\vec p, \vec\upsilon, \cH)\in \sL_\epsilon$.
For this, we recall the blocks defined in \Cref{sec:forjm}.

\medskip
\noindent\textbf{Closeness between blocks and discrete blocks.}
We use the following notion for the correspondence between blocks and discrete blocks.
\begin{definition}  \label{defn:closen}
Take $\Clo>0$.
For any discrete blocks $(\vec p, \vec\upsilon, \cH)$ of $(\hD_{m}^{N_m})^{k_m} \cdots (\hD_{1}^{N_1})^{k_1}$ which are valid in parity, and any blocks $(\vec \bp, \vec\bb,\bH)$ (with times $\vec \bk\in \R_+^m$, and $\bQ_\ell=\sum_{\ell'=1}^\ell \bk_{\ell'}$ for each $\ell \in \llbracket 0, m\rrbracket$),
they are \emph{$\Clo$-close}, or simply \emph{close} if the following conditions hold.
\begin{itemize}
    \item $u=\bu$.
    \item For each $\ell\in\llbracket 1,m\rrbracket$, we have
$|k_\ell- \bk_\ell N^{2/3}|<\Clo$.
    \item For each $\ell\in\llbracket 1,m\rrbracket$, either $\upsilon_\ell=\bb_\ell=\varnothing$, or
$|\upsilon_\ell-\bb_\ell N^{2/3}|<\Clo$ with $\upsilon_\ell\neq\varnothing$.
    \item For each $\ell \in \llbracket 1,m\rrbracket$ and $j\in \llbracket \ell, m+\bu\rrbracket$, we have $\delta_{j,\ell}=\bdel_{j,\ell}$.
    Besides, if we denote the numbers in $\Delta_{j,\ell}$ and $\bDel_{j,\ell}$ by $t_1<\cdots < t_{\delta_{j,\ell}}$ and $x_1<\cdots<x_{\delta_{j,\ell}}$ respectively, we have for each $i\in\llbracket 1, \delta_{j,\ell}\rrbracket$,
    \[
    |t_i - N^{2/3}x_i|,\quad |p^0(t_i)-N^{1/3}\bp^0(x_i)|, \quad |\cH(t_i) - N^{1/3}\bH(x_i)| < \Clo,
    \]
    and
\[
 (\bp^0(x_i)-\bp^0({x_i}-))(p^0(t_i)-p^0(t_i-1))>0,
\]
    and $p_j=0$ on $\llbracket Q_\ell, Q_m \rrbracket$ if and only if $\bp_j=0$ on $[\bQ_\ell,\bQ_m]$.
\end{itemize}
\end{definition}
We note that when $(\vec p, \vec\upsilon, \cH)$ and $(\vec \bp, \vec\bb,\bH)$ are close, necessarily $\delta=\bdel$.

For the sum $\sum_{\vec r\in \cL^{-1}(\vec p, \vec\upsilon, \cH)} w(\vec r)$, we shall approximate it using blocks.

Recall the setup in \Cref{defn:bIs,def:xibi}.
The continuous counterpart of  the sum $\sum_{\vec r\in \cL^{-1}(\vec p, \vec\upsilon, \cH)} w(\vec r)$ is $\bI_{\beta,\vec\bk}[\vec \bp, \vec\bb,\bH]$, and we will explicitly bound the difference between them (with appropriate scaling), using \Cref{lem:RW1}.
We next define the notations in the error terms.
We let
\[
\bF[\vec\bp,\vec\bb, \bH]=2^{-\bdel}\prod_{(x,y)\in\Xi} \bF[x,y],
\]
where
\begin{equation*}
\bF[x,y]=
\begin{cases}
\bF_{0,0}(y-x), & (x,y) \in \Xi_1, \\
\bF_0(y-x;,\bH(y)-\bH(x)), & (x,y) \in \Xi_2, \\
(-1)^{\don[\bp^0(x)<\bp^0(x-)]}\bF_0(y-x;\bH(x)-\bH(y)) , & (x,y) \in \Xi_3, \\
(-1)^{\don[\bp^0(x)<\bp^0(x-)]} \bF(y-x;\bH(x)-\bp^0(x), \bH(y)-\bp^0(y-) ), & (x,y) \in \Xi_4.
\end{cases}
\end{equation*}
And we further denote
\[
\bEr[\vec \bp, \vec\bb,\bH]= \sum_{(x,y)\in \Xi} \bEr[x,y],
\]
where
\begin{equation*}
\bEr[x,y] =
\begin{cases}
N^{-1/40}, & (x,y)\in \Xi_1, \\
        1\wedge \big(N^{-1/40} +N^{-1/3}/(\bH(y)-\bH(x))  \big),
 & (x,y)\in \Xi_2, \\
\!\begin{aligned}
        1\wedge & \big(N^{-3/5}(y-x)^{-3} +N^{-1/40} 
        +  N^{-1/3}/(\bH(x)-\bH(y))  \big),
    \end{aligned}
& (x,y)\in \Xi_3, \\
\!\begin{aligned}
        1\wedge & \big(N^{-3/5}(y-x)^{-3} +N^{-1/40} +  N^{-1/3}/(\bH(x)-\bp^0(x)) \\
       & +N^{-1/3}/(\bH(y)-\bp^0(y-)) \big),
    \end{aligned}
 & (x,y)\in \Xi_4.
\end{cases}
\end{equation*}
Finally, for $\wes>0$, we say that $(\vec\bp,  \vec\bb,\bH)$ is $\wes$-\emph{well-spaced}, or simply \emph{well-spaced}, if
\begin{itemize}
    \item $y-x>\wes N^{-2/3}$ for each $(x,y)\in\Xi$,
    \item  $\bH(y)-\bH(x) > \wes N^{-1/3}$ for each $(x,y)\in\Xi_2$,
    \item  $\bH(x)-\bH(y) > \wes N^{-1/3}$ for each $(x,y)\in \Xi_3$,
    \item $\bH(x)-\bp^0(x), \bH(y)-\bp^0(y-)>\wes N^{-1/3}$ for each  $(x,y)\in\Xi_4$.
\end{itemize}
The notion of being well-spaced is introduced for the asymptotic approximations: for such blocks $\bI(x;h,g)$, $\bI_0(x;h)$, and $\bI_{0,0}(x)$ are well-approximated by their discrete counterparts, as we later establish in \Cref{lem:RW1}.

We next state an asymptotic approximation for $\sum_{\vec r\in \cL^{-1}(\vec p, \vec\upsilon, \cH)} w(\vec r)$.
\begin{prop}  \label{prop:IXHGest}
Take any constants $C_1, \Clo, \wes>0$.
Take any limiting blocks $(\vec \bp, \vec\bb,\bH)$ that is $\wes$-well-spaced, with $\bdel, \bp^0, \bH<C_1$.
Take any $(\vec p, \vec\upsilon, \cH)\in \sL_\epsilon$, and assume that $(\vec p, \vec\upsilon, \cH)$ and $(\vec \bp, \vec\bb,\bH)$ are $\Clo$-close, and that $\cL^{-1}(\vec p, \vec\upsilon, \cH)$ is non-empty.
Also take $\vec \ttt\in \R_{\ge 0}^m$ such that $\ttt_1\le \cdots \le \ttt_m<C_1$, and $|N_\ell-N+\ttt_\ell N^{2/3}|<C_1$ for each $\ell\in\llbracket 1, m\rrbracket$.
Then there exists $C_2=C_2(C_1, \Clo, \wes)$ such that
\begin{multline*}
\Bigg| 2^{-m-\bdel-|\{\ell\in\llbracket 1,m\rrbracket: \bb_\ell\neq\varnothing\}|}N^{ m-\frac{2}{3}\bu+\frac{4}{3}\bdel+\frac{2}{3}|\{\ell\in\llbracket 1,m\rrbracket: \upsilon_\ell\neq\varnothing\}|  } \prod_{\ell=1}^m (2\sqrt{N_\ell N})^{-k_\ell} \sum_{\vec r \in \cL^{-1}(\vec p, \vec\upsilon, \cH)}  w(\vec r) \\ - \exp\left(\sum_{\ell=1}^{m-1}
(\ttt_\ell-\ttt_{\ell+1}) \bH(\bQ_\ell)/2 \right)\bI_{\beta,\vec\bk}[\vec \bp, \vec\bb,\bH] \Bigg| <
C_2\bF[\vec\bp, \vec\bb,\bH]\bEr[\vec\bp, \vec\bb,\bH].
\end{multline*}
\end{prop}
For the proof, we introduce discrete analogs of the functions $\bI$, $\bI_0$, $\bI_{0,0}$ in \Cref{defn:bIs}.

For $X,H,G\in\Z_{\ge 0}$ with $X+H+G$ even, we let $\sF(X;H,G)$ be the collection of all $F:\llbracket 0,X\rrbracket\to \Z_{\ge 0}$, satisfying $|F(t)-F(t-1)|=1$ for each $t\in\llbracket 1,X\rrbracket$, and $F(0)=H$, $F(X)=G$.
When $H\le G$, we denote $\sF^+(X;H,G)$ to be the subset of $\sF(X;H,G)$, consisting of $F$ with $F\ge H$.
\begin{definition}\label{defn:IXHG}
We denote\[
I(X;H,G)=2^{-X-1}\sum_{F\in \sF(X;H,G)} \prod_{\substack{t\in \llbracket 1, X \rrbracket: \\ F(t)=F(t-1)-1 }}\left(1+\frac{2F(t)}{\beta N}\right).
\]
When $H\le G$, we also denote $I^+(X;H,G)$ to be the same, except for that the summation is over $F\in\sF^+(X;H,G)$.
\end{definition}

\begin{proof}[Proof of \Cref{prop:IXHGest}]
In this proof $C,c>0$ may depend on $C_1, \Clo, \wes$, and we take $C_2$ to be large enough depending on all $C,c$.

For the discrete block $(\vec p, \vec\upsilon, \cH)$, and each $\ell\in\llbracket 1,m\rrbracket$, let $\Upsilon_\ell$ be the set consisting of all maximal intervals $\llbracket X,Y\rrbracket\subset \llbracket Q_{\ell-1}, Q_\ell\rrbracket$, such that $p^\ell$ is constant on $\llbracket X,Y\rrbracket$, and the interval contains at most one of $Q_{\ell-1}+\upsilon_\ell$ and $Q_{\ell-1}+\upsilon_\ell-1$.
Then there is an exact correspondence between discrete intervals in $\bigcup_{\ell=1}^m\Upsilon_\ell$ (constructed by $(\vec p, \vec\upsilon, \cH)$), and continuous intervals $\Xi$ constructed by $(\vec \bp, \vec\bb,\bH)$.
This in particular implies that $\sum_{\ell=1}^m |\Upsilon_\ell|=|\Xi|=\delta+m+|\{\ell\in\llbracket 1,m\rrbracket: \upsilon_\ell\neq\varnothing\}|$.

For each $\llbracket X,Y\rrbracket \in \Upsilon_\ell$, if $X=Q_{\ell-1}$, we let \[\sF[X,Y]=\sF^+(Y-X;\cH(X)-p^\ell(X), \cH(Y)-p^\ell(X)),\] and \[I[X,Y]=I^+(Y-X;\cH(X)-p^\ell(X), \cH(Y)-p^\ell(X));\] otherwise, we let  \[\sF[X,Y]=\sF(Y-X;\cH(X)-p^\ell(X), \cH(Y)-p^\ell(X)),\] and \[I[X,Y]=(-1)^{\don[p^\ell(X)<p^\ell(X-1)]}I(Y-X;\cH(X)-p^\ell(X), \cH(Y)-p^\ell(X)).\]

Using the expression from \Cref{defn:wr}, 
for each $\vec r \in \sB_\epsilon^*$ with $\cL(\vec r)=(\vec p, \vec\upsilon, \cH)$, we have
\begin{equation}   \label{eq:bigmnbound}
\Bigg| \frac{2^{-m}N^\delta\prod_{\ell=1}^m (2\sqrt{N_\ell N})^{-k_\ell}|w(\vec r)|}{2^{-Q_m} \prod_{\ell=1}^m N_\ell^{(\cH(Q_{\ell-1})-\cH(Q_\ell))/2} \prod_{\substack{t\in \llbracket Q_{\ell-1}+1, Q_\ell \rrbracket \\ \sH(t)=\sH(t-1)-1}} \left( 1 + \frac{2r_{i_\ell}(t-1)}{\beta N_\ell} \right)} - 1\Bigg| < C N^{-1/3}.
\end{equation}
Indeed, plugging \eqref{eq:weight} into the ratio in the left-hand side, we would get
\begin{align*}
 &\prod_{\ell=1}^m \Biggl[ (N_\ell/N)^{-\bigl|\{t\in\llbracket Q_{\ell-1}+1, Q_\ell \rrbracket:\, \sH(t)=\sH(t-1)-1,\, r_{i_\ell}(t)\neq r_{i_\ell}(t-1)-1\}\bigr|}\notag \\
&\times \prod_{\substack{t\in \llbracket Q_{\ell-1}+1, Q_\ell \rrbracket \cap \Delta \\ \sH(t)=\sH(t-1)-1}}  \left(1+\frac{2r_{i_\ell}(t-1)}{\beta N_\ell}-\frac{|\{j\in \llbracket 1,N\rrbracket: r_j(t-1)\ge r_{i_\ell}(t-1)\}|}{N_\ell}\right)\left( 1 + \frac{2r_{i_\ell}(t-1)}{\beta N} \right)^{-1}  \\
&\times \prod_{\substack{t\in \llbracket Q_{\ell-1}+1, Q_\ell \rrbracket\setminus \Delta \\ \sH(t)=\sH(t-1)-1}}  \left(1+\frac{2r_{i_\ell}(t-1)}{\beta N_\ell}-\frac{|\{j\in \llbracket 1,N\rrbracket: r_j(t-1)\ge r_{i_\ell}(t-1)\}|}{N_\ell}\right)\Biggr].
\end{align*}
Using that $\bdel=\delta<C_1$, each of the three lines above is between $1-CN^{-1/3}$ and $1+CN^{-1/3}$, thus we get \eqref{eq:bigmnbound}.

On the other hand, from the definition of $I[X,Y]$ and \Cref{lem:defcLrev}, we have
\begin{multline}  \label{eq:defnIxyfrom}
\sum_{\vec r \in \sB_\epsilon^*: \cL(\vec r)=(\vec p, \vec\upsilon, \cH)} 2^{-Q_m} \prod_{\ell=1}^m (-1)^{\bigl|\{t\in\llbracket Q_{\ell-1}+1, Q_\ell \rrbracket:\, \sH(t)=\sH(t-1)-1,\, r_{i_\ell}(t)\ge r_{i_\ell}(t-1)\}\bigr|}\\ \times\prod_{\substack{t\in \llbracket Q_{\ell-1}+1, Q_\ell \rrbracket: \\ \sH(t)=\sH(t-1)-1}} \left( 1 + \frac{2r_{i_\ell}(t-1)}{\beta N} \right)  = P[\vec p,\vec\upsilon, \cH]\prod_{\ell=1}^m \prod_{\llbracket X,Y\rrbracket\in \Upsilon_\ell} I[X,Y],    
\end{multline}
where $P[\vec p,\vec\upsilon, \cH]$ is the number of possible $U$ and $\sigma$ for $\vec r$ with $\cL(\vec r)=(\vec p,\vec\upsilon, \cH)$, and is between $N^u(1-CN^{-1/3})$ and $N^u$.
Therefore from \eqref{eq:defnIxyfrom} and \eqref{eq:bigmnbound}, we have that
\begin{multline}  \label{eq:nNdiff}
  \Bigg| 2^{-m}N^\delta\prod_{\ell=1}^m (2\sqrt{N_\ell N})^{-k_\ell} \sum_{\vec r \in \sB_\epsilon^*: \cL(\vec r)=(\vec p, \vec\upsilon, \cH)}  w(\vec r)  -  P[\vec p,\vec\upsilon, \cH]\prod_{\ell=1}^m N_\ell^{(\cH(Q_{\ell-1})-\cH(Q_\ell))/2} \prod_{\llbracket X,Y\rrbracket\in \Upsilon_\ell} I[X,Y]\Bigg|\\ < CN^{-1/3}2^{-Q_m}|\{\vec r \in \sB_\epsilon^*: \cL(\vec r)=(\vec p, \vec\upsilon, \cH)\}|.
\end{multline}
For the right-hand side, it equals
\begin{multline*}
CN^{-1/3}2^{-Q_m} P[\vec p,\vec\upsilon, \cH] \prod_{\ell=1}^m \prod_{\llbracket X,Y\rrbracket\in \Upsilon_\ell} |\sF[X,Y]| <
CN^{-1/3} \cdot N^{u-|\Xi_1| - \frac{2}{3}(|\Xi_2| + |\Xi_3|) - \frac{1}{3}|\Xi_4|} \bF[\vec \bp, \vec\bb,\bH]\\ = CN^{-1/3}\cdot N^{-m+\frac{2}{3}\bu-\frac{1}{3}\bdel-\frac{2}{3}|\{\ell\in\llbracket 1,m\rrbracket: \upsilon_\ell\neq\varnothing\}|}\bF[\vec \bp, \vec\bb,\bH].
\end{multline*}
Here the inequality is due to replacing each $|\sF[X,Y]|$ with $\bF[x,y]$, for $(x,y)\in \Xi$ being the interval in correspondence with $\llbracket X, Y\rrbracket$, using Lemmas \ref{lem:A1}, \ref{lem:sFog}, and \ref{lem:sfoo}. 
The last equality is due to $u=\bu$, and
\begin{equation}  \label{eq:mXi}
 |\Xi_1| + \frac{2}{3}(|\Xi_2| + |\Xi_3|) + \frac{1}{3}|\Xi_4|=   m+\frac{1}{3}\bu+\frac{1}{3}\bdel+\frac{2}{3}|\{\ell\in\llbracket 1,m\rrbracket: \upsilon_\ell\neq\varnothing\}|,
\end{equation}
which follows from $m=|\Xi_1|+|\Xi_2|$,  $m+\bu+|\{\ell\in\llbracket 1,m\rrbracket: \upsilon_\ell\neq\varnothing\}|=|\Xi_1|+|\Xi_3|$, and $\bdel+m+|\{\ell\in\llbracket 1,m\rrbracket: \upsilon_\ell\neq\varnothing\}|=|\Xi|$.

As for the left-hand side of \eqref{eq:nNdiff}, 
for each $\ell\in\llbracket 1,m\rrbracket$,
\[
|\log(N_\ell/N)N^{1/3} + \ttt_\ell| < CN^{-1/3},\quad |N^{-1/3}(\cH(Q_{\ell-1})-\cH(Q_\ell)) - \ttt_\ell (\bH(\bQ_{\ell-1})-\bH(\bQ_\ell))| < CN^{-1/3},
\]
using that $|N_\ell-N+\ttt_\ell N^{2/3}|<C_1$, and the fact $(\vec p, \vec\upsilon, \cH)$ and $(\vec \bp, \vec\bb,\bH)$ are $\Clo$-close.
Thus
\[
|\log(N_\ell / N) (\cH(Q_{\ell-1})-\cH(Q_\ell))/2 + \ttt_\ell (\bH(\bQ_{\ell-1})-\bH(\bQ_\ell))/2 | < CN^{-1/3},
\]
Then by summing over $\ell$ and taking the exponential, 
and using that $1-CN^{-1/3}< P[\vec p,\vec\upsilon, \cH]N^{-u} \le 1$, we get
\begin{equation}   \label{eq:Pnexp}
\left|P[\vec p,\vec\upsilon, \cH]N^{-u}\prod_{\ell=1}^m N_\ell^{(\cH(Q_{\ell-1})-\cH(Q_\ell))/2} - \exp\left(\sum_{\ell=1}^{m-1}
(\ttt_\ell-\ttt_{\ell+1}) \bH(\bQ_\ell)/2 \right)\right|<CN^{-1/3}.
\end{equation}
And for each $\ell\in\llbracket 1,m\rrbracket$ and $\llbracket X,Y\rrbracket \in \Upsilon_\ell$, there is a corresponding $(x,y)\in \Xi$, such that $|X-N^{2/3}x|, |Y-N^{2/3}y|<C$.
Moreover, we have $|\cH(X)-p^\ell(X) - N^{1/3}(\bH(x)-\bp^0(x))|<C$. And if $(x,y) \in \Xi_1\cup\Xi_3$, we have $\cH(Y)=p^\ell(Y)$; and for other $(x,y)$ we have $|\cH(Y)-p^\ell(Y) - N^{1/3}(\bH(y)-\bp^0(y-))|<C$.
Therefore, by \Cref{lem:RW1} below (where the conditions of $x>C_1N^{-2/3}$, and $h, g>C_1N^{-1/3}$ or $h>C_1N^{-1/3}$ or $g-h>C_1N^{-1/3}$, are satisfied by the well-spaced assumption) and $\epsilon>N^{-1/10}$,
we have
\begin{align*}
    |NI[X,Y] - \bI_{\beta}[x,y]| < C\bEr[x,y]\bF[x,y], & \quad  (x,y)\in \Xi_1, \\
    |N^{2/3}I[X,Y] - \bI_{\beta}[x,y]| < C\bEr[x,y]\bF[x,y], & \quad  (x,y)\in \Xi_2\cup \Xi_3, \\
    |N^{1/3}I[X,Y] - \bI_{\beta}[x,y]| < C\bEr[x,y]\bF[x,y], &  \quad (x,y)\in \Xi_4.
\end{align*}
The conclusion follows by combining the above estimates with \eqref{eq:defnibk} and \eqref{eq:nNdiff}, and using \eqref{eq:mXi} again.
\end{proof}

\begin{lemma}  \label{lem:RW1}
For any $C_1, C_2>0$, there is a constant $C_3=C_3(C_1,C_2)>0$, such that for any $C_1N^{-2/3}<x<C_2$ and $0\le h, g <C_2$, and $X, H, G\in \Z_{\ge 0}$, such that $|X-xN^{2/3}|, |H-hN^{1/3}|, |G-gN^{1/3}|<C_2$, the following estimates hold:
\begin{enumerate}
    \item[(1)] When $X+H+G$ is even and $h,g>C_1N^{-1/3}$, we have
\begin{multline*}
\left|N^{1/3}I(X;H,G)-\bI(x;h,g)\right| \\ <C_3\bF(x;h,g)\left(1\wedge (N^{-3/5}x^{-3} + N^{-1/40} + N^{-1/3}/h+N^{-1/3}/g)\right).
\end{multline*}
    \item[(2)] When $X+H$ is even and $h>C_1N^{-1/3}$, we have
\[
\left|N^{2/3}I(X;H,0)-\bI_0(x;h)\right| < C_3\bF_0(x;h)\left(1\wedge (N^{-3/5}x^{-3} + N^{-1/40} +N^{-1/3}/g )\right).
\]
    \item[(3)] When $X+H+G$ is even, $H\le G$, and $g-h>C_1N^{-1/3}$, we have
\begin{multline*}
\left|N^{2/3}I^+(X;H,G)-\exp(-\beta^{-1} xh )\bI_0(x;g-h)\right| \\ <C\bF_0(x;g-h)\left(1\wedge (N^{-3/5}x^{-3} + N^{-1/40} +N^{-1/3}/g)\right).
\end{multline*}
    \item[(4)] When $X$ is even, we have
\[
\left|NI^+(X;H,H)-\exp(-\beta^{-1} xh)\bI_{0,0}(x)\right| < C_3\bF_{0,0}(x)\left(1\wedge (N^{-3/5}x^{-3} + N^{-1/40} )\right).
\]
\end{enumerate}
\end{lemma}
The proof of this lemma is by coupling Bernoulli random walks with Brownian bridges. We postpone the details of the proof until \Cref{ssec:couplgwbb}.

\begin{remark}  \label{rem:lowbdxhg}
We note that if $C_2>C_1$, then the required lower bounds on $x$, $h$, $g$, $g-h$, do not lead to any lower bounds on $X$, $H$, $G$, $G-H$. This is essential when \Cref{lem:RW1} is applied in the proof of \Cref{prop:IXHGest}, and then implicitly used in the next subsection (when upper-bounding \eqref{eq:Asum} there).
\end{remark}

\subsection{Summation over discrete blocks and proof of Proposition \ref{prop:conv}}  \label{ssec:sob}

In the previous subsection we computed the asymptotics of the sum of the walks corresponding to a fixed discrete blocks structure $(\vec p, \vec\upsilon, \cH)$. In this subsection we sum over all discrete blocks in $\sL_\epsilon$, thus arriving at the sum over all walks in $\sB_\epsilon^*$. We will show that the sum approximates an integral over the space of all blocks $(\vec \bp, \vec\bb,\bH)$ as $N\to\infty$. As in the previous subsection, we shall use notations of blocks and their weights defined in \Cref{sec:forjm}.

Due to parity issues in counting discrete walks (which reflects in the necessity to add two terms of different parities in the product in Proposition \ref{prop:conv}), we work under the following setup.
Take any $\vec\fb=\{\fb_\ell\}_{\ell=1}^m\in \{0,1\}^m$ and let $\sB_\epsilon^*[\vec \fb]$ and $\sL_\epsilon[\vec \fb]$ (respectively) denote the corresponding set of walks and blocks of the operator $(\hD_{m}^{N_m})^{k_m+\fb_m} \cdots (\hD_{1}^{N_1})^{k_1+\fb_1}$.

We take $\vec \bk\in \R_+^m$ and denote $\bQ_\ell=\sum_{\ell'=1}^\ell \bk_{\ell'}$ for each $\ell \in \llbracket 0, m\rrbracket$, and take $\vec \ttt\in \R_{\ge 0}^m$ such that $\ttt_1\le \cdots \le \ttt_m$.
The main results of this subsection are the following two statements, including justifying \Cref{defn:core}, plus \Cref{prop:conv}. 
\begin{prop}    \label{prop:fixedepsconv}
Take $C_1>0$, and assume that $|k_\ell-N^{2/3}\bk_\ell|<C_1$, $|N_\ell-N+\ttt_\ell N^{2/3}|<C_1$ for each $\ell\in\llbracket 1, m\rrbracket$.
Then
\begin{multline}  \label{eq:limieq}
\lim_{N\to\infty} 2^{-m}N^m \sum_{\vec\fb\in \{0,1\}^m} \prod_{\ell=1}^m (2\sqrt{N_\ell N})^{-k_\ell-\fb_\ell}  \sum_{\vec r \in \sB_\epsilon^*}  w(\vec r)
\\
=\lim_{C_2\to\infty}\int_{\substack{(\vec\bp,  \vec\bb,\bH)\in\sK_\epsilon[\vec\bk]:  \\ \bp^0, \bH, \bdel < C_2 }} \exp\left(\sum_{\ell=1}^{m-1}
(\ttt_\ell-\ttt_{\ell+1}) \bH(\bQ_\ell)/2 \right)\bI_{\beta,\vec\bk}[\vec \bp, \vec\bb,\bH] \d(\vec\bp,  \vec\bb,\bH),
\end{multline}
and
\begin{multline}  \label{eq:limieqabs}
\lim_{N\to\infty} 2^{-m}N^m \sum_{\vec\fb\in \{0,1\}^m} \prod_{\ell=1}^m (2\sqrt{N_\ell N})^{-k_\ell-\fb_\ell}  \sum_{\vec r \in \sB_\epsilon^*}  |w(\vec r)|
\\
=\lim_{C_2\to\infty}\int_{\substack{(\vec\bp,  \vec\bb,\bH)\in\sK_\epsilon[\vec\bk]:  \\ \bp^0, \bH, \bdel < C_2 }} \exp\left(\sum_{\ell=1}^{m-1}
(\ttt_\ell-\ttt_{\ell+1}) \bH(\bQ_\ell)/2 \right)|\bI_{\beta,\vec\bk}[\vec \bp, \vec\bb,\bH]|\d(\vec\bp,  \vec\bb,\bH),
\end{multline}
and all the limits above exist and are finite.
\end{prop}
In particular, the existence and finiteness of the limit in the right-hand side of \eqref{eq:limieqabs} implies that $(\vec \bp, \vec\bb,\bH)\mapsto \exp\left(\sum_{\ell=1}^{m-1}
(\ttt_\ell-\ttt_{\ell+1}) \bH(\bQ_\ell)/2 \right)\bI_{\beta,\vec\bk}[\vec \bp, \vec\bb,\bH]$ is absolutely integrable in $\sK_\epsilon[\vec\bk]$. 

The next statement is on sending $\epsilon\to 0+$.
\begin{prop}  \label{prop:epszerocov}
$\bL_\beta(\vec\bk, \vec\ttt)$ from \Cref{defn:core} is well-defined, i.e.,
\begin{equation}   \label{eq:epzerlimh}
\lim_{\epsilon\to 0+} \int_{(\vec\bp,  \vec\bb,\bH)\in\sK_\epsilon[\vec\bk]} \exp\left(\sum_{\ell=1}^{m-1}
(\ttt_\ell-\ttt_{\ell+1}) \bH(\bQ_\ell)/2 \right)\bI_{\beta,\vec\bk}[\vec \bp, \vec\bb,\bH]\d(\vec\bp,  \vec\bb,\bH)    
\end{equation}
exists and is finite.
\end{prop}

Below we prove these two statements and \Cref{prop:conv}. The key step is the following pre-limiting estimate.
\begin{prop}   \label{prop:bdsumblo}
Take $C_1, C_2>0$.
Assume that $|k_\ell-N^{2/3}\bk_\ell|<C_1$ and $|N_\ell-N+\ttt_\ell N^{2/3}|<C_1$ for each $N$ and $\ell\in\llbracket 1, m\rrbracket$.
Then there exists $C_3=C_3(C_1, C_2)$ such that
\begin{multline}   \label{eq:bdsumblo}
\Bigg| 2^{-m}N^m \sum_{\vec\fb\in \{0,1\}^m} \prod_{\ell=1}^m (2\sqrt{N_\ell N})^{-k_\ell-\fb_\ell} \sum_{\substack{(\vec p,\vec\upsilon,\cH)\in \sL_\epsilon[\vec \fb]:\\ p^0,\cH<C_2 N^{1/3},\delta<C_2}} \sum_{\vec r \in \cL^{-1}(\vec p, \vec\upsilon, \cH)}  w(\vec r) \\ - \int_{\substack{(\vec\bp,  \vec\bb,\bH)\in\sK_\epsilon[\vec\bk]:  \\ \bp^0, \bH, \bdel < C_2 }} \exp\left(\sum_{\ell=1}^{m-1}
(\ttt_\ell-\ttt_{\ell+1}) \bH(\bQ_\ell)/2 \right)\bI_{\beta,\vec\bk}[\vec \bp, \vec\bb,\bH] \d(\vec\bp,  \vec\bb,\bH)\Bigg| <
C_3\epsilon^{-m/2}N^{-1/40}.
\end{multline}
The same bound holds in absolute values, i.e., if we simultaneously replace $w(r)$ with $|w(r)|$ and $\bI_{\beta,\vec\bk}[\vec\bp,\vec\bb,\bH]$ with $|\bI_{\beta,\vec\bk}[\vec\bp,\vec\bb,\bH]|$.
\end{prop}

\begin{proof}
In this proof $C,c>0$ may depend on $C_1, C_2$, and we take $C_3$ to be large enough depending on all $C,c$.

The general idea is to apply \Cref{prop:IXHGest} to blocks and discrete blocks that are close (in the sense of \Cref{defn:closen}), and sum over discrete blocks and integrate over blocks simultaneously.
To implement this, we need the following `correspondence relation', which is more precise than `closeness' in \Cref{defn:closen}.

\medskip
\noindent\textit{Step 1: correspondence.}
We say that any $(\vec\bp,  \vec\bb,\bH)\in\sK_\epsilon=\sK_\epsilon[\vec\bk]$ and $(\vec p, \vec\upsilon, \cH)\in\bigcup_{\vec\fb\in \{0,1\}^m}\sL_\epsilon[\vec\fb]$ are \emph{in correspondence}, if the following holds:
\begin{itemize}
    \item $u=\bu$.
    \item For each $\ell\in \llbracket 1,m\rrbracket$, either $\upsilon_\ell=\bb_\ell=\varnothing$, or $\upsilon_\ell - \lfloor N^{2/3}\bb_\ell\rfloor \in \{-2, -1\}$.
    \item For each $\ell \in \llbracket 1,m\rrbracket$ and $j\in \llbracket \ell, m+\bu\rrbracket$,
we have $\delta_{j,\ell}=\bdel_{j,\ell}$.
    Besides, if we denote the numbers in the sets $\Delta_{j,\ell}$ and $\bDel_{j,\ell}$ by $t_1<\cdots < t_{\delta_{j,\ell}}$ and $x_1<\cdots<x_{\delta_{j,\ell}}$ respectively, we have for each $i\in\llbracket 1, \delta_{j,\ell}\rrbracket$,
    \[
     t_i - \lfloor N^{2/3}x_i\rfloor \in \{0, 1\}, \quad p^0(t_i)=\lfloor N^{1/3}\bp^0(x_i)\rfloor, \quad \cH(t_i)-\lfloor N^{1/3}\bH(x_i) \rfloor\in \{0, -1\}.    \]
\end{itemize}
We note that for $\bb_\ell$, $t_i$, and $\cH(t_i)$ above, it seems that each is allowed to take one of two values. However, each of them can actually just take one of the two possible values, as the parity for each is already determined.

We next count the number of discrete blocks that are in correspondence to one blocks structure, and consider the measure of all blocks that are in correspondence to one discrete blocks.

\medskip
\noindent\textit{Blocks to discrete blocks.}
We let $\sK_\epsilon^*=\sK_{\epsilon,N}^*$ be the subset of $\sK_\epsilon$, consisting of all $(\vec\bp,  \vec\bb,\bH)$ such that
\begin{itemize}
    \item Any two (finite) numbers in $\bDel\cup\{\bQ_\ell\}_{\ell=0}^m\cup\{\bQ_{\ell-1}+\bb_\ell\}_{\ell=1}^m$ are at least $10N^{-2/3}$ away.
    \item For each $j\in\llbracket 1,m+\bu\rrbracket$, the numbers $\bp_j(x)$ for $x\in\bigcup_{\ell=1}^m\bDel_{j,\ell}$ are mutually different, and any two are at least $10N^{-1/3}$ away.
    In addition, when $j\in\llbracket 1,m\rrbracket$, any one of these numbers is $\ge 10N^{-1/3}$.
    \item For each $\ell\in \llbracket 1,m\rrbracket$, $j\in \llbracket 1, m+\bu\rrbracket$, and $x\in \bDel_{j,\ell}$, unless $\bp_j=0$ on $[x,\bQ_m]$ (i.e., $j\in\llbracket m+1, m+\bu\rrbracket$ and $x=\max\bigcup_{\ell=1}^m \bDel_{j,\ell}$), we must have $\bH(x)\ge \max\{\bp^0(x), \bp^0(x-)\}+10N^{-1/3}$, and $|\bp_j(x)-\bH(x)+\bp^0(x-)|\ge 10N^{-1/3}$.
\end{itemize}
Then any $(\vec\bp,  \vec\bb,\bH)\in\sK_\epsilon^*$ is necessarily $\wes$-well-spaced for some $\wes$ depending only on $C_1$ and $C_2$.

For a fixed $(\vec\bp,  \vec\bb,\bH)\in\sK_\epsilon^*$ with $\bp^0, \bH <C_2-10N^{-1/3}, \bdel < C_2$, we consider all $(\vec p, \vec\upsilon, \cH)\in\bigcup_{\vec\fb\in \{0,1\}^m}\sL_\epsilon[\vec\fb]$ in correspondence with $(\vec\bp,  \vec\bb,\bH)$, satisfying $p^0,\cH<C_2 N^{1/3}$ and $\delta=\bdel<C_2$.
The number\footnote{This is due to that, for each $\ell\in\llbracket 1,m\rrbracket$, $j\in\llbracket \ell, m+\bu\rrbracket$, and $x_i\in \bDel_{j,\ell}$,
unless $\bp_j=0$ on $[x_i,\bQ_m]$ (i.e., $j\in\llbracket m+1, m+\bu\rrbracket$ and $x_i=\max\bigcup_{\ell=1}^m\bDel_{j,\ell}$), there are two choices of the corresponding $\cH(t_i)$.
When $\bp_j=0$ on $[x_i,\bQ_m]$, there is only one choice since necessarily $\cH(t_i)=\lfloor N^{1/3}\bH(x_i) \rfloor-1$. From these, the corresponding $t_i$, as well as $\vec\upsilon$ and $\vec\fb$ are uniquely determined, to ensure that $(\vec p, \vec\upsilon, \cH)$ is valid in parity. 
} of such $(\vec p, \vec\upsilon, \cH)$ is $2^{\bdel-\bu}$.
More generally, for any $(\vec\bp,  \vec\bb,\bH)\in\sK_\epsilon$, the number of $(\vec p, \vec\upsilon, \cH)$ in correspondence with it is $\le 2^{\bdel-\bu}$.

\medskip

\noindent\textit{Discrete blocks to blocks.} On the other hand, for each $\vec\fb\in \{0,1\}^m$, we let $\sL_\epsilon^*[\vec\fb]$ be the subset of $\sL_\epsilon[\vec\fb]$, consisting of all $(\vec p, \vec\upsilon, \cH)$ such that
\begin{itemize}
    \item The (finite) numbers in $\Delta$, $\{Q_\ell+\sum_{\ell'=1}^\ell\fb_{\ell'}\}_{\ell=0}^m$, and $\{Q_{\ell-1}+\upsilon_\ell+\sum_{\ell'=1}^{\ell-1}\fb_{\ell'}\}_{\ell=1}^m$ are mutually different, and any two differ by at least $10$.
    \item For each $\ell\in \llbracket 1,m\rrbracket$, $j\in \llbracket 1, N\rrbracket$, and $t\in \Delta_{j,\ell}$, unless $p_j=0$ on $\llbracket t,Q_m+\sum_{\ell=1}^m\fb_\ell\rrbracket$ (i.e., $j>m$ and $t=\max\bigcup_{\ell=1}^m \Delta_{j,\ell}$), we must have $\cH(t)\ge \max\{p^0(t), p^0(t-1)\}+10$, and $|p_j(t)-\cH(t)+p^0(t-1)|\ge 10$.
\end{itemize}
Then $\cL^{-1}(\vec p, \vec\upsilon, \cH)$ is non-empty for any $(\vec p, \vec\upsilon, \cH)\in\sL_\epsilon^*[\vec\fb]$.

For a fixed $(\vec p, \vec\upsilon, \cH)\in\bigcup_{\vec\fb\in \{0,1\}^m}\sL_\epsilon^*[\vec\fb]$ with $p^0, \cH < C_2N^{1/3}-10$ and $\delta<C_2$, we consider the set of all $(\vec\bp,  \vec\bb,\bH)\in\sK_\epsilon$ that are in correspondence with $(\vec p, \vec\upsilon, \cH)$.
In this case, any $(\vec\bp,  \vec\bb,\bH)\in\sK_\epsilon$ in this set satisfies that $\bp^0, \bH, \bdel < C_2$, and the measure of the set would equal
\begin{multline*}
(2N^{-2/3})^{|\{\ell\in\llbracket 1,m\rrbracket: \upsilon_\ell\neq\varnothing\}|} (2N^{-2/3})^{\delta} (N^{-1/3})^{\delta-|U|} (2N^{-1/3})^{\delta-|U|} \\ = N^{-\frac{4}{3}\delta + \frac{2}{3}|U|-\frac{2}{3}|\{\ell\in\llbracket 1,m\rrbracket: \upsilon_\ell\neq\varnothing\}|} 2^{2\delta-|U|+|\{\ell\in\llbracket 1,m\rrbracket: \upsilon_\ell\neq\varnothing\}|}.
\end{multline*}
More generally, for any $(\vec p, \vec\upsilon, \cH)\in\bigcup_{\vec\fb\in \{0,1\}^m}\sL_\epsilon^*[\vec\fb]$, the measure of $(\vec\bp,  \vec\bb,\bH)$ in correspondence with it is $\le N^{-\frac{4}{3}\delta + \frac{2}{3}|U|-\frac{2}{3}|\{\ell\in\llbracket 1,m\rrbracket: \upsilon_\ell\neq\varnothing\}|} 2^{2\delta-|U|+|\{\ell\in\llbracket 1,m\rrbracket: \upsilon_\ell\neq\varnothing\}|}$.

\medskip
\noindent\textit{Step 2: summation over pairs.}
We now consider all pairs of $(\vec p, \vec\upsilon, \cH)$ and $(\vec\bp,  \vec\bb,\bH)$, such that they are in correspondence, and either
\begin{itemize}
    \item $(\vec\bp,  \vec\bb,\bH)\in\sK_\epsilon^*$ with $\bp^0, \bH <C_2-10N^{-1/3}, \bdel < C_2$, or
    \item $(\vec p, \vec\upsilon, \cH)\in\bigcup_{\vec\fb\in \{0,1\}^m}\sL_\epsilon^*[\vec\fb]$ with $p^0, \cH < C_2N^{1/3}-10$ and $\delta<C_2$.
\end{itemize}
We call any such $(\vec p, \vec\upsilon, \cH)$ and $(\vec\bp,  \vec\bb,\bH)$ an \emph{applicable} pair.
We also put a measure on the space of all applicable pairs, given by the product of the counting measure on discrete blocks, and the measure on blocks.
Then for any $(\vec\bp,  \vec\bb,\bH)$, the number of applicable pairs containing it is at most $2^{\bdel-\bu}$;
and for any $(\vec p, \vec\upsilon, \cH)$, the measure of applicable pairs containing it is at most $N^{-\frac{4}{3}\delta + \frac{2}{3}|U|-\frac{2}{3}|\{\ell\in\llbracket 1,m\rrbracket: \upsilon_\ell\neq\varnothing\}|} 2^{2\delta-|U|+|\{\ell\in\llbracket 1,m\rrbracket: \upsilon_\ell\neq\varnothing\}|}$.

We then apply \Cref{prop:IXHGest} to all applicable pairs, integrate against the measures, and multiply by $2^{-\bdel+\bu}$.
We get
\begin{multline}   \label{eq:aAbd}
\Bigg| 2^{-m}N^m \sum_{\vec\fb\in \{0,1\}^m} \prod_{\ell=1}^m (2\sqrt{N_\ell N})^{-k_\ell-\fb_\ell} \sum_{\substack{(\vec p,\vec\upsilon,\cH)\in \sL_\epsilon[\vec \fb]:\\ p^0,\cH<C_2 N^{1/3},\delta<C_2}}
A(\vec p, \vec\upsilon, \cH) \sum_{\vec r \in \cL^{-1}(\vec p, \vec\upsilon, \cH)}  w(\vec r) \\ - \int_{\substack{(\vec\bp,  \vec\bb,\bH)\in\sK_\epsilon:  \\ \bp^0, \bH, \bdel < C_2 }}
\bA(\vec\bp,  \vec\bb,\bH) \exp\left(\sum_{\ell=1}^{m-1}
(\ttt_\ell-\ttt_{\ell+1}) \bH(\bQ_\ell)/2 \right)\bI_{\beta,\vec\bk}[\vec\bp,\vec\bb,\bH] \d(\vec\bp,  \vec\bb,\bH)\Bigg| \\ <
C \int_{\substack{(\vec\bp,  \vec\bb,\bH)\in\sK_\epsilon:  \\ \bp^0, \bH, \bdel < C_2 }}  \bF[\vec\bp,\vec\bb,\bH]\bEr[\vec\bp,\vec\bb,\bH]\d((\vec\bp,  \vec\bb,\bH),
\end{multline}
where $N^{-\frac{4}{3}\delta + \frac{2}{3}|U|-\frac{2}{3}|\{\ell\in\llbracket 1,m\rrbracket: \upsilon_\ell\neq\varnothing\}|} 2^{2\delta-|U|+|\{\ell\in\llbracket 1,m\rrbracket: \upsilon_\ell\neq\varnothing\}|} A(\vec p, \vec\upsilon, \cH)$ is the measure of applicable pairs for given $(\vec p, \vec\upsilon, \cH)$, and $2^{\bdel-\bu}\bA(\vec\bp,  \vec\bb,\bH)$ is the number of applicable pairs for given $(\vec\bp,  \vec\bb,\bH)$.
Both $\bA(\vec\bp,  \vec\bb,\bH)$ and $A(\vec p, \vec\upsilon, \cH)$ are always in $[0, 1]$; and
\begin{itemize}
    \item $\bA(\vec \bp, \vec\bb,\bH)=1$ for $(\vec\bp,  \vec\bb,\bH)\in\sK_\epsilon^*$ with $\bp^0, \bH <C_2-10N^{-1/3}, \bdel < C_2$,
    \item  $A(\vec p, \vec\upsilon, \cH)=1$ for $(\vec p, \vec\upsilon, \cH)\in\bigcup_{\vec\fb\in \{0,1\}^m}\sL_\epsilon^*[\vec\fb]$ with $p^0, \cH < C_2N^{1/3}-10$ and $\delta<C_2$.
\end{itemize}
We next deduce the conclusion from \eqref{eq:aAbd}.

\medskip
\noindent\textit{Step 3: remaining terms in \eqref{eq:bdsumblo}.}
We consider the difference between the left-hand side of \eqref{eq:bdsumblo} and the left-hand side of \eqref{eq:aAbd}.
It is bounded by the sum of
\begin{equation}  \label{eq:Asum}
 2^{-m}N^m \sum_{\vec\fb\in \{0,1\}^m} \prod_{\ell=1}^m (2\sqrt{N_\ell N})^{-k_\ell-\fb_\ell} \sum_{\substack{(\vec p,\vec\upsilon,\cH)\in \sL_\epsilon[\vec \fb]:\\ p^0,\cH<C_2 N^{1/3},\delta<C_2 \\ A(\vec p, \vec\upsilon, \cH)<1}}
 \Bigg| \sum_{\vec r \in \cL^{-1}(\vec p, \vec\upsilon, \cH)}  w(\vec r) \Bigg|,
\end{equation}
and
\begin{equation}  \label{eq:Asumb}
\int_{\substack{(\vec\bp,  \vec\bb,\bH)\in\sK_\epsilon:  \\ \bp^0, \bH, \bdel < C_2 \\ \bA(\vec\bp,  \vec\bb,\bH)<1 }}
 \exp\left(\sum_{\ell=1}^{m-1}
(\ttt_\ell-\ttt_{\ell+1}) \bH(\bQ_\ell)/2 \right)|\bI_{\beta,\vec\bk}[\vec\bp,\vec\bb,\bH]| \d(\vec\bp,  \vec\bb,\bH).
\end{equation}
We first consider \eqref{eq:Asum}, for which we still use \Cref{prop:IXHGest} and \Cref{rem:lowbdxhg}, and compare the sum on discrete blocks with well-spaced blocks.
For $(\vec p, \vec\upsilon, \cH)\in \bigcup_{\vec\fb\in \{0,1\}^m}\sL_\epsilon[\vec \fb]$ with $p^0,\cH<C_2 N^{1/3},\delta<C_2$ and $A(\vec p, \vec\upsilon, \cH)<1$ (necessarily $(\vec p, \vec\upsilon, \cH)\not\in\bigcup_{\vec\fb\in \{0,1\}^m}\sL_\epsilon^*[\vec \fb]$, or $\max p^0\ge C_2N^{1/3}-10$, or $\max\cH\ge C_2N^{1/3}-10$), and $\cL^{-1}(\vec p, \vec\upsilon, \cH)$ non-empty, we can still find a subset of $\sK_\epsilon$, where each $(\vec\bp,  \vec\bb,\bH)$ satisfies
\begin{itemize}
    \item $\bp^0, \bH, \bdel < C_2$, but either $(\vec\bp,  \vec\bb,\bH)\not\in\sK_\epsilon^*$ or $\max\bp^0\ge C_2-10N^{-1/3}$, or $\max\bH\ge C_2-10N^{-1/3}$;
    \item $(\vec\bp,  \vec\bb,\bH)$ is $\wes$-well-spaced for $\wes$ depending on $C_1$ and $C_2$;
    \item $(\vec p, \vec\upsilon, \cH)$ and $(\vec\bp,  \vec\bb,\bH)$ are $\Clo$-close, for $\Clo$ depending on $C_1$, $C_2$, and $\wes$;
\end{itemize}
and the set has measure $>cN^{-\frac{4}{3}\delta + \frac{2}{3}|U|-\frac{2}{3}|\{\ell\in\llbracket 1,m\rrbracket: \upsilon_\ell\neq\varnothing\}|}$.
We can then apply \Cref{prop:IXHGest} to such $(\vec p, \vec\upsilon, \cH)$ and a uniformly chosen $(\vec\bp,  \vec\bb,\bH)$ in this set, and conclude that \eqref{eq:Asum} is bounded by
\begin{multline}  \label{eq:Asum2}
C \left(\int_{\substack{(\vec\bp,  \vec\bb,\bH)\in\sK_\epsilon:  \\ \bp^0, \bH, \bdel < C_2 }}  |\bI_{\beta,\vec\bk}[\vec\bp,\vec\bb,\bH]|\d(\vec\bp,  \vec\bb,\bH)
-\int_{\substack{(\vec\bp,  \vec\bb,\bH)\in\sK_\epsilon^*:  \\ \bp^0, \bH<C_2-10N^{-1/3}, \bdel < C_2 }}  |\bI_{\beta,\vec\bk}[\vec\bp,\vec\bb,\bH]|\d(\vec\bp,  \vec\bb,\bH)
\right)\\
+C \int_{\substack{(\vec\bp,  \vec\bb,\bH)\in\sK_\epsilon:  \\ \bp^0, \bH, \bdel < C_2 }}  \bF[\vec\bp,\vec\bb,\bH]\bEr[\vec\bp,\vec\bb,\bH]\d(\vec\bp,  \vec\bb,\bH).
\end{multline}
Note that here the second line (i.e., the error term) is the same as the right-hand side of \eqref{eq:aAbd}.
Then using that $|\bI_{\beta,\vec\bk}[\vec\bp,\vec\bb,\bH]|<C\bF[\vec\bp,\vec\bb,\bH]$ whenever $\bp^0<C_2$, we can further bound the first line of \eqref{eq:Asum2} by
\begin{equation}  \label{eq:bddiff99}
C \left(\int_{\substack{(\vec\bp,  \vec\bb,\bH)\in\sK_\epsilon:  \\ \bp^0, \bH, \bdel < C_2 }}  \bF[\vec\bp,\vec\bb,\bH]\d(\vec\bp,  \vec\bb,\bH)
-\int_{\substack{(\vec\bp,  \vec\bb,\bH)\in\sK_\epsilon^*:  \\ \bp^0, \bH<C_2-10N^{-1/3}, \bdel < C_2 }}  \bF[\vec\bp,\vec\bb,\bH]\d(\vec\bp,  \vec\bb,\bH)
\right).
\end{equation}
As for \eqref{eq:Asumb}, using that $|\bI_{\beta,\vec\bk}[\vec\bp,\vec\bb,\bH]|<C\bF[\vec\bp,\vec\bb,\bH]$ whenever $\bp^0<C_2$ and $ \exp\left(\sum_{\ell=1}^{m-1}
(\ttt_\ell-\ttt_{\ell+1}) \bH(\bQ_\ell)/2 \right)\le 1$, it is also bounded by \eqref{eq:bddiff99}.

In summary, to get \eqref{eq:bdsumblo}, it now remains to bound the right-hand side of \eqref{eq:aAbd}, and \eqref{eq:bddiff99}.
This is possible because both are integrals over blocks $(\vec\bp,\vec\bb,\bH)$ in a set with finite total measure; in particular, the blocks are in $\sK_\epsilon$ with $\bdel<C_2$.

\medskip
\noindent\textit{Step 4: bounding \eqref{eq:aAbd}.}
We claim that using the definition of $\bF[\vec\bp,\vec\bb,\bH]$ and $\bEr[\vec\bp,\vec\bb,\bH]$, and integrating over $(\vec\bp,  \vec\bb,\bH)$ with fixed $\Xi_1$, $\Xi_2$, $\Xi_3$, $\Xi_4$,  the right-hand side of \eqref{eq:aAbd} can be bounded by
\begin{multline}  \label{eq:101}
    C \int \Bigg(N^{-1/40}
    +\sum_{(x,y)\in \Xi_2\cup\Xi_3\cup\Xi_4}1\wedge\left(N^{-3/5}(y-x)^{-3}+N^{-1/3}(y-x)^{-1/2}\right) \Bigg)\\ \times \prod_{(x,y)\in \Xi_1}(y-x)^{-3/2} \prod_{(x,y)\in \Xi_2\cup\Xi_3} (y-x)^{-1/2} \d(\Xi_1, \Xi_2, \Xi_3, \Xi_4),
\end{multline}
where the measure on the space of all $(\Xi_1, \Xi_2, \Xi_3, \Xi_4)$ is simply induced by the Lebesgue measure of the endpoints.
In order to prove the bound \eqref{eq:101} , we first imagine that $\bEr[\vec \bp, \vec\bb,\bH]$ of Proposition \ref{prop:IXHGest} is the constant $N^{-1/40}$, and then we get the factors in the second line of \eqref{eq:101} by integrating $\bF[\vec \bp, \vec\bb,\bH]$.
More precisely, the integration of $\bF[\vec \bp, \vec\bb,\bH]$ over $(\vec\bp,  \vec\bb,\bH)$ with fixed $\Xi_1$, $\Xi_2$, $\Xi_3$, $\Xi_4$ is actually an integration over $\bH$, and can be bounded by
\begin{equation}   \label{eq:Xi234p}
C\int \prod_{(x,y)\in\Xi_1} \bF_{0,0}(y-x) \prod_{(x,y)\in\Xi_2\cup \Xi_3} \bF_0(y-x;h_{(x,y)})  \prod_{(x,y)\in\Xi_4} \bF(y-x;h_{(x,y)},g_{(x,y)}) \d \vec{h},
\end{equation}
where $h_{(x,y)}$ corresponds to $\bH(y)-\bH(x)$ for $(x,y)\in\Xi_2$, $\bH(x)-\bH(y)$ for $(x,y)\in\Xi_3$, $\bH(x)-\bp^0(x)$ for $(x,y)\in\Xi_4$, and $g_{(x,y)}$ corresponds to $\bH(y)-\bp^0(y-)$ for $(x,y)\in\Xi_4$.
The integral is over all $\vec{h}=\{h_{(x,y)}\}_{(x,y)\in\Xi_2\cup\Xi_3\cup\Xi_4}\cup \{g_{(x,y)}\}_{(x,y)\in\Xi_4}$, satisfying each $0\le h_{(x,y)}, g_{(x,y)}\le C_2$, and
\[
\sum_{(x,y)\in\Xi_2} h_{(x,y)} - \sum_{(x,y)\in\Xi_3} h_{(x,y)} + \sum_{(x,y)\in\Xi_4} \big( -h_{(x,y)} + g_{(x,y)} \big) = 0.
\]
If $\Xi_2=\Xi_3=\Xi_4=\emptyset$, \eqref{eq:Xi234p} equals $C\prod_{(x,y)\in\Xi_1}\bF_{0,0}(y-x)$.
Otherwise, there exists at least one $\ell\in \llbracket 1, m\rrbracket$ such that $(\bQ_{\ell-1}, \bQ_\ell)\not\in \Xi_1$. Take the smallest such $\ell$.
Then $\bH(\bQ_{\ell-1})=0$, and $\bb_\ell=\varnothing$ by \Cref{Definition_virtual_block}.
Thus, from the definition of $\Xi_1$ in \Cref{ssec:wobvbb}, every $(x,y)\subset (\bQ_{\ell-1}, \bQ_\ell)$ is not in $\Xi_1$, so we can find one $(x_*, y_*)\in \Xi_2\cup\Xi_3\cup\Xi_4$, which is contained in $(\bQ_{\ell-1}, \bQ_\ell)$, with $y_*-x_*>c$.
If $(x_*, y_*)\in \Xi_2\cup\Xi_3$, we bound the factor $\bF_0(y_*-x_*;h_{(x_*,y_*)})$ in \eqref{eq:Xi234p} by $\sup_{h\ge 0}\bF_0(y_*-x_*;h)<C$.
If $(x_*, y_*)\in \Xi_4$, we bound the factor $\bF(y_*-x_*;h_{(x_*,y_*)},g_{(x_*,y_*)})$ in \eqref{eq:Xi234p} by $\sup_{g,h\ge 0}\bF(y_*-x_*;h,g)<C$.
Then the remaining entries of $\vec h$ can be integrated over $[0, C_2]$ in \eqref{eq:Xi234p}.
In summary, we can upper bound \eqref{eq:Xi234p} by
\begin{equation}   \label{eq:Xi234}
C\prod_{(x,y)\in\Xi_1} \bF_{0,0}(y-x) \cdot \prod_{(x,y)\in\Xi_2\cup \Xi_3} \int_0^{C_2}\bF_0(y-x;h) \d h  \cdot \prod_{(x,y)\in\Xi_4} \sup_{g\ge 0}\int_0^{C_2} \bF(y-x;h,g) \d h.
\end{equation}
Note that here we do not explicitly exclude $(x_*, y_*)$ in the product of \eqref{eq:Xi234}, since $\int_0^{C_2}\bF_0(y_*-x_*;h) \d h>c$ (if $(x_*,y_*)\in\Xi_2\cup \Xi_3$) and  $\sup_{g\ge 0}\int_0^{C_2} \bF(y_*-x_*;h,g) \d h>c$ (if $(x_*,y_*)\in\Xi_4$).
In particular, below we bound the integral for the special interval $(x_*, y_*)$ the same way as the other intervals.

The factors in the second line of \eqref{eq:101} then come from the following bounds: for any $x>0$, we have $\bF_{0,0}(x)<Cx^{-3/2}$, $\int_0^\infty \bF_0(x;h) \d h < Cx^{-1/2}$, and $\int_0^\infty \bF(x;h,g) \d h< C$ for any fixed $g\ge 0$.

As for the first factor in the integrand in \eqref{eq:101}, it comes from more careful analysis of $\bEr[\vec \bp, \vec\bb,\bH]$, i.e., the summation over $\bEr[x,y]$ taking into account all four cases appearing in its definition.
In particular, for the term of $1\wedge\big(N^{-1/3}/(\bH(y)-\bH(x))\big)$ for each $(x, y)\in\Xi_2$, and the term of $1\wedge\big(N^{-1/3}/(\bH(x)-\bH(y))\big)$ for each $(x, y)\in\Xi_3$,
we use that for any $x>0$,
\[
\int_0^{C_2} (1\wedge (N^{-1/3}/h)) \bF_0(x;h) \d h < C\left(1\wedge \big(N^{-1/3}x^{-1/2}\big)\right)\cdot x^{-1/2};
\]
and for the term of $1\wedge\big(N^{-1/3}/(\bH(x)-\bp^0(x))+ N^{-1/3}/(\bH(y)-\bp^0(y-))\big)$ for each $(x, y)\in\Xi_4$, 
we use that for any $x>0$ and $g\ge 0$,
\[
\int_0^{C_2} (1\wedge (N^{-1/3}/h)) \bF(x;h,g) \d h < C\left(1\wedge \big(N^{-1/3}x^{-1/2}\big)\right).
\]
These give the terms of $N^{-1/3}(y-x)^{-1/2}$ in the first line of \eqref{eq:101}.

\medskip
\noindent\textit{Step 5: bounding \eqref{eq:101}.}
We next claim that \eqref{eq:101} can be further upper bounded by $C\epsilon^{-m/2}N^{-1/40}$.
For this, we can first upper bound \eqref{eq:101} by
\begin{multline}  \label{eq:76}
    C \sum_{a_\ell \in \llbracket 0, C_2\rrbracket, \ell\in\llbracket 1, m\rrbracket, b_\ell \in \{0, 1\}} \int \Bigg(N^{-1/40}
    +\sum_{\ell=1}^m \sum_{i=1}^{a_\ell} 1\wedge\left(N^{-3/5}x_{\ell,i}^{-3}+N^{-1/3}x_{\ell,i}^{-1/2}\right) \Bigg)\\ \times \prod_{\ell=1}^m \prod_{i=1}^{b_\ell} y_{\ell,i}^{-3/2}\prod_{\ell=1}^m \prod_{i=1}^{a_\ell} x_{\ell,i}^{-1/2} \d (\vec x,\vec y),
\end{multline}
where $a_\ell$, $b_\ell$ correspond to the numbers of intervals in $\Xi_2\cup\Xi_3\cup\Xi_4$ contained in $(\bQ_{\ell-1}, \bQ_\ell)$, and the numbers of intervals in $\Xi_1$ contained in $(\bQ_{\ell-1}, \bQ_\ell)$, respectively. The integral in \eqref{eq:76} is over all $\vec x=\{\{x_{\ell,i}\}_{i=1}^{a_\ell}\}_{\ell=1}^m$ and $\vec y=\{\{y_{\ell,i}\}_{i=1}^{b_\ell}\}_{\ell=1}^m$, satisfying $\sum_{i=1}^{a_\ell}x_{\ell,i} + \sum_{i=1}^{b_\ell}y_{\ell,i}= \bk_\ell$, and each $x_{\ell,i}\ge 0$, and $y_{\ell,i}>\epsilon$.
Here the lower bound of $y_{\ell,i}$ is due to that, for $(\vec\bp,  \vec\bb,\bH)\in\sK_\epsilon$ and $(x,y)\in\Xi_1$, necessarily $y-x>\epsilon$ (which follows from the definition of $\Xi_1$ in \Cref{ssec:wobvbb}, and the definition of $\sK_\epsilon$ in \Cref{defn:core}).

In \eqref{eq:76}, we upper bound each $x_{\ell, i}^{-1}$ or $y_{\ell, i}^{-1}$ by $C$,  whenever $x_{\ell, i}<\bk_\ell/C_2$ or $y_{\ell, i}<\bk_\ell/C_2$ (which holds for at least one $x_{\ell, i}$ or $y_{\ell, i}$ for each $\ell$).
Then for the remaining $x_{\ell,i}$ and $y_{\ell,i}$, we integrate each of them in $[0, \bk_\ell/C_2]$.
The integrals over $y_{\ell, i}$ would give a factor that is at most $C\epsilon^{-m/2}$.
For the integrals over $x_{\ell, i}$, we use that $\int_0^{\bQ_m} N^{-1/40} x^{-1/2} \d x < CN^{-1/40}$, and
\begin{align*}
&\int_0^{\bQ_m} \left(1\wedge\big(N^{-3/5}x^{-3} + N^{-1/3}x^{-1/2} \big)\right) x^{-1/2} \d x \\ \le &  \int_0^{N^{-1/5}} x^{-1/2} \d x + \int_{N^{-1/5}}^{\bQ_m} \big(N^{-3/5}x^{-3} + N^{-1/3}x^{-1/2} \big) x^{-1/2} \d x \\
< & C N^{-1/10} + CN^{-3/5}\cdot N^{1/2} + CN^{-1/3}\cdot\log(N) < C N^{-1/40}.
\end{align*}

\medskip
\noindent\textit{Step 6: bounding \eqref{eq:bddiff99}.}
As for \eqref{eq:bddiff99}, here we integrate $\bF[\vec \bp, \vec\bb,\bH]$, over all $(\vec\bp,  \vec\bb,\bH)\in\sK_\epsilon$, but either $(\vec\bp,  \vec\bb,\bH)\not\in\sK_\epsilon^*$ or $\max\bp^0\ge C_2-10N^{-1/3}$, or $\max\bH\ge C_2-10N^{-1/3}$.
Again we first integrate over $(\vec\bp,  \vec\bb,\bH)$ with fixed $\Xi_1$, $\Xi_2$, $\Xi_3$, $\Xi_4$.
We would get \eqref{eq:Xi234}, except for that, for one $(x,y) \in \Xi_2\cup\Xi_3\cup\Xi_4$ we would replace the domain of integration by an interval of length $<CN^{-1/3}$;
and then we would sum over all choices of this $(x,y)$.
Therefore \eqref{eq:bddiff99} is bounded by
\begin{multline*}
    C\int\Bigg(\sum_{(x,y)\in \Xi_2\cup\Xi_3\cup\Xi_4} 1\wedge (N^{-1/3}(y-x)^{-1/2}) \Bigg)
    \\ \times   \prod_{(x,y)\in \Xi_1}(y-x)^{-3/2} \prod_{(x,y)\in \Xi_2\cup\Xi_3} (y-x)^{-1/2}
     \d(\Xi_1, \Xi_2, \Xi_3, \Xi_4)
    < C \epsilon^{-m/2}N^{-1/40},
\end{multline*}
where we also use that, for any $x>0$ and interval $I\subset (0,\infty)$ of length $N^{-1/3}$, we have  $\int_I \bF_0(x;h) \d h < C(1\wedge N^{-1/3}x^{-1/2})x^{-1/2}$, and $\int_I \bF(x;h,g) \d g< C(1\wedge N^{-1/3}x^{-1/2})$ for any fixed $h$.
The inequality is derived similarly as in bounding \eqref{eq:101}. Therefore \eqref{eq:bdsumblo} follows.

\medskip
\noindent\textit{Step 7: absolute values.}
Finally, in terms of absolute values, we note that for any blocks $(\vec p, \vec\upsilon, \cH)$, the weight $w(\vec r)$ for each $\vec r\in\cL^{-1}(\vec p, \vec\upsilon, \cH)$ have the same sign. Therefore \Cref{prop:IXHGest} still holds if we replace $w(r)$ with $|w(r)|$ and $\bI_{\beta,\vec\bk}[\vec\bp,\vec\bb,\bH]$ with $|\bI_{\beta,\vec\bk}[\vec\bp,\vec\bb,\bH]|$.
Then we can apply \Cref{prop:IXHGest} in the same way as above and deduce \eqref{eq:bdsumblo} in absolute values.
\end{proof}
Given the estimates established so far, we can now deduce the asymptotic results.

\begin{proof}[Proof of \Cref{prop:fixedepsconv}]
Note that $\epsilon>0$ is fixed throughout this proof.
From \Cref{prop:bdsumblo} we have for any $C_2>0$, 
\begin{multline}  \label{eq:fixedpf1}
    \Bigg| \int_{\substack{(\vec\bp,  \vec\bb,\bH)\in\sK_\epsilon[\vec\bk]:  \\ \bp^0, \bH, \bdel < C_2 }} \exp\left(\sum_{\ell=1}^{m-1}
(\ttt_\ell-\ttt_{\ell+1}) \bH(\bQ_\ell)/2 \right)\bI_{\beta,\vec\bk}[\vec \bp, \vec\bb,\bH] \d(\vec\bp,  \vec\bb,\bH) \\ -  2^{-m}N^m \sum_{\vec\fb\in \{0,1\}^m} \prod_{\ell=1}^m (2\sqrt{N_\ell N})^{-k_\ell-\fb_\ell} \sum_{\vec r \in \sB_\epsilon^*}  w(\vec r)\Bigg| <
C_3\epsilon^{-m/2}N^{-1/40} \\ + 2^{-m}N^m \sum_{\vec\fb\in \{0,1\}^m} \prod_{\ell=1}^m (2\sqrt{N_\ell N})^{-k_\ell-\fb_\ell} \sum_{\substack{(\vec p,\vec\upsilon,\cH)\in \sL_\epsilon[\vec \fb]:\\ \max\{ N^{-1/3}p^0, N^{-1/3}\cH, \delta\}\ge C_2}} \sum_{\vec r \in \cL^{-1}(\vec p,\vec\upsilon,\cH)}  |w(\vec r)| ,
\end{multline}
for $C_3=C_3(C_1,C_2)$. Let $A_1(C_2)$ denote the integral in the first line, $A_2(N)$ denote the sum in the second line, and $A_3(C_2, N)$ denote the third line.
By sending $N\to\infty$ with fixed $C_2$, we have
\[
\limsup_{N\to\infty} A_2(N) - \limsup_{N\to\infty} A_3(C_2, N) \le  A_1(C_2)  \le \liminf_{N\to\infty} A_2(N) + \limsup_{N\to\infty} A_3(C_2, N) .
\]
It now remains to show that $\lim_{C_2\to\infty}  \limsup_{N\to\infty} A_3(C_2, N) = 0$.
Since then, $\limsup_{N\to\infty} A_2(N)<\infty$, and by sending $C_2\to \infty$ we have
\[
\limsup_{N\to\infty} A_2(N) \le \liminf_{C_2\to\infty} A_1(C_2) \le  \limsup_{C_2\to\infty} A_1(C_2)  \le \liminf_{N\to\infty} A_2(N),
\]
which implies that both $\lim_{N\to\infty} A_2(N)$ and $\lim_{C_2\to\infty} A_1(C_2)$ exist and are finite, and the limits are equal, i.e., \eqref{eq:limieq} holds.

For \eqref{eq:limieqabs}, by \eqref{eq:bdsumblo} in absolute values, \eqref{eq:fixedpf1} still holds if  $w(\vec r)$ by $|w(\vec r)|$ in the first line and $\bI_{\beta,\vec\bk}[\vec \bp, \vec\bb,\bH]$ by $|\bI_{\beta,\vec\bk}[\vec \bp, \vec\bb,\bH]|$ in the second line, and the same argument as deriving \eqref{eq:limieq} applies.

The rest of this proof is devoted to proving $\lim_{C_2\to\infty}  \limsup_{N\to\infty} A_3(C_2, N) = 0$.
We note that all $C, c$ below are independent of $C_2$.

We apply \Cref{cor:mbdcII} to all $\vec r\in \bigcup_{\vec\fb\in \{0,1\}^m}\sB_\epsilon^*[\vec \fb]$, such that $\max_t \sH(t) \ge C_2N^{1/3}$ or $\delta\ge C_2-m$.
More precisely, from the upper bound in \Cref{cor:mbdcII}, we first multiply $N^m \prod_{\ell=1}^m (2\sqrt{N_\ell N})^{-k_\ell-\fb_\ell}$, sum over $U$ (with fixed $|U|$), and lower bound $|J_\rmVIa|$, $|J_\rmVIb|$, $|J_\rmVIc|$, $|J_\rmC|$ by zero. 
Note that $|J|=m$, and recall (from above \Cref{prop:truncreg}) $\sB_\epsilon^*$ consists of walks $\vec r$ with $J_\rmB=J_\rmC=\emptyset$, $J_\rmIII=J_\rmVIa=J_\rmVIb=J_\rmVIc=\emptyset$, and $\vartheta_\ell\wedge \dot\vartheta_\ell>\epsilon N^{2/3}$ for each $\ell\in J_\rmI\cup J_\rmII\cup J_\rmIV$. 
Also note that the summation over $U$ gives a factor of at most ${N\choose |U|}\le \frac{N^{|U|}}{|U|!}$, so we get
\begin{multline*}
N^{\frac{1}{3}|J_\rmI|+\frac{1}{3}|J_\rmII| -\frac{1}{3}\bigl(|J_\rmIV|+|\{j\in J_\rmIV:\,\vartheta_j=\varnothing\}|\bigr)}
C_5^{\delta+1} \frac{h^{\delta- 2|U|}}{(\delta-2|U|)!}\exp(-C_4h^2) \\
\times \frac{1}{|U|!}\big((|U|+1)\log(|U|+2)\big)^{|U|} \big(h/(|U|+1)\big)^{C_4|U|+m\mathbf{1}_{h\ge |U|}} \prod_{j\in J:\vartheta_j \ne \varnothing} \vartheta_j^{-3/2},
\end{multline*}
where $C_4, C_5$ are constants depending only on $C_1$, and all $C,c$ are allowed to depend on $C_4, C_5$.
We next sum the product of $\prod_{j\in J:\vartheta_j \ne \varnothing} \vartheta_j^{-3/2}$, over $\vartheta_j\ge \epsilon N^{2/3}$ for each $j\in J_\rmI$, $\dot\vartheta_j=\vartheta_j\ge \epsilon N^{2/3}$ for each $j\in J_\rmII$, and $\vartheta_j>\dot\vartheta_j\ge \epsilon N^{2/3}$ or $\vartheta_j=\varnothing$, $\dot\vartheta_j\ge \epsilon N^{2/3}$ for each $j\in J_\rmIV$.
We get a factor of $C\epsilon^{-\frac{1}{2}}N^{-\frac{1}{3}}$ for each $j\in J_\rmI\cup J_\rmII$, a factor of $N^{\frac{2}{3}}$ for each $j\in J_\rmIV$ with $\varepsilon_j=\emptyset$, and a factor of $N^{\frac{1}{3}}$ for each $j\in J_\rmIV$ with $\varepsilon_j\neq \emptyset$.
Thus, the sum is at most
\[
CC_5^{\delta+1} \frac{h^{\delta- 2|U|}}{(\delta-2|U|)!}\exp(-C_4h^2) \frac{1}{|U|!}\big((|U|+1)\log(|U|+2)\big)^{|U|} \big(h/(|U|+1)\big)^{C_4|U|+m\mathbf{1}_{h\ge |U|}}.
\]
By further using that $\frac{(|U|+1)^{|U|}}{|U|!}<C\exp(|U|)$, denoting $a=\delta-2|U|$, $b=|U|$, and taking dyadic $h=2^i$, we have
\[
A_3(C_2, N)<C\epsilon^{-m/2}  \sum_{a, b, i \in \Z_{\ge 0}: \max\{a,b,2^i\}>C_2/4} \frac{(C_62^i)^a}{a!} \exp(-C_42^{2i})
(C_6\log(b+2))^b (2^i/(b+1))^{C_4b},
\]
where $C_6$ is a constant depending only on $C_1$, and all $C,c$ are allowed to depend on $C_6$.

In the rest of this proof we show that the above sum is at most $C\exp(-cC_2)$, from which the conclusion follows.
We can assume that $C_2$ is large enough (depending on $C_4, C_6$), since otherwise the bound immediately follows, as the sum over all $a, b, i \in \Z_{\ge 0}$ is finite.

For any $i\in\Z_{\ge 0}$ we have $\sum_{a=0}^\infty \frac{(C_62^i)^a}{a!} = \exp(C_62^i)$, and
\begin{multline*}
\sum_{b=0}^\infty (C_6\log(b+2))^b(2^i/(b+1))^{C_4b} \\ = 
\sum_{b=0}^{\lceil (i+1)^{2/C_4}\rceil 2^i-1} (C_6\log(b+2))^b(2^i/(b+1))^{C_4b} + \sum_{b=\lceil (i+1)^{2/C_4}\rceil 2^i}^\infty (C_6\log(b+2))^b(2^i/(b+1))^{C_4b}\\
< \lceil (i+1)^{2/C_4}\rceil 2^i\cdot
(C_6\log((i+1)^{2/C_4+1}2^{i}))^{(i+1)^{2/C_4+1}2^{i}}\cdot 2^{C_4(i+1)^{2/C_4+2}2^i} + C
< C \exp(Ci^{2/C_4+2}2^i).
\end{multline*}
Thus,
\begin{multline}   \label{eq:bd111}
\sum_{i \in \Z_{\ge 0}: 2^{2i}>C_2/4 } \exp(-C_42^{2i}) \sum_{a=0}^\infty \frac{(C_62^i)^a}{a!} \sum_{b=0}^\infty (C_6\log(b+2))^b(2^i/(b+1))^{C_4b} \\ <
\sum_{i \in \Z_{\ge 0}: 2^{2i}>C_2/4 } C \exp(-C_42^{2i}+C_62^i+Ci^{2/C_4+2}2^i)
 < C\exp(-cC_2).
\end{multline}
When $2^{2i}\le C_2/4$, we have $\sum_{a=\lceil C_2/4\rceil}^\infty \frac{(C_62^i)^a}{a!}<\frac{(C_6 2^i)^{\lceil C_2/4\rceil}}{\lceil C_2/4\rceil !}\sum_{a=0}^\infty \frac{(C_62^i)^a}{a!}<\frac{(CC_2)^{C_2/8}}{\lceil C_2/4\rceil!}\exp(C_6 2^i)$, and
$\sum_{b=\lceil C_2/4\rceil}^\infty  (C_6\log(b+2))^b (2^i/(b+1))^{C_4b}<\sum_{b=\lceil C_2/4\rceil}^\infty(C_6\log(C_2))^bC_2^{-C_4b/2}<CC_2^{-cC_2}$.
Thus,
\begin{multline}   \label{eq:bd112}
\sum_{i \in \Z_{\ge 0}: 2^{2i}\le C_2/4 }
\exp(-C_42^{2i}) \sum_{a=\lceil C_2/4\rceil}^\infty \frac{(C_62^i)^a}{a!} \sum_{b=0}^\infty  (C_6\log(b+2))^b(2^i/(b+1))^{C_4b} \\ < \frac{(CC_2)^{C_2/8}}{\lceil C_2/4\rceil!} \sum_{i\in \Z_{\ge 0}: 2^{2i}\le C_2/4} \exp(-C_42^{2i}+C_62^i+Ci^{2/C_4+2}2^i) < C\exp(-cC_2)
\end{multline}
and
\begin{multline}   \label{eq:bd113}
\sum_{i \in \Z_{\ge 0}: 2^{2i}\le C_2/4 }
\exp(-C_42^{2i}) \sum_{a=0}^\infty \frac{(C_62^i)^a}{a!} \sum_{b=\lceil C_2/4\rceil}^\infty  (C_6\log(b+2))^b (2^i/(b+1))^{C_4b} \\ < CC_2^{-cC_2} \sum_{i\in \Z_{\ge 0}: 2^{2i}\le C_2/4} \exp(-C_42^{2i}+C_62^i) < C\exp(-cC_2).
\end{multline}
By adding up \eqref{eq:bd111}, \eqref{eq:bd112}, and \eqref{eq:bd113}, the proof concludes.
\end{proof}

\begin{proof}[Proof of \Cref{prop:epszerocov} and \Cref{prop:conv}]
Denote
\[
A_4(N)=2^{-m}\left(\frac{\hP_{k_m}^{N_m,2N/\beta}}{(2\sqrt{N_mN})^{k_m}}+\frac{\hP_{k_m+1}^{N_m,2N/\beta}}{(2\sqrt{N_mN})^{k_m+1}}\right) \cdots \left(\frac{\hP_{k_1}^{N_1,2N/\beta}}{(2\sqrt{N_1N})^{k_1}}+\frac{\hP_{k_1+1}^{N_1,2N/\beta}}{(2\sqrt{N_1N})^{k_1+1}}\right) [1]_{x_1=\dots=x_N=0},  
\]
\[
A_5(\epsilon)= \int_{(\vec\bp,  \vec\bb,\bH)\in\sK_\epsilon[\vec\bk]} \exp\left(\sum_{\ell=1}^{m-1}
(\ttt_\ell-\ttt_{\ell+1}) \bH(\bQ_\ell)/2 \right)\bI_{\beta,\vec\bk}[\vec \bp, \vec\bb,\bH]\d(\vec\bp,  \vec\bb,\bH).
\]
We need to show that $\lim_{N\to\infty}A_4(N) = \lim_{\epsilon\to 0+} A_5(\epsilon)$, and both limits exist and are finite.

From  \Cref{prop:truncreg} and \eqref{eq:limieq}, we have that $\limsup_{N\to\infty} A_4(N)< A_5(\epsilon) + C\sqrt{\epsilon}<\infty$ and $\liminf_{N\to\infty} A_4(N)> A_5(\epsilon) - C\sqrt{\epsilon}$.
Then by sending $\epsilon\to 0+$ we have
\[
\limsup_{\epsilon\to 0+}A_5(\epsilon)\le \liminf_{N\to\infty} A_4(N) \le \limsup_{N\to\infty} A_4(N)\le \liminf_{\epsilon\to 0+}A_5(\epsilon),
\]
so the conclusion follows.
\end{proof}

\subsection{Dyson Brownian Motion: proof of Proposition \ref{prop:DBMconv}}\label{sec:DBMlimit}

The proof of \Cref{prop:DBMconv} is almost the same as the proof of \Cref{prop:conv} and in this section we highlight the necessary modifications.

Instead of Theorem \ref{Theorem_corners_moments}, our starting point is now Theorem \ref{Theorem_DBM_moments}. Arguing similarly to Section \ref{ssec:setdisexp}, we expand the operator $(\hD_{i_m}^{N,\tau_m})^{k_m} \cdots (\hD_{i_1}^{N,\tau_1})^{k_1}$ in terms of walks. The difference is that now all $N_1,\dots,N_m$ are the same and equal to $N$, but $\tau$ (which was fixed in Section \ref{ssec:setdisexp}) is now varying. This leads to an adjustment in the weight of the walk, which we introduced in Definition \ref{defn:wr} and transformed in \eqref{eq:weight}. In the proof of Proposition \ref{prop:DBMconv}, the weight becomes
\begin{equation}   \label{eq:dbmww}
\begin{split}
w(\vec r) = &N^{-\delta}\prod_{\ell=1}^m  (-1)^{|\{t\in\llbracket Q_{\ell-1}+1, Q_\ell \rrbracket: \sH(t)=\sH(t-1)-1, r_{i_\ell}(t)\ge r_{i_\ell}(t-1)\}|} \\ &\times (\beta\tau_\ell N/2)^{k_\ell/2} (\beta\tau_\ell /2)^{(\sH(Q_\ell)-\sH(Q_{\ell-1}))/2}\\
&\times \prod_{\substack{t\in \llbracket Q_{\ell-1}+1, Q_\ell \rrbracket\setminus\Delta \\ \sH(t)=\sH(t-1)-1 }}  \left(1+\frac{2r_{i_\ell}(t-1)}{\beta N}-\frac{|\{j\in \llbracket 1,N\rrbracket: r_j(t-1)\ge r_{i_\ell}(t-1)>0\}|}{N}\right).
\end{split}
\end{equation}
Compared to \eqref{eq:weight}, we replace each $N_\ell$ in the third line, and in $N_\ell^{-\bigl|\llbracket Q_{\ell-1}+1, Q_\ell \rrbracket\cap\Delta\bigr|}$ in the second line, by $N$.
This is due to that, from the definition of the operators \eqref{eq:defhD}, in 2, 3(a), 3(b) of \Cref{defn:tie}, $j$ is allowed to take any number in $\llbracket 1, N\rrbracket$ in the DBM setting.
Then the product $\prod_{\ell=1}^m N^{-\bigl|\llbracket Q_{\ell-1}+1, Q_\ell \rrbracket\cap\Delta\bigr|}$ gives $N^{-\delta}$.
On the other hand, for $N_\ell$ in $(\sqrt{NN_\ell})^{k_\ell}$ and $N_\ell^{(\sH(Q_{\ell-1})-\sH(Q_\ell))/2}$ in the second line of \eqref{eq:weight}, we replace each by $\beta\tau_\ell/2$, and get $(\beta\tau_\ell N/2)^{k_\ell/2}$ and $(\beta\tau_\ell /2)^{(\sH(Q_\ell)-\sH(Q_{\ell-1}))/2}$.
This is because these factors come from counting the number of down-steps versus the number of up-steps: in \Cref{defn:tie} each $-$ or $\top$ step in $\llbracket Q_{\ell-1}+1, Q_\ell\rrbracket$ corresponds to a factor of $N_\ell$, while in the DBM setting each $+$ step gives a factor of $\beta\tau_\ell/2$.

The classification of walks developed in \Cref{ssec:classi} is still valid. The bound of Proposition \ref{prop:mbdc} in Section \ref{ssec:generabd} can be deduced verbatim, except for that the factor $\prod_{\ell=1}^m (2\sqrt{NN_\ell})^{k_\ell}$ in the right-hand side of \eqref{eq:wbdfdt} is replaced by $\prod_{\ell=1}^m (2\beta\tau_\ell N)^{k_\ell/2}$. 
In particular, for the factor $\prod_{\ell=1}^m N_\ell^{(\sH(Q_{\ell-1})-\sH(Q_\ell))/2}$, essentially it is upper bounded by $\exp(Ch)$, and the same holds for $\prod_{\ell=1}^m(\beta\tau_\ell /2)^{(\sH(Q_\ell)-\sH(Q_{\ell-1}))/2}$.
Also, $N^{-\delta}$ is of the same order (up to a factor between $1$ and $1+C\delta N^{-1/3}$) as $\prod_{\ell=1}^m N_\ell^{-\bigl|\llbracket Q_{\ell-1}+1, Q_\ell \rrbracket\cap\Delta\bigr|}$. Thus these factors are bounded the same way as their counterparts in the proof of \Cref{prop:mbdc}.

The bounds of Propositions \ref{prop:csBJ} and \ref{prop:sBJe} of Section \ref{ssec:buc} follow in the same way, with the same replacement of $\prod_{\ell=1}^m (2\sqrt{NN_\ell})^{k_\ell}$ by $\prod_{\ell=1}^m (2\beta\tau_\ell N)^{k_\ell/2}$ in the right-hand sides. From there, in parallel to Section \ref{ssec:pair}, we develop an analogue of Proposition \ref{prop:truncreg}, which is a bound:  
\[
\left| N^m \sum_{\vec r \in \sB_\epsilon^*} w(\vec r) - \hP_{k_m}^{N,\tau_m} \cdots \hP_{k_1}^{N,\tau_1} [1]_{x_1=\dots=x_N=0} \right| < C \prod_{\ell=1}^m (2\beta\tau_\ell N)^{k_\ell/2} (\sqrt{\epsilon} + N^{-1/12}).
\]

Further, in parallel with Sections \ref{ssec:bdecom} and \ref{ssec:sob}, we analyze $N^{m}\sum_{\vec r\in \sB_\epsilon^*}w(\vec r)$ as $N\to\infty$ and show that it converges towards the integral over $\sK_\epsilon[\vec\bk]$ in \Cref{defn:core}. The same block decomposition as in \Cref{ssec:bdecom} applies and an analogue of the estimate \Cref{prop:IXHGest} holds, with $\prod_{\ell=1}^m (2\sqrt{NN_\ell})^{-k_\ell}$ replaced by $\prod_{\ell=1}^m (2\beta\tau_\ell N)^{-k_\ell/2}$;
and in its proof, $\prod_{\ell=1}^m N_\ell^{(\cH(Q_{\ell-1})-\cH(Q_\ell))/2}$ would be replaced by $\prod_{\ell=1}^m(\beta\tau_\ell/2)^{(\cH(Q_{\ell})-\cH(Q_{\ell-1}))/2}$. 
In particular, \eqref{eq:Pnexp} would be replaced by
\[
\left|P[\vec p,\vec\upsilon, \cH]N^{-u}\prod_{\ell=1}^m(\beta\tau_\ell/2)^{(\cH(Q_{\ell})-\cH(Q_{\ell-1}))/2} - \exp\left(\sum_{\ell=1}^{m-1}
(\ttt_\ell-\ttt_{\ell+1}) \bH(\bQ_\ell)/2 \right)\right|<CN^{-1/3},
\]
which can be derived using $|\log(\beta\tau_\ell/(2N))N^{1/3}-\ttt_\ell|<CN^{-1/3}$ (instead of $|\log(N_\ell/N)N^{1/3} + \ttt_\ell| < CN^{-1/3}$), and the other estimates in obtaining \eqref{eq:Pnexp}.

Finally, everything in \Cref{ssec:sob} (\Cref{prop:bdsumblo} and its proof, and the final proof of \Cref{prop:conv}) goes through verbatim, with each appearance of $\prod_{\ell=1}^m (2\sqrt{N_\ell N})^{-k_\ell-\fb_\ell}$ replaced by $\prod_{\ell=1}^m (2\beta\tau_\ell N)^{-(k_\ell+\fb_\ell)/2}$.

\section{Random process convergence} \label{Section_process_convergence}
In this section we finish the proofs of Theorems \ref{thm:main}, \ref{thm:cor-conv}, and \ref{thm:dbm-conv}. Our task is to upgrade the convergence of joint moments in Theorems \ref{thm:multil} and \ref{thm:multit} to distributional convergence in the space of continuous curves equipped with uniform topology (on compact intervals of time). The two main ingredients are:
\begin{itemize}
 \item A uniform bound of Section \ref{Section_increment_bound} on the moments of the time increments for the G$\beta$E corners process, the Dyson Brownian Motion, and the Airy$_\beta$ line ensemble. This is used both for proving tightness as $N\to\infty$ in the space of continuous curves, and for the existence of continuous modification for the limiting Airy$_\beta$ line ensemble. The proof relies on an algebraic identity for Dunkl operators, which can be of independent interest.
     
 \item An indirect argument of Section \ref{Section_identification} deducing the distributional convergence from the convergence of the moments, taking into account that the moments problem for the individual random variables $\exp(\bk \AB_i(\ttt)/2)$ of Theorem \ref{thm:main} is not uniquely determined (the moments grow too fast as a function of $m$).
\end{itemize}

We introduce the notations for the quantities in Theorems \ref{thm:cor-conv} and \ref{thm:dbm-conv}. Let  $\{y_i^n\}_{0\le i\le n}$ be the (infinite rank) G$\beta$E corners process with variance $1$ from \Cref{eq_GBE_corners}. 
For any $\ttt<N^{1/3}$ and $i\in\llbracket 1, N-\ttt N^{2/3}\rrbracket$, denote
\[
\by_i^{(N)}(\ttt)=\by_i(\ttt)=N^{2/3}\left(\frac{y_i^{\lfloor N-\ttt N^{2/3} \rfloor}}{\sqrt{2\beta\lfloor N-\ttt N^{2/3}\rfloor}}-1\right).
\]
Let $\{Y_i(\tau)\}_{1\leq i\leq N, \tau>0}$ be the $N$ dimensional DBM started from zero initial condition from \Cref{Definition_DBM}.
For any $\ttt>-N^{1/3}$ and $i\in\llbracket 1, N\rrbracket$, denote
\[
\bY_i^{(N)}(\ttt)=\bY_i(\ttt)=N^{2/3}\left( \frac{Y_{i}(2N/\beta +2N^{2/3}\ttt/\beta)}{2N\sqrt{1+\ttt N^{-1/3}}}-1\right).
\]

\subsection{Uniform bound on time increments} \label{Section_increment_bound}
In order to upgrade \Cref{thm:multil} and \Cref{thm:multit} to the statements about convergence of continuous functions, we need to have control over time increments. In the spirit of the classical Kolmogorov Continuity Theorem, for this it would be sufficient to prove the following two estimates on the fourth moments.

\begin{prop}   \label{prop:mom-cor}
For any $C_1>0$, there exists $C=C(C_1)>0$ such that the following holds.
For any $N\in\N$, $N>C_1^3$, $0\le \ttt_1<\ttt_2\le C_1$, $\ttt_2-\ttt_1\ge N^{-2/3}$, $C_1^{-1}N^{2/3}<k<C_1N^{2/3}$, we have
\begin{multline*}
\E \left[  \left( \sum_{i=1}^{\lfloor N-\ttt_1 N^{2/3} \rfloor} \left(1+N^{-2/3}\by_i^{(N)}(\ttt_1)/2 \right)^k - \sum_{i=1}^{\lfloor N-\ttt_2 N^{2/3} \rfloor} \left( 1+N^{-2/3}\by_i^{(N)}(\ttt_2)/2 \right)^k\right)^4  \right] \\ < C (\ttt_2-\ttt_1)^2 + C(\ttt_2-\ttt_1)N^{-1/3}.
\end{multline*}
\end{prop}

\begin{prop}   \label{prop:mom-dbm}
For any $C_1>0$, there exists $C=C(C_1)>0$ such that the following holds.
For any $N\in\N$, $0\le \ttt_1<\ttt_2 \le C_1$, $C_1^{-1}N^{2/3}<k<C_1N^{2/3}$, we have
\[
\E \left[  \left( \sum_{i=1}^N \left(1+N^{-2/3}\bY_i^{(N)}(\ttt_1)/2 \right)^k - \sum_{i=1}^N \left( 1+N^{-2/3}\bY_i^{(N)}(\ttt_2)/2\right)^k\right)^4  \right] < C (\ttt_2-\ttt_1)^2.
\]
\end{prop}
The proofs of the above estimates are obtained by expressing the moments in terms of Dunkl operators (as in Theorems \ref{Theorem_corners_moments} and \ref{Theorem_DBM_moments}), and bounding them using estimates in \Cref{sec:gme}. In order to get the right-hand sides of the inequalities proportional to $(\ttt_2-\ttt_1)^2$, rather than $|\ttt_2-\ttt_1|$, we rely on an interesting purely algebraic identity for Dunkl operators, which we state next.

Recall the operators
\[
 \D_i^N= \frac{\partial}{\partial x_i} + \frac{\beta}{2} \sum_{\begin{smallmatrix} j\in\llbracket 1,N\rrbracket,\\ j\ne i\end{smallmatrix}} \frac{1-\sigma_{ij}}{x_i-x_j},\quad\quad
 \hD_i^{N,\tau}=  \D_i^N + \tau x_i,\quad\quad
 \hP_k^{N,\tau}=\sum_{i=1}^N (\hD_i^{N,\tau})^k,\]
acting on polynomials of $x_1, x_2, \cdots$.
We further denote
\[
\baP_k^{N,\tau}= \sum_{w=0}^{k-1} \sum_{i=1}^N
(\hD_i^{N,\tau})^wx_i(\hD_i^{N,\tau})^{k-1-w},
\]
which can be thought of as the $\tau$ derivative of $\hP_k^{N,\tau}$.
For any two operators $\P$ and $\P'$, we denote the commutator $[\P, \P']=\P\P'-\P'\P$.
Then we can also write $\baP_k^{N,\tau}=[\hP_k^{N,\tau}, \cS]$, where $\cS$ is the operator of multiplication by the function $\sum_{i=1}^N \frac{x_i^2}{2}$.
\begin{prop}  \label{prop:piden}
For any $N,k\in\N$, $k\ge 2$, and $\tau>0$, we have $[[\baP_{k}^{N,\tau}, \hP_{k}^{N,\tau}], \hP_{k}^{N,\tau}]=0$.
\end{prop}
\begin{proof}
In this proof we fix $N, \tau$.
We write $\hD_i=\hD_i^{N,\tau}$, $\hP_k=\hP_k^{N,\tau}$, $\baP_k=\baP_k^{N,\tau}$, for each $i\in\llbracket 1,N\rrbracket$ and $k\in\N$.
The identity to prove can be written as $[[[\cS, \hP_{k}], \hP_{k}], \hP_{k}]=0$.

Recall that Dukl operators commute, as stated in \Cref{Lemma_Dunkls_commute}. Hence, using $[\D_i, x_j]=\frac{\beta}{2} \sigma_{ij}$, valid for $i\ne j$, we conclude that $\hD_i$ also commute. Therefore, by linearity and product rule $[A,BC]=[A,B]C+B[A,C]$, the expression $[[[\cS, \hP_{k}], \hP_{k}], \hP_{k}]$ can be written as a linear combination of the terms of the kind $A [[[x_{i_1}^2, \hD_{i_2}], \hD_{i_3}], \hD_{i_4}] B$ for various operators $A$, $B$ (which are products of operators $\hD_i$) and various indices $(i_1, i_2, i_3, i_4)\in\llbracket 1, N\rrbracket^4$. Hence, it remains to show that $[[[x_{i_1}^2, \hD_{i_2}], \hD_{i_3}], \hD_{i_4}]=0$ for each $(i_1, i_2, i_3, i_4)\in\llbracket 1, N\rrbracket^4$.

It is straightforward to check that, for each $i,j\in\llbracket 1,N\rrbracket$, $i\neq j$,
\[
[\hD_i, x_j] = \frac{\beta}{2}\sigma_{ij}, \quad [\hD_i, x_j^2] = \frac{\beta}{2} (x_i+x_j)\sigma_{ij};
\]
and for each $i\in \llbracket 1, N\rrbracket$, 
\[
[\hD_i, x_i] = 1 + \frac{\beta}{2}\sum_{j\in\llbracket 1,N\rrbracket: j\neq i} \sigma_{ij},\quad [\hD_i, x_i^2] = 2x_i + \frac{\beta}{2}\sum_{j\in\llbracket 1,N\rrbracket: j\neq i} (x_i+x_j) \sigma_{ij};
\]
and for each $i, j \in \llbracket 1, N\rrbracket$, $i\neq j$, we have $[\hD_{i'},\sigma_{ij}] = 0$ for each $i'\in\llbracket 1, N\rrbracket$.
Using these we do get the desired identity: just note that $[x_{i_1}^2, \hD_{i_2}]$ is a linear combination of $x_i\sigma_{i,j}$, then $[[x_{i_1}^2, \hD_{i_2}], \hD_{i_3}]$ is a linear combination of $\sigma_{i,j}$, and then $[[[x_{i_1}^2, \hD_{i_2}], \hD_{i_3}], \hD_{i_4}]=0$.
\end{proof}

\begin{corollary}  \label{cor:piden}
For any $N,k\in\N$, $k\ge 2$, and $\tau>0$, we have
\[
-\baP_{k}^{N,\tau}(\hP_{k}^{N,\tau})^3 + 3\hP_{k}^{N,\tau}\baP_{k}^{N,\tau}(\hP_{k}^{N,\tau})^2 - 3(\hP_{k}^{N,\tau})^2\baP_{k}^{N,\tau}\hP_{k}^{N,\tau} +
(\hP_{k}^{N,\tau})^3\baP_{k}^{N,\tau}=0.
\]
\end{corollary}
\begin{proof}
The above identity can be written as $[[[\baP_{k}^{N,\tau}, \hP_{k}^{N,\tau}], \hP_{k}^{N,\tau}], \hP_{k}^{N,\tau}]=0$, which follows from \Cref{prop:piden}.
\end{proof}

\begin{proof}[Proof of \Cref{prop:mom-dbm}]
For any $\ttt\in\R$ we denote $\bX(\ttt)=\sum_{i=1}^N \left(1+N^{-2/3}\bY_i^{(N)}(\ttt)/2 \right)^k$. Our goal is to bound $\E[(\bX(\ttt_2)-\bX(\ttt_1))^4]$.
We will use (within this proof) $C$ to denote a constant that is large enough (depending on $C_1$ in the statement), and its value may change from line to line.

We also denote $\tau_\iota=2N/\beta+2N^{2/3}\ttt_\iota/\beta$, for $\iota=1, 2$;
then we have $1+N^{-2/3}\bY_i(\ttt_\iota)/2=\frac{Y_i(\tau_\iota)}{\sqrt{2\beta N\tau_\iota}}$.
We next show that
\begin{multline}  \label{eq:Xdiffbd}
\left| \E\left[
\bX(\ttt_2)\bX(\ttt_1)^3- \bX(\ttt_1)^4
\right] - (\tau_2-\tau_1) \frac{\left(\baP_k^{N,\tau_1}-k(2\tau_1)^{-1}\hP_k^{N,\tau_1}\right)(\hP_k^{N,\tau_1})^3 [1]_{x_1=\dots=x_N=0}}{(2\beta N\tau_1)^{2k}} \right|\\ < C(\ttt_2-\ttt_1)^2.
\end{multline}
By \Cref{Theorem_DBM_moments}, and as explained below \eqref{eq:defhP}, we can write the expectation as
\[
\frac{\left((\tau_1/\tau_2)^{k/2}\hP_k^{N,\tau_2}-\hP_k^{N,\tau_1}\right)(\hP_k^{N,\tau_1})^3 [1]_{x_1=\dots=x_N=0}}{(2\beta N\tau_1)^{2k}}.
\]
Thus, it suffices to prove
\begin{multline}   \label{eq:abovleft}
\left|\left((\tau_1/\tau_2)^{k/2}\hP_k^{N,\tau_2}-\hP_k^{N,\tau_1}
-(\tau_2-\tau_1)\left(\baP_k^{N,\tau_1} -k(2\tau_1)^{-1}\hP_k^{N,\tau_1}\right)\right)(\hP_k^{N,\tau_1})^3[1]_{x_1=\dots=x_N=0}\right| \\ < CN^{-4/3}(2\beta N\tau_1)^{2k}(\tau_2-\tau_1)^2.
\end{multline}
This can be viewed as a Taylor expansion of $(\tau_1/\tau_2)^{k/2}\hP_k^{N,\tau_2}$ in $\tau_2-\tau_1$: by taking the $\tau_2$ derivative of $(\tau_1/\tau_2)^{k/2}\hP_k^{N,\tau_2}$ at $\tau_2=\tau_1$, we get $\baP_k^{N,\tau_1} -k(2\tau_1)^{-1}\hP_k^{N,\tau_1}$.

By \Cref{lem:exp}, we have
\[
(\hP_k^{N,\tau_1})^4[1]_{x_1=\dots=x_N=0}
=\sum_{(i_1,i_2,i_3,i_4) \in \llbracket 1, N \rrbracket^4} \sum_{\vec r\in\sB[i_1,i_2,i_3,i_4]} w(\vec r),
\]
where $\sB[i_1,i_2,i_3,i_4]$ is the set of all walks (of $(\hD_{i_4}^{N})^{k} \cdots (\hD_{i_1}^{N})^{k}$) from \Cref{defn:wallblock}, and $w(\vec r)$ is defined via \eqref{eq:dbmww} for the operator $(\hD_{i_4}^{N,\tau_1})^{k} \cdots (\hD_{i_1}^{N,\tau_1})^{k}$.
Similarly, (for the same $\sB[i_1,i_2,i_3,i_4]$ and $w(\vec r)$) we have 
\[
\hP_k^{N,\tau_2}(\hP_k^{N,\tau_1})^3[1]_{x_1=\dots=x_N=0}
=\sum_{\substack{(i_1,i_2,i_3,i_4) \in \llbracket 1, N \rrbracket^4 \\ \vec r\in\sB[i_1,i_2,i_3,i_4]}} (\tau_2/\tau_1)^{|\{t\in \llbracket 3k+1, 4k\rrbracket:\sH(t)=\sH(t-1)+1\}|} w(\vec r),
\]
and
\begin{multline*}
\sum_{i=1}^N
(\hD_i^{N,\tau})^wx_i(\hD_i^{N,\tau})^{k-1-w}(\hP_k^{N,\tau_1})^3[1]_{x_1=\dots=x_N=0} \\
=\sum_{\substack{(i_1,i_2,i_3,i_4) \in \llbracket 1, N \rrbracket^4 \\ \vec r\in\sB[i_1,i_2,i_3,i_4]}} \tau_1^{-1} \don[\sH(4k-w)=\sH(4k-w-1)+1] w(\vec r),
\end{multline*}
for each $w\in\llbracket 0, k-1\rrbracket$. By summing over $w$ we have
\[
\baP_k^{N,\tau_1}(\hP_k^{N,\tau_1})^3[1]_{x_1=\dots=x_N=0}
=\sum_{\substack{(i_1,i_2,i_3,i_4) \in \llbracket 1, N \rrbracket^4 \\ \vec r\in\sB[i_1,i_2,i_3,i_4]}} \tau_1^{-1}|\{t\in \llbracket 3k+1, 4k\rrbracket:\sH(t)=\sH(t-1)+1\}| w(\vec r).
\]
Thus, the left-hand side of \eqref{eq:abovleft} can be written as
\[
\left|\sum_{\substack{(i_1,i_2,i_3,i_4) \in \llbracket 1, N \rrbracket^4 \\ \vec r\in\sB[i_1,i_2,i_3,i_4]}}  \left( (\tau_1/\tau_2)^{\sH(3k)/2} - 1 + (\tau_2-\tau_1)(2\tau_1)^{-1}\sH(3k) \right)  w(\vec r)\right|,
\]
using that $\sH(3k)=\sH(3k)-\sH(4k)=k-2|\{t\in \llbracket 3k+1, 4k\rrbracket:\sH(t)=\sH(t-1)+1\}|$.

We then handle blow-up terms via cancellations, as done in \Cref{ssec:buc}. Let $\cC$ be the map in \Cref{defn:fP}. For any $\vec r\in \sB[i_1,i_2,i_3,i_4]$ satisfying $\cC(\vec r)=\vec r$ (i.e., without type II indices), we have
\[
\left| (\tau_1/\tau_2)^{\sH(3k)/2} - 1 + (\tau_2-\tau_1)(2\tau_1)^{-1}\sH(3k) \right| <C N^{-2}(\max_t \sH(t) )^2 (\tau_2-\tau_1)^2.
\]
Use this, \Cref{lem:eq:caneq}, and \Cref{cor:sumoverUh}, and sum over all $(i_1,i_2,i_3,i_4) \in \llbracket 1, N \rrbracket^4$ and $\vec r$ without type II indices, we get the bound $CN^{-4/3}(2\beta N\tau_1)^{2k}(\tau_2-\tau_1)^2$.
Therefore \eqref{eq:Xdiffbd} follows.

We can similarly estimate
\begin{itemize}
\item
$\E\left[\bX(\ttt_2)^2\bX(\ttt_1)^2- \bX(\ttt_2)\bX(\ttt_1)^3\right]$,
\item $\E\left[\bX(\ttt_2)^3\bX(\ttt_1)- \bX(\ttt_2)^2\bX(\ttt_1)^2\right]$,
\item $\E\left[\bX(\ttt_2)^4- \bX(\ttt_2)\bX(\ttt_1)^3\right]$
\end{itemize}
through the actions of the operators
\begin{itemize}
    \item $\hP_k^{N,\tau_1}\left(\baP_k^{N,\tau_1}-k(2\tau_1)^{-1}\hP_k^{N,\tau_1}\right)(\hP_k^{N,\tau_1})^2$,
    \item $(\hP_k^{N,\tau_1})^2\left(\baP_k^{N,\tau_1}-k(2\tau_1)^{-1}\hP_k^{N,\tau_1}\right)\hP_k^{N,\tau_1}$,
    \item $(\hP_k^{N,\tau_1})^3\left(\baP_k^{N,\tau_1}-k(2\tau_1)^{-1}\hP_k^{N,\tau_1}\right)$,
\end{itemize}
respectively. 
Then since
\begin{multline*}
\E[(\bX(\ttt_2)-\bX(\ttt_1))^4] = - \E\left[
\bX(\ttt_2)\bX(\ttt_1)^3- \bX(\ttt_1)^4
\right]  + 3\E\left[\bX(\ttt_2)^2\bX(\ttt_1)^2- \bX(\ttt_2)\bX(\ttt_1)^3\right] \\ - 3\E\left[\bX(\ttt_2)^3\bX(\ttt_1)- \bX(\ttt_2)^2\bX(\ttt_1)^2\right] + \E\left[\bX(\ttt_2)^4- \bX(\ttt_2)\bX(\ttt_1)^3\right],
\end{multline*}
via a linear combination of these estimates, and \Cref{cor:piden}, the conclusion follows.
\end{proof}

\begin{proof}[Proof of \Cref{prop:mom-cor}]
We follow the same strategy as the previous proof. However, estimates in this proof are more technical. The main reason is that in the setting of DBM, we need to consider $\hP_k^{N,\tau}$ with different $\tau$, and the set of walks remain the same; while for corners we need to consider different $N$, which lead to different sets of walks for different terms in a version of \eqref{eq:abovleft}.

Again we use $C$ to denote a constant that is large enough (depending on $C_1$ in the statement), whose value may change from line to line.
For any $\ttt\in\R$ we denote $\bx(\ttt)=\sum_{i=1}^N \left(1+N^{-2/3}\by_i^{(N)}(\ttt)/2 \right)^k$, and we want to show $\E[(\bx(\ttt_2)-\bx(\ttt_1))^4]<C(\ttt_2-\ttt_1)^2+C(\ttt_2-\ttt_1)N^{-1/3}$.

Write $N_\iota = \lfloor N-\ttt_\iota N^{2/3}\rfloor$ for $\iota=1, 2$; then we have $1+N^{-2/3}\by_i(\ttt_\iota)/2=\frac{y_i^{N_\iota}}{\sqrt{2\beta N_\iota}}$.
Now we shall prove an analogue of \eqref{eq:Xdiffbd}:
\begin{multline}  \label{eq:xdiffbd}
\left| \E\left[
\bx(\ttt_2)\bx(\ttt_1)^3- \bx(\ttt_1)^4
\right] - \frac{2(N_1-N_2)}{\beta}\cdot \frac{\left(\baP_k^{N_1}-k\beta(4N)^{-1}\hP_k^{N_1}\right)(\hP_k^{N_1})^3 [1]_{x_1=\dots=x_N=0}}{(4NN_1)^{2k}} \right|\\ < C(\ttt_2-\ttt_1)^2 +C(\ttt_2-\ttt_1)N^{-1/3},
\end{multline}
where we let $\hP_k^{N_\iota} = \hP_k^{N_\iota,2N/\beta}$ and $\baP_k^{N_\iota} = \baP_k^{N_\iota,2N/\beta}$ for $N_\iota=1,2$.

By \Cref{Theorem_corners_moments}, and as explained below \eqref{eq:defhP}, the expectation equals (Note that $\{y_i^n\}_{0\le i\le n}$ is the G$\beta$E corners process with variance $1$, while we used instead $2N/\beta$ throughout Section \ref{sec:gme}.)
\[
\frac{\left(N_1^{k/2}N_2^{-k/2}\hP_k^{N_2}-\hP_k^{N_1}\right)(\hP_k^{N_1})^3 [1]_{x_1=\dots=x_N=0}}{(4NN_1)^{2k}},
\]
therefore \eqref{eq:xdiffbd} is implied by an analogue of \eqref{eq:abovleft} and our assumption $\ttt_2-\ttt_1\ge N^{-2/3}$:
\begin{multline}  \label{eq:xdiffbd2}
\left|\left(2N_1^{k/2}N_2^{-k/2}\hP_k^{N_2}-2\hP_k^{N_1}
- (N_1-N_2)\left(4\baP_k^{N_1}/\beta -kN^{-1}\hP_k^{N_1}\right)\right)(\hP_k^{N_1})^3[1]_{x_1=\dots=x_N=0}\right| \\ < C(4NN_1)^{2k}\left(N^{-4/3}(N_1-N_2)^2 + N^{-1}(N_1-N_2)\right).
\end{multline}

For $\iota=1, 2$, we let $\sB_\iota[i_1,i_2,i_3,i_4]$ be the set of all walks from \Cref{defn:wallblock}, and $w_\iota(\vec r)$ be from \Cref{defn:wr} or \eqref{eq:weight}, for the operators $(\hD_{i_4}^{N_1})^{k}(\hD_{i_3}^{N_1})^{k}(\hD_{i_2}^{N_1})^{k}(\hD_{i_1}^{N_1})^{k}$ and $(\hD_{i_4}^{N_2})^{k}(\hD_{i_3}^{N_1})^{k}(\hD_{i_2}^{N_1})^{k}(\hD_{i_1}^{N_1})^{k}$, respectively.
Besides, we further denote $\sB_+=\bigcup_{(i_1,i_2,i_3,i_4) \in \llbracket 1, N_1 \rrbracket^4} \sB_1[i_1,i_2,i_3,i_4]$, and $\sB_-=\bigcup_{(i_1,i_2,i_3,i_4) \in \llbracket 1, N_1 \rrbracket^3\times\llbracket 1, N_2 \rrbracket} \sB_2[i_1,i_2,i_3,i_4]$.
We also let $\sB_*$ denote the collection of all $\vec r\in\sB_+$, such that $i_4$ is of type II, and $i_4 \neq i_1, i_2, i_3$.

Now similar to \eqref{eq:abovleft}, the left-hand side of \eqref{eq:xdiffbd2} can be bounded by the sum of the following:
\begin{equation}  \label{eq:mlbd0}
 2(N_1/N_2)^{k/2} \Bigg|\sum_{\vec r\in\sB_*} (N_2/N_1) w_2(\vec r) + \sum_{\vec r\in\sB_+\setminus \sB_*}  w_2(\vec r) - \sum_{\vec r\in\sB_-}  w_2(\vec r) \Bigg|,
\end{equation}
\begin{equation}  \label{eq:mlbd1}
\Bigg|  2(N_1/N_2)^{k/2} \Bigg( \sum_{\vec r\in\sB_*} (N_2/N_1) w_2(\vec r) + \sum_{\vec r\in\sB_+\setminus \sB_*} w_2(\vec r) \Bigg)
- \sum_{\vec r\in \sB_+} 2(N_2/N_1)^{\sH(3k)/2}w_1(\vec r) \Bigg|,
\end{equation}
and
\begin{equation}  \label{eq:mlbd2}
    \Bigg|\sum_{\vec r\in\sB_+}  \big( 2(N_2/N_1)^{\sH(3k)/2} - 2 + N^{-1}(N_1-N_2)\sH(3k) \big)  w_1(\vec r)\Bigg|.
\end{equation}

We next bound them one-by-one. For \eqref{eq:mlbd0} and \eqref{eq:mlbd1}, we need to exploit cancellations, as done in \Cref{ssec:buc}.

\smallskip

\noindent\textbf{Bound \eqref{eq:mlbd0}.}
Let $\cC$ be the map in \Cref{defn:fP}. Similarly to \Cref{lem:eq:caneq}, for any $\vec r \in \sB_-$ such that $\cC(\vec r)=\vec r$,  $i_4$ is of type I, and $i_4 \neq i_1, i_2, i_3$, we have
\begin{multline*}
 \Bigg| \sum_{\vec s\in \sB_*\setminus \sB_-, \cC(\vec s)=\vec r} (N_2/N_1) w_2(\vec s) - \sum_{\vec s\in \sB_-\cap \sB_*, \cC(\vec s)=\vec r} \frac{N_1-N_2}{N_1} w_2(\vec s) \Bigg| \\ < C(|U|+1)N^{-1}\prod_{j\in J_I, j\neq i_4}\left( |U|N^{-1}+ N^{-1}\vartheta_{j}\max_t \sH(t)\right)|w_2(\vec r)|.
\end{multline*}
In addition, if $U\cap \llbracket N_2+1, N_1\rrbracket=\emptyset$ and $i_1, i_2, i_3\le N_2$, the upper bound can be improved to
\[
C(|U|+1)(N_1-N_2)N^{-2}\prod_{j\in J_I, j\neq i_4}\left( |U|N^{-1}+ N^{-1}\vartheta_{j}\max_t \sH(t) \right)|w_2(\vec r)|.
\]
From the above two bounds and using \Cref{prop:mbdc}, we get
\begin{multline}  \label{eq:mlbd00}
 2(N_1/N_2)^{k/2} \Bigg| \sum_{\vec r\in \sB_*\setminus \sB_-, \cC(\vec r)\in \sB_-} (N_2/N_1) w_2(\vec r) - \sum_{\vec r\in \sB_-\cap \sB_*} \frac{N_1-N_2}{N_1} w_2(\vec r) \Bigg| \\ < C(4NN_1)^{2k}  N^{-5/3}(N_1-N_2).
\end{multline}
Besides, as in \Cref{lem:eq:caneq}, for any $\vec r\in\sB_+$ with $\cC(\vec r)=\vec r\not\in\sB_-$, we have
\[
 \Bigg| \sum_{\vec s\in \sB_*, \cC(\vec s)=\vec r} (N_2/N_1) w_2(\vec s) + \sum_{\vec s\in \sB_+\setminus \sB_*, \cC(\vec s)=\vec r} w_2(\vec s) \Bigg| < C\prod_{j\in J_I}\left( |U|N^{-1}+ N^{-1}\vartheta_{j}\max_t \sH(t) \right)|w_2(\vec r)|.
\]
Then using \Cref{cor:sumoverUh}, we get
\begin{multline}  \label{eq:mlbd01}
 2(N_1/N_2)^{k/2} \Bigg| \sum_{\vec r\in \sB_*, \cC(\vec r)\not\in \sB_-} (N_2/N_1) w_2(\vec r) + \sum_{\vec r\in \sB_+\setminus \sB_*, \cC(\vec r)\not\in \sB_-} w_2(\vec r)  \Bigg|\\ <C(4NN_1)^{2k}  N^{-1}(N_1-N_2).
\end{multline}
We next add \eqref{eq:mlbd00} and \eqref{eq:mlbd01}.
Note that for $\vec r\in\sB_*$ with $\cC(\vec r)\not\in\sB_-$, necessarily $\vec r\not\in \sB_-$; and for $\vec r\not\in\sB_*$, $\vec r\in \sB_-$ if and only if $\cC(\vec r)\in\sB_-$. Therefore,
\[
\sB_* = \{\vec r\in \sB_*\setminus\sB_-: \cC(\vec r)\in\sB_-\}\cup (\sB_-\cap \sB_*) \cup \{\vec r\in \sB_*: \cC(\vec r)\not\in\sB_-\},
\]
\[
\sB_- = (\sB_-\cap \sB_*)\cup\{\vec r\in \sB_+\setminus\sB_*: \cC(\vec r)\in\sB_-\}.
\]
Thus, adding up \eqref{eq:mlbd00} and \eqref{eq:mlbd01} implies that \eqref{eq:mlbd0} is bounded by $C(4NN_1)^{2k} N^{-1}(N_1-N_2)$.

\smallskip

\noindent\textbf{Bound \eqref{eq:mlbd1}.}
Take any  $\vec r\in \sB_+$ satisfying $\cC(\vec r)=\vec r$,
and $\vec s\in \sB_+$ with $\cC(\vec s)=\vec r$. If $\vec s\not\in \sB_*$, then we have $w_2(\vec s) - \frac{w_2(\vec r)w_1(\vec s)}{w_1(\vec r)} = 0$. If $\vec s\in \sB_*$, we have
\[
\left| (N_2/N_1) w_2(\vec s) - \frac{w_2(\vec r)w_1(\vec s)}{w_1(\vec r)}\right| < C N^{-3}(N_1-N_2) \vartheta_{i_4} |w_2(r)|\max_t \sH(t).
\]
Then by summing over $\vec s$, we have
\begin{multline*}
 2(N_1/N_2)^{k/2}\Bigg|   \sum_{\vec s\in\sB_*, \cC(\vec s)=\vec r} (N_2/N_1) w_2(\vec s) + \sum_{\vec s\in\sB_+\setminus \sB_*, \cC(\vec s)=\vec r} w_2(\vec s)  \\
-  \frac{w_2(\vec r)}{w_1(\vec r)} \sum_{\vec s\in\sB_+, \cC(\vec s)=\vec r} w_1(\vec s) \Bigg| < C\exp(CN^{-1/3}\max_t \sH(t)) N^{-5/3}(N_1-N_2) \vartheta_{i_4} |w_1(r)|.
\end{multline*}
Besides, we have
\begin{multline*}
\Bigg|  2(N_1/N_2)^{k/2}w_2(\vec r) - 2(N_2/N_1)^{\sH(3k)/2}w_1(\vec r) \Bigg| \\ < C(\exp(C(\delta N^{-1}+N^{-4/3}\max_t \sH(t) )(N_1-N_2))-1)  |w_1(\vec r)|.
\end{multline*}
The above two together imply
\begin{multline*}
\Bigg|  2(N_1/N_2)^{k/2} \Bigg( \sum_{\vec s\in\sB_*, \cC(\vec s)=\vec r} (N_2/N_1) w_2(\vec s) + \sum_{\vec s\in\sB_+\setminus \sB_*, \cC(\vec s)=\vec r} w_2(\vec s) \Bigg) \\
- \sum_{\vec s\in\sB_+, \cC(\vec s)=\vec r} 2(N_2/N_1)^{\sH(3k)/2}w_1(\vec s) \Bigg| < C\exp(CN^{-1/3}\max_t \sH(t)) N^{-5/3}(N_1-N_2) \vartheta_{i_4} |w_1(\vec r)| \\
+  C(\exp(C(\delta N^{-1}+N^{-4/3}\max_t \sH(t))(N_1-N_2))-1) \Bigg|\sum_{\vec s\in\sB_+, \cC(\vec s)=\vec r}w_1(\vec s)\Bigg|.
\end{multline*}
Then by \Cref{lem:eq:caneq}, and taking a summation over $\vec r$ using \Cref{prop:mbdc}, we can bound \eqref{eq:mlbd2} by $C(4NN_1)^{2k} N^{-1}(N_1-N_2)$.

\smallskip

\noindent\textbf{Bound \eqref{eq:mlbd2}.}
For any  $\vec r\in \sB_+$ we have
\[
 \left| 2(N_2/N_1)^{\sH(3k)/2} - 2 + N^{-1}(N_1-N_2)\sH(3k) \right|< CN^{-2}(N_1-N_2)^2(\max_t \sH(t) )^2 .
\]
Then with \Cref{lem:eq:caneq} and \Cref{cor:sumoverUh}, we can bound \eqref{eq:mlbd2} by  $C(4NN_1)^{2k}N^{-4/3}(N_1-N_2)^2$.

So far we obtained \eqref{eq:xdiffbd2}, therefore \eqref{eq:xdiffbd}. As in the previous proof, we can similarly estimate
\begin{itemize}
\item
$\E\left[\bx(\ttt_2)^2\bx(\ttt_1)^2- \bx(\ttt_2)\bx(\ttt_1)^3\right]$,
\item $\E\left[\bx(\ttt_2)^3\bx(\ttt_1)- \bx(\ttt_2)^2\bx(\ttt_1)^2\right]$,
\item $\E\left[\bx(\ttt_2)^4- \bx(\ttt_2)\bx(\ttt_1)^3\right]$,
\end{itemize}
respectively, through the actions of the operators
\begin{itemize}
    \item $\hP_k^{N_1}\left(\baP_k^{N_1}-k\beta(4N)^{-1}\hP_k^{N_1}\right)(\hP_k^{N_1})^2$,
    \item $(\hP_k^{N_1})^2\left(\baP_k^{N_1}-k\beta(4N)^{-1}\hP_k^{N_1}\right)\hP_k^{N_1}$,
    \item $(\hP_k^{N_1})^3\left(\baP_k^{N_1}-k\beta(4N)^{-1}\hP_k^{N_1}\right)$.
\end{itemize}
A linear combination of these estimates (as in the previous proof) and \Cref{cor:piden} imply the conclusion.
\end{proof}

\subsection{Topological statement}
In this section we present a general topological statement, which connects (uniform in compacts) convergence of families of continuous functions with convergence of sums of large powers of the kind we have in Theorems \ref{thm:multil} and \ref{thm:multit}. All functions in this section are deterministic.
Throughout this and the next subsection we work with $\R\cup\{-\infty\}$, whose topology is generated by open intervals $(a, b)$ and $[-\infty, a)$ for all $a<b\in\R$.

For any $x\in\R$, $\alpha>0$, $N\in\N$, we denote
\[
M_+[x,\alpha,N]= \frac{1}{2}(1+N^{-2/3}x)^{2\lfloor \alpha N^{2/3}/2\rfloor},\quad
M_-[x,\alpha,N]= \frac{1}{2}(1+N^{-2/3}x)^{2\lfloor \alpha N^{2/3}/2\rfloor+1},
\]
and $M[x,\alpha,N]=M_+[x,\alpha,N]+M_-[x,\alpha,N]$.

\begin{prop}  \label{prop:keyconv}
Take any compact interval $I\subset\R$.
For integers $1\le i\le N$, let $f^N_i:I\to\R$ be continuous functions, satisfying $f^N_i\ge f^N_{i+1}$ for each $1\le i \le N-1$.
Suppose there exists an increasing sequence of positive integers $N_1<N_2<\cdots$, such that for any $\alpha\in\N$ or $\alpha^{-1}\in\N$, as $k\to\infty$, $\sum_{i=1}^{N_k} M[f_i^{N_k},\alpha,N_k]$ converges uniformly towards a continuous function from $I$ to $\R$, and $\sum_{i=1}^{N_k} M_+[f_i^{N_k},\alpha,N_k]$ is bounded by a constant independent of $k$ (the constant allowed to depend on the sequence and $\alpha$).
Then for each $i\in\N$, as $k\to\infty$,
we have uniform convergence of $f_i^{N_k}$, as continuous functions from $I$ to $\R\cup\{-\infty\}$.
Moreover, if we denote the limit by $f_i$, then for any $\alpha>0$,
$\sum_{i=1}^{N_k} M[f_i^{N_k},\alpha,N_k]\to \sum_{i=1}^\infty \exp(\alpha f_i)$ uniformly over $I$ as $k\to\infty$.
\end{prop}
\begin{proof}
We also denote (for any $x\in\R$, $\alpha>0$, $N\in\N$)
\[
M_+'[x,\alpha,N] = \left(\max\{1+N^{-2/3}x, 0\}\right)^{2\lfloor \alpha N^{2/3}/2\rfloor}.
\]
Then for $x>-N^{2/3}$,
\[
|M_+'[x,\alpha,N]-M[x,\alpha,N]| = N^{-2/3}|x|M_+[x,\alpha,N].
\]
Take any $\delta_-, \delta_+>0$.
Note that $|x|\le \delta_+^{-1}(1+N^{-2/3}x)^{\delta_+ N^{2/3}}$ when $x>0$ and $|x|\le \delta_-^{-1}(1+N^{-2/3}x)^{-\delta_- N^{2/3}}$ when $-N^{2/3}<x<0$. So for $x>-N^{2/3}$, $\alpha>\delta_-$, and $N$ large enough depending on $\delta_-, \delta_+$, we have
\begin{equation}  \label{eq:mpmdif}
|M_+'[x,\alpha,N]-M[x,\alpha,N]| \le 2 N^{-2/3}(\delta_+^{-1}M_+[x,\alpha+\delta_+,N]+\delta_-^{-1} M_+[x,\alpha-\delta_-,N]).
\end{equation}
For $-\infty<x<-N^{2/3}$ we have
\[
|M_+'[x,\alpha,N]-M[x,\alpha,N]| = |2+xN^{-2/3}|M_+[x,\alpha,N].
\]
Let $x'=-2N^{2/3}-x>-N^{2/3}$. Then $|2+xN^{-2/3}|M_+[x,\alpha,N]=|x'|M_+[x',\alpha,N]$, $M_+[x,\alpha+\delta_+,N]=M_+[x',\alpha+\delta_+,N]$, and $M_+[x,\alpha-\delta_-,N]=M_+[x',\alpha-\delta_-,N]$.
Thus by applying the $x>-N^{2/3}$ case arguments with $x$ replaced by $x'$, we also get \eqref{eq:mpmdif}.
When $x=-N^{2/3}$, \eqref{eq:mpmdif} also holds since $M_+'[-N^{-2/3},\alpha,N]=M[-N^{-2/3},\alpha,N]=0$

Next, we combine \eqref{eq:mpmdif} with $\sum_{i=1}^{N_k} M_+[f_i^{N_k},\alpha,N_k]$ being bounded for all $\alpha\in\N$ and all $\alpha^{-1}\in\N$. We conclude that for any $\alpha>0$,  as $k\to\infty$, $\sum_{i=1}^{N_k} M_+'[f_i^{N_k}(s_k),\alpha,N_k]$ converges
if and only if $\sum_{i=1}^{N_k} M[f_i^{N_k}(s_k),\alpha,N_k]$ converges. The limit would be the same if both converge.

We now prove the convergence of $f_1^{N_k}$. The idea is that for a large $\alpha$ only the first term in $\sum_{i=1}^{N_k} M_+'[f_i^{N_k}(s_k),\alpha,N_k]$ significantly contributes to the $k\to\infty$ limit. For the formal proof, we argue by contradiction, and suppose that the convergence does not hold.
Note that $\limsup_{k\to\infty}\sup_I f_1^{N_k}<\infty$, because $M_+[f_1^{N_k},1,N_k]$ is uniformly bounded.
Then there exist two disjoint subsequences of $N_1<N_2<\cdots$, denoted by $P_1<P_2<\cdots$ and $Q_1<Q_2<\cdots$, and some $t_1, t_2, \cdots \to t$ and $s_1, s_2, \cdots \to t$ in $I$, such that $f_1^{P_k}(t_k)\to a$ and $f_1^{Q_k}(s_k)\to b$ as $k\to\infty$, where $a,b \in \R\cup\{-\infty\}$ with $a<b$.

Note that for any $\alpha\in\N$, since $\lim_{k\to\infty} f_1^{Q_k}(t_k)=b$, we have
\[\lim_{k\to\infty}\sum_{i=1}^{P_k} M_+'[f_i^{P_k}(t_k),\alpha,P_k]=\lim_{k\to\infty}\sum_{i=1}^{Q_k} M_+'[f_i^{Q_k}(s_k),\alpha,Q_k]\ge \exp(\alpha b).\]
On the other hand, since $\lim_{k\to\infty} f_1^{P_k}(t_k)=a$, and $f_1^{P_k}(t_k)\ge f_i^{P_k}(t_k)$ for any $i\in\N$, we have
\[
\lim_{k\to\infty} \sum_{i=1}^{P_k}  M_+'[f_i^{P_k}(t_k),\alpha,P_k] \le \exp((\alpha-1)a) \lim_{k\to\infty} \sum_{i=1}^{P_k} M_+'[f_i^{P_k}(t_k),1,P_k].
\]
Combining the above two inequalities, we obtain
\[
\exp(\alpha(b-a)+a) \le \lim_{k\to\infty} \sum_{i=1}^{P_k} M_+'[f_i^{P_k}(t_k),1,P_k]
\]
Since $a<b$, and the right-hand side is finite, the above inequality can not hold for large enough $\alpha$, and we arrive at a contradiction. Therefore we get the uniform convergence of $f_1^{N_k}$ as $k\to\infty$.

From the convergence of $f_1^{N_k}$, $\sum_{i=2}^{N_k} M[f_i^{N_k},\alpha,N_k]$ converges uniformly as $k\to\infty$, and  $\sum_{i=2}^{N_k} M_+[f_i^{N_k},\alpha,N_k]$ is bounded by a constant independent of $k$.
Then we can deduce the convergence of $f_2^{N_k}$ using the same arguments; and further, we can inductively get the convergence of $f_i^{N_k}$ for each $i\in\N$.

We next take any $\alpha>0$, and prove
\begin{equation} \label{eq:alphaconv}
    \sum_{i=1}^{N_k} M_+'[f_i^{N_k},\alpha,N_k]\to \sum_{i=1}^\infty \exp(\alpha f_i).
\end{equation}
Denote $X=\limsup_{k\to\infty}\sum_{i=1}^{N_k} M_+'[f_i^{N_k},\alpha,N_k]$.
Take $\alpha_1^{-1}\in\N$ and $\alpha_2\in\N$ such that $\alpha_1<\alpha<\alpha_2$.
From $\sum_{i=1}^{N_k} M_+'[f_i^{N_k},\alpha,N_k]\le \sum_{i=1}^{N_k} M_+'[f_i^{N_k},\alpha_1,N_k]+\sum_{i=1}^{N_k} M_+'[f_i^{N_k},\alpha_2,N_k]$ we have that $X<\infty$.
Take any $j\in\N$.
Note that for any $j'\in \llbracket j, N_k\rrbracket$,
\[
 M_+'[f_{j'}^{N_k},\alpha,N_k] \le j^{-1}\sum_{i=1}^{N_k} M_+'[f_i^{N_k},\alpha,N_k].
\]
Therefore,
\[
\limsup_{k\to\infty}\sum_{i=j}^{N_k} M_+'[f_i^{N_k},\alpha,N_k]
\le (j^{-1}X)^{1-\alpha_1/\alpha} \lim_{k\to\infty} \sum_{i=1}^{N_k} M_+'[f_i^{N_k},\alpha_1,N_k].
\]
This implies that
\[
\lim_{j\to\infty}\limsup_{k\to\infty}\sum_{i=j}^{N_k} M_+'[f_i^{N_k},\alpha,N_k] = 0.
\]
Then by the convergence of $f_i^{N_k}\to f_i$ for each $i\in\N$, we get \eqref{eq:alphaconv}.
\end{proof}

\subsection{Distributional convergence and uniqueness of the limiting process} \label{Section_identification}
Finally we prove the uniqueness of Airy$_\beta$ line ensemble, as well as the convergence for both the DBM and the G$\beta$E corners processes, thereby finishing the proofs of our main results.

\begin{proof}[Proof of \Cref{thm:main}, \Cref{thm:cor-conv}, and \Cref{thm:dbm-conv}]
Our proof consists of two steps. 
First, we show that both $\{\by_i^{(N)}\}_{i=1}^\infty$ and $\{\bY_i^{(N)}\}_{i=1}^\infty$ are tight, and the moments of Laplace transform for any subsequential limit are given by $\bL_\beta(\vec\bk,\vec\ttt)$ from \Cref{defn:core}.
This gives the existence part of \Cref{thm:main}.
Second, we prove the uniqueness part, with which we also get \Cref{thm:cor-conv} and \Cref{thm:dbm-conv}.
\medskip

\noindent\textbf{Step 1: tightness and subsequential limit.}
It is classical (by e.g., \cite[Theorem 7.3]{Bil}) that from \Cref{thm:multit} and \Cref{prop:mom-dbm}, for any fixed $\bk>0$,
$\sum_{i=1}^N M[\bY_i^{(N)}/2,\bk,N]$ is tight, as continuous functions from $[0,\infty)$ to $\R$, under the uniform in compact topology. 
Note that $\{Y_i(\tau)\}_{1\leq i\leq N, \tau>0}$ and $\{-Y_i(\tau)\}_{1\leq i\leq N, \tau>0}$ have the same law, therefore $\sum_{i=1}^N (M_+[\bY_i^{(N)}/2,\bk,N]-M_-[\bY_i^{(N)}/2,\bk,N])$ has the same law as $\sum_{i=1}^N M[\bY_i^{(N)}/2,\bk,N]$, and is also tight.
Taking the sum or the difference of the above two processes together, we have that $\sum_{i=1}^N M_+[\bY_i^{(N)}/2,\bk,N]$ and $\sum_{i=1}^N M_-[\bY_i^{(N)}/2,\bk,N]$ are also tight. (See \cite[Chapter 3]{EthierKurtz} for general information on tightness in the spaces of continuous functions.)\footnote{We note that the use of symmetry (in deriving the tightness of $\sum_{i=1}^N M_+[\bY_i^{(N)}/2,\bk,N]$ and $\sum_{i=1}^N M_-[\bY_i^{(N)}/2,\bk,N]$) is not essential. 
Alternatively, one-point tightness of $\sum_{i=1}^N M_+[\bY_i^{(N)}/2,\bk,N]$ and $\sum_{i=1}^N M_-[\bY_i^{(N)}/2,\bk,N]$ can be obtained using \Cref{prop:csBJ}, 
and bounds on time increments are again from
\Cref{prop:mom-dbm}.}

Thus, we can take a subsequence so that both $\sum_{i=1}^N M_+[\bY_i^{(N)}/2,\bk,N]$ and $\sum_{i=1}^N M_-[\bY_i^{(N)}/2,\bk,N]$ converge for each $\bk\in\N$ and $\bk^{-1}\in\N$; and by Skorokhod's representation theorem (see e.g., \cite[Chapter 1, Theorem 6.7]{Bil}) we can make it almost sure convergence.
By \Cref{prop:keyconv}, $\bY_i^{(N)}$ for each $i$ converges along this subsequence. 

Take any  $m\in \N$, $\vec \bk \in \R^m_+$. By \Cref{thm:multit}, $\prod_{j=1}^m (\sum_{i=1}^N M[\bY_i^{(N)}/2,\bk_j,N])$ is uniformly integrable in $N$.
Then  from the second part of \Cref{prop:keyconv} and \Cref{thm:multit}, the Laplace transforms of the limit of $\{\bY_i^{(N)}\}_{i=1}^\infty$ have moments given by \Cref{thm:main}. 

Using the same arguments, and \Cref{thm:multil}, \Cref{prop:mom-cor} instead of \Cref{thm:multit}, \Cref{prop:mom-dbm}, 
we also get the convergence of the rescaled Gaussian $\beta$ corners process $\{\by_i^{(N)}\}_{i=1}^\infty$, and the Laplace transforms of the limit have moments given by \Cref{thm:main}.
\footnote{For G$\beta$E corners process, $\by_i^{(N)}(\ttt)$ is not continuous in $\ttt$ due to discretization, and is also not defined when $i>N-\ttt N^{2/3}$. Thus to apply \Cref{prop:keyconv}, we need to consider a different $\tilde\by_i^{(N)}(\ttt)$, constructed as follows. For $\ttt\in N^{-2/3}\Z$, $0\le \ttt <N^{1/3}$, we let $\{\tilde\by_i^{(N)}(\ttt)\}_{i=1}^N$ contain the same numbers as $\{\by_i^{(N)}(\ttt)\}_{i=1}^{N-\ttt N^{2/3}}$ and in addition $\ttt N^{2/3}$ copies of $-N^{2/3}$ (note that $M_+[-N^{-2/3},\bk,N]=M_-[-N^{-2/3},\bk,N]=0$). For generic $\ttt$ we linearly interpolate between $\ttt\in N^{-2/3}\Z$. Then \Cref{prop:keyconv} applies and the same arguments go through.}

Note that so far we have established the uniform in compact subsequential convergence for $\bY_i^{(N)}$ and $\by_i^{(N)}$
and moments for the Laplace transform of the limit for $\ttt\in [0,\infty)$, and we need to extend this to $\ttt\in\R$.
For $\bY_i^{(N)}$, using that the DBM is invariant under diffusive scaling, the convergence can be extended to $[-C,\infty)$ for any $C>0$, thereby to $\R$.
For $\by_i^{(N)}$, note that for any $C>0$ we also have the same convergence for $\by_i^{(\lfloor N+CN^{2/3}\rfloor)}$; thus we have the convergence of $\bY_i^{(N)}$ on $[-C,\infty)$, thereby we get the extension to $\R$.
\medskip

\noindent\textbf{Step 2: uniqueness.}
Take  $\{\AB_{1,i}(\ttt)\}_{i=1}^\infty$ to be a subsequential limit of $\{\bY_i^{(N)}(\ttt)\}_{i=1}^\infty$ constructed in Step 1. Suppose that there is another processes $\{\AB_{2,i}(\ttt)\}_{i=1}^\infty$, such that both $\{\AB_{1,i}(\ttt)\}_{i=1}^\infty$ and $\{\AB_{2,i}(\ttt)\}_{i=1}^\infty$ have continuous sample paths in $\ttt$, and both have moments of their Laplace transforms given by \Cref{thm:main} and \Cref{defn:core}. 

Now for each $\iota=1, 2$ and any $\ttt, h\in \R$, we denote $f_\iota(\ttt, h)=\max\{i\in \N: \AB_{\iota,i}(\ttt)>h\}$ (with the convention $\max\emptyset=0$).
We would like to prove that $\{\AB_{1,i}(\ttt)\}_{i=1}^\infty$ and $\{\AB_{2,i}(\ttt)\}_{i=1}^\infty$ have the same distributions and for that it is sufficient to show that $f_1$ and $f_2$ have the same finite dimensional distributions.

From the moments of Laplace transforms in \eqref{eq_expected_Laplace}, via an integration by parts, we have that for any $\vec\bk\in\R_+^m$ and $\vec \ttt\in\R^m$,
\[
\int_{\R^m} \exp\left(\sum_{\ell=1}^m \bk_\ell h_\ell\right) \E \left[ \prod_{\ell=1}^m f_1(\ttt_\ell, h_\ell)\right]  \d \vec h = \int_{\R^m} \exp\left(\sum_{\ell=1}^m \bk_\ell h_\ell\right) \E \left[ \prod_{\ell=2}^m f_2(\ttt_\ell, h_\ell)\right]  \d \vec h.
\]
Then by the uniqueness theorem for Laplace transform (i.e. Lerch's theorem \cite{Lerch}), for any $\vec h\in \R^m$ except for a Lebesgue measure zero set,
\[
\E \left[ \prod_{\ell=1}^m f_1(\ttt_\ell, h_\ell)\right] = \E \left[ \prod_{\ell=2}^m f_2(\ttt_\ell, h_\ell)\right].
\]
Since the expectations are right continuous in each $h_\ell$, this equality holds for any $\vec h\in \R^m$.
Namely, we have that $f_1$ and $f_2$ equal in any finite mixed moments.

Next, we claim that for any fixed $\ttt, h$, the tails of $f_1(\ttt, h)$ have exponential decay. With this in hand, we have that the expectation of $\exp(\epsilon f_1(\ttt,h))$ is finite for some $\epsilon>0$, hence the one-dimensional marginal of $f_{1}(\ttt,h)$ is uniquely determined by its one-dimensional moments, see e.g \cite[Theorem 30.1]{Bill}. As for the finite dimensional distribution, Peterson's theorem (see \cite[Theorem 14.6]{momentprob}) says that it is uniquely determined by its mixed moments as long as this is the case for all its marginals. Therefore we have that the mixed moments of $f_{1}$ uniquely determine its finite dimensional distribution, which implies that for any $\vec{\tau}=(\ttt_1,...,\ttt_{m})\in \R^{m}$, $\vec{h}=(h_{1},...,h_{m})\in \R^{m}$, $(f_1(\ttt_1,h_1),...,f_1(\ttt_m,h_m))$ and $(f_2(\ttt_1,h_1),...,f_2(\ttt_m,h_m))$, must be equal in distribution, because they share the same mixed moments. This further implies that $\{\AB_{1,i}(\ttt)\}_{i=1}^\infty$ and $\{\AB_{2,i}(\ttt)\}_{i=1}^\infty$ also have the same finite dimensional distribution.

To get the exponential tail probability of $f_{1}(\ttt,h)$, note that for any $i\in\N$, $\PP[f_1(\ttt, h)\ge i]=\PP[\AB_{1,i}(\ttt)>h]$. Moreover, since we obtain $\{\AB_{1,i}(\ttt)\}_{i=1}^{\infty}$ as a subsequential limit of  $\{\bY_i^{(N)}(\ttt)\}_{i=1}^\infty$, for any fixed $\ttt$, $\{-\AB_{1,i}(\ttt)\}_{i=1}^{\infty}$ is equal in distribution to
the eigenvalues of the Stochastic Airy Operator, by \cite[Theorem 1.1]{RRV}.
Then the exponential decay of the tail of $f_1(\ttt, h)$ follows from \Cref{lem:bdabci} below (deduced from \cite[Theorem 1.2]{RRV}).
This finishes the proof of the uniqueness.
\end{proof}
Let us state the tail bound of the Stochastic Airy Operator, used in the previous proof.
Let $\Lambda_0\le \Lambda_1 \le \cdots$ denote the eigenvalues of the Stochastic Airy Operator, as in \cite{RRV}.
They are also the edge limit of the Gaussian $\beta$ Ensemble from \eqref{eq_GBE_t}, see \cite[Theorem 1.1]{RRV}. 
\begin{lemma}  \label{lem:bdabci}
For any $h\in\R$, there exists some $0<c<1$ (depending on $h$), such that $\PP[\Lambda_{i-1}<-h]<c^i$ for any $\tau\in\R$ and $i\in\N$.
\end{lemma}

\begin{proof}
For $h\in\R$, let $p=p_h$ be a one-dimensional diffusion on $t\in[h, \infty)$ defined by the It\^{o} equation:
\begin{equation}\label{eq_ricatti}
     dp(t)=\left(t-p^2(t)\right)dt+\frac{2}{\sqrt{\beta}}dB(t),
\end{equation}
such that it starts from $+\infty$ at time $h$; and whenever it arrives at $-\infty$, it starts again from $+\infty$.
\cite[Theorem 1.2]{RRV} states that
for any $i\in\N$, $\PP[\Lambda_{i-1}<-h]$ equals the probability of the event that $p_h$ comes through $-\infty$ at least $i$ times at times $t\ge h$.

From the diffusion $p_h$, it can also be derived that, if $q(h)$ is the probability that $p_h$ arrives at $-\infty$ in finite time, then (A) $q(h)<1$ for any $h$, and (B) $q(h)$ is non-increasing in $h$.

 To see (A), note that by \cite[Theorem 1.2]{RRV},
 \begin{equation}\label{eq_qh}
     q(h)=1-\PP[\Lambda_{0}\ge-h],
 \end{equation}
    which is in $(0,1)$ for all $h\in \R$ (see \cite[Theorem 1.3]{RRV} for a tail estimate of $\Lambda_{0}$). (B) also follows directly from (\ref{eq_qh}).   

From (B) we have $\PP[\Lambda_{i-1}<-h]<(q(h))^i$, and by (A) the conclusion follows.
\end{proof}

The just presented proof can be also used to deduce the following continuity estimate for the Laplace transform of $\{\AB_i(\ttt)\}_{i=1}^\infty$. This implies that the Laplace transform is H\"older $\frac{1}{4}-$ function in $\ttt$; with additional efforts the estimate can be upgraded to H\"older $\frac{1}{2}-$ class.
\begin{corollary}   \label{cor:last}
For any $C_1>0$, there exists $C=C(C_1)>0$ such that the following holds.
For any $-C_1\le \ttt_1<\ttt_2 \le C_1$, $C_1^{-1}<\bk<C_1$, we have
\[
\E \left[  \left( \sum_{i=1}^\infty \exp(\bk \AB_i(\ttt_1)/2) - \sum_{i=1}^\infty \exp(\bk \AB_i(\ttt_2)/2) \right)^4  \right]
< C (\ttt_2-\ttt_1)^2.
\]
\end{corollary}
\begin{proof}
From \Cref{thm:main} and Theorems \ref{thm:multil}, \ref{thm:multit}, we have that the above expectation is precisely the $N\to\infty$ limit of the expectations in \Cref{prop:mom-cor} or \Cref{prop:mom-dbm}. Thus by \Cref{prop:mom-cor} or \Cref{prop:mom-dbm} the conclusion follows.
\end{proof}
\begin{remark}
One can write the expectation in \Cref{cor:last} through \Cref{thm:main}, then Taylor expand in $\ttt_1-\ttt_2$. \Cref{cor:last} says that the coefficient of $\ttt_1-\ttt_2$, which is a certain combination of integrals, miraculously vanishes. This is a non-trivial integral identity, which eventually follows from the Dunkl operator identity of \Cref{cor:piden}.
\end{remark}

\appendix

\section{Quantitative approximation of random walks}

\subsection{Upper bounds for numbers of walks}  \label{ssec:expvar}
The following estimates are used in the proof of \Cref{prop:mbdc}.

We let  $\overline\sF(X;H,G)$ be the collection of all $F:\llbracket 0,X\rrbracket\to \Z$, satisfying $|F(t)-F(t-1)|=1$ for each $t\in\llbracket 1,X\rrbracket$, and $F(0)=H$, $F(X)=G$. 
We also let $\sF(X;H,G)$ be the subset of $\overline\sF(X;H,G)$, consisting of those $F$ which stay non-negative.
And we use $C,c>0$ to denote large and small universal constants, whose values may change from line to line.
\begin{lemma}  \label{lem:countwk}
Take any $X, H \in \Z_{\ge 0}$, and $M\ge 0$. There exist $C,c>0$ such that:
\begin{itemize}
    \item[(i)] $\sum_{G\in \Z_{\ge 0}, |G-H|\ge M}|\sF(X;H,G)|\le C2^X\exp(-cM^2/X)$.
    \item[(ii)] $|\sF(X;H,G)|\le CX^{-1/2}2^X\exp(-cM^2/X)$, for any $G \in \Z_{\ge 0}$, such that $|G-H|\ge M$.
    \item[(iii)] $\sum_{G\in \Z_{\ge 0}, G\ge M}|\sF(X;0,G)|\le CX^{-1/2}2^X\exp(-cM^2/X)$.
    \item[(iv)]  $|\sF(X;H,0)|\le CX^{-1}2^X\exp(-cM^2/X)$, whenever $H\ge M$.
    \item[(v)] $|\sF(X;0,0)|\le CX^{-3/2}2^X$.
\end{itemize}
\end{lemma}
\begin{proof}
For some choices of parity of $X$, $H$, $G$, the set $\sF(X;H,G)$ is empty and the inequalities are trivially true. Throughout the proof we silently assume that the parities are chosen so that  $\sF(X;H,G)$ has at least one element.
Besides, without loss of generality, we assume that $M$ is an even integer.

We repeatedly use the reflection principle throughout the proof (see, e.g.~\cite[Theorem 4.9.1]{DBOOK}). Namely, for any $X, H, G\in \Z_{\ge 0}$, we have
\[
|\sF(X;H,G)| = |\overline\sF(X;H,G)| - |\overline\sF(X;H,-G-2)|.
\]
Also recall that
\[
|\overline\sF(X;H,G)| = {X \choose (X+H-G)/2}.
\]
We will also use asymptotic bounds on binomial coefficients: for any integers $0<A\le B$, we have ${B\choose A} < 2^BB^{-1/2} \exp(-(2A-B)^2/(2B))$.
This can be obtained by the Stirling's formula.

The bound (i) follows from $\sum_{G\in \Z_{\ge 0}, |G-M|\ge M}|\sF(X;H,G)|\le \sum_{G\in \Z, |G-H|\ge M}|\overline\sF(X;H,G)| = \sum_{G\in \Z, |G-H|\ge M}{X\choose (X+H-G)/2} \le C2^X\exp(-cM^2/X)$.

For (ii), it follows from that $|\sF(X;H,G)|\le |\overline\sF(X;H,G)|={X\choose (X+H-G)/2}$.

For (iii), by reflection principle, the sum equals $|\overline\sF(X;0,M)|$ or $|\overline\sF(X;0,M+1)|$, depending on the parity of $X+M$. Therefore the estimate follows from the bound on the binomial coefficient.

For (iv), by reflection principle, $|\sF(X;H,0)|$ equals $|\overline\sF(X;H,0)|-|\overline\sF(X;H,-2)|={X\choose (X+H)/2}-{X\choose (X+H)/2 + 1}={X\choose (X+H)/2} \cdot \frac{H+1}{(X+H)/2+1}$, which is bounded as desired.

For (v), it is the Catalan number ${X\choose X/2} \cdot \frac{1}{X/2+1}$, and the bound holds.
\end{proof}

We will need the following statement on the monotonicity of random walk bridges.
\begin{lemma}   \label{lem:mon}
Take any $X, H, G, R, L, H', G', R', L' \in \Z_{\ge 0}$, with $X+H+G$ and $X+H'+G'$ even, and $H\le H'$, $G\le G'$, $R\le R'$, $L\le L'$.
Sample $\mathfrak{f}:\llbracket 0, X\rrbracket \to \Z_{\ge 0}$ uniformly from $\overline\sF(X;H,G)$, conditional on $R\le \min \mathfrak{f}\le \max \mathfrak{f}\le L$; and $\mathfrak{f}':\llbracket 0, X\rrbracket \to \Z_{\ge 0}$ uniformly from $\overline\sF(X;H',G;)$, conditional on $R'\le\min \mathfrak{f}'\le \max \mathfrak{f}'\le L'$.
Then we can couple $\mathfrak{f}$ and $\mathfrak{f}'$ in such a way that $\mathfrak{f}(t)\le \mathfrak{f}'(t)$ for all $t\in\llbracket 0, X\rrbracket$.
\end{lemma}
\Cref{lem:mon} can be proved using a coupling of Monte-Carlo Markov chains for sampling $\mathfrak{f}$ and $\mathfrak{f}'$. 
See e.g.~Lemmas 2.6 and 2.7 in \cite{CH} and their proofs.

\begin{lemma}  \label{lem:eva}
Take any $X, H, G, L \in \Z_{\ge 0}$ with $X+H+G$ even, and sample $\mathfrak{f}:\llbracket 0, X\rrbracket \to \Z_{\ge 0}$ uniformly from $\sF(X;H,G)$, conditional on $\max \mathfrak{f}\le L$.
Take $k\in\N$ and discrete intervals $\llbracket X_1, Y_1\rrbracket, \ldots, \llbracket X_k, Y_k\rrbracket$ contained in $\llbracket 0, X\rrbracket$.
Assume that they are disjoint from each other, except for at endpoints  and also assume that $Y_i-X_i\ge 1$ for all $i$. Take $a_1,\ldots, a_k \in \N$, and denote $a=\sum_{i=1}^k a_i$.  For each $i\in\llbracket 1, k\rrbracket$, denote $V_i = \max_{\llbracket X_i, Y_i\rrbracket}\mathfrak{f} - \min_{\llbracket X_i, Y_i\rrbracket}\mathfrak{f}+1$.
We have
\begin{equation}  \label{eq:eva11}
\E\left[\prod_{i=1}^k \frac{V_i^{a_i}}{a_i!} \right] < C^a \log(a+1)^a (a^{-a/2}X^{a/2} + a^{-a} |H-G|^a )
\end{equation}
and for any $K\in \N$,
\begin{multline}  \label{eq:eva22}
\E\left[\prod_{i=1}^k \frac{V_i^{a_i}}{a_i!} (\don[\max \mathfrak{f} \ge \max\{H,G\}+K] + \don[\min \mathfrak{f} \le \min\{H,G\}-K]) \right] \\ < C^a \log(a)^a (a^{-a/2}X^{a/2} + a^{-a} (|H-G|+K)^a ) \exp(-cK^2/X).
\end{multline}
We emphasize that the constant $C$ does not depend on $k$, i.e., the number of intervals.
\end{lemma}

\begin{proof}
Without loss of generality we assume that $H\le G$.

We start by deriving \eqref{eq:eva11}.
For each $i\in \llbracket 1,k\rrbracket$, we write $Z_i=\frac{V_i-(G-H)(Y_i-X_i)/X}{(Y_i-X_i)^{1/2}}$.
\begin{claim} \label{cla}
    We have that $\PP[Z_i>x] < C\exp(-cx^2)$ for any $x>0$, where $C,c>0$ are universal constants (i.e., independent of $X, H, G, L, X_i, Y_i$).
\end{claim}
We postpone the proof of this claim, and proceed assuming it.

We write $Z=\max\{0, Z_1, \ldots, Z_k\}$.
Then $\PP[Z>x] < Ck\exp(-cx^2)$ for any $x>0$, and $\PP[Z>x\sqrt{\log(k+1)}] < Ck\exp(-c(k+1)x^2) < C\exp(-cx^2)$.
This implies that $\frac{Z}{C\sqrt{\log(k+1)}}$ is stochastically dominated by the absolute value of a standard normal random variable.
Thus we have $\E[Z^m]<(C\log(k+1)m)^{m/2}$ for any $m\in \N$.

Now to prove \eqref{eq:eva11}, we note that for any $i\in\llbracket 1, k\rrbracket$,
\begin{equation}  \label{eq:evapff0}
\frac{V_i^{a_i}}{a_i!} = \sum_{p_i=0}^{a_i} \frac{Z_i^{p_i} (Y_i-X_i)^{p_i/2} ((G-H)(Y_i-X_i)/X)^{q_i} }{p_i! q_i!},
\end{equation}
where $q_i=a_i-p_i$.
Therefore we consider
\begin{equation}  \label{eq:evapff1}
\E\left[ \prod_{i=1}^k \frac{Z^{p_i} (Y_i-X_i)^{p_i/2} ((G-H)(Y_i-X_i)/X)^{q_i} }{p_i! q_i!} \right],
\end{equation}
where $p_i, q_i \in \Z_{\ge 0}$ and $p_i+q_i=a_i$ for each $i\in \llbracket 1, k\rrbracket$.
Denote $p=\sum_{i=1}^k p_i$ and $q=\sum_{i=1}^k q_i$.
We have $\E[Z^{p}] < (C\log(k+1) p )^{p/2}$, and $\prod_{i=1}^k p_i! q_i! > c^a \prod_{i=1}^k p_i^{p_i}q_i^{q_i}$.
Also, since the logarithm function is concave, we have
\[
\sum_{i=1}^k p_i \log\left(p_i^{-1}(Y_i-X_i)\right) \le p \log\left( p^{-1}\left(\sum_{i=1}^k Y_i-X_i\right) \right) \le p \log(p^{-1}X),
\]
\[
\sum_{i=1}^k q_i \log\left(q_i^{-1}(Y_i-X_i)\right) \le q \log\left( q^{-1}\left(\sum_{i=1}^k Y_i-X_i\right) \right) \le q \log(q^{-1}X),
\]
so
\[\prod_{i=1}^k(Y_i-X_i)^{p_i/2} \le \frac{\prod_{i=1}^k p_i^{p_i/2}}{ p^{p/2} } X^{p/2} ,\quad \prod_{i=1}^k(Y_i-X_i)^{q_i} \le \frac{\prod_{i=1}^k q_i^{q_i}}{ q^{q} } X^{q} .\]
Thus, we bound \eqref{eq:evapff1} by
\[
C^a\log(a+1)^a \frac{X^{p/2}}{ p^{p/2} } \cdot\frac{(G-H)^{q} }{q^q}.
\]
The inequality $y<e^y$ applied to $y=a/p,a/q$ leads to $(a/p)^p, (a/q)^q<e^a$, and, therefore, we have $\frac{1}{p^{p/2}q^q}<\frac{e^{3a/2}}{a^{p/2+q}}$.
Then we can further bound \eqref{eq:evapff1} by
\[
C^a\log(a+1)^a \frac{X^{p/2}(G-H)^{q}}{ a^{p/2+q}} < C^a\log(a+1)^a (a^{-a/2}X^{a/2}+a^{-a}(G-H)^a),
\]
where the inequality is by $\frac{X^{p/2}(G-H)^{q}}{ a^{p/2+q}}=\left(\frac{X^{1/2}}{a^{1/2}}\right)^p\left(\frac{G-H}{a}\right)^q \le \max\left\{\frac{X^{1/2}}{a^{1/2}}, \frac{G-H}{a}\right\}^{p+q}$, and $p+q=a$.
By summing over all (at most $2^a$) possible choices of $p_i, q_i$, and using \eqref{eq:evapff0}, we get \eqref{eq:eva11}.

For \eqref{eq:eva22}, note that $\PP[\max \mathfrak{f}\ge G+K]<C\exp(-cK^2/X)$. (This follows from the same arguments as bounding the probability of $\cE_2$ in the proof of \Cref{cla} below.)
For $\mathfrak{f}$ conditional on $\max \mathfrak{f}\ge G+K$, it can be sampled as follows.
First take a $Y\in \llbracket 0, X\rrbracket$ such that $Y+H+G+K$ is even.
Then take $\mathfrak{f}_-$ uniformly random from $\sF(Y;H,G+K)$ conditional on $\le G+K$, and $\mathfrak{f}_+$ uniformly random from $\sF(X-Y;G+K,G)$ conditional on $\le L$.
Finally one gets $\mathfrak{f}$ by concatenating $\mathfrak{f}_-$ and $\mathfrak{f}_+$.
It now suffices to show that
\begin{equation}   \label{eq:evaproof2}
\E\left[\prod_{i=1}^{k} \frac{V_{i,-}^{a_{i,-}}}{a_{i,-}!} \right]
\E\left[\prod_{i=1}^{k} \frac{V_{i,+}^{a_{i,+}}}{a_{i,+}!} \right] < C^a \log(a+1)^a (a^{-a/2}X^{a/2} + a^{-a} (G-H+K)^a ),
\end{equation}
where
\begin{itemize}
    \item $V_{i,-} = \max_{\llbracket X_i, Y_i\rrbracket} \mathfrak{f}_- - \min_{\llbracket X_i, Y_i\rrbracket} \mathfrak{f}_-$ and $a_{i,-}=a_i$ if $Y_i \le Y$; \\$V_{i,-}=1$ and $a_{i,-}=0$ if $X_i\ge Y$; \\ $V_{i,-} = \max_{\llbracket X_i, Y\rrbracket} \mathfrak{f}_- - \min_{\llbracket X_i, Y_i\rrbracket} \mathfrak{f}_-$ and $a_{i,-}$ be any number in $\llbracket 0, a_i\rrbracket$ if $X_i<Y<Y_i$.
    \item $V_{i,+} = \max_{\llbracket X_i-Y, Y_i-Y\rrbracket} \mathfrak{f}_+ - \min_{\llbracket X_i-Y, Y_i-Y\rrbracket} \mathfrak{f}_+$ and $a_{i,+}=a_i$ if $X_i \ge Y$; \\ $V_{i,+}=1$ and $a_{i,+}=0$ if $Y_i\le Y$; \\ $V_{i,+} = \max_{\llbracket 0, Y_i-Y\rrbracket} \mathfrak{f}_+ - \min_{\llbracket 0, Y_i-Y\rrbracket} \mathfrak{f}_-$ and $a_{i,+}= a_i-a_{i,-}$ if $X_i<Y<Y_i$.
\end{itemize}
We have that \eqref{eq:evaproof2} directly follows from applying \eqref{eq:eva11} to $\mathfrak{f}_-$ and $\mathfrak{f}_+$, respectively.
Then from \eqref{eq:evaproof2}, by summing over all
(at most $a+1$) possible choices of $a_{i,-}$ (when there is one $i\in \llbracket 1, k\rrbracket$ with $X_i<Y<Y_i$), and averaging over all $Y$, the conclusion follows.
The estimate for $\mathfrak{f}$ conditional on $\max \mathfrak{f}\le H-K$ follows similarly.
\end{proof}

\begin{proof}[Proof of \Cref{cla}]
Under the event where $Z_i>x$, one of the following must happen
\begin{itemize}
    \item[$\cE_1$:] $|\mathfrak{f}(Y_i)-\mathfrak{f}(X_i)|>(G-H)(Y_i-X_i)/X + x(Y_i-X_i)^{1/2}/3-1$; or
    \item[$\cE_2$:] $\max_{\llbracket X_i, Y_i\rrbracket} \mathfrak{f} > \max\{\mathfrak{f}(X_i), \mathfrak{f}(Y_i)\} + x(Y_i-X_i)^{1/2}/3$; or
    \item[$\cE_3$:] $\min_{\llbracket X_i, Y_i\rrbracket} \mathfrak{f} < \min\{\mathfrak{f}(X_i), \mathfrak{f}(Y_i)\} - x(Y_i-X_i)^{1/2}/3$.
\end{itemize}
We next show that the event $\cE_1$ is contained in the union of the following two events:
\begin{itemize}
    \item[$\cE_1'$:] $|\mathfrak{f}(Y_i)-\mathfrak{f}(X_i)|>(\mathfrak{f}(Y_i)-H)(Y_i-X_i)/Y_i + x(Y_i-X_i)^{1/2}/6-1$; or
    \item[$\cE_1''$:] $|\mathfrak{f}(Y_i)-\mathfrak{f}(X_i)|>(G-\mathfrak{f}(X_i))(Y_i-X_i)/(X-X_i) + x(Y_i-X_i)^{1/2}/6-1$.
\end{itemize}
Indeed, if none of $\cE_1'$ or $\cE_1''$ hold, while $\cE_1$ holds, then necessarily $(\mathfrak{f}(Y_i)-H)/Y_i>(G-H)/X$, thus, $\mathfrak{f}(Y_i)>H+(G-H)Y_i/X$, and similarly $\mathfrak{f}(X_i)<G-(G-H)(X-X_i)/X=H+(G-H)X_i/X$.
In particular $\mathfrak{f}(Y_i)>\mathfrak{f}(X_i)$.
Then, by taking the reverse inequality in $\cE_1'$ multiplied by $Y_i$, plus the reverse inequality in $\cE_1''$ multiplied by $X-X_i$, we obtain
\[
X(\mathfrak{f}(Y_i)-\mathfrak{f}(X_i)) \le (G-H)(Y_i-X_i) + (Y_i+X-X_i)(x(Y_i-X_i)^{1/2}/6-1).
\]
Dividing by $X$, we get a contradiction with $\cE_1$. This proves the desired $\cE_1\subset \cE'_1\cup \cE''_1$.

We next show that $\PP[\cE_1'\mid \mathfrak{f}(Y_i)]<C\exp(-cx^2)$. Note that while the left-hand side is a conditional probability, the bound is uniform and does not depend on $\mathfrak{f}(Y_i)$.

We first analyze the inequality in $\cE'_1$ without absolute value; namely, we prove
\begin{equation}  \label{eq:monapp}
\PP[\mathfrak{f}(X_i)< \mathfrak{f}(Y_i)X_i/Y_i +H(Y_i-X_i)/Y_i - x(Y_i-X_i)^{1/2}/6+1 \mid \mathfrak{f}(Y_i)] < C\exp(-cx^2).
\end{equation}
By \Cref{lem:mon}, it suffices to prove \eqref{eq:monapp} with $\mathfrak{f}$ replaced by $\mathfrak{f}'$, uniformly sampled from $\overline\sF(Y_i; H, \mathfrak{f}(Y_i))$ conditional on $\le \max\{\mathfrak{f}(Y_i), H\}$.
In order to prove the inequality involving $\mathfrak{f}'$, we count walks in $\overline\sF(Y_i; H, \mathfrak{f}(Y_i))$ satisfying certain conditions.
First, the number of walks in $\overline\sF(Y_i; H, \mathfrak{f}(Y_i))$ that are $\le \max\{\mathfrak{f}(Y_i), H\}$ equals $|\sF(Y_i;|H-\mathfrak{f}(Y_i)|, 0)|$, which (using the reflection principle) is the same as
\begin{equation}  \label{eq:spt1}
\begin{split}
&|\overline\sF(Y_i;|H-\mathfrak{f}(Y_i)|, 0)| - |\overline\sF(Y_i;|H-\mathfrak{f}(Y_i)|, -2)| \\ &= {Y_i\choose (Y_i+|H-\mathfrak{f}(Y_i)|)/2} - {Y_i\choose (Y_i+|H-\mathfrak{f}(Y_i)|+2)/2} \\ 
&=
\frac{2|H-\mathfrak{f}(Y_i)|+2}{ Y_i+|H-\mathfrak{f}(Y_i)|+2}|\overline\sF(Y_i;|H-\mathfrak{f}(Y_i)|, 0)|.
\end{split}    
\end{equation}
On the other hand, for walks in $\overline\sF(Y_i; H, \mathfrak{f}(Y_i))$ that are $\le \max\{\mathfrak{f}(Y_i), H\}$, and equal $\max\{\mathfrak{f}(Y_i), H\}-K$ at $X_i$ for some fixed $K\in \Z_{\ge 0}$, the number of such walks equals $|\sF(X_i;|H-\mathfrak{f}(Y_i)|, K)| \cdot |\sF(Y_i-X_i;0, K)|$ or $|\sF(X_i;0, K)| \cdot |\sF(Y_i-X_i;|H-\mathfrak{f}(Y_i)|, K)|$, depending on whether $\mathfrak{f}(Y_i)\ge H$ or not.
In the former case, using the reflection principle, the number of walks is bounded by
\begin{equation}  \label{eq:spt2}
\frac{(|H-\mathfrak{f}(Y_i)|+1)(K+1)}{X_i+2} |\overline\sF(X_i;|H-\mathfrak{f}(Y_i)|, K)| \cdot \frac{2K+2}{Y_i-X_i+2} |\overline\sF(Y_i-X_i;0, K)|.
\end{equation}
Note that by Stirling's formula, we have
\begin{multline*}
\frac{|\overline\sF(X_i;|H-\mathfrak{f}(Y_i)|, K)| \cdot |\overline\sF(Y_i-X_i;0, K)|}{|\overline\sF(Y_i;|H-\mathfrak{f}(Y_i)|, K)|}\\
< C\sqrt{\frac{Y_i+2}{(X_i+2)(Y_i-X_i+2)}} \exp\left( 
 - c \frac{(K-|H-\mathfrak{f}(Y_i)|(Y_i-X_i)/Y_i)^2(Y_i+2)}{(X_i+2)(Y_i-X_i+2)}\right)    .
\end{multline*}
Using this, if we sum \eqref{eq:spt2} over all $K> |H-\mathfrak{f}(Y_i)|(Y_i-X_i)/Y_i +x(Y_i-X_i)^{1/2}/6-1$, and divide by \eqref{eq:spt1}, the result would be upper bounded by $C\exp(-cx^2)$.
When $\mathfrak{f}(Y_i)<H$ the same upper bound can be derived similarly. 
These imply \eqref{eq:monapp}.

By the same argument we can analyze the inequality in $\cE_1'$ with the left-hand side replaced by $\mathfrak{f}(X_i)-\mathfrak{f}(Y_i)$ without absolute value. Then we get $\PP[\cE_1'\mid \mathfrak{f}(Y_i)]<C\exp(-cx^2)$, and similarly $\PP[\cE_1''\mid \mathfrak{f}(X_i)]<C\exp(-cx^2)$.
These together imply that $\PP[\cE_1]<C\exp(-cx^2)$.

We then consider $\cE_2\setminus \cE_1$. Without loss of generality, we assume that $Y_i-X_i$ is even, and the odd case follows similarly.

Conditional on $\mathfrak{f}(X_i)$ and $\mathfrak{f}(Y_i)$, the function $\mathfrak{f}$ on $\llbracket X_i, Y_i\rrbracket$ is a random Bernoulli random bridge, conditional on $\ge 0$ and $\le L$.
By \Cref{lem:mon}, $\PP[\cE_2\mid \mathfrak{f}(X_i), \mathfrak{f}(Y_i)]$ is bounded by the probability of $\max \mathfrak{f}' > x(Y_i-X_i)^{1/2}/3$, where $\mathfrak{f}'$ is uniformly chosen from $\sF(Y_i-X_i; \lfloor (Y_i-X_i)^{1/2}\rfloor, \lfloor (Y_i-X_i)^{1/2}\rfloor)$.
It then suffices to consider the number of walks in $\sF(Y_i-X_i; \lfloor (Y_i-X_i)^{1/2}\rfloor, \lfloor (Y_i-X_i)^{1/2}\rfloor)$ with maximum $>x(Y_i-X_i)^{1/2}/3$, and $|\sF(Y_i-X_i; \lfloor (Y_i-X_i)^{1/2}\rfloor, \lfloor (Y_i-X_i)^{1/2}\rfloor)|$, and upper bound their ratio. 
Using the reflection principle, we have
\begin{align*}
&|\sF(Y_i-X_i; \lfloor (Y_i-X_i)^{1/2}\rfloor, \lfloor (Y_i-X_i)^{1/2}\rfloor)| \\
&=
|\overline\sF(Y_i-X_i; 0, 0)| - |\overline\sF(Y_i-X_i; 0, 2\lfloor (Y_i-X_i)^{1/2}\rfloor+2)|
> \frac{1}{10} |\overline\sF(Y_i-X_i; 0, 0)|.
\end{align*}
On the other hand, for walks in $\sF(Y_i-X_i; \lfloor (Y_i-X_i)^{1/2}\rfloor, \lfloor (Y_i-X_i)^{1/2}\rfloor)$ with maximum $>x(Y_i-X_i)^{1/2}/3$,
their number is at most the number of walks in $\overline\sF(Y_i-X_i; 0, 0)$, with maximum $>(Y_i-X_i)^{1/2}(x/3-1)$.
Using the reflection principle, this is at most
$|\overline\sF(Y_i-X_i; 0, \lfloor x(Y_i-X_i)^{1/2}/3\rfloor)|$ (when $x>6$) or $|\overline\sF(Y_i-X_i; 0, 0)|$ (otherwise).
In either case, using Stirling's formula, the ratio of the last expression and $|\overline\sF(Y_i-X_i; 0, 0)|$ is $<C\exp(-cx^2)$.
We then conclude that $\PP[\cE_2\mid \mathfrak{f}(X_i), \mathfrak{f}(Y_i)]<C\exp(-cx^2)$, thus $\PP[\cE_2]<C\exp(-cx^2)$.
Similarly we have $\PP[\cE_3]<C \exp(-cx^2)$.

Putting all of the above bounds together, we get $\PP[Z_i>x] < C\exp(-cx^2)$.
\end{proof}

\subsection{Exact asymptotic walk counts}   \label{Appendix_2}

Recall the notations from \Cref{defn:bIs} and \Cref{defn:IXHG} (and also at the beginning of \Cref{ssec:expvar}).
We next provide more precise estimates on the size $|\sF(X;H,G)|$, which are used in \Cref{ssec:bdecom}.

\begin{lemma}  \label{lem:A1}
For any $C_1,C_2>0$, there exists a constant $C_3=C_3(C_1,C_2)>0$, such that for each $N=1,2,\dots$, and  $x,h,g\in\mathbb R_{\ge 0}$, and  $X, H, G\in \Z_{\ge 0}$, satisfying: $C_1 N^{-2/3}<x<C_2$ and $C_1 N^{-1/3}< h,g < C_2 N^{1/100}$, and $|X-xN^{2/3}|, |H-hN^{1/3}|, |G-gN^{1/3}|<C_2$, and $X+H+G$ is even, we have
\begin{multline}  \label{eq:A1bound}
\left|2^{-X-1}N^{1/3}|\sF(X;H,G)| - \bF(x;h,g) \right| \\ <  C_3\bF(x;h,g) \left(1\wedge (N^{-3/5}x^{-3} + N^{-1/10} + N^{-1/3}/h+N^{-1/3}/g )\right).
\end{multline}
\end{lemma}

\begin{proof}
Recall that \[
\bF(x;h,g) = \frac{1}{\sqrt{2\pi x}} \left( \exp\left(-\frac{(g-h)^2}{2x}\right)-\exp\left(-\frac{(g+h)^2}{2x}\right)\right).
\]
In this proof we use $C,c>0$ to denote large and small constants that can depend on $C_1, C_2$.

Without loss of generality we assume that $G\le H$.
When $X$ is small (i.e., bounded by any fixed constant), $x$ is of order at most $N^{-2/3}$, thus, the second factor in the right-hand side of \eqref{eq:A1bound} equals $1$. 
Note that $2^{-X-1}N^{1/3}|\sF(X;H,G)|$ is of order $N^{1/3}$.
If $|H-G|\le X$, then $\bF(x;h,g)>cN^{1/3}$, because $x$ is of order $N^{-2/3}$, $g-h$ is of order $N^{-1/3}$, and $h$ is of order at least $N^{-1/3}$. Hence, the conclusion follows by taking $C_3$ in the right-hand side of \eqref{eq:A1bound} to be large. On the other hand, if $|H-G|>X$, then $|\sF(X;H,G)|=0$, and the conclusion follows by taking an arbitrary $C_3>1$ in the right-hand side of \eqref{eq:A1bound}.

Below we assume that $X$ is large enough. Note that large $X$ implies large $N$, because we assumed $x< C_2$, and $|X-x N^{2/3}|<C_2$.

When $|G-H|<X$ we have via reflection principle:
\[
|\sF(X;H,G)| = {X \choose (X+H-G)/2} \left(1 - \frac{((X-H-G)/2)_{G+1}}{((X+H-G)/2+1)_{G+1}}\right).
\]
Stirling's approximation implies that when $|G-H|<X$, we have:
\begin{multline}  \label{eq:boundsitr}
\Bigg|\log\left( {X \choose (X+H-G)/2}  \right) - \log\left(\frac{2X}{\pi (X+H-G)(X-H+G)}\right)/2 \\ - X\log(X) + (X+H-G)\log((X+H-G)/2)/2 \\ + (X-H+G)\log((X-H+G)/2)/2 \Bigg| < \frac{C}{X+H-G}+\frac{C}{X-H+G}.
\end{multline}
For $(X+H-G)\log(X+H-G)/2+(X-H+G)\log(X-H+G)/2-X\log(X)$, it can be written as 
\[
X\left( (1+(H-G)/X)\log(1+(H-G)/X)/2 + (1-(H-G)/X)\log(1-(H-G)/X)/2 \right).
\]
By taking a Taylor expansion in the variable $(H-G)/X$, this is between $\frac{(G-H)^2}{2X}+C_4^{-1}\frac{(G-H)^4}{X^3}$ and $\frac{(G-H)^2}{2X}+C_4\frac{(G-H)^4}{X^3}$, for $C_4>0$ being a large constant that depends only on $C_1, C_2$.
Therefore, we have
\[
|\sF(X;H,G)| = 2^{X+1} \frac{1}{\sqrt{2\pi X}} \exp\left(-\frac{(G-H)^2}{2X}\right)
\left(1 - \frac{((X-H-G)/2)_{G+1}}{((X+H-G)/2+1)_{G+1}}\right)\Err,
\]
where
\begin{multline}  \label{eq:bounderrterm}
 \sqrt{\frac{X^2}{X^2-(G-H)^2}}\exp\left(-C_4\frac{(G-H)^4}{X^3}-\frac{C}{X-|G-H|}\right) \\ <\Err < \sqrt{\frac{X^2}{X^2-(G-H)^2}}\exp\left(-C_4^{-1}
\frac{(G-H)^4}{X^3} + \frac{C}{X-|G-H|}\right)
\end{multline}
and the factor $\sqrt{\frac{X^2}{X^2-(G-H)^2}}$ is from the second term in the first line of \eqref{eq:boundsitr}.
We further have \begin{equation} \label{eq:bdgh}
\left|\frac{(g-h)^2}{2x}-\frac{(G-H)^2}{2X}\right| <CN^{-1/3}\frac{|g-h|}{x}+ CN^{-2/3}\frac{(g-h)^2}{x^2}.
\end{equation}
\noindent\textbf{Case 1.} When $x<N^{-1/5}$,
we have
\begin{equation}  \label{eq:onefive}
    1\wedge (N^{-3/5}x^{-3} + N^{-1/10} + N^{-1/3}/h+N^{-1/3}/g ) = 1.
\end{equation}
Thus, in this case, it suffices to show that  $2^{-X-1}N^{1/3}|\sF(X;H,G)| <C\bF(x;h,g)$.
When $|G-H|>X$ this is obvious, since then $|\sF(X;H,G)| = 0$.
When $|G-H|=X$, necessarily $|\sF(X;H,G)| = 1$, and $||g-h|-N^{1/3}x|<CN^{-1/3}$. Thus, since $h,g>C_1N^{-1/3}$, we have $gh/x>c$, and
\[
\bF(x;h,g) = \frac{1}{\sqrt{2\pi x}} \exp\left(-\frac{(g-h)^2}{2x}\right)(1-\exp(-2gh/x)) > c N^{1/3}X^{-1/2}\exp(-X/2),
\]
and \eqref{eq:A1bound} follows.

Below we assume that $|G-H|<X$.
Note that since $X$ is large enough, there is some $C_5>0$ depending only on $C_1, C_2$, such that $N^{-4/3}\frac{(g-h)^4}{x^4}<C_5\frac{(G-H)^4}{X^3}$.
Then we have that the right-hand side of \eqref{eq:bdgh} is bounded by
\[
C + (2C_4C_5)^{-1}N^{-4/3}\frac{(g-h)^4}{x^4}<C+(2C_4)^{-1}\frac{(G-H)^4}{X^3},
\]
by taking $C$ large and using $C_4$ from \eqref{eq:bounderrterm}.
By combining this with the upper bound of $\Err$ in \eqref{eq:bounderrterm}, we have
\[
\exp\left(-\frac{(G-H)^2}{2X}\right) \Err 
< C\exp\left(-\frac{(g-h)^2}{2x}\right) \sqrt{\frac{X^2}{X^2-(G-H)^2}} \exp\left(-(2C_4)^{-1}
\frac{(G-H)^4}{X^3} \right).
\]
Thus, we further have
\begin{equation}  \label{eq:a1bound001}
\exp\left(-\frac{(G-H)^2}{2X}\right) \Err < C\exp\left(-\frac{(g-h)^2}{2x}\right).
\end{equation}

On the other hand, we have
\begin{multline*}
1-\frac{((X-H-G)/2)_{G+1}}{((X+H-G)/2+1)_{G+1}}
= \sum_{i=0}^H  \frac{((X-H-G)/2+i+1)_{G+1}-((X-H-G)/2+i)_{G+1}}{((X+H-G)/2+1)_{G+1}} \\
= \sum_{i=0}^H \frac{(G+1)((X-H-G)/2+i)_G}{((X+H-G)/2+1)_{G+1}}  
\le \sum_{i=0}^H \frac{G+1}{X} = \frac{(G+1)(H+1)}{X},
\end{multline*}
so
\begin{equation}  \label{eq:a1bound002}
1-\frac{((X-H-G)/2)_{G+1}}{((X+H-G)/2+1)_{G+1}} \le 1 \wedge \frac{(G+1)(H+1)}{X} <C(1-\exp(-2gh/x)).
\end{equation}
Therefore, by putting together \eqref{eq:a1bound001} and \eqref{eq:a1bound002}, we have
\[
2^{-X-1}N^{1/3}|\sF(X;H,G)| <C\exp\left(-\frac{(g-h)^2}{2x}\right)(1-\exp(-2gh/x)) = C\bF(x;h,g),
\]
and the conclusion holds due to \eqref{eq:onefive} in this case.

\noindent\textbf{Case 2.} When $x\ge N^{-1/5}$, necessarily $|G-H|\le G+H<X$, since $h,g <C_2N^{1/100}$.
Now we have
\[\frac{2^{-X-1}N^{1/3}|\sF(X;H,G)|}{\bF(x;h,g)} = \frac{\frac{N^{1/3}}{\sqrt{2\pi X}} \exp\left(-\frac{(G-H)^2}{2X}\right)}{\frac{1}{\sqrt{2\pi x}} \exp\left(-\frac{(g-h)^2}{2x}\right)}\cdot \frac{1 - \frac{((X-H-G)/2)_{G+1}}{((X+H-G)/2+1)_{G+1}}}{1-\exp\left(-\frac{2gh}{x}\right)} \cdot \Err.\]
We are going to show that each factor in the right-hand side is close to $1$.

The right-hand side of \eqref{eq:bdgh} is smaller than $CN^{-1/10}$, and $|N^{1/3}\sqrt{x/X}-1|<CN^{-2/3+1/5}=CN^{-7/15}$.
Thus
\begin{equation}  \label{eq:86bd}
\left|\frac{\frac{N^{1/3}}{\sqrt{2\pi X}} \exp\left(-\frac{(G-H)^2}{2X}\right)}{\frac{1}{\sqrt{2\pi x}} \exp\left(-\frac{(g-h)^2}{2x}\right)} - 1\right| < CN^{-1/10}.
\end{equation}

On the other hand, we have
\begin{multline*}
\left| \log\left(\frac{((X-H-G)/2)_{G+1}}{((X+H-G)/2+1)_{G+1}} \right) + \frac{2gh}{x} \right| 
= \left| \sum_{i=0}^G \log\left(1-\frac{2H+2}{X+H-G+2i+2}\right) + \frac{2gh}{x}\right|
\\<
\left|\frac{2(H+1)(G+1)}{X} - \frac{2gh}{x}\right|
+C (H+1)^2(G+1)/X^2 
< CN^{-1/3}(gh^2/x^2+gh/x^2+g/x+h/x),
\end{multline*}
where in the first inequality we used that $\left|\log\left(1-\frac{2H+2}{X+H-G+2i+2}\right) + \frac{2H+2}{X}\right|<\frac{C(H+1)^2}{X^2}$ for each $i\in\llbracket 0, G\rrbracket$.
As a result,
\begin{multline}   \label{eq:resut}
 \left|\frac{1 - \frac{((X-H-G)/2)_{G+1}}{((X+H-G)/2+1)_{G+1}}}{1-\exp\left(-\frac{2gh}{x}\right)} -1 \right| < \frac{\left(\exp(CN^{-1/3}(gh^2/x^2+gh/x^2+g/x+h/x))-1\right)\exp\left(-\frac{2gh}{x}\right)}{1-\exp\left(-\frac{2gh}{x}\right)} \\ < C(N^{-1/10}+N^{-1/3}/h+N^{-1/3}/g).
\end{multline}
Here the last inequality is due to the following reason. Since $x\ge N^{-1/5}$ and $h,g <C_2N^{1/100}$, we have $N^{-1/3}(g/x+h/x)<N^{-1/100}$.
If $gh/x<N^{1/100}$, we can upper bound the second expression by $\frac{CN^{-1/3}(gh^2/x^2+gh/x^2+g/x+h/x)}{gh/x}$; otherwise, we get the upper bound $C\exp(-cN^{1/50})$.

Besides, using that $x\ge N^{-1/5}$ and $h,g <C_2N^{1/100}$, we have
\[
1\le \frac{(G-H)^4}{X^3} < C N^{-2/3} \frac{(g-h)^4}{x^3} < CN^{-3/5}x^{-3}, \quad \frac{1}{X-|G-H|} < N^{-1/10},
\]
\[
\sqrt{\frac{X^2}{X^2-(G-H)^2}} < 1+C\frac{(G-H)^2}{X^2} < 1+CN^{-1/10}.
\]
Therefore, by \eqref{eq:bounderrterm}, we obtain $|\Err-1| < C(N^{-3/5}x^{-3}+N^{-1/10})$.

Finally, the bound on $|\Err-1|$ with \eqref{eq:86bd} and \eqref{eq:resut} imply that 
\[
\left|\frac{2^{-X-1}N^{1/3}|\sF(X;H,G)|}{\bF(x;h,g)} - 1\right| 
< C(N^{-3/5}x^{-3}+N^{-1/10}+N^{-1/3}/h+N^{-1/3}/g).    \qedhere
\]
\end{proof}

\begin{lemma}  \label{lem:sFog}
For any $C_1,C_2>0$, there exists a constant $C_3=C_3(C_1,C_2)>0$, such that for each $N=1,2,\dots$, and  $x,h\in\mathbb R_{\ge 0}$, and  $X, H\in \Z_{\ge 0}$, satisfying: $C_1 N^{-2/3}<x<C_2$ and $C_1 N^{-1/3}< h < C_2 N^{1/100}$, and $|X-xN^{2/3}|, |H-hN^{1/3}|<C_2$, and $X+H$ is even, we have
\[
\left|2^{-X-1}N^{2/3}|\sF(X;H,0)| - \bF_0(x;h) \right| \\ <  C_3\bF_0(x;h)\left(1\wedge (N^{-3/5}x^{-3} + N^{-1/10} + N^{-1/3}/h )\right).
\]
\end{lemma}
\begin{proof}
Recall that
\[
\bF_0(x;h) = \frac{2h}{\sqrt{2\pi x^3}} \exp\left(-\frac{h^2}{2x}\right).
\]
Similarly to the previous proof, we assume that $X$ is large enough, and when $H<X$ we have
\[
|\sF(X;H,0)| = 2^{X+1} \frac{2H+2}{X+H+2}\cdot \frac{1}{\sqrt{2\pi X}} \exp\left(-\frac{H^2}{2X}\right)\Theta_0,
\]
where
\[
\sqrt{\frac{X^2}{X^2-H^2}}\exp\left(-C_4\frac{H^4}{X^3}-\frac{C}{X-H}\right) \\ <\Theta_0 < \sqrt{\frac{X^2}{X^2-H^2}}\exp\left(-C_4^{-1}
\frac{H^4}{X^3} + \frac{C}{X-H}\right),
\]
for some large $C_4>0$, depending only on $C_1, C_2$.
We further have
\begin{equation} \label{eq:bdgh2}
\left|\frac{h^2}{2x}-\frac{H^2}{2X}\right| <CN^{-1/3}\frac{h}{x}+ CN^{-2/3}\frac{h^2}{x^2}.
\end{equation}
\noindent\textbf{Case 1.} When $x<N^{-1/5}$, it suffices to bound  $2^{-X-1}N^{2/3}|\sF(X;H,0)|$ by $\bF_0(x;h)$, up to a constant factor.
When $H>X$ this is obvious, since then $|\sF(X;H,0)| = 0$. When $H=X$, necessarily $|\sF(X;H,0)| = 1$, and $|h-N^{1/3}x|<CN^{-1/3}$. Thus, since $h>C_1N^{-1/3}$, we have $gh/x>c$, 
\[
\bF_0(x;h) > c N^{2/3}HX^{-3/2}\exp(-X/2),
\]
and the above stated bound also follows.

Below we assume that $H<X$.
The right-hand side of \eqref{eq:bdgh2} is smaller than
$C + (2C_4)^{-1}
\frac{H^4}{X^3}$,
by taking $C$ large.
By plugging this into the upper bound for $\Theta_0$, we get that $2^{-X-1}N^{1/3}|\sF(X;H,0)| < C\bF_0(x;h)$,
and the conclusion holds.

\noindent\textbf{Case 2.} When $x\ge N^{-1/5}$, necessarily $H<X$ since $h<C_2N^{1/100}$.
The right-hand side of \eqref{eq:bdgh2} is bounded by $CN^{-1/10}$.
We have $|\Theta_0-1| < C(N^{-2/3}x^{-3}+N^{-1/10})$ in this case, and $\left|\frac{2H}{X+H+2}\cdot\frac{xN^{1/3}}{2h}-1\right| < C(N^{-1/10}+N^{-1/3}/h)$.
Thus, the conclusion also holds.
\end{proof}
We also record without proof the following asymptotics of Calatan numbers.
\begin{lemma}  \label{lem:sfoo}
For any $C_1,C_2>0$, there exists a constant $C_3=C_3(C_1,C_2)>0$, such that for each $N=1,2,\dots$, and  $x\in\mathbb R_{\ge 0}$, and  $X\in \Z_{\ge 0}$, satisfying: $C_1 N^{-2/3}<x<C_2$ and $|X-xN^{2/3}|<C_2$, and $X$ is even, we have
\[
\left|2^{-X-1}N|\sF(X;0,0)| - \bF_{0,0}(x) \right| <  C_3N^{-2/3}x^{-1}\bF_{0,0}(x).
\]
\end{lemma}

\subsection{Coupling between walks and Brownian bridges}  \label{ssec:couplgwbb}
We next establish that a uniformly random element from $\sF(X;H,G)$ is close to a corresponding Brownian bridge.
For this, we need the following statement.
\begin{lemma}  \label{lem:couple}
Take any $X, H, G, H', G'\in \Z_{\ge 0}$, with both $X+H+G$ and $X+H'+G'$ even, and $\mathfrak{f}, \mathfrak{f}':\llbracket 0, X\rrbracket\to \Z_{\ge 0}$, uniformly random from $\sF(X;H,G)$ and $\sF(X;H',G')$ respectively. We can couple them so that almost surely, $|\mathfrak{f}-\mathfrak{f}'|\le \max\{|G-G'|, |H-H'|\}$ on $\llbracket 0, X\rrbracket$.

Take any $x, h, g, h', g'>0$, and Brownian bridges $B$ and $B'$ on $[0,x]$ with $B(0)=h$, $B(x)=g$, conditional on $B>0$; and $B'(0)=h'$, $B'(x)=g'$, conditional on $B'>0$. We can couple them so that almost surely, $|B-B'|\le \max\{|g-g'|, |h-h'|\}$ on $[0,x]$.
\end{lemma}
This statement is in a similar spirit as \Cref{lem:mon}, and can also be proved using Monte-Carlo Markov chains. 
See Lemmas 2.6 and 2.7 in \cite{CH} and their proofs.

Our main coupling statement is as follows.
\begin{lemma}  \label{lem:coupleeb}
For any $C_1>0$ and $C_2>1$, there is a constant $C_3=C_3(C_1,C_2)>0$, such that the following is true for all $N\in \Z_+$.
Take any $C_1N^{-2/3}<x<(C_2/2)^{10}$ and $0\le h,g < (C_2-x^{1/10})N^{1/100}$, and $X, H, G\in \Z_{\ge 0}$, such that $|X-xN^{2/3}|, |H-hN^{1/3}|, |G-gN^{1/3}|<C_2$, and $X+H+G$ is even.
Take $\mathfrak{f}:\llbracket 0,X\rrbracket\to \Z_{\ge 0}$ uniformly random from $\sF(X;H,G)$, and a Brownian bridge $B$ on $[0,x]$ with $B(0)=h$, $B(x)=g$, conditional on $B\ge 0$. Then we can couple $\mathfrak{f}$ and $B$ together, such that with probability $>1-C_3\log(X)XN^{-7/10}$, we have $|\mathfrak{f}(Y)-N^{1/3}B(y)|<C_3\log(N)N^{3/10}+C_3\log(X)N^{1/10}$ for all $Y\in\llbracket 0, X\rrbracket$ and $y\in [0, x]$ with $|Y-yN^{2/3}|<1$.
\end{lemma}
\begin{proof}
In this proof we use $C,c>0$ to denote constants that may depend on $C_1, C_2$. 
We shall take $C_3$ to be large enough depending on $C_1, C_2$ and all $C,c$, and we assume without loss of generality that $N$ is large enough depending on $C_1, C_2, C_3$ and all $C,c$.

We prove by an induction in $X$.

For the base case, we consider $X<N^{3/5}$. 
Take any $U, W \in \llbracket 0, X\rrbracket$, $U<W$, and let $V=\lfloor (U+W)/2\rfloor$ or $V=\lceil (U+W)/2\rceil$.
Similarly to the proof of \Cref{cla}, we have that 
\[
\PP\left( |\mathfrak{f}(U)+\mathfrak{f}(W)-2\mathfrak{f}(V)| > \frac{1}{10}(W-U)^{1/2} \log(N) \right) < C\exp(-c(\log(N))^2).
\]
Now we take $U, W$ to be $\lfloor a2^{-b}X\rfloor, \lfloor (a+1)2^{-b}X\rfloor$, for any $b\in\Z_{\ge 0}$ and $a\in \llbracket 0, 2^b-1\rrbracket$, such that $U<W$.
By a union bound, we have that with probability $>1-C\exp(-c(\log(N))^2)$, for all such $U, V, W$ we have $|\mathfrak{f}(U)+\mathfrak{f}(W)-2\mathfrak{f}(V)| \le 0.1(W-U)^{1/2} \log(N)$.
This, in particular, implies that 
\begin{equation}   \label{eq:EYuniformbound}
\PP[|\mathfrak{f}(Y)-H-(G-H)Y/X|<X^{1/2}\log(N), \; \forall Y\in\llbracket 0,X\rrbracket] >1-C\exp(-c(\log(N))^2).
\end{equation}
Similarly we have that 
\begin{equation}    \label{eq:EYuniformboundB}
\PP[|B(y)-h-(g-h)y/x|<x^{1/2}\log(N), \; \forall y\in\llbracket 0,x\rrbracket] >1-C\exp(-c(\log(N))^2).    
\end{equation}
Therefore the statement holds.

For the induction step, now we assume that $X\ge N^{3/5}$, and that the conclusion holds for any smaller $X$.
The idea is to first couple the random variables $\mathfrak{f}(\lfloor X/2\rfloor)$ and $N^{1/3}B(\lfloor X/2\rfloor N^{-2/3})$ and then couple distributions conditional on the values of these variables.

We first deduce that
\begin{equation}  \label{eq:ndcbod}
\PP\left(N^{-1/3}\mathfrak{f}(\lfloor X/2\rfloor )\ge \left(C_2-\frac{99}{100}x^{1/10}\right)N^{1/100}\right) <CN^{-1/10},
\end{equation}
when both $h,g < (C_2-x^{1/10})N^{1/100}$.
By \Cref{lem:mon} it suffices to consider the case where $H=\lceil (C_2-x^{1/10})N^{1/3+1/100}+C_2\rceil$, and $G=H$ or $H+1$ (depending on the parity of $X$).
In this case we can estimate the number of walks in $\sF(X;H,G)$ that is $>\big(C_2-\frac{99}{100}x^{1/10}\big)N^{1/3+1/100}$ at $\lfloor X/2\rfloor$, using \Cref{lem:countwk}(i), and get an upper bound of $C2^X\exp(-cx^{1/5}N^{1/50}) < C2^X\exp(-cN^{1/150})$;
and by reflection principle we can lower bound $|\sF(X;H,G)|$ by $cX^{-1/2}2^X$.
Then their ratio is $<CN^{-1/10}$.

We now couple $\mathfrak{f}(\lfloor X/2\rfloor)$ and $N^{1/3}B(\lfloor X/2\rfloor N^{-2/3})$ so that with probability $>1-CN^{-1/10}$, they differ by $\le N^{1/10}$.

We take $\mathfrak{f}'$  uniformly sampled from $\sF(X;H+\lfloor N^{1/10}/3\rfloor,G+\lfloor N^{1/10}/3\rfloor)$, and a Brownian bridge $B'$ on $[0,x]$ with $B'(0)=h+N^{-7/30}/3$, $B'(x)=g+N^{-7/30}/3$, conditional on $B'>0$.
Take any $i\in\Z_{\ge 0}$ such that $X+H+\lfloor N^{1/10}/3\rfloor+\lfloor X/2\rfloor+i$ is even, and $i<C_2N^{1/3+1/100}$.
The probability $\PP[\mathfrak{f}'(\lfloor X/2\rfloor)=i]$ equals
\[
\frac{|\sF(\lfloor X/2\rfloor; H+\lfloor N^{1/10}/3\rfloor,i)|\cdot |\sF(\lceil X/2\rceil; i, G+\lfloor N^{1/10}/3\rfloor)|}{|\sF(X;H+\lfloor N^{1/10}/3\rfloor,G+\lfloor N^{1/10}/3\rfloor)|}.
\]
By \Cref{lem:A1}, this equals
\begin{equation}   \label{eq:Fratio}
2N^{-1/3}\frac{\bF(\lfloor X/2\rfloor N^{-2/3};h+N^{-7/30},\tau) \bF(x-\lfloor X/2\rfloor N^{-2/3};\tau,g+N^{-7/30}) }{\bF(x;h+N^{-7/30},g+N^{-7/30})} \Phi,    
\end{equation}
for any $\tau > 0$ with $|i-\tau N^{1/3}|\le 2$, 
where $\Phi$ is a number satisfying
\begin{equation}   \label{eq:Phibound}
|\Phi-1| < C(1\wedge (N^{-1/10} + N^{-1/3}/\tau)).
\end{equation}
On the other hand, the probability density of $B(\lfloor X/2\rfloor N^{-2/3})$ at $\tau$ is precisely the ratio in \eqref{eq:Fratio}.
Thus, if we consider $\mathfrak{f}'(\lfloor X/2\rfloor)$ and $2\lfloor N^{1/3}B'(\lfloor X/2\rfloor N^{-2/3})/2 \rfloor$ or $2\lfloor N^{1/3}B'(\lfloor X/2\rfloor N^{-2/3})/2 \rfloor+1$ (depending on the parity), the total variation distance between their distributions is bounded by
\begin{multline}  \label{eq:pplusinte}
\PP[\mathfrak{f}'\lfloor X/2\rfloor \ge C_2N^{1/3+1/100}] \\+ \int_0^{C_2N^{1/100}+N^{-1/3}} \frac{\bF(\lfloor X/2\rfloor N^{-2/3};h+N^{-7/30},\tau) \bF(x-\lfloor X/2\rfloor N^{-2/3};\tau,g+N^{-7/30}) }{\bF(x;h+N^{-7/30},g+N^{-7/30})} |\Phi-1| \d \tau.
\end{multline}
By \eqref{eq:ndcbod}, the first term is bounded by $CN^{-1/10}$.
As for the integral, for $\tau<x^{1/2}N^{-1/10}$ we can bound the ratio by $Cx^{-1/2}$, and $|\Phi-1|$ by $C$; so the integral for such $\tau$ is bounded by $CN^{-1/10}$. For $\tau>x^{1/2}N^{-1/10}$, we have $\tau>cN^{-2/15}$, and therefore, $|\Phi-1|<CN^{-1/10}$, so we get an upper bound of $CN^{-1/10}$ times the integral of the ratio.
Note that since the ratio is precisely the probability density of $B(\lfloor X/2\rfloor N^{-2/3})$ at $\tau$, its integral for $\tau$ from $0$ to $\infty$ equals $1$.
Thus, the integral in \eqref{eq:pplusinte} is bounded by $CN^{-1/10}$.

In other words, we can couple $\mathfrak{f}'$ and $B'$ so that with probability $>1-CN^{-1/10}$, we have
\[|\mathfrak{f}'(\lfloor X/2\rfloor)-N^{1/3}B'(\lfloor X/2\rfloor N^{-2/3})|\le 2.\]
We then couple $\mathfrak{f}$ and $\mathfrak{f}'$ using \Cref{lem:couple} so that almost surely they differ by $\le N^{1/10}/3$; similarly we couple $B$ and $B'$ using \Cref{lem:couple} so that almost surely they differ by $\le N^{-7/30}/3$.
Therefore, we have $|\mathfrak{f}(\lfloor X/2\rfloor)-N^{1/3}B(\lfloor X/2\rfloor N^{-2/3})|\le N^{1/10}$ with probability $>1-CN^{-1/10}$.

We then use the induction hypothesis for $\mathfrak{f}$ on $\llbracket 0, \lfloor X/2\rfloor\rrbracket$ and $\llbracket \lfloor X/2\rfloor, X\rrbracket$ respectively, conditional on $\mathfrak{f}(\lfloor X/2\rfloor)$ and $B(\lfloor X/2\rfloor N^{-2/3})$.
Note that here we only couple them so that (with probability $>1-CN^{-1/10}$) they differ by $\le N^{1/10}$, rather than $\le C_2$.
Therefore to use the induction hypothesis, we need one more procedure.

We consider Brownian bridges $B_-$ and $B_+$ on $[0,\lfloor X/2\rfloor N^{-2/3}]$ and $[\lfloor X/2\rfloor N^{-2/3}, x]$ respectively, with boundary conditions $B_-(0)=h$, $B_-(\lfloor X/2\rfloor N^{-2/3})=N^{-1/3}\mathfrak{f}(\lfloor X/2\rfloor)$, and  $B_+(\lfloor X/2\rfloor N^{-2/3})=N^{-1/3}\mathfrak{f}(\lfloor X/2\rfloor)$, $B_+(x)=g$, conditional on $B_-, B_+>0$.
From the bound \eqref{eq:ndcbod} and that $(C_2-\frac{99}{100}x^{1/10})N^{1/100}$ smaller than either $(C_2-(\lfloor X/2\rfloor N^{-2/3})^{1/10})N^{1/100}$ or $(C_2-(x-\lfloor X/2\rfloor N^{-2/3})^{1/10})N^{1/100}$, we have that
with probability $>1-CN^{-1/10}$,
we can couple $B_-$ and $B_+$ with $\mathfrak{f}$ on the two intervals respectively, using the induction hypothesis. 
We then couple $B_-$ and $B_+$ with $B$ on $[0,\lfloor X/2\rfloor N^{-2/3}]$ and $[\lfloor X/2\rfloor N^{-2/3}, x]$ respectively, conditional on $B(\lfloor X/2\rfloor N^{-2/3})$, using \Cref{lem:couple}.
Under the event that $|\mathfrak{f}(\lfloor X/2\rfloor)-N^{1/3}B(\lfloor X/2\rfloor N^{-2/3})|\le N^{1/10}$, we have that $B$ and $B_-$, $B_+$ differ by $\le N^{-7/30}$ almost surely.

Finally, using that $CN^{-1/10}+C_3\log(\lceil  X/2 \rceil)XN^{-7/10}<C_3\log(X)XN^{-7/10}$, and that $C_3\log(\lceil  X/2 \rceil)N^{1/10} + N^{1/10}<C_3\log(X)N^{1/10}$, we see that the coupling satisfies the requirements.
Thus the conclusion follows by the principle of induction.
\end{proof}

\begin{remark}
The bound in \Cref{lem:coupleeb} is stated in this way for the convenience of the proof by induction. It is straightforward to see that, under the same coupling, with probability $>1-N^{-1/31}$ we have $|N^{-1/3} \mathfrak{f}(Y)-B(y)|<N^{-1/31}$ for all $Y\in\llbracket 0, X\rrbracket$ and $y\in [0, x]$ with $|Y-yN^{2/3}|<1$.     
\end{remark}

Finally, we use this coupling to prove \Cref{lem:RW1}.

\begin{proof}[Proof of \Cref{lem:RW1}]
In this proof we use $C,c>0$ to denote constants that may depend on $C_1, C_2$, and we take $C_3$ to be large enough depending on all $C,c$

We first prove (1).
Take $\mathfrak{f}:\llbracket 0,X\rrbracket\to \Z_{\ge 0}$ uniformly random from $\sF(X;H,G)$, and a Brownian bridge $B$ on $[0,x]$ with $B(0)=h$, $B(x)=g$, conditional on $B>0$; and couple them using \Cref{lem:coupleeb}.

Note that when $H\ge G$, each up-step of the path has a matching down-step, except for $H-G$ up-steps. Thus, we have
\[
\sum_{\substack{t\in \llbracket 1, X \rrbracket: \\ \mathfrak{f}(t)=\mathfrak{f}(t-1)-1 }}\log\left(1+\frac{2\mathfrak{f}(t)}{\beta N}\right) = 
\sum_{\substack{t\in \llbracket 1, X \rrbracket: \\ \mathfrak{f}(t)=\mathfrak{f}(t-1)+1 }}\log\left(1+\frac{2\mathfrak{f}(t)-2}{\beta N}\right) + \sum_{i=G}^{H-1}\log\left(1+\frac{2i}{\beta N}\right);
\]
and when $H<G$, we have
\[
\sum_{\substack{t\in \llbracket 1, X \rrbracket: \\ \mathfrak{f}(t)=\mathfrak{f}(t-1)-1 }}\log\left(1+\frac{2\mathfrak{f}(t)}{\beta N}\right) + \sum_{i=H}^{G-1}\log\left(1+\frac{2i}{\beta N}\right)= 
\sum_{\substack{t\in \llbracket 1, X \rrbracket: \\ \mathfrak{f}(t)=\mathfrak{f}(t-1)+1 }}\log\left(1+\frac{2\mathfrak{f}(t)-2}{\beta N}\right).
\]
Thus (since $G, H < CN^{1/3}$) we have
\[
\left|\sum_{\substack{t\in \llbracket 1, X \rrbracket: \\ \mathfrak{f}(t)=\mathfrak{f}(t-1)-1 }}\log\left(1+\frac{2\mathfrak{f}(t)}{\beta N}\right)- \sum_{\substack{t\in \llbracket 1, X \rrbracket: \\ \mathfrak{f}(t)=\mathfrak{f}(t-1)+1 }}\log\left(1+\frac{2\mathfrak{f}(t)-2}{\beta N}\right)\right| < CN^{-1/3}.
\]
Let $\cE$ be the event where $\max \mathfrak{f}, \max B<N^{1/2}$, and $|N^{-1/3} \mathfrak{f}(t)-B(y)|<N^{-1/31}$ for all $t\in\llbracket 0, X\rrbracket$ and $y\in [0, x]$ with $|t-yN^{2/3}|<1$.
Then under $\cE$, for each $t\in\llbracket 1, X\rrbracket$ with $t<xN^{2/3}$, we have
\[
\left| \log\left(1+\frac{2\mathfrak{f}(t)+\iota}{\beta N}\right) - 2\beta^{-1}\int_{(t-1)N^{-2/3}}^{tN^{-2/3}} B(y) \d y \right| < CN^{-2/3-1/31},
\]
where $\iota=0$ or $-2$.
Thus, under $\cE$, we have
\[
\left| \sum_{\substack{t\in \llbracket 1, X \rrbracket: \\ \mathfrak{f}(t)=\mathfrak{f}(t-1)-1 }}\log\left(1+\frac{2\mathfrak{f}(t)}{\beta N}\right) - \beta^{-1}\int_0^x B(y) \d y \right| < CN^{-1/31},
\]
which implies that
\begin{equation}  \label{eq:pfa91}
\left|\prod_{\substack{t\in \llbracket 1, X \rrbracket: \\ \mathfrak{f}(t)=\mathfrak{f}(t-1)-1 }}\left(1+\frac{2\mathfrak{f}(t)}{\beta N}\right)
- \exp\left( \beta^{-1}\int_0^x B(y) \d y \right) \right| < CN^{-1/31}\exp\left( \beta^{-1}\int_0^x B(y) \d y \right).
\end{equation}
On the other hand, using the same arguments as deriving \eqref{eq:EYuniformbound} and \eqref{eq:EYuniformboundB} in the proof of \Cref{lem:coupleeb}, we have $\PP(\max \mathfrak{f} > mN^{1/3})<C\exp(-cm^2)$ and $\PP(\max B > m)<C\exp(-cm^2)$, for any $m>0$.
Thus,
\begin{equation}  \label{eq:pfa92}
\E \left[\exp\left( \beta^{-1}\int_0^x B(y) \d y \right) \right]< \int_0^\infty  C\exp(-cm^2) \exp(Cm) \d m < C,
\end{equation}
and $\PP[\cE^c]<CN^{-1/40}$ (where $\cE^c$ denotes the complement of $\cE$), and
\[
\E \left[\don[\cE^c]\prod_{\substack{t\in \llbracket 1, X \rrbracket: \\ \mathfrak{f}(t)=\mathfrak{f}(t-1)-1 }}\left(1+\frac{2\mathfrak{f}(t)}{\beta N}\right) \right]< \int_0^\infty  C(\exp(-cm^2)\wedge \PP[\cE^c]) \exp(Cm) \d m < C\log(N)N^{-1/31},
\]
\[
\E \left[\don[\cE^c]\exp\left( \beta^{-1}\int_0^x B(y) \d y \right) \right]< \int_0^\infty  C(\exp(-cm^2)\wedge \PP[\cE^c]) \exp(Cm) \d m < C\log(N)N^{-1/31}.
\]
These two bounds together with \eqref{eq:pfa91} and \eqref{eq:pfa92} imply that
\[
\left|\E\prod_{\substack{t\in \llbracket 1, X \rrbracket: \\ \mathfrak{f}(t)=\mathfrak{f}(t-1)-1 }}\left(1+\frac{2\mathfrak{f}(t)}{\beta N}\right)
- \E \exp\left( \beta^{-1}\int_0^x B(y) \d y \right) \right| < CN^{-1/40}.
\]
This leads to
\[
\left|
N^{1/3}I(X;H,G) - \frac{\bI(x;h,g)\cdot 2^{-X-1}N^{1/3}|\sF(X;H,G)|}{\bF(x;h,g)}
\right| < CN^{-1/40}\bF(x;h,g),
\]
using that $2^{-X-1}N^{1/3}|\sF(X;H,G)| < C\bF(x;h,g)$, which can be deduced from \Cref{lem:A1}.
Applying \Cref{lem:A1} again to the left-hand side, and using that $\bI(x;h,g)< C\bF(x;h,g)$, the estimate (1) follows.

For (2), (3), and (4), we similarly apply \Cref{lem:coupleeb} and \Cref{lem:sFog,lem:sfoo} to deduce that
\[
\left|N^{2/3}I(X;H,0)-\frac{\bI_0(x;h)\cdot 2^{-X-1}N^{2/3}|\sF(X;H,0)|}{\bF_0(x;h)}\right| < CN^{-1/40}\bF_0(x;h),
\]
and
\begin{multline*}
\left|N^{2/3}I^+(X;H,G)-\frac{\exp(-\beta^{-1}xh) \bI_0(x;g-h)\cdot 2^{-X-1}N^{2/3}|\sF^+(X;H,G)|}{\bF_0(x;g-h)}\right| \\ < CN^{-1/40}\bF_0(x;g-h),
\end{multline*}
and
\[
\left|NI^+(X;H,H)-\frac{\exp(-\beta^{-1}xh)\bI^\epsilon_0(x;h)\cdot 2^{-X-1}N|\sF^+(X;H,H)|}{\bF_{0,0}(x)}\right| < CN^{-1/40}\bF_{0,0}(x).
\]
These together with \Cref{lem:sFog,lem:sfoo} imply (2)--(4) of \Cref{lem:RW1}.
\end{proof}

\bibliographystyle{alpha}
\bibliography{bibliography}

\end{document}